\documentclass[a4paper,12pt,reqno,twoside,openright]{memoir}
\usepackage[utf8]{inputenc}
\usepackage[french]{babel}
\usepackage{amsmath}
\usepackage{amssymb}
\usepackage{amsthm}
\usepackage[all]{xy}

\usepackage[T1]{fontenc}

\usepackage{txfonts}

\usepackage{tikz-cd}
\usetikzlibrary{cd}
\usepackage{color}
\usepackage{enumerate}
\usepackage{a4wide}
\usepackage{rotating}
\usepackage{longtable}
\usepackage{imakeidx}
\makeindex[intoc]
\usepackage{hyperref}

\makeatletter

\DeclareMathOperator{\GL}{GL}

\DeclareMathOperator{\SL}{SL}

\DeclareMathOperator{\SO}{SO}

\DeclareMathOperator{\upU}{U}

\newcommand{\0}{\textbf{0}}
\renewcommand{\1}{\textbf{1}}

\DeclareMathOperator{\tr}{Tr}
\DeclareMathOperator{\ad}{ad}
\DeclareMathOperator{\Ad}{Ad}

\DeclareMathOperator{\diag}{diag}

\renewcommand\Re{\operatorname{Re}}

\renewcommand{\min}{{\textup{min}}}

\DeclareMathOperator{\horn}{Horn}
\DeclareMathOperator{\LR}{LR}
\DeclareMathOperator{\sing}{Singulier}
\DeclareMathOperator{\vect}{Vect}
\DeclareMathOperator{\e}{e}
\DeclareMathOperator{\s}{s}
\DeclareMathOperator{\herm}{Herm}
\DeclareMathOperator{\cc}{c}

\DeclareMathOperator{\sym}{Sym}

\theoremstyle{plain}
\newtheorem{theorem}{Th\'eor\`eme}[section]
\newtheorem{prop}[theorem]{Proposition}
\newtheorem{lem}[theorem]{Lemme}
\newtheorem{coro}[theorem]{Corollaire}

\newtheorem{hyp}[theorem]{Hypoth\`ese}

\theoremstyle{definition}
\newtheorem{definition}[theorem]{D\'efinition}
\newtheorem{exemple}[theorem]{Exemple}
\newtheorem{rem}[theorem]{Remarque}

\numberwithin{equation}{section}

\newcommand{\agot}{\ensuremath{\mathfrak{a}}}
\newcommand{\bgot}{\ensuremath{\mathfrak{b}}}
\newcommand{\cgot}{\ensuremath{\mathfrak{c}}}

\newcommand{\ggot}{\ensuremath{\mathfrak{g}}}
\newcommand{\hgot}{\ensuremath{\mathfrak{h}}}
\newcommand{\kgot}{\ensuremath{\mathfrak{k}}}

\newcommand{\mgot}{\ensuremath{\mathfrak{m}}}
\newcommand{\pgot}{\ensuremath{\mathfrak{p}}}
\newcommand{\sogot}{\ensuremath{\mathfrak{so}}}
\newcommand{\qgot}{\ensuremath{\mathfrak{q}}}
\newcommand{\ngot}{\ensuremath{\mathfrak{n}}}

\newcommand{\sgot}{\ensuremath{\mathfrak{s}}}
\newcommand{\tgot}{\ensuremath{\mathfrak{t}}}
\newcommand{\ugot}{\ensuremath{\mathfrak{u}}}
\newcommand{\Ngot}{\ensuremath{\mathfrak{N}}}
\newcommand{\Rgot}{\ensuremath{\mathfrak{R}}}
\newcommand{\Sgot}{\ensuremath{\mathfrak{S}}}
\newcommand{\Xgot}{\ensuremath{\mathfrak{X}}}

\newcommand{\glgot}{\ensuremath{\mathfrak{gl}}}

\newcommand{\Acal}{\ensuremath{\mathcal{A}}}
\newcommand{\Bcal}{\ensuremath{\mathcal{B}}}
\newcommand{\Ccal}{\ensuremath{\mathcal{C}}}

\newcommand{\Ecal}{\ensuremath{\mathcal{E}}}

\newcommand{\Fcal}{\ensuremath{\mathcal{F}}}

\newcommand{\Lcal}{\ensuremath{\mathcal{L}}}

\newcommand{\Ocal}{\ensuremath{\mathcal{O}}}
\newcommand{\Pcal}{\ensuremath{\mathcal{P}}}
\newcommand{\Rcal}{\ensuremath{\mathcal{R}}}
\newcommand{\Scal}{\ensuremath{\mathcal{S}}}
\newcommand{\Tcal}{\ensuremath{\mathcal{T}}}
\newcommand{\Xcal}{\ensuremath{\mathcal{X}}}
\newcommand{\Ycal}{\ensuremath{\mathcal{Y}}}
\newcommand{\Zcal}{\ensuremath{\mathcal{Z}}}
\newcommand{\Ucal}{\ensuremath{\mathcal{U}}}
\newcommand{\Vcal}{\ensuremath{\mathcal{V}}}
\newcommand{\Wcal}{\ensuremath{\mathcal{W}}}

\newcommand{\J}{\ensuremath{\mathbb{J}}}
\newcommand{\Z}{\ensuremath{\mathbb{Z}}}
\newcommand{\C}{\ensuremath{\mathbb{C}}}
\newcommand{\B}{\ensuremath{\mathbb{B}}}
\newcommand{\F}{\ensuremath{\mathbb{F}}}
\newcommand{\G}{\ensuremath{\mathbb{G}}}

\newcommand{\Q}{\ensuremath{\mathbb{Q}}}
\newcommand{\R}{\ensuremath{\mathbb{R}}}
\newcommand{\N}{\ensuremath{\mathbb{N}}}
\newcommand{\Pbb}{\ensuremath{\mathbb{P}}}

\newcommand{\croc}{\ensuremath{\hookrightarrow}}

\newcommand{\T}{\ensuremath{\hbox{T}}}

\newcommand{\crit}{\ensuremath{\mathrm{Crit}}}

\newcommand{\mini}{\ensuremath{\mathrm{s}}}

\newcommand{\reg}{\operatorname{\mathrm{reg}}}

\newcommand{\leng}{\operatorname{\mathrm{long}}}

\newcommand{\adiag}{{\ensuremath{A}}}

\newcommand{\infrp}{\operatorname{\mathrm{inf-RP}}}

\newcommand{\rp}{\operatorname{\mathrm{RP}}}

\def \tG {\widetilde{G}}
\def \tK {\widetilde{K}}
\def \tU {\widetilde{U}}
\def \tB {\widetilde{B}}
\def \tP {\widetilde{P}}
\def \tPbb {\widetilde{\Pbb}}

\def \tW {\widetilde{W}}

\def \tlambda {\tilde{\lambda}}

\def \tggot {\tilde{\ggot}}

\def \tagot {\tilde{\agot}}
\def \tkgot {\tilde{\kgot}}
\def \tpgot {\tilde{\pgot}}
\def \tugot {\tilde{\ugot}}
\def \ttgot {\tilde{\tgot}}
\def \tngot {\tilde{\ngot}}
\def \tNgot {\tilde{\Ngot}}

\def \tsgot {\tilde{\sgot}}
\def \tXgot {\tilde{\Xgot}}

\def \tRgot {\widetilde{\Rgot}}

\def \ta {\tilde{a}}

\def \tk {\tilde{k}}
\def \tp {\tilde{p}}
\def \tu {\tilde{u}}
\def \tv {\tilde{v}}
\def \tg {\tilde{g}}
\def \tgamma {\tilde{\gamma}}

\def \hgamma {\widehat{\gamma}}
\def \hC {\widehat{C}}
\def \hx {\hat{x}}
\def \hpi {\hat{\pi}}

\def \deg {\mathrm{deg}}

\def \tFcal{\widetilde{\mathcal{F}}}
\def \tX {\widetilde{X}}
\def \tY {\widetilde{Y}}

\def \tT {\widetilde{T}}
\def \tP {\widetilde{P}}
\def \teta{\widetilde{\eta}}
\def \txi {\tilde{\xi}}

\def \tw {\tilde{w}}

\def \tzeta {\tilde{\zeta}}

\title{Polytopes de Kirwan dans le cadre r\'eel }

\author{Paul-Emile Paradan}
\date{Décembre 2025}

\begin{document}
	
	\maketitle
	

\vspace*{\fill}

\small\textit{Address:} University of Montpellier, CNRS, IMAG

\small\textit{E-Mail:} \texttt{paul-emile.paradan@umontpellier.fr}

\small\textit{2020 Mathematics Subject Classification:} 

\small\textit{Keywords: involution, convexity, moment map, Schubert calculus, symmetric spaces} 

\cleardoublepage

\setcounter{tocdepth}{1}

\tableofcontents


\chapter*{Introduction}

Voici le prototype de r\'esultat de convexit\'e que nous discuterons dans cette monographie. \`A chaque matrice hermitienne $X$ de taille $n\times n$, 
on peut associer son spectre $\e(X)=(\e_1,\ldots, \e_n)$ qui est form\'e de ses valeurs propres rang\'ees en ordre d\'ecroissant. 
La partie r\'eelle de $X$, not\'ee $\Re(X)$, est une matrice r\'eelle sym\'etrique. Il est naturel de vouloir comparer les spectres respectifs des matrices 
$X$ et $\mathrm{Re}(X)$. Pour cela, on pose 
$$
\mathcal{E}_{\mathrm{I}}(n):=\Big\{(\e(X),\e(\Re(X))),\ X\  \mbox{matrice\ hermitienne } n\times n\Big\} .
$$
Nous montrerons que les $\mathcal{E}_{\mathrm{I}}(n)$ sont des cones convexes poly\'edraux qui admettent une description r\'ecursive. 
Par exemple, lorsque $n=4$, nous verrons que 
$(x,y)\in(\R^4)^2$ appartient \`a $\mathcal{E}_{\mathrm{I}}(4)$ si et seulement si les 17 in\'egalit\'es suivantes sont satisfaites: 
\begin{equation*}
\boxed{
\begin{array}{rcl}
x_1 \geq  x_1\geq  x_3\geq  x_4 & , &y_1 \geq  y_1\geq  y_3\geq  y_4\\
x_1+x_2+x_3 +x_4& = & y_1+y_2+y_3+y_4\\
x_1 &\geq & y_1 \\
x_2 &\geq & y_3 \\
x_1+x_2& \geq & y_1+y_2\\
x_1+x_3 & \geq & y_2+y_3\\
x_1+x_3 & \geq & y_1+y_4\\
x_1+x_4& \geq& y_3+y_4\\
x_2+x_3 & \geq & y_3+y_4\\
x_4 &\leq & y_4 \\
x_3 &\leq & y_2 .\\
\end{array}
}
\end{equation*}

\medskip

Le probl\`eme pr\'ec\'edent se trouve avoir de fortes connexions avec le c\^one de Horn: 
$$
\horn(n):=\Big\{(\e(X),\e(Y),\e(X+Y)),\ X,Y\  \mbox{matrices\ hermitiennes } n\times n\Big\} .
$$
Ce dernier a fait l'objet de nombreuses \'etudes qui remontent aux in\'egalit\'es de Weyl obtenues en 1912. En 1999, 
les travaux conjugu\'es de Knutson-Tao \cite{Knutson-Tao-99} et Klyacko \cite{Klyachko} ont montr\'e que les 
$\horn(n)$ sont des cones convexes poly\'edraux qui admettent 
une description r\'ecursive: ceci r\'epondait \`a une conjecture formul\'ee par Horn \cite{Horn} dans les ann\'ees 60.  Belkale \cite{Belkale01} 
et Knutson-Tao-Woodward \cite{KTW04} ont ensuite 
donn\'e le nombre minimal d'in\'egalit\'es d\'ecrivant $\horn(n)$. Par exemple, pour $\horn(4)$, 52 in\'egalit\'es sont n\'ecessaires.

On voit assez facilement que pour tout $(x,y)\in \mathcal{E}_{\mathrm{I}}(n)$, le triplet $(x,x,2y)$ appartient \`a $\horn(n)$. Le point cl\'e ici 
est que l'implication r\'eciproque soit vraie: ce r\'esultat est une cons\'equence du th\'eor\`eme de O'Shea-Sjamaar \cite{OSS} sur lequel nous reviendrons 
plus tard. Nous savons donc que $\mathcal{E}_{\mathrm{I}}(n)$ s'identifie avec l'intersection du cone $\horn(n)$ avec le sous-espace 
$V=\{(x,x,2y), x,y\in\R^n\}$. Les \'equations de $\horn(n)$, lorsqu'on les restreint \`a $V$, fournissent un syst\`eme redondant d'in\'egalit\'es pour d\'ecrire 
$\mathcal{E}_{\mathrm{I}}(n)$. Un des principaux objects de cette monographie est d'expliquer comment fournir un syst\`eme r\'eduit d'in\'egalit\'es pour 
$\mathcal{E}_{\mathrm{I}}(n)$ qui se r\'evelera \^etre lui m\^eme r\'ecursif. Par exemple, pour d\'ecrire $\horn(4)$, nous avons besoin de 52 in\'egalit\'es, 
tandis que $\mathcal{E}_{\mathrm{I}}(4)$ n\'ecessite seulement 17 in\'egalit\'es.

L'\'etude du cone $\horn(n)$ rel\`eve de la g\'eom\'etrie hamiltonienne, en particulier parce que l'application $(X,Y)\mapsto X+Y$ est le 
prototype d'application moment. A contrario, l'application $X\mapsto \Re(X)$ est une application moment que l'on a tordu au moyen 
de l'application de conjugaison $X\mapsto \overline{X}$: on appelera cela une application moment r\'eelle.

\medskip

Dans ce travail,  les questions de convexit\'e sont abord\'ees \`a travers l'\'etude des actions hamiltoniennes d'un groupe de Lie compact sur une 
vari\'et\'e symplectique en pr\'esence d'une involution sur le groupe et d'une involution antisymplectique sur la vari\'et\'e. L'ensemble des points 
fixes de l'involution sur la vari\'et\'e est un sous-vari\'et\'e Lagrangienne. Dans ce contexte, nous pouvons consid\'erer deux polytopes de Kirwan: 
celui de la vari\'et\'e initiale et celui associ\'e \`a la sous-vari\'et\'e Lagrangienne form\'ee des points fixes de l'involution. Ce dernier sera appel\'e 
le polytope de Kirwan r\'eel. Il y a 25 ans, O'Shea et Sjamaar obtenaient un r\'esultat remarquable que l'on peut r\'esumer ainsi: le polytope de Kirwan r\'eel correspond \`a l'intersection du polytope de Kirwan usuel avec le sous-espace form\'e des vecteurs anti-invariant pour l'involution \cite{OSS}.

Le premier objectif de ce m\'emoire de pr\'esenter une m\'ethode g\'eom\'etrique pour obtenir les in\'egalit\'es qui d\'ecrivent le polytope de Kirwan r\'eel. 
Cette méthode repose sur l'adaptation du concept de  \og paires bien couvrantes \fg\, au cadre avec involution : ce concept a été introduit par Nicolas Ressayre dans le cadre projectif \cite{Ressayre10}, puis adapté au cadre de Kähler \cite{pep-ressayre}. Ensuite, nous appliquerons ce résultat pour décrire les cônes convexes associés aux représentations isotropes des espaces symétriques.

\medskip

Nous concluons cette introduction en détaillant chaque chapitre qui compose cette monographie.

\medskip

Dans le chapitre \ref{chapter: convexity}, nous commençons par rappeler les résultats classiques de convexité en géométrie hamiltonienne. 
Le premier concerne le théorème de convexité de Kirwan \cite{Kir84b}, tandis que le second est le théorème de convexité obtenu dans le cadre 
avec involution par O'Shea et Sjamaar \cite{OSS} . Nous terminons ce chapitre en introduisant deux exemples remarquables de cones convexes 
polyédraux : $\LR(U,\tU)$ et $\horn_\pgot(K,\tK)$.

\medskip

Dans le second chapitre, nous nous plaçons dans le cadre des variétés de Kähler $U$-hamiltoniennes. Nous rappelons comment dans ce contexte les 
\emph{paires de Ressayre} permettent de décrire le polytope de Kirwan. Dans le cadre avec involution, nous  obtenons le premier résultat de cet article : 
une description du polytope moment réel au moyen de \emph{paires de Ressayre réelles} (voir le théorème \ref{th:real-ressayre-pairs}).

\medskip

Le chapitre 3 est dédié à la preuve du théorème \ref{th:real-ressayre-pairs}. Comme dans le cas classique, la stratification à la Kirwan-Ness 
des variétés de Kähler $U$-hamiltoniennes est l'outil principal pour obtenir les inégalités qui décrivent le polytope moment réel. Pour cela, nous analysons le comportement des strates par rapport aux involutions (sur le groupe et sur la variété).

\medskip

Depuis les travaux de Klyachko \cite{Klyachko}, Berenstein-Sjamaar \cite{Berenstein-Sjamaar} et Belkale \cite{Belkale01,Belkale06}, on sait que le calcul de Schubert est un outil remarquable pour paramétrer les inégalités des polytopes de Kirwan. Nous passons en revue cet outil dans la quatrième section, ainsi que l'amélioration proposée par Belkale-Kumar \cite{BK06}, connue sous le nom de Lévi-mobilité. Dans le cadre avec involution, nous décrivons les cellules de Bruhat qui admettent un point fixé par l'involution anti-holomorphe.

\medskip

Dans certains exemples que nous détaillerons, les inégalités des polytopes de Kirwan réels sont paramétrées au moyen du calcul de Schubert dans les variétés à $2$-drapeaux $\F(r,n-r;n)$. Grâce aux formules multiplicatives de Ressayre-Richmond, nous pourrons paramétrer ces inégalités au moyen de deux conditions cohomologiques sur des grassmaniennes. Le chapitre 5 est dédié à ces formules multiplicatives, et nous proposerons une version raffinée des résultats de 
Ressayre et Richmond \cite{Res11b,Ric12,RR11}.

\medskip

Dans le sixième chapitre, nous rappelons la description générale des cône classiques $\LR(U,\tU)$ obtenues dans \cite{Berenstein-Sjamaar,BK06,Ressayre10}. 
Nous détaillerons en particulier les deux cas particuliers $\horn(n)$ et $\LR(m,n)$.

\medskip

Au chapitre 7, nous présentons le deuxième résultat principal de cette monographie, qui consiste en la description des cônes $\horn_\pgot(K,\tK)$ associés aux représentations isotropes des espaces symétriques. Nous verrons la grande similitude avec la description des cônes classiques $\LR(U,\tU)$ : nous travaillons avec des groupes de Weyl restreints au lieu des groupes de Weyl standard, et nous considérons des équations cohomologiques portées par des variétés de Schubert qui sont globalement invariantes sous l'involution anti-holomorphe. Lorsque $\tK=K^s$, notre description affine celle obtenue précédemment par Kapovich-Leeb-Millson \cite{KLM-memoir-08}.

\medskip

Les quatre chapitres suivants sont consacrés à la description de certains cônes  $\horn_\pgot(K,\tK)$. 
Voici un bref aperçu des cas étudiés. 

\medskip

Dans le chapitre 8, nous obtenons:

\begin{itemize}
\item  une description récursive du cône $\mathcal{E}_{\mathrm{I}}(n)$ qui a été défini au début de cette introduction.
\item  une description récursive du cône $\mathcal{E}_{\mathrm{II}}(n)$: celui-ci compare le spectre d'une matrice symétrique $A\in\sym(2n)$ avec celui de 
la matrice hermitienne $\pi(A)\in\herm(n)$ associée. Ce dernier cône a été abordé dans des travaux récents de Chenciner, Chenciner-Pérez et Heckman-Zhao 
\cite{Chenciner, Chenciner-Perez,Heckman-Zhao}.
\end{itemize}

\medskip

Au chapitre 9, nous étudions le cône de Horn singulier, noté $\sing(p,q)$, qui est défini comme l'ensemble des triplets de spectres singuliers 
$(\s(A),\s(B),\s(A+B))$ où $A,B$ parcourent les matrices complexes de taille $p\times q$. Nous avons ici une première utilisation des formules 
multiplicatives de Ressayre-Richmond, qui nous permettrent de paramétrer les inégalités de $\sing(p,q)$ au moyen de deux conditions cohomologiques 
sur des grassmanniennes.

\medskip

Au chapitre 10, nous étudions deux types de cônes qui comparent le spectre singulier de matrices rectangulaires avec le 
spectre de matrices hermitiennes. Nous obtenons une description 

\begin{itemize}
\item  du cône $\Acal(p,q)$: celui-ci compare le spectre d'une matrice hermitienne $X$ avec le spectre singulier de son bloc hors diagonale $\pi_{p,q}(X)$.
\item  du cône $\Bcal(n)$: celui-ci compare le spectre singulier d'une matrice carrée complexe $X$ avec le spectre de la matrice hermitienne $\frac{1}{2}(X+X^*)$.
\end{itemize}

Fomin-Fulton-Li-Poon ont obtenu, à l'aide d'une méthode ad hoc, une description plus fine du cône $\Acal(p,q)$, avec un système d'inégalités qui est contenu 
dans celui que je décris \cite{FFLP}. Il est raisonnable de conjecturer que ces deux systèmes d'inégalités sont identiques et qu'ils fournissent une description minimale du cône $\Acal(p,q)$.

\medskip

Au chapitre 11, nous décrivons deux types de cônes qui comparent le spectre singulier d'une matrice carrée $X$ avec les spectres singuliers des matrices 
rectangulaires formées à partir de $X$.

\chapter{Th\'eor\`emes de convexit\'e pour les vari\'et\'es hamiltoniennes}\label{chapter: convexity}

Dans cette section, nous rappelons les th\'eor\`emes de convexit\'e associ\'es aux actions hamiltoniennes de groupes de Lie compacts.

Tout au long de ce chapitre, $U\subset \tU$ désigne des groupes de Lie compacts et connexes, et $\ugot\subset \tugot$ désigne leurs algèbres de Lie. 
Nous notons $\pi_{\tugot,\ugot}:\tugot^*\to \ugot^*$ la projection associ\'ee.

Pour une introduction à la notion d'actions hamiltoniennes, le lecteur peut consulter les références suivantes \cite{Benoist02,GSj05,CdS08,Berline-Vergne11,Woodward11,Audin12}.

\section{Actions hamiltoniennes}

Consid\'erons une action du groupe $U$ sur une vari\'et\'e symplectique $(M,\Omega)$. On note $X\in\ugot\mapsto X_M \in \mbox{Vect}(M)$ l'action infinit\'esimale de l'alg\`ebre de Lie: ici $X_M(m)=-X\cdot m =\frac{d}{dt}\vert_{t=0}e^{-tX}m$ est le champ de vecteurs engendr\'e par $X\in \ugot$. On suppose que l'action du groupe $U$ laisse la $2$-forme symplectique invariante, ou, de manière équivalente, nous exigeons que $\forall X\in \ugot$, la $1$-forme $\iota(X_M)\Omega$ est ferm\'ee.

L'action de $U$ sur $(M,\Omega)$ est dite hamiltonienne s'il existe une application \'equivariante $\Phi_\ugot : M\to \ugot^*$ satisfaisant la relation 
\begin{equation}\label{eq:hamiltonien}
d\langle\Phi_\ugot, X\rangle=-\iota(X_M)\Omega, \quad\mbox{pour tout}\  X\in \ugot. 
\end{equation}
L'application $\Phi_\ugot$ est appel\'ee une application moment.

\begin{exemple}
Si $U$ agit de fa\c{c}on unitaire sur un espace vectoriel hermitien $(V,\langle\cdot,\cdot\rangle_V)$, l'application moment relative \`a 
la structure symplectique $\Omega=-\mbox{Im}(\langle\cdot,\cdot\rangle_V)$ est d\'efinie par la relation suivante: 
$\langle \Phi_\ugot(v),X\rangle=\frac{i}{2} \langle X v,v\rangle_V$, $\forall v\in V,\forall X\in \ugot$.
\end{exemple}

\begin{exemple}
Reprenons le cas d'une action unitaire de $U$ sur un espace vectoriel hermitien $(V,\langle\cdot,\cdot\rangle_V)$. La vari\'et\'e projective 
$\Pbb V$ est munie de la $2$-forme symplectique de Fubini-Study $\Omega_{\Pbb V}$, qui se d\'efinit explicitement comme suit. 
Soient $[v]\in \Pbb V$ et l'ouvert $\Ucal_{[v]}:=\{ [y]\in \Pbb V, \langle v,y\rangle_V\neq 0\}$. Soit $(e_1,\ldots, e_{\ell -1})$ 
une base orthonorm\'ee de l'espace vectoriel orthogonal à $[v]$. Sur l'ouvert $\Ucal_{[v]}$,  la $2$-forme $\Omega_{\Pbb V}$ est d\'efinie par la relation
$$
\Omega_{\Pbb V}= i\sum_{k=1}^{\ell-1} dz_k\wedge d\overline{z}_k, 
$$
par rapport aux coordonnées locales $(z_1,\ldots, z_{\ell -1})\mapsto \left[\frac{v}{\|v\|}+\sum_{k=1}^{\ell-1}z_k e_k\right]\in \Ucal_{[v]}$.

L'application moment $\Phi_{\Pbb V} : \Pbb V\to \ugot^*$ associ\'ee \`a l'action de $U$ sur $\Pbb V$, est alors d\'efinie par la relation 
$$
\langle\Phi_{\Pbb V}([y]),\lambda \rangle= i\,\frac{\langle \lambda y,y\rangle_V}{\langle y,y\rangle_V}.
$$
\end{exemple}

\begin{exemple}
Soit $\Ocal\subset \tugot^*$ une orbite coadjointe de $\tU$. La structure symplectique de Kirillov-Kostant-Souriau sur $\Ocal$ est d\'efinie par la relation 
\begin{equation}
\Omega_\Ocal(X_\Ocal,Y_\Ocal)\vert_\xi=\langle\xi,[X,Y]\rangle,\quad \forall \xi\in\Ocal,\quad \forall X,Y\in\tugot.
\end{equation}
Dans ce cas, l'application moment $\Ocal \croc \tugot^*$ est l'inclusion. L'application moment $\Phi_{\ugot}:\Ocal\to \ugot^*$ associ\'ee \`a l'action du sous-groupe 
$U$ sur $\Ocal$ est \'egal \`a la compos\'ee de l'inclusion $\Ocal \croc \tugot^*$ avec la projection $\pi_{\tugot,\ugot}:\tugot^*\to \ugot^*$.
\end{exemple}

\begin{exemple} \label{exemple-O-lambda}
Consid\'erons le groupe unitaire $\upU_n$ et son alg\`ebre de Lie $\ugot(n)$ form\'ee des matrices $n\times n$ anti-hermitiennes. 
Notons $\herm(n)$ l'espace vectoriel des matrices hermitiennes. Utilisons l'isomorphisme $\upU_n$-équivariant 
\begin{equation}\label{eq:t}
\mathrm{t}: \herm(n)\to \ugot(n)^*,
\end{equation} 
d\'efini par la relation $\langle \mathrm{t}(X),Y\rangle= \mathrm{Im}(\mathrm{Tr}(XY))$, $\forall X\in \herm(n)$, $\forall Y\in \ugot(n)$. 
Fixons un $n$-uplet de r\'eels $\lambda=(\lambda_1\geq\ldots\geq \lambda_n)$, et notons $\Ocal_\lambda$ l'ensemble des matrices 
$n\times n$ hermitiennes de spectre \'egal \`a $\lambda$. \`A travers l'isomorphisme (\ref{eq:t}), $\Ocal_\lambda$ s'identifie avec une orbite coadjointe de 
$\upU_n$. La forme symplectique de Kirillov-Kostant-Souriau sur $\Ocal_\lambda$ est d\'efinie par la relation 
\begin{equation}\label{eq:2-forme-O-lambda}
\Omega_\lambda\vert_A([Y_1,A],[Y_2,A])= \mathrm{Im}\Big(\mathrm{Tr}(A[Y_1,Y_2])\Big),\quad \forall A\in \Ocal_\lambda,\quad \forall Y_1,Y_2\in\ugot(n).
\end{equation}
et l'application moment $\Phi:\Ocal_\lambda\to\ugot(n)^*$ est la compos\'ee de l'inclusion $\Ocal_\lambda\croc \herm(n)$ avec (\ref{eq:t}).
\end{exemple}

\begin{exemple}
Consid\'erons l'action de $U$ sur une vari\'et\'e $M$.  La projection naturelle $\mbox{p} : T^*M\to M$ 
permet de d\'efinir la $1$-forme de Liouville $\lambda$ par la relation 
$$
\lambda(x)= x\circ T \mbox{p}(x),\quad \forall x\in T^*M.
$$ 
La vari\'et\'e cotangente $T^*M$ admet une $2$-forme symplectique canonique, $-d\lambda$, et l'action de $U$ sur $T^*M$ est hamiltonienne. L'application moment associ\'ee 
$\Phi_\ugot :T^*M\to \ugot^*$ est d\'efinie par la relation suivante:
$$
\langle \Phi_\ugot(x),X\rangle = \langle x,X\cdot \mbox{p}(x)\rangle,\quad \forall x\in T^*M.
$$
\end{exemple}

\begin{exemple}\label{ex:U-tilde-U}
Consid\'erons l'action de $\tU\times U$ sur $\tU$ d\'efinie par $(\tg,g)\cdot a=\tg a g^{-1}$. L'action de $\tU\times U$ sur $T^*\tU$ est hamiltonienne. 
Si on identifie $T^* \tU$ \`a $\tU\times \tugot^*$ au moyen des translations \`a gauche, l'application moment 
$\Phi_{\tugot,\ugot}: \tU\times \tugot^*\to\tugot^*\times\ugot^*$ est d\'efinie par les relations
\begin{equation}\label{eq:Phi-u-u-prime}
\Phi_{\tugot,\ugot}(\tk,\txi)=\Big(\tk\txi,-\pi_{\tugot,\ugot}(\txi)\Big),\quad \forall \tk\in \tU,\forall \txi\in\tugot^*.
\end{equation}
\end{exemple}

\medskip

Reprenons le contexte d'une application moment $\Phi_\ugot$ associ\'ee \`a l'action hamiltonienne de $U$ sur $(M,\Omega)$. 
Soit $T$ un tore maximal du groupe $U$, d'alg\`ebre de Lie $\tgot$. Le dual $\tgot^*$ s'identifie avec les vecteurs $T$-invariants de $\ugot^*$. 
Fixons une chambre de Weyl $\tgot^*_+\subset \tgot^*$. Les orbites coadjointes de $U$ sont param\'etr\'ees par $\tgot^*_+$: 
l'application $\xi\in \tgot^*_+\longmapsto U\xi\in \ugot^*/U$ est une bijection. Ainsi les orbites de $U$ dans $\Phi_\ugot(M)$ 
sont param\'etr\'ees par l'ensemble 
\begin{equation}\label{eq:polytope-kirwan}
\Delta_\ugot(M):= \Phi_\ugot(M)\cap \tgot^*_+.
\end{equation}

L'un des question centrale de ce m\'emoire est la description de ces ensembles $\Delta_\ugot(M)$. On commence avec un r\'esultat remarquable de convexit\'e.

\begin{theorem}\label{theo:convexe}
\begin{itemize}
\item Supposons $M$ compacte. Alors 
\begin{enumerate}
\item Si $U$ est ab\'elien, $\Delta_\ugot(M)$ est \'egal \`a l'enveloppe convexe de l'ensemble fini $\Phi_\ugot(M^U)$.
\item Si $U$ est non-abélien, $\Delta_\ugot(M)$ est un polytope convexe.
\end{enumerate}
\item Supposons $M$ non-compacte. Si l'application moment $\Phi_\ugot: M\to\ugot^*$ est propre, alors $\Delta_\ugot(M)$ 
est un sous-ensemble fermé, convexe, et localement polyédral.
\end{itemize}
\end{theorem}

Ce th\'eor\`eme de convexit\'e a fait l'objet de nombreuses contributions. Pour l'action d'un tore maximal sur une orbite coadjointe, le 
th\'eor\`eme \ref{theo:convexe} a \'et\'e prouv\'e par Kostant \cite{Kos74} \'etendant des r\'esultats pr\'ec\'edents de Schur et Horn. 
Il a \'et\'e g\'en\'eralis\'e aux actions de sous-groupes par Heckman \cite{Heckman82}. Pour les actions de tores sur les vari\'et\'es hamiltoniennes, 
le th\'eor\`eme a \'et\'e prouv\'e ind\'ependamment par Atiyah \cite{Ati82} et par Guillemin et Sternberg \cite{GS82}. Dans le cas non ab\'elien, 
le th\'eor\`eme a \'et\'e démontré par Mumford \cite{Mumford84} et Guillemin et Sternberg \cite{GS84} pour les vari\'et\'es projectives. 
La premi\`ere preuve compl\`ete pour les actions hamiltoniennes de groupes non-ab\'eliens sur des vari\'et\'es symplectiques  
a \'et\'e obtenue par Kirwan \cite{Kir84b}. Plusieurs preuves alternatives du r\'esultat de Kirwan ont \'et\'e publi\'ees, 
voir \cite{CDM88,HNP94,Sjamaar98,LMTW}, qui traitent le cas o\`u l'application moment est propre.

Dans cette monographie, nous nous concentrons principalement sur les versions constructives du théorème de Kirwan, car l'un des inconvénients de ce résultat de convexité est qu'il ne fournit pas de description explicite de $\Delta_\ugot(M)$ lorsque le groupe $U$ est non abélien.

\section{Le cadre avec involution}

Soit $U$ un groupe de Lie compact connexe muni d'une involution $\sigma$, c'est-\`a- dire un automorphisme qui satisfait 
$\sigma^2=Id_G$. La d\'eriv\'ee de $\sigma$ induit des involutions sur l'alg\`ebre de Lie $\ugot$ et sur son dual $\ugot^*$, 
que nous d\'esignons tous deux par $\sigma$ par simplicit\'e. On note $K$ la composante connexe du sous-groupe $U^\sigma:=\{u\in U, \sigma(u)=u\}$.

Nous reprenons la d\'efinition donn\'ee par \cite{OSS}.

\begin{definition}Une vari\'et\'e $(U,\sigma)$-hamiltonienne est un quadruplet $(M,\Omega,\tau,\Phi_\ugot)$ constitu\'e d'une vari\'et\'e symplectique $(M,\Omega)$ \'equip\'ee 
d'une action symplectique du groupe $U$, d'une involution antisymplectique $\tau$ sur $M$ (c'est-\`a-dire un diff\'eomorphisme $\tau$ 
satisfaisant $\tau^2 = id_M$ et $\tau^*\Omega=-\Omega$), et d'une application moment $\Phi_\ugot$ pour l'action de $U$ sur $(M,\Omega)$. 
De plus, nous imposons les deux conditions suivantes :
\begin{align}
\Phi_\ugot(\tau(m))&=-\sigma\left(\Phi_\ugot(m)\right)\label{eq:condition-phi-tau}\\
\tau(um)&= \sigma(u)\tau(m),\label{eq:condition-sigma-tau}
\end{align}
pour tout $m\in M$ et $u\in U$. 

\end{definition}

\begin{exemple} \label{ex:O-lambda-involution}
Reprenons l'exemple \ref{exemple-O-lambda}: notons $\Ocal_\lambda$ l'ensemble des matrices 
$n\times n$ hermitiennes de spectre \'egal \`a $\lambda$, et $\Omega_\lambda$ la $2$-forme de Kirillov-Kostant-Souriau. 
On consid\`ere l'involution $\sigma(u)=\overline{u}$ sur $\upU_n$, et l'involution
$\tau(X)=\overline{X}$ sur $\herm(n)$. On constate que $\Ocal_\lambda$ est stable sous l'involution $\tau$, et on v\'erifie \`a l'aide de 
la formule (\ref{eq:2-forme-O-lambda}) que $\tau^*\Omega_\lambda=-\Omega_\lambda$. L'application moment 
$\Phi_\lambda:\Ocal_\lambda\to\ugot(n)^*$ v\'erifie la condition (\ref{eq:condition-phi-tau}) car $\mathrm{t}(\tau(X))=-\sigma(\mathrm{t}(X))$, 
$\forall X\in \herm(n)$. On voit ainsi que $(\Ocal_\lambda,\Omega_\lambda,\tau,\Phi_\lambda)$ est une vari\'et\'e
$(\upU_n,\sigma)$-Hamiltonienne.
\end{exemple}

Consid\'erons l'ensemble $M^\tau=\{m\in M, \tau(m)=m\}$ : si $M^\tau$ est {\em non vide}, la condition $\tau^*\Omega=-\Omega$ impose \`a $M^\tau$ 
d'\^etre un sous-vari\'et\'e Lagrangienne de $(M,\Omega)$, non n\'ecessairement connexe. La relation (\ref{eq:condition-sigma-tau}) montre 
que $M^\tau$ est stable sous l'action du sous-groupe $K\subset U$, et la condition (\ref{eq:condition-phi-tau}) donne 
$$
\Phi_\ugot(M^\tau)\subset (\ugot^*)^{-\sigma}.
$$

\begin{exemple} \label{ex:O-K-involution}
Consid\'erons une orbite coadjointe $\Ocal$ de $U$ telle que $\Ocal\cap (\ugot^*)^{-\sigma}\neq \emptyset$. Ainsi $\Ocal$ est laiss\'e stable par l'involution 
$\tau(\xi)=-\sigma(\xi),\forall \xi\in\ugot^*$. On v\'erifie alors que $(\Ocal,\Omega_\Ocal,\tau)$ est une vari\'et\'e $(U,\sigma)$-hamiltonienne. 
Soit $\Phi_{\kgot}:\Ocal\to (\ugot^*)^{\sigma}$ l'application moment relative à l'action de $K$ sur $\Ocal$. Alors 
$$
\Ocal^\tau=\Ocal\cap(\ugot^*)^{-\sigma}=\Phi_{\kgot}^{-1}(0)
$$
est connexe, \'egal \`a une $K$-orbite. En particulier, on a $K\xi=U^{-\sigma}\xi$ pour tout $\xi\in \Ocal\cap(\ugot^*)^{-\sigma}$.
\end{exemple}

Nous avons le r\'esultat important suivant \cite{OSS}.

\begin{theorem}[O'Shea-Sjamaar]\label{theo:OSS}
$$
\Phi_\ugot(M)\bigcap (\ugot^*)^{-\sigma}=\Phi_\ugot(M^\tau).
$$
\end{theorem}

\medskip

Reformulons ce th\'eor\`eme en termes de polytopes de Kirwan. Choisissons un tore maximal $T\subset K$, $\sigma$-invariant, et 
tel que le sous-espace vectoriel $\tgot^{-\sigma}=\{X\in \tgot,\sigma(X)=-X\}$ est de dimension maximale. 

\begin{definition}
Une chambre de Weyl $\tgot^*_+\subset \tgot^*$ est \emph{adaptée} à l'involution $\sigma$ si l'intersection 
$\tgot^*_+\cap (\tgot^*)^{-\sigma}$ param\`etre les $K$-orbites dans $(\ugot^*)^{-\sigma}$, c'est \`a dire si l'application
\begin{eqnarray*}
\tgot^*_+\cap(\tgot^*)^{-\sigma}&\longrightarrow&  (\ugot^*)^{-\sigma}/K \\
\xi&\longmapsto & K\xi
\end{eqnarray*}
est bijective. Il existe toujours des chambres de Weyl adaptées: voir, par exemple, l'appendice B dans \cite{OSS}.
\end{definition}

\begin{rem}\label{rem:orbite-symetrique}
Choisissons une chambre de Weyl $\tgot^*_+$ adaptée. Alors pour tout $\lambda\in \tgot^*_+$, l'intersection 
$U\lambda\cap (\ugot^*)^{-\sigma}$ est non-vide si et seulement si $\lambda\in \tgot^*_+\cap(\tgot^*)^{-\sigma}$, et dans ce cas 
$U\lambda\cap (\ugot^*)^{-\sigma}= K\lambda$.
\end{rem}

Avec un choix de chambre de Weyl adaptée, les ensembles
$$
\Delta_\ugot(M)=\Phi_\ugot(M)\cap \tgot^*_+\qquad \mbox{et}\qquad \Delta(M^\tau)=\Phi_\ugot(M^\tau)\cap \tgot^*_+
$$
param\`etrent respectivement les $U$-orbites dans $\Phi_\ugot(M)$ et les $K$-orbites dans $\Phi_\ugot(M^\tau)$.
Le th\'eor\`eme de O'Shea-Sjamaar peut \^etre reformul\'e au moyen de l'\'egalit\'e

\begin{equation}\label{eq:OSS}
\Delta_\ugot(M)\cap (\tgot^*)^{-\sigma}=\Delta(M^\tau).
\end{equation}

Le th\'eor\`eme de convexit\'e \ref{theo:convexe} nous dit que $\Delta_\ugot (M) $ est un sous-ensemble convexe, localement poly\'edral. 
Gr\^ace \`a (\ref{eq:OSS}), nous voyons que $\Delta(M^\tau)$ est également un sous-ensemble convexe localement poly\'edral: 
on l'appellera le polytope de Kirwan {\em r\'eel}. On verra \`a la section \ref{sec:description-polytope-reel} comment d\'ecrire 
$\Delta(M^\tau)$ en terme de paires de Ressayre réelles.

\begin{exemple} Reprenons l'exemple \ref{ex:O-lambda-involution}. \`A une matrice $n\times n$ hermitienne (ou sym\'etrique r\'eelle) $X$, 
on associe son spectre $\e(X)=(\e_1(X)\geq\ldots\geq \e_n(X))$. Fixons deux $n$-uplets $\mu=(\mu_1\geq\ldots\geq \mu_n)$ et 
$\lambda=(\lambda_1\geq\ldots\geq \lambda_n)$, et consid\'erons la $(\upU_n,\sigma)$-vari\'et\'e hamiltonienne $\Ocal_\lambda\times\Ocal_\mu$. 
L'application moment $\Phi_\ugot: \Ocal_\lambda\times\Ocal_\mu\to\ugot(n)^*$ et l'involution $\tau$ sont d\'efinis respectivement par 
$\Phi_\ugot(X,Y)=\mathrm{t}(X+Y)$, et $\tau(X,Y)=(\overline{X},\overline{Y})$.
Ici, la vari\'et\'e $(\Ocal_\lambda\times\Ocal_\mu)^\tau$ est \'egale \`a $\Ocal^\R_\lambda\times\Ocal^\R_\mu$ o\`u $\Ocal^\R_\nu$ d\'esigne 
l'ensemble des matrice $n\times n$ sym\'etriques r\'eelles $X$ telle que $\e(X)=\nu$. Dans cet exemple, les chambres de Weyl $\tgot^*_+$  et  
$\tgot^*_+\cap (\tgot^*)^{-\sigma}$ sont \'egales, et correspondent, \`a travers l'isomorphisme 
$\mathrm{t}: \herm(n)\to \ugot(n)^*$, \`a l'ensemble des matrices diagonales r\'eelles $\diag(a_1,\ldots,a_n)$ avec $(a_1\geq\ldots\geq a_n)$. 
Les polytopes de Kirwan sont ici 
\begin{align*}
\Delta_\ugot(\Ocal_\lambda\times\Ocal_\mu)&=\left\{\e(X+Y), X\in \Ocal_\lambda,\ Y\in \Ocal_\mu\right\},\\
\Delta((\Ocal_\lambda\times\Ocal_\mu)^\tau)&=\left\{\e(A+B), A\in \Ocal^\R_\lambda,\ B\in \Ocal^\R_\mu\right\}.
\end{align*}
Il est imm\'ediat que $\Delta((\Ocal_\lambda\times\Ocal_\mu)^\tau)\subset \Delta_\ugot(\Ocal_\lambda\times\Ocal_\mu)$, 
et le th\'eor\`eme d'O'Shea-Sjamaar permet de voir que 
$\Delta((\Ocal_\lambda\times\Ocal_\mu)^\tau)=\Delta_\ugot(\Ocal_\lambda\times\Ocal_\mu)$.
\end{exemple}

\section{Deux exemples remarquables}

\subsection{Le c\^one $\LR(U,\tU)$}

On consid\`ere le groupe $\tU$ munie de l'action de $\tU\times U$ d\'efinie par $(\tg,g)\cdot a=\tg a g^{-1}$. Alors l'action de $\tU\times U$ sur la vari\'et\'e 
cotangente $T^*\tU$ est hamiltonienne (voir l'exemple \ref{ex:U-tilde-U}). Si on identifie $T^* \tU$ \`a $\tU\times \tugot^*$ au moyen des translations \`a gauche, l'application moment 
$\Phi_{\tugot,\ugot}: \tU\times \tugot^*\to\tugot^*\times\ugot^*$ est d\'efinie par les relations
$$
\Phi_{\tugot,\ugot}(\tk,\txi)=\Big(\tk\txi,-\pi_{\tugot,\ugot}(\txi)\Big),\quad \forall \tk\in \tU,\forall \txi\in\tugot^*.
$$

On remarque ais\'ement que l'application $\Phi_{\tugot,\ugot}$ est propre, ainsi le th\'eor\`eme de convexit\'e \ref{theo:convexe} s'applique ici. 
Fixons deux chambres de Weyl $\ttgot_+^*\subset \ttgot^*$ et $\tgot_+^*\subset \tgot^*$.  
Le polytope de Kirwan associ\'e \`a l'application moment $\Phi_{\tugot,\ugot}$ est d\'efini par la relation 
$$
\Delta_{\tugot\times\ugot}(T^*\tU):=\Big\{(\txi,\xi)\in \ttgot_+^*\times \tgot_+^*,\ \exists \tk\in \tU, \ \xi =-\pi_{\tugot,\ugot}(\tk\txi)\Big\}.
$$

Afin de travailler avec les alg\`ebres de Lie (et non leurs duaux), on introduit un produit scalaire invariant sur $\tugot$: 
celui-ci induit les identification $\tugot\simeq \tugot^*$ et $\ugot\simeq \ugot^*$, 
et d'autre part, il permet de d\'efinir la projection orthogonale $\pi_{\tugot,\ugot}:\tugot\to\ugot$. 
Fixons deux chambres de Weyl $\ttgot_+\subset \ttgot$ et $\tgot_+\subset \tgot$.

\begin{definition}
On d\'efinit le c\^one
$$
\LR(U,\tU):=\Big\{(\ta,a)\in \ttgot_+\times \tgot_+,\ U a \subset \pi_{\tugot,\ugot}(\tU\ta)\Big\}.
$$
\end{definition}

Notons $W=N(T)/T$ le groupe de Weyl, et $w_o\in W$ son \'el\'ement le plus long. L'application $\xi\mapsto \xi^\vee=- w_o\xi$ est une 
involution de la chambre de Weyl 
$\tgot^*_+\simeq \tgot_+$. On voit ainsi que $(\txi,\xi)\in \Delta_{\tugot\times\ugot}(T^*\tU)$ si et seulement si $(\txi,\xi^\vee)\in \LR(U,\tU)$.

On expliquera \`a la section \ref{sec:LR-U-tilde-U} comment d\'ecrire le c\^one convexe poly\'edral $\LR(U,\tU)$ au moyen du calcul de Schubert.

\subsection{Les c\^ones $\horn_\pgot(K,\tK)$}

On suppose que le groupe $\tU$ est muni d'une involution $\sigma$, et que le sous-groupe $U$  est stable par rapport \`a $\sigma$. 
Notons respectivement $\tK$ et $K$  les composantes connexes de l'identit\'e des sous-groupes $\tU^{\sigma}$ et $U^\sigma$. 
On munit la vari\'et\'e $\tU\times U$-hamiltonienne $T^*\tU\simeq \tU\times \tugot^*$ de l'involution anti-symplectique 
$$
\tau(\tk,\txi)=(\sigma(\tk),-\sigma(\txi)).
$$
Ainsi la vari\'et\'e lagrangienne $(T^*\tU)^\tau$ est isomorphe \`a $\tU^{\sigma}\times (\tugot^*)^{-\sigma}$. 

Voici une conséquence du th\'eor\`eme d'O'Shea-Sjamaar \ref{theo:OSS}.

\begin{prop}\label{prop:OSS-orbites}
Pour tout  $(\txi,\xi)\in (\tugot^*)^{-\sigma}\times (\ugot^*)^{-\sigma}$, les \'enonc\'es suivants sont \'equivalents:
\begin{enumerate}
\item[$i)$] $U \xi \subset \pi_{\tugot,\ugot}(\tU\txi)$,
\item[$ii)$] $U^{-\sigma}\xi\subset \pi_{\tugot,\ugot}(\tU^{\sigma}\txi)$,
\item[$iii)$] $K\xi\subset \pi_{\tugot,\ugot}(\tK\txi)$.
\end{enumerate}
Ici, les \'enonc\'es $ii)$ et $iii)$ sont identiques car $K\xi=U^{-\sigma}\xi$ et $\tK\txi=\tU^{\sigma}\txi$.
\end{prop}

Choisissons deux chambres de Weyl $\ttgot_+^*\subset \ttgot^*$ et $\tgot_+^*\subset \tgot^*$ adapt\'ees \`a l'involution $\sigma$.

\begin{definition}
On d\'efinit le c\^one
$$
\horn_\pgot(K,\tK):=\Big\{(\tX,X)\in \ttgot_+\cap \ttgot^{-\sigma}  \times \tgot_+\cap \tgot^{-\sigma},\ K X \subset \pi_{\tugot,\ugot}(\tK\tX)\Big\}.
$$
\end{definition}

Le th\'eor\`eme d'O'Shea-Sjamaar permet de voir que 
$$
\horn_\pgot(K,\tK)=\LR(U,\tU)\bigcap \ttgot^{-\sigma}  \times \tgot^{-\sigma},
$$
et de montrer ainsi que $\horn_\pgot(K,\tK)$ est un c\^one convexe poly\'edral. L'un des principaux objectifs de ce m\'emoire est de param\'etrer, 
de mani\`ere minimale si possible, les in\'egalit\'es qui d\'efinissent ces c\^ones.

\chapter{Th\'eor\`emes de convexit\'e dans le cadre K\"ahl\'erien}

Les groupes compacts connexes $\tU\subset\tU$ admettent des complexifications $U_\C\subset \tU_\C$, c'est \`a dire des groupes r\'eductifs 
complexes contenant $U\subset \tU$ comme sous-groupes compacts maximaux, et 
dont les alg\`ebres de Lie sont $\ugot_\C=\ugot\oplus i\ugot$, et $\tugot_\C=\tugot\oplus i\tugot$.

\section{Variétés de Kähler $U$-Hamiltoniennes}

Nous consid\'erons dans cette section des vari\'et\'es de K\"{a}hler $U$-Hamiltonienne $(M,\Omega)$. Nous entendons par l\`a~ :

\begin{itemize}
\item Le groupe r\'eductif complexe $U_\C$ agit de mani\`ere holomorphe sur la vari\'et\'e complexe $M$.
\item La structure complexe de $M$, notée $\J$, est compatible avec la structure symplectique, i.e. $\Omega(\cdot,\J\cdot)$ d\'efinit une structure riemanienne sur $M$.
\item La forme de K\"{a}hler $\Omega$ est $U$-invariante.
\item Il existe une application moment $U$-\'equivariante $\Phi_\ugot : M\to \ugot^*$ satisfaisant la relation \ref{eq:hamiltonien}. 
\end{itemize}

\medskip

Dans les prochaines sections, on explique comment d\'ecrire le polytope de Kirwan $\Delta_{\ugot}(M)$ lorsque l'on travaille 
avec une vari\'et\'e de K\"{a}hler $U$-Hamiltonienne.

\medskip

\begin{exemple}[Orbites coajointes]\label{ex:orbite-coadjointe-Kahler}

Soit $\xi\in\tgot^*_+$. L'orbite coadjointe $U\xi$ admet une structure complexe compatible avec la structure symplectique KKS. 
La structure de variété complexe munie d'une action holomorphe de $U_\C$ peut être visualisée à travers les isomorphismes 
$$
U\xi\simeq U/U_\xi\simeq U_\C/P(\xi), 
$$
où\footnote{Ici, on voit $\xi$ comme un élément de l'algèbre de Lie $\tgot$, grâce à une identification équivariante $\ugot^*\simeq\ugot$.}  
\begin{equation}\label{eq:parabolique-def}
P(\xi)=\left\{g \in U_\C, \lim_{t\to\infty}\exp(-it\xi)g\exp(it\xi)\ \mbox{existe}\right\}
\end{equation} 
est un sous-groupe parabolique de $U_\C$. 

Dans la suite, on désigne par $(U\xi)^o$  l'orbite coadjointe $U\xi$ munie de la structure k\"ahlérienne opposée à la structure KKS. 
Comme variété complexe, on a l'isomorphisme $(U\xi)^o\simeq U_\C/P(-\xi)$.
\end{exemple}

\begin{exemple}[Exemple fondamental]\label{ex:U-tilde-C-Kahler}
Les translations \`a gauche fournissent des isomorphismes $T^*\tU \simeq \tU\times \tugot^*$ et $T\tU \simeq \tU\times \tugot$. 
On obtient alors un isomorphisme \'equivariant $\varphi_2: T^* \tU\simeq T\tU$ au moyen de l'identification 
$\tugot^*\simeq \tugot$ donn\'ee par un produit scalaire invariant. 

La groupe $\tU_\C$ admet une d\'ecomposition de Cartan, c'est \`a dire un diff\'eomorphisme
$\tU\times\tugot \longrightarrow \tU_\C$, $(\tk,\tX)\longmapsto \tk \exp(i\tX)$. Cela d\'efinit un isomorphisme $\varphi_1: T\tU\simeq \tU_\C$.  
Si nous utilisons l'isomorphisme \'equivariant
$$
\varphi_1\circ \varphi_2: T^* \tU\longrightarrow \tU_\C
$$ 
pour transporter la forme symplectique canonique de $T^*\tU$, la structure symplectique sur $\tU_\C$ qui en r\'esulte est 
invariante par rapport à l'action $\tU\times \tU$, et elle est compatible avec la structure complexe de $\tU_\C$ (voir [10], §3).

Conclusion: le groupe r\'eductif $\tU_\C$ est une vari\'et\'e de K\"{a}hler $\tU\times U$-Hamiltonienne, et l'application moment $\Phi_{\tugot,\ugot}:\tU_\C\to\tugot\times\ugot$ correspondante est d\'efinie par
$$
\Phi_{\tugot,\ugot}\left(\tk \exp(i\tX)\right)=\left(\tk \tX, -\pi_{\tugot,\ugot}(\tX)\right).
$$
\end{exemple}

\section{Paires de Ressayre}

Dans tout cette section, on consid\`ere une vari\'et\'e de K\"{a}hler $U$-Hamiltonienne $(M,\Omega)$ 
avec une application moment $\Phi_\ugot$ propre.

\subsection{\'El\'ements admissibles}

Introduisons la notion d'\'el\'ements admissibles associ\'es \`a l'action infinitésimale de $\ugot$ sur $M$. 

Soit $T\subset U$ un tore maximal d'algèbre de Lie $\tgot$. Le groupe $\mbox{Hom}(U(1),T)$ s'identifie naturellement au r\'eseau $\wedge:=\frac{1}{2\pi}\ker(\exp : \tgot\to T)$. 
Un vecteur $\gamma\in \tgot$ est appel\'e {\em rationnel} s'il appartient au $\Q$-espace vectoriel $\tgot_\Q$ engendr\'e par $\wedge$. 

Le groupe stabilisateur de $x\in M$ est not\'e $U_x:=\{k\in U ; k x=x\}$, et son alg\`ebre de Lie est not\'ee $\ugot_x$.

\begin{definition}\label{def:admissible-U-action} D\'efinissons $\dim_\ugot(\Xcal):=\min_{x\in\Xcal}\dim(\ugot_x)$ pour tout sous-ensemble $\Xcal\subset M$. 
Un \'el\'ement non nul $\gamma\in\tgot$ est appel\'e {\em admissible} par rapport à l'action $U\circlearrowright M$ 
si $\gamma$ est rationnel et si $\dim_\ugot(M^\gamma)-\dim_U(M)\in\{0,1\}$.
\end{definition}

\begin{rem}\label{rem:cas-generic-stabilizer}
$\dim_\ugot(M)=0$ lorsque le stabilisateur g\'en\'erique infinit\'esimal de l'action $U$ est r\'eduit à $\{0\}$. Dans ce cas, un \'el\'ement rationnel 
$\gamma$ est admissible si 
$\dim_\ugot(M^\gamma)\!=\!1$.
\end{rem}

\subsection{Paires de Ressayre infinit\'esimales}

Pour tout $x\in M$, l'action infinit\'esimale de $\ugot_\C$ sur $M$ d\'efinit 
une application $\C$-lin\'eaire $U_x$-\'equivariante
\begin{eqnarray}\label{eq:rho-m}
\rho_x:\ugot_\C & \longrightarrow & \T_x M\\ 
X & \longmapsto & X\cdot x .\nonumber
\end{eqnarray}

\begin{definition}
\begin{enumerate}
\item Pour tout $\gamma\in\ugot$, posons $\gamma_o:=\frac{1}{i}\gamma$.
\item Consid\'erons l'action lin\'eaire $\rho : U\to \GL(E)$ d'un groupe de Lie compact sur un espace vectoriel complexe $E$. 
Pour tout $(\gamma,a)\in\ugot\times\R$, nous d\'efinissons le sous-espace vectoriel $E^{\gamma_o=a}=\{v\in E, d\rho(\gamma_o)v= av\}$.  
Ainsi, pour tout $\gamma\in\ugot$, nous avons la d\'ecomposition suivante $E=E^{\gamma_o>0}\oplus E^{\gamma_o=0}\oplus E^{\gamma_o<0}$ 
o\`u $E^{\gamma_o>0}=\sum_{a>0}E^{\gamma_o=a}$, et $E^{\gamma_o<0}=\sum_{a<0}E^{\gamma_o=a}$.
\item Pour tout $\gamma\in\ugot$, on pose $\tr(\gamma_o \circlearrowright E^{\gamma_o>0}):=\sum_{a>0}a\,\dim(E^{\gamma_o=a})$.
\item Soit $\varphi:E\to F$ un morphisme lin\'eaire \'equivariant entre deux $U$-modules. Pour tout $\gamma\in\ugot$, 
l'application $\varphi$ se sp\'ecialise en une application lin\'eaire $E^{\gamma_o>0}\longrightarrow F^{\gamma_o>0}$.
\end{enumerate}
\end{definition}

\medskip

Soit $\Rgot$ l'ensemble des racines pour l'action du tore $T$ sur $\ugot_\C$. Le choix d'une chambre de Weyl $\tgot^*_+$ d\'etermine 
un sous-ensemble $\Rgot^+$ de racines positives de telle sorte que $\tgot^*_+:=\{\xi\in\ugot^*,\ (\xi,\alpha)\geq 0,\ \forall \alpha\in \Rgot^+\}$. 
On consid\`ere la d\'ecomposition 
\begin{equation}\label{eq:decomposition-ugot}
\ugot_\C=\ngot\oplus\tgot_\C\oplus \overline{\ngot}
\end{equation} 
o\`u $\ngot=\sum_{\alpha\in\Rgot^+}(\ugot_\C)_\alpha$. 
D\'esignons par $B\subset U_\C$ le sous-groupe de Borel ayant $\bgot:=\tgot_\C\oplus \ngot$ pour alg\`ebre de Lie.

\medskip

Consid\'erons $(x,\gamma)\in M\,\times\,\tgot$ tel que $x\in  M^\gamma$. Le morphisme $U_x$-\'equivariant
(\ref{eq:rho-m}) induit une application lin\'eaire complexe 
\begin{equation}\label{eq:rho-gamma}
\rho_x^\gamma:\ngot^{\gamma_o>0}\longrightarrow (\T_x M)^{\gamma_o>0}.
\end{equation}

\medskip

\begin{definition}\label{def:infinitesimal-B-ressayre-pair}
Soit $\gamma\in\tgot$ un \'el\'ement non nul, et soit $C\subset M^\gamma$ une composante connexe. La donn\'ee $(\gamma,C)$ est appel\'ee 
une {\em paire de Ressayre infinit\'esimale} de $M$ si $\exists x\in C$, tel que $\rho_x^\gamma$ est un isomorphisme. 
Si de plus nous avons $\dim_\ugot(C)-\dim_\ugot(M)\in\{0,1\}$, et que $\gamma$ est rationnel, nous appelons $(\gamma,C)$ une 
{\em paire de Ressayre infinit\'esimale r\'eguli\`ere} de $M$.
\end{definition}

\medskip

\begin{rem}
Lorsque $U=T$ est ab\'elien, un couple $(\gamma,C)$ est une paire de Ressayre infinit\'esimale lorsque le fibr\'e vectoriel $(\T M\vert_C)^{\gamma_o>0}$ est r\'eduit \`a $0$.
\end{rem}

\begin{rem}
Puisque nous travaillons avec des vecteurs rationnels, nous identifierons les paires $(\gamma,C)$ et $(q\gamma,C)$, pour tout $q\in \Q^{>0}$.
\end{rem}

\subsection{Paires de Ressayre}

Introduisons maintenant une notion plus restrictive, celle de paire de Ressayre. Soient $\gamma\in\tgot$ un \'el\'ement non nul, 
et  $C\subset M^\gamma$ soit une composante connexe. Nous introduisons la sous-vari\'et\'e de Bia\l ynicki-Birula 
\begin{equation}\label{eq:BB}
C^-:= \Big\{m\in M, \lim_{t\to\infty} \exp(t\gamma_o) m\ \in C \Big\},
\end{equation}
qui est une sous-vari\'et\'e complexe localement ferm\'ee de $M$ \cite{Carrell-Sommese78}. On remarque que pour tout $x\in C$, $(\T_x M)^{\gamma_o\leq 0}=\T_x C^-$. 

Consid\'erons maintenant le sous-groupe parabolique $P(\gamma)\subset U_\C$ d\'efini par
\begin{equation}\label{eq:P-gamma}
P(\gamma)=\Big\{g \in U_\C, \lim_{t\to\infty}\exp(t\gamma_o)g\exp(-t\gamma_o)\ \mbox{existe}\Big\}.
\end{equation}
On remarque que l'alg\`ebre de Lie de $P(\gamma)$ est $(\ugot_\C)^{\gamma_o\leq 0}$.
Comme $C^-$ est invariant sous l'action de $P(\gamma)$, nous pouvons consid\'erer la vari\'et\'e complexe 
$B\times_{B\cap P(\gamma)} C^-$ et l'application holomorphe 
\begin{equation}\label{eq:q-gamma}
q_\gamma : B\times_{B\cap P(\gamma)} C^-\longrightarrow M
\end{equation}
qui envoie $[b,x]$ vers $bx$. 

\begin{definition}\label{def:B-ressayre-pair}
La donn\'ee $(\gamma,C)$ est appel\'ee une {\em paire de Ressayre} de $M$ s'il existe des ouverts 
$\Vcal\subset M$ et $\Ucal\subset C^-$,  tels que 
\begin{itemize}
\item $\Vcal$ est dense et $B$-invariant.
\item $\Ucal$ est dense, $B\cap P(\gamma)$-invariant, et intersecte $C$. 
\item $q_{\gamma}$ d\'efinit un diff\'eomorphisme $B\times_{B\cap P(\gamma)}\Ucal\simeq \Vcal$.
\end{itemize}
Si en outre $\dim_\ugot(C)-\dim_\ugot(M)\in\{0,1\}$ et que $\gamma$ est rationnel, nous appelons $(\gamma,C)$ une paire de Ressayre  r\'eguli\`ere de $M$.
\end{definition}

\begin{rem}\label{rem:B-paire}
La notion de paire de Ressayre dépend du choix du sous-groupe de Borel $B$. Lorsque cela nécessaire, on parlera de 
$B$-paire de Ressayre. Par exemple, dans la section \ref{sec:RR-mutliplicative}, on utilisera la remarque suivante: 
si $k\in K$ normalise l'algèbre de Lie $\tgot$, alors $(\gamma,C)$ est une $B$-paire de Ressayre si et seulement si 
$(Ad(k)(\gamma),kC)$ est une $kBk^{-1}$-paire de Ressayre.
\end{rem}

Pour tout $x\in C$, l'espace tangent $T\vert_{[e,x]}(B\times_{B\cap P(\gamma)} C^-)$ est isomorphe \`a 
$\ngot^{\gamma_o>0}\times (T_x M)^{\gamma_o\leq 0}$, et l'application tangente  
$\T q_\gamma \vert_{[e,x]}:T\vert_{[e,x]}(B\times_{B\cap P(\gamma)} C^-)\to \T_x M$ est d\'efinie par : 
$(X, v)\mapsto X\cdot x +v$. Ainsi $\T q_\gamma \vert_{[e,x]}$ est un isomorphisme si et seulement si 
$\rho_x^\gamma$ est un isomorphisme. On voit ainsi que la notion \og paire de Ressayre \fg{} raffine  celle de  \og paire de Ressayre infinit\'esimale \fg.

\subsection{Lemme fondamental}\label{sec:lemme-fondamental}

Pour caractériser les paires de Ressayre, nous utiliserons à de plusieurs reprises le résultat suivant. Soient $M$ une variété de Kähler $U$-hamiltonienne, 
$\beta\in\ugot$ et $C$ une composante connexe de $M^\beta$. La variété de Bialynicki-Birula 
$C^-= \Big\{m\in M, \lim_{t\to\infty} \exp(t\beta_o) m\ \in C \Big\}$ est muni de l'action du groupe parabolique $P(\beta)$. 
Soit $H\subset U_\C$ un sous-groupe complexe fermé tel que son algèbre de Lie $\hgot$ contienne $\C\beta$.

On considère l'application holomorphe 
$$
\Theta_H : H\times_{H\cap P(\beta)}C^-\longrightarrow M
$$
qui envoie $[h:x]$ sur $hx$. L'ouvert $\{z, d\Theta_H\vert_z \ \mathrm{est\ un\ isomorphisme}\}$ est de la forme $H\times_{H\cap P(\beta)}(C^-)_{reg}$. De plus, 
$z\in (C^-)_{reg}\cap C$ si et seulement si l'application linéaire
$$
X\in \hgot^{\beta_o>0}\longrightarrow X\cdot z\in (T_z M)^{\beta_o >0}
$$
est un isomorphisme.

On a le résultat classique suivant (voir \cite[Lemma 2.5]{pep-ressayre} ou \cite[Lemme 3.10]{Brion-bourbaki}).

\begin{lem}\label{lem:RP-fondamental}
Supposons que $(C^-)_{reg}\neq\emptyset$. Alors $C\cap (C^-)_{reg}\neq\emptyset$ si et seulement si 
\begin{equation}\label{condition-numerique-2}
\tr(\beta_o\circlearrowright \hgot^{\beta_o>0})=\tr(\beta_o\circlearrowright (TM\vert_{C})^{\beta_o>0}).
\end{equation}
Dans ce cas, $(C^-)_{reg}=\pi^{-1}(C\cap (C^-)_{reg})$ où $\pi : C^-\to C$ désigne la projection.
\end{lem}

\subsection{Caract\'erisation des paires de Ressayre}\label{sec:Characterization-RP}

Nous avons la caractérisation suivante des paires de Ressayre.

\begin{prop}\label{prop:caracteriser-RP}
\begin{itemize}
\item Si $(\gamma,C)$ est une paire de Ressayre infinit\'esimale, alors
\begin{enumerate}
\item[$(A_1)$] $\dim(\ngot^{\gamma_o>0})=\mathrm{rank} (\T M\vert_C)^{\gamma_o>0}$,
\item[$(A_2)$] $\tr(\gamma_o \circlearrowright \ngot^{\gamma_o>0})=\tr(\gamma_o \circlearrowright (\T M\vert_C)^{\gamma_o>0})$.
\end{enumerate}
\item $(\gamma,C)$ est une paire de Ressayre infinit\'esimale si et seulement si les conditions $(A_1)$, $(A_2)$, et
$$
 (A_3')\qquad \Big\{y\in M, q_\gamma^{-1}(y)\neq\emptyset \Big\}\quad \textrm{est d'int\'erieur non vide},
$$
sont satisfaites.
\item  $(\gamma,C)$ est une paire de Ressayre si et seulement si les conditions $(A_1)$, $(A_2)$, et
$$
 (A_3)\qquad \Big\{y\in M, \sharp \{q_\gamma^{-1}(y)\}=1 \Big\}\quad \textrm{contient un ouvert dense},
$$
sont satisfaites.
\end{itemize}
\end{prop}

\begin{proof} 
L'existence de $x\in C$ pour lequel l'application lin\'eaire $\rho_x^\gamma:\ngot^{\gamma_o>0}\longrightarrow (\T_x M)^{\gamma_o>0}$ 
est un isomorphisme entraine directement les identit\'es $(A_1)$ et $(A_2)$. Le premier point est v\'erifi\'e.

L'existence de $x\in C$ pour lequel l'application tangente  $\T q_\gamma \vert_{[e,x]}$ est un isomorphisme entraine que l'image de 
$q_\gamma$ est d'int\'erieur non-vide. On a donc montr\'e que si $(\gamma,C)$ est une paire de Ressayre infinit\'esimale, alors 
les conditions $(A_1)$, $(A_2)$ et $(A'_3)$ sont satisfaites. 

Supposons maintenant que les conditions $(A_1)$, $(A_2)$ et $(A'_3)$ sont satisfaites. Gr\^ace \`a $(A_1)$, on sait que $q_\gamma$ est une 
application holomorphe entre deux vari\'et\'es de m\^eme dimension. Notons $N_{reg}$ l'ensemble des points 
$y\in N:=B\times_{B\cap P(\gamma)} C^-$ pour lesquels l'application tangente $\T q_\gamma \vert_{y}$ est un isomorphisme. 
L'hypoth\`ese $(A'_3)$ montre, gr\^ace au Lemme de Sard, que $N_{reg}\neq \emptyset$. Comme $N_{reg}$ est invariant sous l'action du groupe $B$, 
on a $N_{reg}=B\times_{B\cap P(\gamma)} (C^-)_{reg}$ où  $(C^-)_{reg}=C^- \cap N_{reg}$ est un ouvert $B\cap P(\gamma)$-invariant de $C^-$.

Notons $\pi_C: C^-\to C$ la projection. Finalement, l'hypoth\`ese $(A_2)$ montre que $(C^-)_{reg}=(\pi_C)^{-1}(C\cap N_{reg})$ 
(voir le lemme \ref{lem:RP-fondamental}), et pour tout 
$x\in C\cap N_{reg}$, l'application $\rho_x^\gamma$ est un isomorphisme. On d\'emontre ainsi que $(\gamma,C)$ est une paire de Ressayre infinit\'esimale.

A ce stade, nous avons d\'emontr\'e les deux premiers points.

Pour le troisi\`eme point, supposons tout d'abord que $(\gamma,C)$ satisfait les conditions $(A_1)$, $(A_2)$ et $(A_3)$, et notons $M_{reg}$ l'ensemble des points $y\in M$ pour lesquels $q_\gamma^{-1}(y)$ est un singleton.
Nous avons vu pr\'ec\'edemment que l'ouvert $N_{reg}$ est non-vide. Dans ce cas, son compl\'ementaire  $N-N_{reg}$ est un sous-ensemble analytique de la vari\'et\'e complexe $N$, en particulier  $N_{reg}$ est dense.

\begin{lem}\label{lem:caracteriser-RP}
\begin{enumerate}
\item $q_\gamma^{-1}(M_{reg})\subset N_{reg}$.
\item $q_\gamma$ est injectif sur l'ouvert $N_{reg}$.
\end{enumerate}
\end{lem}
\begin{proof} Soit $x\in N$ tel que $y= q_\gamma(x)\in M_{reg}$. Si l'application $\T q_\gamma \vert_{x}$ n'est pas un isomorphisme, le point $x\in N$ n'est pas 
isol\'e dans la fibre $q_\gamma^{-1}(y)$, ce qui contradictoire avec le fait que $q_\gamma^{-1}(y)$ est un singleton. Donc $x\in N_{reg}$. 

Supposons qu'il existe deux points $a\neq b$ de $N_{reg}$ tels que $c:=q_\gamma(a)=q_\gamma(b)$. Alors, il existe des voisinages ouverts (euclidiens) 
respectifs de $a,b$ et $c$, not\'es $\Vcal_a\subset N$,  $\Vcal_b\subset N$ et $\Vcal_c,\Vcal_c'\subset M$, pour lesquels l'application 
$q_\gamma$ d\'efinit des bijections $\Vcal_a\simeq \Vcal_c$ et $\Vcal_b\simeq \Vcal_c'$, et tels que $\Vcal_a\cap\Vcal_b=\emptyset$. 
Comme $M_{reg}$ est dense, il existe $y\in \Vcal_c\cap \Vcal'_c\cap M_{reg}$. Alors la fibre $q_\gamma^{-1}(y)$ intersecte \`a la fois 
$\Vcal_a$ et $\Vcal_b$, ce qui contredit le fait que $q_\gamma^{-1}(y)$ est un singleton. On a ainsi montr\'e que l'application 
$q_\gamma$ est injective sur l'ouvert $N_{reg}$.
\end{proof}

D'apr\`es le lemme pr\'ec\'edent, on sait que $q_\gamma$ d\'efinit un diff\'eomorphisme holomorphe entre les ouverts $N_{reg}\subset N$ et 
$q_\gamma(N_{reg})\subset M$. Comme $M_{reg}\subset q_\gamma(N_{reg})$, et que $M_{reg}$ est dense, on peut conclure que 
$q_\gamma(N_{reg})$ un ouvert dense de $M$. On a d\'ej\`a expliqu\'e que la condition $(A_2)$ impose que l'ouvert $N_{reg}$ intersecte la sous-vari\'et\'e 
$C\subset N$. On a finalement montr\'e que $(\gamma,C)$ est une paire de Ressayre. 

R\'eciproquement, supposons que $(\gamma,C)$ est une {\em paire de Ressayre} de $M$: il existe des ouverts denses $\Ucal_N\subset N$ et $\Ucal_M\subset M$, 
tous deux $B$-invariants, tel que $q_{\gamma}$ d\'efinit un diff\'eomorphisme $\Ucal_N\simeq \Ucal_M$, et de plus $\Ucal_N\cap C\neq \emptyset$. Cela implique que 
$(\gamma,C)$ est une {\em paire de Ressayre infinit\'esimale}, et donc $(\gamma,C)$ v\'erifie les conditions $(A_1)$ et $(A_2)$. 
Il nous reste à montrer que $(\gamma,C)$ satisfait la condition $(A_3)$. Par d\'efinition de l'ouvert $N_{reg}$, on a les inclusions  
$\Ucal_N\subset N_{reg}$ et $\Ucal_M\subset q_\gamma(N_{reg})$. La m\^eme d\'emarche que celle employ\'ee au lemme 
\ref{lem:caracteriser-RP} permet de voir que $q_\gamma$ est injectif sur l'ouvert $N_{reg}$. 
Ainsi $q_\gamma$ d\'efinit un diff\'eomorphisme entre les ouverts denses $N_{reg}\subset N$ et $q_\gamma(N_{reg})$. 
Le lemme suivant compl\`ete la preuve du dernier point de la proposition \ref{prop:caracteriser-RP}.

\begin{lem}
\begin{enumerate}
\item L'image $q_\gamma(N-N_{reg})\subset M$ est d'int\'erieur vide.
\item $q_\gamma(N_{reg})-q_\gamma(N-N_{reg})\subset M_{reg}$.
\item $M_{reg}$ contient un ouvert dense de $M$.
\end{enumerate}
\end{lem}
\begin{proof} 
Le premier point est une cons\'equence imm\'ediate du lemme de Sard. L'inclusion $q_\gamma(N_{reg})-q_\gamma(N-N_{reg})\subset M_{reg}$ vient du fait que 
$q_\gamma$ d\'efinit une bijection entre $N_{reg}$ est son image $q_\gamma(N_{reg})$. 
D'apr\`es les deux premiers points, $M_{reg}$ est dense car il contient l'intersection 
$q_\gamma(N_{reg})\cap \Vcal$, o\`u $\Vcal$ est l'ouvert dense \'egal \`a l'int\'erieur de $M-q_\gamma(N-N_{reg})$.
\end{proof}
\end{proof}

Dans les exemples que nous traiterons, les variétés sont algébriques. Dans ce cadre, nous adaptons la notion de paire de Ressayre.

\medskip

Si $f: \Xcal \to \Ycal$ est un morphisme dominant entre deux variétés algébriques complexes lisses de même dimension, on note $f^*:\C(\Ycal) \longrightarrow \C(\Xcal)$ 
l'extension obtenue au niveau des corps de fonctions rationnelles. On note $\deg(f)\in \N-\{0\}\cup\{\infty\}$ le degré de cette extension. Lorsque le degré $\deg(f)$ 
est fini, il existe un ouvert de Zariski $\Vcal\subset \Ycal$ tel que $\sharp \{f^{-1}(y)\} = \deg(f)$ pour tout $y\in\Vcal$. Ainsi, lorsque $\deg(f)=1$, le morphisme $f$ définit une correspondance birationnelle entre $\Xcal$ et $\Ycal$.

\medskip

Soient $M$ une variété de Kähler $U$-hamiltonienne possédant une structure de variété algébrique. Soit 
$q_\gamma : B\times_{B\cap P(\gamma)} C^-\longrightarrow M$ le morphisme associé à $(\gamma,C)$.

\begin{definition}\label{def:B-ressayre-pair-alg}
La donn\'ee $(\gamma,C)$ est appel\'ee une {\em paire de Ressayre algébrique} de $M$ si les conditions 
suivantes sont satisfaites
\begin{enumerate}
\item $\deg(q_\gamma)=1$, 
\item Condition $(A_2)$ : $\tr(\gamma_o \circlearrowright \ngot^{\gamma_o>0})=\tr(\gamma_o \circlearrowright (\T M\vert_C)^{\gamma_o>0})$.
\end{enumerate}

La donn\'ee $(\gamma,C)$ est appel\'ee une {\em paire de Ressayre de degré fini} de $M$ si les conditions 
suivantes sont satisfaites
\begin{enumerate}
\item[1.'] $\deg(q_\gamma)\in \N-\{0\}$, 
\item[2.] Condition $(A_2)$ : $\tr(\gamma_o \circlearrowright \ngot^{\gamma_o>0})=\tr(\gamma_o \circlearrowright (\T M\vert_C)^{\gamma_o>0})$.
\end{enumerate}

Si en outre $\dim_\ugot(C)-\dim_\ugot(M)\in\{0,1\}$ et que $\gamma$ est rationnel, nous appelons $(\gamma,C)$ une paire de Ressayre algébrique (ou de degré fini) r\'eguli\`ere de $M$.
\end{definition}

Grâce à la proposition \ref{prop:caracteriser-RP}, on voit immédiatement qu'une paire de Ressayre algébrique est une paire de Ressayre, et qu'une 
paire de Ressayre de degré fini est une paire de Ressayre infinitésimale.

\section{Description du polytope de Kirwan}

La fonction $x\mapsto \langle\Phi_\ugot(x),\gamma\rangle$ est localement constante sur $M^\gamma$: on désigne donc par $\langle\Phi_\ugot(C),\gamma\rangle$ sa valeur sur 
une composante connexe $C\subset M^\gamma$.

\medskip

Voici une version constructive du th\'eor\`eme de Kirwan, d\'emontr\'ee par Ressayre dans le cadre projectif \cite{Ressayre10} et 
\'etendu au cadre K\"ahl\'erien dans \cite{pep-ressayre}.

\begin{theorem}\label{th:infinitesimal-ressayre-pairs} 
Soit $(M,\Omega)$ une vari\'et\'e de K\"{a}hler $U$-hamiltonienne, \'equip\'ee d'une application moment propre.  
Pour $\xi\in\tgot^*_+$, les énoncés suivants sont équivalents :
\begin{itemize}
\item $\xi\in \Delta_{\ugot}(M)$.
\item Pour toute paire de Ressayre infinit\'esimale r\'eguli\`ere  $(\gamma,C)$, on a $\langle \xi,\gamma\rangle\geq \langle \Phi_\ugot(C),\gamma\rangle$.
\item Pour toute paire de Ressayre r\'eguli\`ere  $(\gamma,C)$, nous avons 
$\langle \xi,\gamma\rangle\geq \langle \Phi_\ugot(C),\gamma\rangle$.
\end{itemize}
\end{theorem}

\begin{rem}Le résultat précédent est encore valable sans la condition de \emph{régularité} sur les paires de Ressayre.

\end{rem}

\section{Vari\'et\'es de K\"{a}hler $(U,\sigma)$-Hamiltoniennes}\label{sec:U-sigma-hamiltonien}

Soit $\sigma$ une involution sur le groupe compact $U$. Nous d\'esignons par $K$ la composante connexe du sous-groupe de point fixe 
$U^\sigma$. Au niveau des alg\`ebres de Lie, nous avons une d\'ecomposition  $\ugot=\ugot^{\sigma}\oplus\ugot^{-\sigma}$ o\`u 
$\ugot^{\pm\sigma}=\{X\in \ugot, \sigma(X)=\pm X\}$. 

On d\'esigne encore par $\sigma$ l'involution \emph{anti-holomorphe} sur $U_\C$ 
d\'efinie par les relations
\begin{equation}\label{eq:sigma-anti-holomorphe}
\sigma(ke^{iX})=\sigma(k)e^{-i\sigma(X)},\qquad \forall (k,X)\in U\times\ugot.
\end{equation}

Soit $G$ la composante connexe du sous-groupe ferm\'e de $U_\C$ fix\'e par $\sigma$ : c'est un sous-groupe r\'eductif r\'eel qui est stable sous 
l'involution de Cartan. Son algèbre de Lie est \'egale \`a $\ggot=\kgot\oplus \pgot$ o\`u $\kgot=\ugot^{\sigma}$ et 
$$
\pgot=i \ugot^{-\sigma}.
$$ 
Le groupe de Lie réductif réel $G$ admet la d\'ecomposition de Cartan $K\times\pgot \to G$, $(k,X)\mapsto k e^X$.

\begin{definition}\label{def:U-sigma-kahler}
Une vari\'et\'e de K\"{a}hler $(U,\sigma)$-hamiltonienne est un tuple $(M,\Omega,\tau)$ constitu\'e d'une vari\'et\'e complexe 
$(M,\J)$ \'equip\'ee d'une action holomorphe du groupe $U_\C$, d'une forme symplectique $U$-invariante $\Omega$ compatible avec la structure complexe $\J$, 
d'une involution $\tau$ sur $M$, et d'une application moment $\Phi_\ugot$ pour l'action de $U$ sur $(M,\Omega)$. 
De plus, nous imposons les conditions suivantes :
\begin{align*}
\tau^*(\J)&=-\J, \quad \tau \ \mathrm{est\ anti-holomorphe},\\
\tau^*(\Omega)&=-\Omega, \quad \tau \ \mathrm{est\ anti-symplectique},\\
\Phi_\ugot(\tau(m))&=-\sigma\left(\Phi_\ugot(m)\right),\\
\tau(gm)&= \sigma(g)\tau(m),
\end{align*}
pour tout $m\in M$ et $g\in U_\C$. 
\end{definition}

\begin{exemple}
Soit $\xi\in(\tgot^*)^{-\sigma}$. L'orbite coadjointe $U\xi$ est isomorphe au quotient $U_\C/P(\xi)$, où le sous groupe parabolique $P(\xi)\subset U_\C$ est invariant sous l'involution $\sigma$ (car $\sigma(i\xi)=i\xi$). Ainsi l'involution anti-holomorphe $\tau_\xi: U_\C/P(\xi)\to U_\C/P(\xi)$ peut être définie par $\tau_\xi([g])=[\sigma(g)]$.

La sous-variété fixée par l'involution $\tau_\xi$ est connexe, égale \`a $K\xi$, et l'isomorphisme 
$U\xi\simeq U_\C/P(\xi)$ descend en un isomorphisme $K\xi\simeq G/ G\cap P(\xi)$  où $G\cap P(\xi)$ est un sous-groupe parabolique réel de $G$.
\end{exemple}

Lorsque $M$ est une vari\'et\'e de K\"{a}hler $(U,\sigma)$-hamiltonienne, nous accordons une attention particuli\`ere \`a l'application $K$-\'equivariante 
$$
\Phi_\pgot : M\to \pgot^*
$$ 
d\'efinie par les relations $\langle\Phi_\pgot,\zeta\rangle=\langle\Phi_\ugot,i\zeta\rangle$, pour tout $\zeta\in\pgot$. Si nous d\'esignons par 
$j:\pgot\to\ugot^{-\sigma}$ l'isomorphisme $\zeta\mapsto i\zeta$, nous avons $\Phi_\pgot=j^*\circ \pi\circ\Phi_{\ugot}$, o\`u 
$j^* : (\ugot^*)^{-\sigma}\to \pgot^*$ est l'application duale et $\pi:\ugot^*\to(\ugot^*)^{-\sigma}$ est la projection canonique.

\begin{lem}
La vari\'et\'e $M$ est muni de la structure Riemannienne $(-,-)_M:=\Omega(-,\J-)$. Pour tout $\zeta\in\pgot$, le champ de vecteur $\zeta_M$ 
est le champ de vecteur gradient de la fonction $\langle\Phi_\pgot,\zeta\rangle$.
\end{lem}
\begin{proof} L'identit\'e (\ref{eq:hamiltonien}) montre que
$d\langle\Phi_\pgot,\zeta\rangle=d\langle\Phi_\ugot, i\zeta\rangle=-\Omega(\J\zeta_M,-)=(\zeta_M,-)_M$.
\end{proof}

Soit $\agot\subset\pgot$ une sous alg\`ebre ab\'elienne maximale, et notons $\Sigma(\ggot)\subset \agot^*$ l'ensemble des racines restreintes relatives \`a l'action de $\agot$ sur $\ggot$. \`A chaque racine restreinte $\alpha\in\Sigma$, on associe $\ggot_\alpha:=\{X\in \ggot, [H,X]=\alpha(H)X, \ \forall H\in\agot\}$. On a la décomposition
\begin{equation}\label{eq:decomposition-ggot}
\ggot=\mgot\oplus\agot\oplus \bigoplus_{\alpha\in \Sigma(\ggot)}\ggot_\alpha,
\end{equation}  
où $\mgot$ est la sous-algèbre centralisatrice de $\agot$ dans $\kgot$.

Le choix d'un sous-ensemble $\Sigma^+\subset \Sigma(\ggot)$ de racines positives d\'etermine une chambre de Weyl $\agot^*_+\subset \agot^*$ : celle-ci 
est un domaine fondamental pour l'action de $K$ sur $\pgot^*$.

\begin{definition}
Pour toute partie $K$-invariante $\Xcal\subset M$, le sous-ensemble $\Delta_\pgot(\Xcal):=\Phi_\pgot(\Xcal)\cap\agot_+^*$ param\`etre les $K$-orbites dans l'image
$\Phi_\pgot(\Xcal)$.
\end{definition}

Choisissons un tore maximal $\sigma$-invariant  $T\subset U$ dont l'alg\`ebre de Lie $\tgot$ contient $i\agot$. Nous pouvons considérer une chambre de Weyl $\tgot^*_+\subset \tgot^*$ telle que $\tgot^*_+\cap (\tgot^*)^{-\sigma}{\simeq}_{j^*} \agot^*_+ $ (voir l'appendice dans \cite{OSS}).

Nous pouvons reformuler le th\'eor\`eme d'O'Shea-Sjamaar  dans le cadre K\"ahlérien.

\begin{theorem}[\cite{OSS}]\label{theo:OSS-bis}
Soit $(M,\Omega,\tau)$ une vari\'et\'e de K\"{a}hler $(U,\sigma)$-hamiltonienne telle que $M^\tau\neq\emptyset$. Alors
$$
\Delta_\ugot(M)\bigcap (\tgot^*)^{-\sigma}\underset{j^*}{\simeq} \Delta_\pgot(M^\tau).
$$
\end{theorem}

Dans la section suivante, nous allons expliquer comment paramétrer les faces de $\Delta_\pgot(M^\tau)$ en termes de paires de Ressayre {\em réelles}.

\medskip

\section{Paire de Ressayre r\'eelles}\label{sec:paires-reelles}

Dans le reste de cette section, nous supposons que $M^\tau\neq \emptyset$. Nous verrons à la section \ref{sec:preuve-OSS} que le théorème d'O'Shea-Sjamaar 
admet le raffinement suivant: pour toute composante connexe $\Zcal\subset M^\tau$, nous avons $\Delta_\pgot(M^\tau)=\Delta_\pgot(\Zcal)$.

Nous fixons pour le reste de la section une composante connexe $\Zcal$ de $M^\tau$: il s'agit d'une variété riemannienne équipée des actions des groupes 
$K$ et $G$.

\subsection{Elements admissibles} 

Nous commençons par introduire la notion d'\og éléments admissibles \fg{} pour l'action infinitésimale de $\pgot$ sur $\Zcal$.
Un vecteur $\zeta\in \agot$ est dit {\em rationnel} si $i\zeta$ appartient au $\Q$-vectoriel $\tgot_\Q$ engendré par le réseau $\wedge$.

Les sous-groupes stabilisateurs de $m\in M$ par rapport aux actions de $K$ et $G$ sont respectivement notés $L_m$ et $G_m$ : 
leurs algèbres de Lie sont notées $\kgot_m$ et $\ggot_m$. Nous porterons une attention particulière au sous-espace 
$\pgot_m=\{X\in\pgot, X\cdot m=0\}\subset\ggot_m$. Si $\zeta\in\agot$,  la sous-variété $\Zcal^\gamma=\{z\in\Zcal, \gamma\cdot z=0\}$ 
est stable sous l'action du sous-groupe stabilisateur $G^\gamma=\{g\in G,g\gamma=\gamma\}$.

\begin{definition}\label{def:reel-admissible} Définissons $\dim_\pgot(\Xcal):=\min_{z\in\Xcal}\,\dim(\pgot_z)$ pour tout sous-ensemble $\Xcal\subset \Zcal$. 
Un élément non nul $\zeta\in\agot$  est dit {\em admissible} par rapport à l'action $\pgot\circlearrowright \Zcal$ si $\zeta$ est rationnel et si 
$\dim_\pgot(\Zcal^\zeta)-\dim_\pgot(\Zcal)\in\{0,1\}$.
\end{definition}

\subsection{Paires de Ressayre dans un cadre avec involution}\label{sec:RP-involution}

Pour tout $z\in \Zcal$, l'action infinitésimale de $\ggot$ sur $\Zcal$ définit une application $\R$-linéaire 
\begin{eqnarray}\label{eq:rho-z}
\rho^\R_z:\ggot & \longrightarrow & \T_z \Zcal\\
X & \longmapsto & X\cdot m \nonumber
\end{eqnarray}
qui est équivariante sous l'action du sous-groupe stabilisateur $G_z$.

\begin{definition}
Considérons un endomorphisme symétrique $\Lcal(\zeta)$ d'un espace vectoriel euclidien $E$. Nous associons l'espace propre
$E^{\zeta=a}=\{v\in E, \Lcal(\zeta)v= av\}$ à tout $a\in\R$. Nous avons la décomposition
$E=E^{\zeta>0}\oplus E^{\zeta=0}\oplus E^{\zeta<0}$ où $E^{\zeta>0}=\sum_{a>0}E^{\zeta=a}$, et $E^{\zeta<0}=\sum_{a<0}E^{\zeta=a}$.
\end{definition}

\medskip

Considérons $(x,\zeta)\in \Zcal\times\agot$ tel que $x\in \Zcal^\zeta$. L'action infinitésimale de $\zeta$ définit 
des endomorphismes symétriques $\Lcal(\zeta):\T_x \Zcal\to \T_x \Zcal$ et $\Lcal(\zeta):\ggot\to \ggot$ satisfaisant 
$\Lcal(\zeta)\circ\rho_x^\R=\rho^\R_x\circ\Lcal(\zeta)$.
Ainsi, le morphisme (\ref{eq:rho-z}) induit une application $\R$-linéaire
\begin{equation}\label{eq:rho-zeta}
\rho_x^{\R,\zeta}:\Ngot^{\zeta>0}\longrightarrow (\T_x \Zcal)^{\zeta>0},
\end{equation}
où $\Ngot=\sum_{\alpha\in\Sigma^+}\ggot_\alpha$.

\begin{definition}\label{def:infinitesimal-real-ressayre-pair}
Soit $\zeta\in\agot$ un élément non nul, et $C\subset M^\zeta$ une composante connexe telle que $C\cap\Zcal\neq\emptyset$. La donnée $(\zeta,C)$ est appelée 
une {\em paire de Ressayre infinitésimale réelle} de $\Zcal$ si $\exists x\in C\cap\Zcal$, tel que (\ref{eq:rho-zeta}) est un isomorphisme. 
Si, en outre, nous avons $\dim_\pgot(C\cap\Zcal)-\dim_\pgot(\Zcal)\in\{0,1\}$, et que $\zeta$ est rationnel, nous appelons $(\zeta,C)$ 
une {\em paire de Ressayre infinitésimale réelle régulière} de $\Zcal$.
\end{definition}

\begin{rem}
Lorsque $U=T$ est abélien, un couple $(\zeta,C)$ est une paire de Ressayre infinitésimale réelle  lorsque il existe une composante connexe 
$\Ccal'\subset C\cap \Zcal$ tel que le fibré vectoriel 
$(\T\Zcal\vert_{\Ccal'})^{\zeta>0}$ est égal au fibré « nul ».
\end{rem}

\begin{prop}
Soit $\zeta\in\agot$ un élément non nul, et soit $C\subset M^\zeta$ une composante connexe telle que $C\cap\Zcal\neq\emptyset$. 
Les conditions suivantes sont équivalentes
\begin{enumerate}
\item[$i)$] $(\zeta,C)$ est une paire de Ressayre infinitésimale réelle.
\item[$ii)$] $(i\zeta,C)$ est une paire de Ressayre infinitésimale.
\end{enumerate}
\end{prop}
\begin{proof}Comparons les sous-algèbres nilpotentes $\ngot\subset \ugot_\C$ et $\Ngot\subset \ggot$. 
Soit $\zeta_0$ un \'el\'ement qui appartient \`a l'int\'erieur de la chambre de Weyl $\agot_+:=\{\zeta\in\agot, \langle\alpha,\zeta\rangle \geq 0, \forall \alpha\in\Sigma(\ggot)^+\}$. Alors
$(\ugot_\C)^{\zeta_0>0}=\ggot^{\zeta_0>0}\otimes\C$ est un sous-espace vectoriel complexe stable par l'involution 
$\sigma$ et (\ref{eq:decomposition-ugot}) permet de voir que 
$(\ugot_\C)^{\zeta_0>0}=\ngot^{\zeta_0>0}$. Grâce \`a (\ref{eq:decomposition-ggot}), on sait que $\Ngot=\ggot^{\zeta_0>0}$, donc 
$\ngot^{\zeta_0>0}=\Ngot\otimes\C$. Ainsi, pour tout $\zeta\in\agot$, on a 
$$
\ngot^{\zeta>0}=\ngot^{\zeta>0,\zeta_0>0}=\Ngot^{\zeta>0}\otimes\C.
$$
Pour tout $x\in C\cap \Zcal$, les espaces vectoriels r\'eels $(\T_x\Zcal)^{\zeta>0}\subset T_x \Zcal$ sont respectivement \'egaux aux points fixes de l'involution 
anti-lin\'eaire $\tau$ sur $(\T_x C)^{\zeta>0}\subset T_x C$. 

On a ainsi montr\'e que, pour tout $x\in C\cap \Zcal$, l'application $\C$-lin\'eaire $\rho_x^{i\zeta}:\ngot^{\zeta>0}\longrightarrow (\T_x C)^{\zeta>0}$ 
est la complexification de l'application $\R$-lin\'eaire $\rho_x^{\R,\zeta}:\Ngot^{\zeta>0}\longrightarrow (\T_x \Zcal)^{\zeta>0}$. A ce stade, on a montr\'e que 
$i) \Longrightarrow ii)$. 

Supposons maintenant que $(i\zeta,C)$ est une paire de Ressayre infinitésimale. Alors $C_{reg}:=\{x\in C,\ \rho_x^{i\zeta}\ \mathrm{est\ inversible}\}$ 
est un ouvert de Zariski non-vide de la vari\'et\'e complexe $C$. Comme $C\cap \Zcal$ est une sous-variété lagrangienne de 
$C$ on a $\Zcal\cap C_{reg}\neq\emptyset$ (voir l'appendice). On a donc montr\'e que $(\zeta,C)$ est une paire de Ressayre infinitésimale réelle. 
\end{proof}

\begin{rem}\label{rem:RPreel-vs-RPstandard}
Bien que les notions de \og paire de Ressayre infinitésimale réelle\fg{} et \og paire de Ressayre infinitésimale\fg{} coincident, il est important 
de noter que les conditions de {\em régularit\'e} diffèrent d'un cas à l'autre.
\end{rem}

\begin{rem}\label{rem:trace-reelle-versus-complexe}Comme l'application $\C$-lin\'eaire $\rho_x^{i\zeta}$ est la complexification de 
$\R$-lin\'eaire $\rho_x^{\R,\zeta}$. Alors 
$\dim(\ngot^{\zeta>0})=\mathrm{rank} (\T M\vert_C)^{\zeta>0}\quad \Longleftrightarrow 
\quad \dim_\R(\Ngot^{\zeta>0})=\mathrm{rank}_\R (\T \Zcal\vert_{C\cap \Zcal})^{\zeta>0}$
et \break $\tr(\zeta \circlearrowright \ngot^{\zeta>0})=\tr(\zeta \circlearrowright (\T M\vert_C)^{\zeta>0})
\quad \Longleftrightarrow \quad 
\tr_\R(\zeta \circlearrowright \Ngot^{\zeta>0})=\tr_\R(\zeta \circlearrowright (\T \Zcal\vert_{C\cap \Zcal})^{\zeta>0})$.
\end{rem}

\medskip

Introduisons maintenant une notion plus restrictive, celle de paire de Ressayre réelle.

\medskip


Soit $\zeta\in\agot$ un élément non nul, et soit $C\subset M^\zeta$ une composante connexe intersectant $\Zcal$. La sous-variété de Kähler $C$ 
est stable sous l'involution $\tau$ et sous l'action du sous-groupe stabilisateur $U^\zeta_\C=\{g\in U_\C, g\zeta=\zeta\}$.  
Nous travaillons avec la sous-variété de Bialynicki-Birula $C^-:=\{m\in M, \lim_{t\to\infty} \exp(t\zeta) m\ \in C\}$.
Considérons maintenant le sous-groupe parabolique défini par
\begin{equation}\label{eq:P-gamma-reel}
\Pbb(\zeta)=\{g\in U_\C, \lim_{t\to\infty}\exp(t\zeta)g\exp(-t\zeta)\ \mathrm{existe}\}.
\end{equation}
Le sous-groupe\footnote{Nous avons changé la notation du sous-groupe parabolique pour mettre en évidence qu'il est stable pour l'involution $\sigma$.} 
$\Pbb(\zeta)=P(i\zeta)$, d'alg\`ebre de Lie $\ugot_\C^{\zeta\leq 0}$, est stable sous l'involution $\sigma$.

\begin{definition}\label{def:parabolique-P}
Si l'on prend $\zeta_0\in\agot$ tel que $\langle\alpha,\zeta_0\rangle >0$ pour toute racine $\alpha\in\Sigma^+$, le sous-groupe parabolique $\Pbb(-\zeta_0)\subset U_\C$ est noté $\Pbb$.
\end{definition}

La sous-variété $C^-$ étant invariante sous l'action de $\Pbb(\zeta)$, on peut  considérer la variété complexe 
$\Pbb\times_{\Pbb\cap \Pbb_\gamma} C^-$ et l'application holomorphe 
$$
\mathrm{q}^\R_\zeta: \Pbb\times_{\Pbb\cap \Pbb(\zeta)} C^- \to M
$$
qui envoie $[p,m]$ vers $pm$. La variété complexe $\Pbb\times_{\Pbb\cap \Pbb(\zeta)} C^-$ est équipée d'une involution anti-holomorphe naturelle
$\tau_\gamma:[p,m]\mapsto [\sigma(p),\tau(m)]$ telle que $\tau\circ \mathrm{q}^\R_\zeta =\mathrm{q}^\R_\zeta\circ \tau_\gamma$.

\begin{definition}\label{def:ressayre-pair}
Soit $\zeta\in\agot-\{0\}$ et $C\subset M^\zeta$ un composante connexe telle que $C\cap \Zcal\neq\emptyset$. Le couple $(\zeta,C)$ est appelé une 
\emph{paire de Ressayre réelle} de $\Zcal$ s'il existe des ouverts $\Vcal\subset M$ et $\Ucal\subset C^-$, invariants par $\tau$, et tels que 
\begin{itemize}
\item $\Vcal$ est dense et $\Pbb$-invariant,
\item $\Ucal$ est dense, $\Pbb\cap \Pbb(\zeta)$-invariant, et intersecte $C$, 
\item l'application $\mathrm{q}^\R_{\zeta}$ définit un difféomorphisme $\Pbb\times_{\Pbb\cap \Pbb(\zeta)}\Ucal\simeq \Vcal$.
\end{itemize}
Si, en outre, nous avons $\dim_\pgot(\Ccal)-\dim_\pgot(\Zcal)\in\{0,1\}$, et que $\zeta$ est rationnel, nous appelons $(\zeta,\Ccal)$ 
une \emph{paire de Ressayre réelle régulière}.
\end{definition}

Nous avons la caractérisation suivante des paires de Ressayre réelles.

\begin{prop}\label{prop:caracteriser-RP-reel}
Soit $\zeta\in\agot-\{0\}$ et $C\subset M^\zeta$ un composante connexe telle que $C\cap \Zcal\neq\emptyset$. 
Les conditions suivantes sont équivalentes
\begin{enumerate}
\item $(\zeta,C)$ est un paire de Ressayre réelle de $\Zcal$.
\item $(i\zeta,C)$ est une paire de Ressayre de $M$.
\item $(\zeta,C)$ satisfait les conditions  
\begin{align*}
(A_1) \qquad& \dim(\ngot^{\zeta>0})=\mathrm{rank} (\T M\vert_C)^{\zeta>0}\\
(A_2) \qquad& \tr(\zeta \circlearrowright \ngot^{\zeta>0})=\tr(\zeta \circlearrowright (\T M\vert_C)^{\zeta>0})\\
(A_3) \qquad& \Big\{y\in M, \sharp \{(\mathrm{q}^\R_\zeta)^{-1}(y)\}=1 \Big\} \quad \textrm{contient un ouvert dense}.
\end{align*}
\end{enumerate}
\end{prop}

\begin{proof} Nous commençons avec le lemme suivant qui concerne les trois sous-groupes $B$, $\Pbb$ et $\Pbb(\zeta)$ de $U_\C$.

\begin{lem} 
\begin{enumerate}
\item Nous avons la relation ensembliste $\Pbb = B(\Pbb\cap\Pbb(\zeta))$. 
\item Pour toute $\Pbb\cap \Pbb(\zeta)$-variété $Q$, on a un isomorphisme canonique 
\begin{equation}\label{eq:B-P-isomorphisme}
B\times_{B\cap \Pbb(\zeta)} Q\simeq \Pbb\times_{\Pbb\cap \Pbb(\zeta) }Q.
\end{equation}
\end{enumerate}
\end{lem}
\begin{proof} Les algèbres de Lie des groupes  $B$, $\Pbb$ et $\Pbb(\zeta)$ sont respectivement
\begin{eqnarray*}
\bgot&=&\tgot_\C\oplus \sum_{\beta>0,\sigma(\beta)=\beta}(\ugot_\C)_\beta\oplus \sum_{\beta>0,\sigma(\beta)\neq\beta}(\ugot_\C)_\beta,\\
\mathrm{Lie}(\Pbb)&=&\tgot_\C\oplus \sum_{\beta,\sigma(\beta)=\beta}(\ugot_\C)_\beta\oplus \sum_{\beta>0,\sigma(\beta)\neq\beta}(\ugot_\C)_\beta,\\
\mathrm{Lie}(\Pbb(\zeta))&=&\tgot_\C\oplus \sum_{\beta,\sigma(\beta)=\beta}(\ugot_\C)_\beta\oplus \sum_{\langle\beta,i\zeta\rangle\leq 0,\sigma(\beta)\neq\beta}(\ugot_\C)_\alpha.
\end{eqnarray*}
Ainsi, $\bgot\subset\mathrm{Lie}(\Pbb)$ et donc $B\subset\Pbb$. La decomposition de Levi de $\Pbb$ montre que 
$\Pbb= N L$ ou $L$ est le sous groupe réductif connexe d'algèbre de Lie $\tgot_\C\oplus \sum_{\beta,\sigma(\beta)=\beta}(\ugot_\C)_\beta$ et 
$N$ est le sous groupe unipotent d'algèbre de Lie $\sum_{\beta>0,\sigma(\beta)\neq\beta}(\ugot_\C)_\beta$. Comme $N\subset B$ et que 
$L\subset \Pbb(\zeta)$, on obtient $\Pbb\subset B\Pbb(\zeta)$. Le relations  $B\subset\Pbb\subset B\Pbb(\zeta)$ entrainent que $\Pbb = B(\Pbb\cap\Pbb(\zeta))$.

Soit $Q$ une variété munie d'une action du groupe $\Pbb\cap \Pbb(\zeta)$. L'identité $\Pbb = B(\Pbb\cap\Pbb(\zeta))$ permet de voir que l'application 
$[b,x]\in B\times_{B\cap \Pbb(\zeta)} Q\mapsto [b,x]\in \Pbb\times_{\Pbb\cap \Pbb(\zeta) }Q$ détermine un difféomorphisme $B$-équivariant.
\end{proof}

Si $(\zeta,C)$ est un paire de Ressayre réelle de $\Zcal$, nous avons un isomorphisme $\Pbb\times_{\Pbb\cap \Pbb(\zeta) } \Ucal\simeq \Vcal$, 
où $\Ucal$ est un ouvert dense $\Pbb\cap \Pbb(\zeta)$-invariant de $C^-$ et intersectant $C$, et où $\Vcal=\Pbb\Ucal=B\Ucal$ est un ouvert dense $\Pbb$-invariant de $M$. 
Grâce \`a (\ref{eq:B-P-isomorphisme}, on peut conclure que nous avons un isomorphisme $B\times_{B\cap \Pbb(\zeta) } \Ucal\simeq \Vcal$, i.e. $
(i\zeta,C)$ est une paire de Ressayre de $M$. On vient de v\'erifier l'implication \emph{1}. $\Longrightarrow$ \emph{2}., et  
l'implication \emph{2}. $\Longrightarrow$ \emph{3}. est démontrée dans la proposition \ref{prop:caracteriser-RP}.

Supposons maintenant que $(\zeta,C)$ satisfait les conditions  $(A_1)$, $(A_2)$ et $(A_3)$. Considérons la $\Pbb$-variété $N= \Pbb\times_{\Pbb\cap \Pbb(\zeta)} C^-$ et l'application holomorphe $q^\R_\zeta: N \to M$, $\tau$-équivariante. Grâce à $(A_1)$ et $(A_3)$ on sait que 
$$
N_{reg}=\{x\in N, Tq^\R_\zeta\vert_x\ \mathrm{est\ un\ isomorphisme}\}
$$
est un ouvert de Zariski non-vide, $\Pbb$-invariant, de la forme $N_{reg}= \Pbb\times_{\Pbb\cap \Pbb(\zeta)} (C^-)_{reg}$ où $(C^-)_{reg}$ est un ouvert 
$\Pbb\cap \Pbb(\zeta)$-invariant de $C^-$. L'identit\'e  $(A_2)$ 
permet de voir que 
$$
(C^-)_{reg}=(\pi_C)^{-1}(C\cap N_{reg}),
$$ 
où $C\cap N_{reg}$ est un ouvert de Zariski non-vide $\Pbb\cap U_\zeta$-invariant de $C$. Finalement, $(A_3)$ 
montre que $q^\R_\zeta$ définit un difféomorphisme entre $N_{reg}$ et l'ouvert dense $q^\R_\zeta(N_{reg})$ 
(voir la preuve de la proposition \ref{prop:caracteriser-RP} où ces arguments ont été démontré en détail). On a ainsi montr\'e que 
$(\zeta,C)$ est un paire de Ressayre réelle de $\Zcal$.
\end{proof}

\medskip

Comme \`a la remarque \ref{rem:RPreel-vs-RPstandard}, il est important de noter que même si les concepts de \og paire de Ressayre  réelle\fg{} et 
\og paire de Ressayre\fg{} coincident, les conditions de {\em régularit\'e} diffèrent d'un cas à l'autre.

\section{Description du polytope de Kirwan r\'eel}\label{sec:description-polytope-reel}

Soit  $(M,\Omega, \tau)$ une vari\'et\'e de K\"{a}hler $(U,\sigma)$-hamiltonienne. On suppose que la sous-variété $M^\tau$ est non-vide et on note $\Zcal$ 
l'une de ses composantes connexes. Voici l'un des principaux résultats de cette monographie.

\begin{theorem}\label{th:real-ressayre-pairs}
Pour $\xi\in\agot^*_{+}$, les affirmations suivantes sont équivalentes :
\begin{itemize}
\item $\xi\in\Delta_\pgot(M^\tau)$.
\item $\xi\in\Delta_\pgot(\Zcal)$.
\item Pour toute paire de Ressayre infinitésimale réelle régulière $(\zeta,C)$ de $\Zcal$, on a $\langle \xi,\zeta\rangle\geq \langle \Phi_\pgot(C),\zeta\rangle$.
\item Pour toute paire de Ressayre réelle régulière $(\zeta,C)$ de $\Zcal$, on a $\langle \xi,\zeta\rangle\geq \langle \Phi_\pgot(C),\zeta\rangle$.
\end{itemize}
\end{theorem}

\begin{rem}
Dans le théorème précédent, le résultat reste valable si l'on supprime l'hypothèse de régularité sur $(\zeta,\Ccal)$.
\end{rem}

\chapter{Preuve du th\'eor\`eme principal}

Dans toute cette section nous choisissons un produit scalaire rationnel invariant sur l'algèbre de Lie $\ugot$ du groupe de Lie compact $U$ 
(voir l'appendice). Par rationnel, nous entendons que pour le tore maximal $T\subset U$ avec algèbre de Lie $\tgot$, le produit scalaire 
prend des valeurs entières sur le réseau $\wedge:=\frac{1}{2\pi}\ker(\exp:\tgot\to T)$. Notons $\wedge^*\subset \tgot^*$ le réseau dual :
$\wedge^*=\hom(\wedge,\Z)$. Nous associons aux réseaux $\wedge$ et $\wedge^*$ les $\Q$-espaces vectoriels $\tgot_\Q$ et $\tgot^*_\Q$
 générés par ceux-ci : les vecteurs appartenant à $\tgot_\Q$ et $\tgot^*_\Q$ sont dits rationnels.

Le produit scalaire invariant sur $\ugot$ induit une identification $\ugot^*\simeq \ugot,\xi\mapsto \xi^\flat $ telle que $\tgot_\Q\simeq\tgot^*_\Q$.
Pour simplifier notre notation, nous ne ferons pas de distinction entre $\xi$ et $\xi^\flat$. Par exemple, pour tout $\beta\in\ugot^*$, nous écrivons $M^{\beta}$ pour 
la sous-variété fixée par l'action du groupe engendré par $\beta^\flat$.

\section{Stratifications à la Kirwan-Ness}

Dans la première partie nous travaillons avec une variété de K\"ahler $U$-hamiltonienne $(M,\Omega)$ admettant une application moment propre $\Phi_\ugot:M\to\ugot^*$.

\subsection{Le cadre usuel}

Soit $f_\ugot:=\frac{1}{2}(\Phi_\ugot,\Phi_\ugot):M\longrightarrow \R$ la norme quadratique de l'application moment. Remarquons que $f_\ugot$ est une fonction propre $U$-invariante sur $M$.

\begin{definition}
Le champ de vecteurs de Kirwan $\kappa_\ugot$ est défini par la relation
$$
\kappa_\ugot(m)=\Phi_\ugot(m)\cdot m,\quad \forall m\in M.
$$
\end{definition}

Considérons le gradient $\nabla f_\ugot$ de la fonction $f_\ugot$ par rapport à la métrique riemannienne $\Omega(-,\J-)$. 
Nous avons les faits bien connus suivants \cite{Woodward11}.

\begin{prop}
\begin{enumerate}
\item Le gradient de $f_\ugot$ est $\nabla f_\ugot=\J(\kappa_\ugot)$.
\item L'ensemble des points critiques de la fonction $f_\ugot$ est $\crit(f_\ugot)=\{\kappa_\ugot=0\}$.
\item Nous avons la décomposition
$\Phi_\ugot(\crit(f_\ugot))=\bigcup_{\lambda\in \Bcal_\Phi} U\lambda$
où l'ensemble $\Bcal_\ugot\subset \tgot^*_{+}$ est discret. $\Bcal_\ugot$ est appelé l'ensemble des \emph{types} de $M$.
 
\item Nous avons 
$$
\crit(f_\ugot)=\bigcup_{\lambda\in \Bcal_\ugot}Z_\lambda
$$
où $Z_\lambda=\crit(f_\ugot)\cap\Phi_\ugot^{-1}(U\lambda)$ est égal à $U(M^{\lambda}\cap \Phi_\ugot^{-1}(\lambda))$.
\end{enumerate}
\end{prop}

\medskip

On désigne par $\varphi^t_\ugot:M\to M,t\geq 0$, le flot de gradient de la fonction $- f_\ugot$ : 
$$
\frac{d}{dt} \varphi^t_\ugot(m)= -\nabla f_\ugot \left(\varphi^t_\ugot(m)\right).
$$

Comme $f_\ugot$ est propre, $\varphi^t_\ugot$ existe pour tout $t\in [0,\infty[$, et d'après un résultat de Duistermaat \cite{Lerman05}, nous savons que toute 
trajectoire de $\varphi^t_\ugot$ admet une limite lorsque $t\to\infty$. Pour tout 
$m\in M$, posons 
$$
m_\infty:=\lim_{t\to\infty} \varphi^t_\ugot(m)\quad \in\ \crit(f_\ugot).
$$

La construction de la stratification de Kirwan-Ness se fait comme suit. Pour chaque $\lambda\in\Bcal_\ugot$, notons $M_\lambda$ 
l'ensemble des points de $M$ tendant vers $Z_\lambda$ : $M_{\langle\lambda\rangle}:=\{m\in M; m_\infty\in Z_\lambda\}$. De par sa définition même, l'ensemble 
$M_\lambda$ est contenu dans $\{m\in M, f_\ugot(m)\geq \frac{1}{2}\|\lambda\|^2\}$. La stratification de Kirwan-Ness est la décomposition
 \cite{Kir84a}, \cite{Ness84} : 
\begin{equation}\label{eq:KN-stratification}
M=\bigcup_{\lambda\in \Bcal_\ugot} M_{\langle\lambda\rangle}.
\end{equation}

Lorsque $0\in\ugot^*$ appartient à l'image de $\Phi_\ugot$, la strate $M_{\langle 0\rangle}$ correspond au sous-ensemble ouvert dense des points semi-stables analytiques :
\begin{equation}\label{eq:strate-M-0}
M_{\langle 0\rangle}=\left\{m\in M; \ \overline{U_\C\, m}\cap \Phi_\ugot^{-1}(0)\neq\emptyset\right\}.
\end{equation}

Expliquons maintenant la géométrie de $M_{\langle\lambda\rangle}$ pour un type $\lambda$ non nul. Soient $U_\lambda$ est le sous-groupe stabilisateur de $\lambda$, 
et $C_\lambda$ l'union des composantes connexes de $M^{\lambda}$ intersectant $\Phi_\ugot^{-1}(\lambda)$. Alors $C_\lambda$ est une variété de 
Kähler $U_\lambda$-hamiltonienne avec une application moment propre $\Phi_\lambda:=\Phi_\ugot\vert_{C_\lambda}-\lambda$.

La sous-variété complexe de Bialynicki-Birula
$$
C^-_\lambda:=\left\{m\in M, \lim_{t\to\infty} \exp(-it\lambda)\cdot m\ \in C_\lambda\right\}
$$
correspond à l'ensemble des points de $M$ tendant vers $C_\lambda$ sous l'effet du flot de gradient de $-\langle\Phi_\ugot,\lambda\rangle$, 
lorsque $t\to\infty$. La limite du flot définit une projection $C^-_\lambda\to C_\lambda$. Remarquons que $C^-_\lambda$ est invariante 
sous l'action du sous-groupe parabolique $P(\lambda)\subset U_\C$.

Considérons maintenant la stratification de Kirwan-Ness de la variété de Kähler $U_\lambda$-hamiltonienne $C_\lambda$. Soit 
$C_{\lambda,\langle 0\rangle}$ l'ouvert dense de 
$C_\lambda$ correspondant au type $0$ :
$$
C_{\lambda,\langle 0\rangle}=\left\{x\in C_{\lambda};\ \overline{(U_\lambda)_\C\, x}\cap\Phi_\ugot^{-1}(\lambda)\neq \emptyset\right\}.
$$
Soit $C^-_{\lambda,\langle 0\rangle}$ l'ouvert de $C^{-}_\lambda$ égal à l'image réciproque de $C_{\lambda,\langle 0\rangle}$ 
par la projection $C^-_\lambda\to C_\lambda$.

\begin{theorem}[\cite{Kir84a}]\label{Kirwan-stratification}
Soit $(M,\Omega)$ une variété de Kähler $U$-hamiltonienne avec une application moment propre. Pour chaque type non nul $\lambda$, 
$M_{\langle\lambda\rangle}$ est une sous-variété complexe localement fermée  $U_\C$-invariante de $M$, et 
\begin{align*}
U_\C\times_{P(\lambda)}C^-_{\lambda,\langle 0\rangle}&\longrightarrow M_{\langle\lambda\rangle}\\ 
[g,z] &\longmapsto g\cdot z
\end{align*}
est un difféomorphisme holomorphe $U_\C$-équivariant.
\end{theorem}

Kirwan a donné une preuve de ce théorème lorsque $M$ est une variété compacte \cite{Kir84a}. Lorsque $M$ n'est pas compacte mais que l'application moment est propre, 
on pourra trouver une preuve dans \cite{HSS08} (voir également \cite{Woodward11, pep-ressayre-hkkn}).

Les faits standard suivants vont être très importants dans le cadre réel. Soit $\lambda_{\mini}$ la projection orthogonale de 
$0$ sur le polytope convexe fermé $\Delta_\ugot(M)$. 

\begin{prop}\label{prop:strate-ouverte-complexe}
\begin{enumerate}
\item[a)] $\lambda_{\mini}\in \Bcal_\ugot$.
\item[b)]$\|\lambda\| >\|\lambda_{\mini}\|$, $\forall \lambda\in \Bcal_\ugot-\{\lambda_{\mini}\}$.
\item[c)] Si $\lambda\neq \lambda_{\mini}$, alors la strate $M_{\langle\lambda\rangle}$ est d'intérieur vide.
\item[d)] $M_{\langle\lambda_{\mini}\rangle}$ est un ouvert de Zariski de $M$, connexe et $U_\C$-invariant.
\end{enumerate}
\end{prop}

\subsection{Le cadre avec involution}

Dans cette section, nous supposons que notre variété de Kähler $U$-hamiltonienne $(M,\Omega)$ admet une involution anti-holomorphe 
$\tau$ compatible avec une involution anti-holomorphe $\sigma$ sur $U_\C$ (voir la définition \ref{def:U-sigma-kahler}).

Le but de cette section est d'expliquer comment la stratification de 
Kirwan-Ness $M=\bigcup_{\lambda\in \Bcal_\ugot} M_{\langle\lambda\rangle}$ interagit avec les involutions $\sigma$ et $\tau$. Nous 
supposons ici que la variété $M^\tau$ est {\em non vide}.

Comme à la section \ref{sec:U-sigma-hamiltonien}, nous travaillons avec une chambre de Weyl 
$\tgot^*_+\subset \tgot^*$ telle que $\tgot^*_+\cap (\tgot^*)^{-\sigma}$ est un domaine fondamental pour l'action de $K=(U^\sigma)_0$ 
sur $(\ugot^*)^{-\sigma}$. Ainsi, pour tout $\xi\in \tgot^*_+$, l'intersection $U\xi\cap (\ugot^*)^{-\sigma}$ est non-vide si et seulement si 
$\xi\in\tgot^*_+\cap (\tgot^*)^{-\sigma}$.

\begin{definition}
Notons par $\sigma_+ :\tgot^*_+\to\tgot^*_+$ l'involution de la chambre de Weyl définie par les relations : $-\sigma(U\xi)=-U\sigma(\xi)=:U\sigma_+(\xi)$ pour tout $\xi\in\tgot^*_+$.
\end{definition}

Rappelons que $\varphi^t_\ugot:M\to M$ désigne le flot du champ de vecteur $-\nabla f_\ugot$.
 
\begin{prop}\label{prop:fundamental-stratification-involution}
Soit $\Zcal$ une composante connexe de $M^\tau$.
\begin{enumerate}
\item Pour tout $(m,t)\in M\times \R_{\geq 0}$, on a $\tau(\varphi^t_\ugot(m))= \varphi^t_\ugot(\tau(m))$.
\item Si $m\in\Zcal$, alors $\varphi^t_\ugot(m)\in\Zcal$ pour tout $t\geq 0$.
\item Nous avons $\tau(m_\infty)=(\tau(m))_\infty$ pour tout $m\in M$.
\item Si $m\in \Zcal$, alors $m_\infty\in \Zcal$.
\item Pour tout $\lambda\in\Bcal_\ugot$, nous avons $\tau(M_{\langle\lambda\rangle})=M_{\langle \sigma_+(\lambda)\rangle}$.
\item Pour tout $\lambda\in\Bcal_\ugot$ tel que $M_{\langle\lambda\rangle}\cap \Zcal\neq \emptyset$, nous avons
\begin{itemize}
\item[i)] $\Zcal$ intersecte l'ensemble critique $Z_{\lambda}=U(M^{\lambda}\cap \Phi_\ugot^{-1}(\lambda))$.
\item[ii)] $\sigma(\lambda)=-\lambda$.
\item[iii)]$\tau(M_{\langle\lambda\rangle})=M_{\langle\lambda\rangle}$.
\item[iv)]$\tau(C_\lambda)=C_\lambda$ et $\tau(C^{-}_\lambda)=C^{-}_\lambda$
\end{itemize}
\item $\sigma(\lambda_{\mini})=-\lambda_{\mini}$.
\item $M_{\langle\lambda_\mini\rangle}\cap \Zcal$ est un sous-ensemble ouvert dense de $\Zcal$.
\end{enumerate}
\end{prop}

\begin{proof} Le premier point est une conséquence directe du fait que, pour tout $m\in M$, l'application tangente $\T_m\tau :\T_m M\to\T_{\tau(m)}M$ 
envoie le vecteur $\J(\kappa_\ugot(m))$ sur  $\J(\kappa_\ugot(\tau(m)))$. Les points \emph{2.}, \emph{3.} et \emph{4.} découlent du premier.
  
Par définition, $m\in M_{\langle\lambda\rangle} \Leftrightarrow \Phi_\ugot(m_\infty)\in U\lambda$.  Alors, si $m\in M_{\langle\lambda\rangle}$, on a
 $$
 \Phi_\ugot((\tau(m))_\infty)=\Phi_\ugot(\tau(m_\infty))=-\sigma(\Phi_\ugot(m_\infty))\in -U\sigma(\lambda)=U\sigma_+(\lambda),
$$
ce qui implique que $\tau(m)\in M_{\langle \sigma_+(\lambda)\rangle}$. Ainsi l'identité 
$\tau(M_{\langle\lambda\rangle})=M_{\langle \sigma_+(\lambda)\rangle}$ est démontrée. 

Soit $m\in M_{\langle\lambda\rangle}\cap \Zcal$. Grâce au point {\em 4.}, nous savons que 
$$
m_\infty\in Z_{\lambda}\cap \Zcal\subset \Phi_\ugot^{-1}(U\lambda)\cap \Zcal.
$$
Cela montre le point {\em 6. i)}, et de plus, on a  $\Phi_\ugot(m_\infty)\in U\lambda\cap (\ugot^*)^{-\sigma}$.  
Le fait que $U\lambda\cap (\ugot^*)^{-\sigma}$ ne soit pas vide impose au poids 
dominant $\lambda$ la relation $\sigma(\lambda)=-\lambda$ (voir la remarque \ref{rem:orbite-symetrique}). Le point {\em 6. ii)} est démontré et, 
grâce au point \emph{5.},  le point {\em 6. iii)} en découle. 
La variété $C_\lambda$ est égale à la réunion des composantes connexes $C$ de $M^\lambda$ telles que $\lambda\in\Phi_\ugot(C)$. Comme $\sigma(\lambda)=-\lambda$, 
on a $\tau(M^\lambda)=M^\lambda$, e.g. $\tau$ permute les composantes connexes de $M^\lambda$. De plus, la relation $\lambda\in\Phi_\ugot(C)$ est équivalente à 
$\lambda\in\Phi_\ugot(\tau(C))$. On montre ainsi que $\tau(C_\lambda)=C_\lambda$. Un élément $m\in M$ appartient à $C_\lambda^-$ si et seulement si 
la limite $[m]_\lambda:=\lim_{t\to\infty}e^{-it\lambda}m$ existe et appartient à $C_\lambda$. Comme $\tau(e^{-it\lambda}m)=e^{-it\lambda}\tau(m)$, on a 
$\tau([m]_\lambda)=[\tau(m)]_\lambda$, et cette dernière relation montre que $C^{-}_\lambda$ est stable pour l'involution $\tau$.

L'ensemble $M^\tau$ est une sous-variété  lagrangienne de la variété symplectique $(M,\Omega)$, tandis que la strate $M_{\langle\lambda_\mini\rangle}\subset M$ 
est un ouvert de Zariski. Alors, pour n'importe quelle composante connexe $\Zcal$ de $M^\tau$, l'intersection 
$\Zcal\cap M_{\langle\lambda_\mini\rangle}$ est un ouvert dense de $\Zcal$ (voir l'appendice). On en déduit, d'après le point {\em 6.},  que $\sigma(\lambda_{\mini})=-\lambda_{\mini}$. 
\end{proof}

Dans la suite, nous nous intéressons particulièrement au cas où $\lambda_{\mini}\neq 0$. Alors, nous avons une application holomorphe $\tau$-équivariante
$$
p^\R_{\lambda_{\mini}}: U_\C\times_{P(\lambda_{\mini})}C_{\lambda_{\mini}}^{-}\longrightarrow M.
$$
 Le prochain résultat est une réécriture du théorème \ref{Kirwan-stratification}.

\begin{theorem}\label{th:stratification-kirwan-bis}
Supposons $\lambda_{\mini}\neq 0$. Alors, il existe\footnote{$\Ucal:=C_{\lambda_{\mini},\langle 0\rangle}^-$ et $\Vcal= M_{\lambda_{\mini}}$} des sous-ensembles ouverts $\Vcal\subset M$ et 
$\Ucal\subset C^{-}_{\lambda_{\mini}}$, invariants par $\tau$, et satisfaisant
\begin{enumerate}
\item $\Ucal$ est  un ouvert dense, $P(\lambda_{\mini})$-invariant, de $C^-_{\lambda_{\mini}}$.
\item $\Ucal$ intersecte $C_{\lambda_{\mini}}$.
\item $\Vcal$ est  un ouvert  dense, $U_\C$-invariant, de $M$.
\item L'application $p^\R_{\lambda_{\mini}}$ définit un difféomorphisme holomorphe $U_\C\times_{P(\lambda_{\mini})}\Ucal\overset{\sim}{\longrightarrow} \Vcal$.
\end{enumerate}
\end{theorem}

\subsection{Stratification dans le cadre réel}\label{sec:real-stratification}

Dans cette partie, nous expliquons comment (\ref{eq:KN-stratification}) fournit une stratification de la sous-variété $\Zcal\subset M^\tau$. Rappelons que 
l'isomorphisme $j^*:(\ugot^*)^{-\sigma}\simeq \pgot^*$ induit un bijection $\tgot^*_+\cap (\tgot^*)^{-\sigma}{\simeq}_{j^*} \agot^*_+ $.

Soit $\Bcal_\ugot^\Zcal$ l'ensemble formé par les éléments $\lambda\in \Bcal_\ugot$ tels que 
$M_{\langle\lambda\rangle}\cap \Zcal\neq \emptyset$. Nous savons d'après la proposition \ref{prop:fundamental-stratification-involution} 
que $\sigma(\lambda)=-\lambda$ lorsque $\lambda\in\Bcal_\ugot^\Zcal$ : dans ce cas, nous écrivons $\tlambda=j^*(\lambda)\in\agot^*_+$ et
 $$
 \Zcal_{\langle\tlambda\rangle}= M_{\langle\lambda\rangle}\cap \Zcal.
 $$

Soit $\Bcal_\pgot^\Zcal=j^*(\Bcal_\ugot^\Zcal)\subset\agot^*_+$. Nous avons une partition 
$$
\Zcal=\bigcup_{\tlambda\in\Bcal^\Zcal_\pgot} \Zcal_{\langle\tlambda\rangle}
$$ 
en sous-variétés réelles localement fermées, que nous allons interpréter en termes d'action du groupe $G$ sur $\Zcal$ et de la fonction gradient 
$\Phi_\pgot\vert_\Zcal:\Zcal\to\pgot^*$.

Soit $f_\pgot:=\tfrac{1}{2}(\Phi_\pgot,\Phi_\pgot):M\to\R$. Vérifions d'abord que $\Bcal_\pgot^\Zcal$ paramétrise l'ensemble des points 
critiques de la fonction $f_\pgot\vert_\Zcal$. Rappelons que le point \emph{8.} de la proposition \ref{prop:fundamental-stratification-involution} 
nous dit que $0\in \Bcal_\pgot^\Zcal\Leftrightarrow \Phi_\ugot^{-1}(0)\cap\Zcal\neq \emptyset\Leftrightarrow \Phi_\ugot^{-1}(0)\neq \emptyset$.

\begin{lem}
L'ensemble des points critiques de la fonction $f_\pgot\vert_\Zcal$ admet la décomposition
$$
\crit(f_\pgot\vert_{\Zcal})=\bigcup_{\tlambda\in \Bcal^\Zcal_\pgot}K(\Zcal^{\tlambda}\cap \Phi_\pgot^{-1}(\tlambda)).
$$
\end{lem}

\begin{proof}
Nous avons $f_\ugot= \frac{1}{2}(\Phi_{\ugot^\sigma},\Phi_{\ugot^\sigma}) + f_\pgot$, où la fonction  
$\Phi_{\ugot^\sigma}:M\to(\ugot^\sigma)^*$ s'annule sur $\Zcal$. Par conséquent $\crit(f_\pgot\vert_\Zcal)=\crit(f_\ugot)\cap \Zcal$. 
Puisque $\crit(f_\ugot)=\bigcup_{\lambda\in \Bcal_\ugot}U(M^{\lambda}\cap \Phi_\ugot^{-1}(\lambda))$, on obtient
$\crit(f_\pgot\vert_\Zcal)=\bigcup_{\lambda\in \Bcal_\ugot} Z_\lambda\cap \Zcal$. Tout d'abord 
$Z_\lambda\cap \Zcal\neq\empty$ si et seulement si $M_{\langle\lambda\rangle}\cap \Zcal\neq \emptyset$ (voir Proposition 
\ref{prop:fundamental-stratification-involution}). Expliquons maintenant pourquoi 
$$
Z_\lambda\cap \Zcal = K(\Zcal^{\tlambda}\cap \Phi_\pgot^{-1}(\tlambda))
$$
lorsque $\lambda\in \Bcal_\ugot^\Zcal$. L'inclusion $K(\Zcal^{\tlambda}\cap \Phi_\pgot^{-1}(\tlambda))\subset Z_\lambda\cap \Zcal$ est immédiate. Pour démontrer 
l'inclusion réciproque, on considère un point $m\in Z_\lambda\cap \Zcal$: alors $\Phi_\ugot(m)\cdot m =0$ et 
$\Phi_\ugot(m)\in U\lambda\cap (\ugot^*)^{-\sigma}=K\lambda$. Ainsi il existe $k\in K$ tel que $m=km'$ avec $\Phi_\pgot(m')=\tlambda$ et $\tlambda\cdot m'=0$.
\end{proof}

\medskip

Considérons le cas où $0\in\Bcal_\ugot$, i.e. $0\in\Bcal_\pgot^\Zcal$. 

\begin{lem}\label{lem:Z-0}
Si $0\in\Bcal_\pgot^\Zcal$, alors $\Zcal_{\langle 0\rangle}=\left\{z\in \Zcal, \ \overline{G\, z}\cap \Phi_\pgot^{-1}(0)\neq\emptyset\right\}$.
\end{lem}

\begin{proof}
Par définition,  $\Zcal_{\langle 0\rangle}=M_{\langle 0\rangle}\cap \Zcal= \left\{z\in \Zcal,\ z_\infty\in \Phi_\ugot^{-1}(0))\right\}$. 
Ici, $m_\infty=\underset{t\to\infty}{\lim}\varphi^t_\ugot(m)$, où $\varphi^t_\ugot:M\to M$ désigne le flot du champ de vecteurs $-\nabla f_\ugot$. 
Lorsque $z\in \Zcal$, nous savons que $\varphi^t_\ugot(z)\in \Zcal, \forall t\geq 0$, et on voit de de plus que le vecteur tangent 
$$
\frac{d}{dt}\varphi^t_\ugot(z)= -\nabla f_\ugot(\varphi^t_\ugot(z))=-\nabla f_\pgot(\varphi^t_\ugot(z))
$$
appartient au sous-espace $\ggot\cdot\varphi^t_\ugot(z)=\{X\cdot\varphi^t_\ugot(z), X\in\ggot\}$. 
Cela montre que $\varphi^t_\ugot(z)\in Gz$ 
pour tout $t\geq 0$, et donc que $z_\infty\in \overline{Gz},\forall z\in \Zcal$. On voit alors que  
$\Zcal_{\langle 0\rangle}\subset\left\{z\in \Zcal, \ \overline{G\, z}\cap \Phi_\pgot^{-1}(0)\neq\emptyset\right\}$. 
D'autre part, pour tout $z\in \Zcal$, 
$\overline{G\, z}\cap \Phi_\pgot^{-1}(0)\neq\emptyset\Longrightarrow \overline{U_\C\, z}\cap \Phi_\ugot^{-1}(0)\neq\emptyset 
\Longrightarrow z\in \Zcal\cap M_{\langle 0\rangle}=\Zcal_{\langle 0\rangle}$.
Ainsi, la preuve du lemme est complète. 
\end{proof}

\medskip

Soit $\lambda\in\Bcal_\ugot^\Zcal$ un élément non nul et soit $\tlambda=j^*(\lambda)\in \Bcal_\pgot$. L'intersection 
$\Ccal_{\tlambda}:= C_{\lambda}\cap \Zcal$ correspond à l'union des composantes connexes de $\Zcal^{\tlambda}$ intersectant 
$\Phi_\pgot^{-1}(\tlambda)$, et l'intersection $C_{\lambda}^-\cap \Zcal$ correspond à la sous-variété de Bialynicki-Birula {\em réelle}
$$
\Ccal_{\tlambda}^-:=\{z\in \Zcal,  \lim_{t\to\infty} e^{t\tlambda} z\ \in \Ccal_{\tlambda}\}.
$$

De la même manière, $\Ccal_{\tlambda,\langle 0\rangle}:= C_{\lambda,\langle 0\rangle}\cap \Zcal$ est le sous-ensemble 
ouvert et dense de $\Ccal_{\tlambda}$ formé par les éléments $z\in \Ccal_{\tlambda}$ tels que 
$\overline{G_{\tlambda}\, z}\cap \Phi^{-1}_\pgot(\tlambda)\neq \emptyset$ (voir le lemme \ref{lem:Z-0}). 
Enfin, $\Ccal_{\tlambda,\langle 0\rangle}^-:=C_{\lambda,0}^-\cap \Zcal$ est égal au tiré en arrière 
du sous-ensemble ouvert $\Ccal_{\tlambda,\langle 0\rangle}$ par la projection 
$\Ccal_{\tlambda}^-\to \Ccal_{\tlambda}$.

Le sous-groupe parabolique complexe $P(\lambda)=\Pbb(\tlambda)\subset U_\C$  est stable sous l'involution anti-holomorphe 
$\sigma$ et le sous-groupe groupe $\Pbb(\tlambda)\cap G$ est un sous-groupe parabolique réel  de $G$ défini par\footnote{Ici $\tlambda$ est vu 
comme un élément de $\agot$ à travers l'identification $\agot^*\simeq\agot$ fournie par un produit scalaire invariant sur $\ugot_\C$.}
\begin{equation}\label{eq:parabolique-G}
\Pbb(\tlambda)\cap G=\{g\in G, \lim_{t\to\infty}\exp(t\tlambda)g\exp(-t\tlambda)\ \mathrm{existe}\}.
\end{equation}
Puisque $M_{\langle\lambda\rangle}$ est isomorphe à $U_\C\times_{\Pbb(\tlambda)}C^-_{\lambda,0}$, nous pouvons maintenant conclure que la strate 
$\Zcal_{\langle\tlambda\rangle}:=M_{\langle\lambda\rangle}\cap \Zcal$ admet la description suivante.

\begin{prop}\label{prop:description-strata}
Soit $\tlambda\in\Bcal_\pgot^\Zcal$ un élément non nul. L'application $[g,z]\mapsto gz$ induit un difféomorphisme 
\begin{equation}\label{eq:decomposition-reelle}
G\times_{\Pbb(\tlambda)\cap G}\Ccal_{\tlambda,{\langle 0\rangle}}^{-}\overset{\sim}{\longrightarrow} \Zcal_{\langle\tlambda\rangle}.
\end{equation}
\end{prop}

\medskip

L'existence d'une stratification $\Zcal=\bigcup_{\tlambda\in\Bcal_\pgot}\Zcal_{\tlambda}$ où chaque strate est décrite par l'isomorphisme
(\ref{eq:decomposition-reelle}) a  été obtenue par P. Heinzner, G.W. Schwarz et H. St\"{o}tzel dans un cadre beaucoup plus général \cite{HSS08}. 
Néanmoins, nous obtenons dans ce cadre particulier une information cruciale pour la suite : une seule strate a un intérieur non vide, 
c'est la strate ouverte $\Zcal_{\langle\tlambda_{\mini}\rangle}$ attachée à l'élément $\lambda_{\mini}\in\Bcal_\ugot$ de norme minimale.

\medskip

\section{Preuve du théorème d'O'Shea-Sjamaar dans le cadre k\"ahlérien}\label{sec:preuve-OSS}

Soit  $(M,\Omega, \tau)$ une vari\'et\'e de K\"{a}hler $(U,\sigma)$-hamiltonienne. On suppose que la sous-variété $M^\tau$ est non-vide. 
Nous allons démontrer le raffinement suivant du théorème d'O'Shea-Sjamaar.

\begin{theorem}
Nous avons $\Delta_\ugot(M)\bigcap (\tgot^*)^{-\sigma}\underset{j^*}{\simeq} \Delta_\pgot(\Zcal)$ pour toute composante connexe $\Zcal$ de $M^\tau$.
\end{theorem}

Par définition, nous avons $ \Delta_\pgot(\Zcal)\subset j^*(\Delta_\ugot(M)\bigcap (\tgot^*)^{-\sigma})$. Il nous faut donc démontrer que, 
pour tout $\xi\in \Delta_\ugot(M)\bigcap (\tgot^*)^{-\sigma}$, on a $j^*(\xi)\in \Delta_\pgot(\Zcal)$.

Premier cas : $\xi=0$. Supposons que $0\in \Delta_\ugot(M)$. Cela signifie que $\lambda_{\mini}=0$ et d'après la 
proposition \ref{prop:fundamental-stratification-involution}, on sait que $\Zcal$ intersecte le sous-ensemble critique $Z_{0}= \Phi_\ugot^{-1}(0)$. 
On montre ainsi que $0\in \Delta_\pgot(\Zcal)$.

\medskip

Deuxième cas : $\xi\neq 0$. Soit $\xi\in \Delta_\ugot(M)\bigcap (\tgot^*)^{-\sigma}$ non-nul. Nous allons montrer que $j^*(\xi)\in \Delta_\pgot(\Zcal)$ 
grâce à l'astuce de décalage\footnote{En anglais, c'est le \og Shifting Trick\fg.}. On travaille avec la vari\'et\'e de K\"{a}hler $(U,\sigma)$-hamiltonienne
$$
N=M\times (U\xi)^o,
$$
où $(U\xi)^o$ désigne l'orbite coadjointe $U\xi$ avec comme structure k\"ahlérienne l'opposée de la structure KKS 
(voir l'exemple \ref{ex:U-tilde-C-Kahler}). L'application moment 
$\Phi^N_\ugot : N\to \ugot^*$ est définie par la relation $\Phi^N_\ugot(m,g\xi)=\Phi_\ugot(m)-g\xi$, et l'involution anti-holomorphe 
$\tau_N : N\to N$ est définie par
$$
\tau_N(m,g\xi)=(\tau(m),\sigma(g)\xi).
$$
Ainsi, le produit cartésien $\Zcal\times K\xi$ est une composante connexe de la sous variété $N^{\tau_N}$. 

Comme $\xi\in \Delta_\ugot(M)\bigcap (\tgot^*)^{-\sigma}$, on a $0\in \Delta_\ugot(N)$, et d'après le premier cas, 
cela implique que $0\in \Delta_\pgot(\Zcal\times K\xi)$. Nous avons prouvé que $j^{*}(\xi)\in \Phi_\pgot(Z)\cap \agot^*_+=\Delta_\pgot(\Zcal)$.

\section{Construction de paires de Ressayre réelles}\label{sec:construction-RP}

Nous travaillons avec un produit scalaire $U$-invariant $(-,-)$ sur $\ugot_\C$, satisfaisant les conditions de l'appendice. 
Grâce à lui, nous avons un isomorphisme $\xi\in V^*=\hom(V,\R)\to \xi^\flat \in V$ pour n'importe quel sous-espace vectoriel réel $V\subset\ugot_\C$. La commutativité du diagramme suivant est fréquemment utilisée dans ce qui suit:
$$
\xymatrixcolsep{5pc}\xymatrix{
(\ugot^*)^{-\sigma}\ar[d]^{\flat} \ar[r]^{j^*} & \pgot^*\ar[d]^{\flat} \\
\ugot^{-\sigma}          & \pgot\ar[l]^{j} .
}
$$

Grâce au théorème d'O'Shea-Sjamaar, on sait que $\Delta_\pgot(\Zcal)$ est un convexe ferm\'e contenu dans la chambre de Weyl $\agot^*_{+}$. 

Dans toute la suite, $\tilde{a}$ désigne un élément de l'intérieur de la chambre de Weyl $\agot^*_{+}$ qui n'appartient pas \`a $\Delta_\pgot(\Zcal)$. 
On note :
\begin{itemize}
\item $\tilde{a}'\in\agot^*_{+}$ la projection orthogonale de $\tilde{a}$ sur $\Delta_\pgot(\Zcal)$, 
\item $\zeta_{\tilde{a}}\in\agot$ tel que $\zeta_{\tilde{a}}=(\tilde{a}'-\tilde{a})^\flat$.
\end{itemize}

\medskip

On commence avec le résultat classique suivant.

\begin{lem}\label{lem:inegalite-classique}
Soient $a,b\in\tgot^*_+$. Alors, on a $\|ga-b\|\geq \|a-b\|, \, \forall g\in U$, et l'égalité est vraie si et seulement si $ga\in U_{b}\cdot a$.
\end{lem}
\begin{proof}
Il s'agit d'un résultat classique de géométrie hamiltonienne dont nous rappelons brièvement les arguments. Nous avons
$\|ga-b\|^2=\|a\|^2+\|b\|^2-2\phi(ga)$ où $\phi$  est la composante $b$-ième de l'application moment sur $Ua$. 
La fonction $\phi: Ua\to \R$ admet un unique maximum local qui est atteint sur une orbite du sous-groupe stabilisateur $U_b$ 
(voir \cite{Ati82,GS82}). Enfin, il n'est pas difficile de vérifier que le point $a$ appartient à cette orbite.
\end{proof}

Considérons l'application $\Phi_\pgot: M\to\pgot^*$ et le sous-ensemble (non-vide) $\Zcal\cap\Phi_\pgot^{-1}(\tilde{a}')$.

\begin{lem}\label{lem:fibre}
$\Zcal\cap\Phi_\pgot^{-1}(\tilde{a}')$ est contenu dans la sous-variété $\Zcal^{\zeta_{\tilde{a}}}$.
\end{lem}
\begin{proof}
Vérifions tout d'abord que $\|\tilde{a}'-\tilde{a}\|^2$ est la valeur minimale de la fonction \break $\|\Phi_\pgot-\tilde{a}\|^2:\Zcal\to\R$. Si $z\in \Zcal$, on a 
$\Phi_\pgot(z)=k\teta$, avec $k\in K$ et $\teta\in\Delta_\pgot(\Zcal)$, et alors 
$$
 \|\Phi_\pgot(z)-\tilde{a}\|^2=\|k\teta-\tilde{a}\|^2\underset{(1)}{\geq} \|\teta-\tilde{a}\|^2\underset{(2)}{\geq} \|\tilde{a}'-\tilde{a}\|^2.
 $$
 L'inégalité (1) est une conséquence que Lemme \ref{lem:inegalite-classique}, et l'inégalité (2) vient du fait que $\tilde{a}'$ est 
 la projection orthogonale de $\tilde{a}$ sur $\Delta_\pgot(\Zcal)$.
 
Ainsi, si $z\in \Zcal\cap\Phi_\pgot^{-1}(\txi')$, la différentielle de $\|\Phi_\pgot-\txi\|^2:\Zcal\to\R$ s'annule en $z$. Mais 
$$
d\|\Phi_\pgot-\txi\|^2\vert_z=2d\langle\Phi_\pgot,\gamma_{\txi}\rangle\vert_z=2(\gamma_{\txi}\cdot z,-)_\Zcal.
$$
Ainsi, $\zeta_{\tilde{a}}\cdot z=0$. 
\end{proof}

\medskip

Soit $C_{\zeta_{\tilde{a}}}$ l'union des composantes connexes de $M^{\zeta_{\tilde{a}}}$ intersectant $\Zcal\cap\Phi_\pgot^{-1}(\tilde{a}')$. Le résultat suivant est 
l'outil principal pour mettre en évidence des paires de Ressayre réelles sur $\Zcal$.

\medskip

\begin{theorem}\label{theo:construction-RP}
Pour toute composante connexe $C\subset C_{\zeta_{\tilde{a}}}$, le couple $(\zeta_{\tilde{a}}, C)$ constituent une paire de Ressayre réelle sur $\Zcal$.
\end{theorem}

\begin{proof}

Soit $a\in \tgot^*_+\cap(\ugot^*)^{-\sigma}$ tel que $\tilde{a}=j^*(a)$. Puisque $\tilde{a}\notin\Delta_\pgot(\Zcal)$, 
nous savons que $a\notin \Delta_\ugot(M)$. De plus, comme $\tilde{a}$ appartient à l'intérieur de la chambre de Weyl $\agot^*_{+}$, le sous-groupe parabolique 
$\Pbb(-\tilde{a})$ est égal à $\Pbb$ (voir définition \ref{def:parabolique-P}), et la variété complexe $(Ua)^o$ est isomorphe à $U_\C/\Pbb$.

Nous travaillons avec la variété de Kähler $(U,\sigma)$-hamiltonienne 
$$
N:=M\times (Ua)^o.
$$

La sous-variété $\Zcal_N:=\Zcal\times Ka$ est une composante connexe de la sous-variété fixée par l'involution 
$\tau_N$. Soit $\Phi_\ugot^N:N\to \ugot^*$ l'application moment propre relative à l'action de $U$ sur $N$.

Le résultat suivant précise le lemme \ref{lem:fibre}.

\begin{lem}Soit $a'\in\tgot^*_+$ la projection orthogonale de ``$a$'' sur $\Delta_\ugot(M)$.
\begin{itemize}
\item La fonction $\|\Phi^N_\ugot\|:N\to\R$ atteint son minimum sur $U(\Phi_\ugot^{-1}(a')\times\{a\})$.
\item Nous avons $\sigma(a')=-a'$. En d'autres termes, nous avons $\tilde{a}'=j^*(a')$ où $\tilde{a}'$ est la projection orthogonale de $\tilde{a}$ sur $\Delta_\pgot(\Zcal)$.
\end{itemize}
\end{lem}

\begin{proof}
Si $n=(m,ka)\in N$, on écrit  $\Phi_\ugot(m)=g\eta$ avec $\eta\in\Delta_\ugot(M)$. Alors
$$
\|\Phi^N_\ugot(n)\|=\|g\eta-ka\|\geq \|\eta-\xi\|\geq \|a'-a\|,
$$
et, d'après le lemme \ref{lem:inegalite-classique}, l'égalité $\|\Phi^N_\ugot(n)\|= \|a'-a\|$ est vraie si et seulement si 
$\eta=ai'$ et $L^{-1}g\in U_{a}U_{a'}$. Il s'ensuit que l'ensemble critique
$$
Z_{\lambda_{\mini}}=\{n\in N,\, \|\Phi^N_\ugot(n)\|= \|a'-a\|\}\subset \crit\left(\|\Phi^N_\ugot\|^2\right)
$$
est égal $U(\Phi_\ugot^{-1}(a')\times\{a\})$.

Grâce au point \emph{8.} de la proposition \ref{prop:fundamental-stratification-involution}, nous savons que $Z_{\lambda_{\mini}}\cap \Zcal_N$ n'est pas vide. 
Considérons $(g,m)\in U\times \Phi_\ugot^{-1}(\xi')$ tel que $g(m,a)\in Z_{\lambda_{\mini}}\cap \Zcal_N$, alors $gm\in \Zcal$ et 
$\Phi_\ugot(gm)=ga'$, ce qui implique que $\Phi_\ugot(gm)$ appartient à $Ua'\cap (\ugot^*)^{-\sigma}$. 
Il s'ensuit que $\sigma(a')=-\sigma(a')$ (voir la remarque \ref{rem:orbite-symetrique}). 
\end{proof}

\bigskip

Grâce au lemme précédent, nous savons que le type minimal de $N\simeq M\times U_\C/\Pbb$ est l'élément non nul 
$a'-a\in (\tgot^*)^{-\sigma}$ que nous identifions à $\zeta_{\tilde{a}}=\tilde{a}'-\tilde{a}\in \agot^*\simeq \agot$. 
Nous allons maintenant décrire la strate ouverte et dense $N_{\langle\zeta_{\tilde{a}}\rangle}$ de $N$.
 
L'ensemble critique associé au type minimal $\zeta_{\tilde{a}}$ est $U(\Phi_\ugot^{-1}(a')\times\{a\})$. Soit 
$C_{N,\tilde{a}}$ la composante connexe de $N^{\zeta_{\tilde{a}}}$ contenant $\Phi_\ugot^{-1}(a')\times\{a\}$ : nous avons
 $$
C_{N,\tilde{a}}\simeq C_{\tilde{a}}\times (U_{\zeta_{\tilde{a}}})_\C/(U_{\zeta_{\tilde{a}}})_\C\cap \Pbb,
 $$ 
 où $C_{\tilde{a}}$ est la composante connexe de $M^{\zeta_{\tilde{a}}}$ contenant $\Phi_\ugot^{-1}(a')$. La sous-variété complexe de Bialynicki-Birula  est alors
$C_{N,\tilde{a}}^-\simeq C_{\tilde{a}}^-\times \Pbb(\zeta_{\tilde{a}})/\Pbb\cap \Pbb(\zeta_{\tilde{a}})$. 

Afin de terminer la preuve du théorème \ref{theo:construction-RP}, nous devons expliquer le lien entre les deux applications holomorphes suivantes:
$$
q^\R_{\zeta_{\tilde{a}}}: \Pbb\times_{\Pbb\cap \Pbb(\zeta_{\tilde{a}})} C_{\tilde{a}}^-\longrightarrow M
$$
et
$$
p^\R_{\zeta_{\tilde{a}}}: U_\C\times_{\Pbb(\zeta_{\tilde{a}})}\left(C_{\tilde{a}}^-\times\Pbb(\zeta_{\tilde{a}})/\Pbb\cap \Pbb(\zeta_{\tilde{a}})\right)\longrightarrow N.
$$

Le théorème \ref{th:stratification-kirwan-bis} nous permet de voir qu'il existe des sous-ensembles ouverts denses $\Vcal\subset M\times U_\C/\Pbb$ et 
 $\Ucal\subset C_{\tilde{a}}^-\times\Pbb(\zeta_{\tilde{a}})/\Pbb\cap \Pbb(\zeta_{\tilde{a}})$, invariant  pour l'involution $\tau_N$, et tels que

\begin{itemize}
\item $\Ucal$ est  $\Pbb(\zeta_{\tilde{a}})$-invariant, et intersecte $C_{\tilde{a}}\times (U_{\zeta_{\tilde{a}}})_\C/(U_{\zeta_{\tilde{a}}})_\C\cap \Pbb$,
\item $\Vcal$ est  $U_\C$-invariant 
\item la fonction $p^\R_{\gamma_\xi}$ définit un diffeomorphisme $U_\C\times_{\Pbb(\zeta_{\tilde{a}})}\Ucal\overset{\sim}{\longrightarrow} \Vcal$.
\end{itemize}

Puisque $\Vcal\subset M\times U_\C/\Pbb\simeq U_\C\times_{\Pbb} M $ est  $U_\C$-invariant et dense, nous avons 
$\Vcal\simeq U_\C\times_\Pbb\Vcal_0$ où $\Vcal_0\subset M$ est un 
sous-ensemble ouvert, $\Pbb$-invariant,  et dense,  défini par les relations $m\in \Vcal_0\Leftrightarrow (m,[e])\in\Vcal$. De même, nous avons 
$\Ucal\simeq \Pbb(\zeta_{\tilde{a}})\times_{\Pbb\cap \Pbb(\zeta_{\tilde{a}})}\Ucal_0$ où $\Ucal_0\subset C_{\gamma_\xi}^-$ est l'ouvert 
$\Pbb\cap \Pbb(\zeta_{\tilde{a}})$-invariant et dense, défini par les relations $x\in \Ucal_0\Leftrightarrow (x,[e])\in\Ucal$. 
On vérifie facilement que $\Ucal_0$ et $\Vcal_0$ sont invariants sous l'involution $\tau$ et que $\Vcal_0$ a une intersection non vide avec $C_{\tilde{a}}$.

Enfin, nous remarquons que $q^\R_{\zeta_{\tilde{a}}}$ définit un difféomorphisme $\Pbb\times_{\Pbb\cap \Pbb(\zeta_{\tilde{a}})}  \Ucal_0 \simeq \Vcal_0$. Cela peut être facilement vérifié à l'aide du diagramme commutatif
$$
\xymatrixcolsep{5pc}\xymatrix{
U_\C\times_{\Pbb(\zeta_{\tilde{a}})} \Ucal\ar[d]^{p^\R_{\zeta_{\tilde{a}}}} \ar[r]^{\sim} & U_\C\times_{\Pbb} \left(\Pbb\times_{\Pbb\cap \Pbb(\zeta_{\tilde{a}})}  \Ucal_0\right)\ar[d]^{1\times q^\R_{\zeta_{\tilde{a}}}} \\
\Vcal\ar[d] \ar[r]^{\sim}          & U_\C\times_{\Pbb} \Vcal_0\ar[d].\\
M\times U_\C/\Pbb\ar[r]^{\sim} & U_\C\times_{\Pbb} M.
}
$$
La preuve du théorème \ref{theo:construction-RP} est alors complète. 
\end{proof}

\section{Théorème de Coupure Principale}\label{sec:principal-cross-section}

Nous proposons ici un théorème de coupure principale, dans l'esprit de \cite{LMTW}, et qui s'applique à la $G$-variété $\Zcal$.  
Puisque $\Delta_\pgot(\Zcal)$ est un sous-ensemble convexe fermé de $\agot^*_+$, il existe une face ouverte unique $\tsgot$ de $\agot^*_+$ telle que  
\begin{itemize}
\item $\Delta_\pgot(\Zcal)\cap\tsgot\neq \emptyset$,
\item $\Delta_\pgot(\Zcal)$ est contenu dans l'adhérence de $\tsgot$.
\end{itemize}
Remarquons que $\Delta_\pgot(\Zcal)\cap\tsgot$ est dense dans $\Delta_\pgot(\Zcal)$. Rappelons que l'isomorphisme linéaire $j^*:(\tgot^*)^{-\sigma}\to \agot^*$ induit une bijection 
$\tgot^*_+\cap(\tgot^*)^{-\sigma}\simeq \agot^*_+$.

\begin{lem}
\begin{enumerate}
\item Il existe une unique face ouverte $\sgot$  de la chambre de Weyl $\tgot^*_+$ telle que $\sgot\cap(\tgot^{-\sigma})^*{\simeq}_{j^*} \tsgot$.
\item Le sous-groupe stabilisateur $K_{\txi}$ ne dépend pas de $\txi\in\tsgot$ : il est noté $K_{\tsgot}$.
\item La sous-algèbre stabilisatrice $\ggot_{\txi}$ ne dépend pas de $\txi\in\tsgot$ : elle est notée $\ggot_{\tsgot}$. Nous avons la décomposition : 
$\ggot_{\tsgot}=\kgot_{\tsgot}\oplus\pgot_{\tsgot}$ où $\kgot_{{\tsgot}}$ est l'algèbre de Lie de $K_{\tsgot}$.
\end{enumerate}
\end{lem}

\begin{proof} Tout $\xi_o\in\tgot^*_+$ appartient à la face ouverte $\sgot(\xi_o)\subset \tgot^*_+$ définie comme suit : $\xi\in \sgot(\xi_o)$ si et seulement si  
$\xi\in\tgot^*_+$ et $(\beta,\xi)=0\Longleftrightarrow(\beta,\xi_o)=0$ pour tout $\beta\in \Rgot(\ugot,\tgot)$. Il est maintenant facile de vérifier que 
la face $\sgot(\xi_o)$ ne dépend pas de $\xi_o\in (j^*)^{-1}(\tsgot)\subset\tgot^*_+\cap(\tgot^{-\sigma})^*$ : cette face, notée $\sgot$, satisfait la relation 
$\sgot\cap(\tgot^{-\sigma})^* {\simeq}_{j^*} \tsgot$.

Tous les points de la face ouverte $\sgot$ ont le même sous-groupe stabilisateur connexe $U_{\sgot}$. Si l'on prend $\xi\in \sgot\cap(\tgot^{-\sigma})^*$, on voit que
 $U_{\xi}=U_{\sgot}$ est stable sous $\sigma$ et que $K\cap U_{\sgot}=K\cap U_{\xi} $ est égal à $K_{\txi}$ où $\txi=j^*(\xi)\in\tsgot$.
Le deuxième point est réglé et le troisième point est laissé au lecteur. 
\end{proof}

\medskip

Le sous-ensemble  
$$
\Tcal_{\tsgot}=\{z\in\Zcal, \Phi_\pgot(z)\in \tsgot\}.
$$
est l'objet clé de notre théorème de coupure principale.

\begin{theorem}\label{theo:principal-cross-section}
\begin{enumerate}
\item $\Tcal_{\tsgot}$ est une sous-variété $K_{\tsgot}$-invariante de $\Zcal$.
\item L'application $K\times_{K_{\tsgot}}\Tcal_{\tsgot}\to \Zcal$, $[k,y]\mapsto ky$ est un difféomorphisme sur un sous-ensemble ouvert, dense, et  $K$-invariant de $\Zcal$.
\end{enumerate}
\end{theorem}

\begin{proof} Nous considérons le sous-ensemble ouvert suivant de $(\ugot_{\sgot})^*$ :
$$
\Vcal_{\sgot}:=U_{\sgot}\Big\{\xi\in\tgot^*_+, \ U_\xi\subset U_{\sgot}\Big\}.
$$

Le tiré en arrière $Y_{\sgot}=\Phi_\ugot^{-1}(\Vcal_{\sgot})$ est la section symplectique en $\sgot$ \cite{LMTW}. Il s'agit d'une variété symplectique 
$U_{\sgot}$-invariante de $M$ telle que l'application $U\times_{U_{\sgot}} Y_{\sgot}\longrightarrow M, [g,y]\mapsto gy$ définit un 
difféomorphisme sur le sous-ensemble ouvert et dense $U Y_{\sgot}\subset M$.

\begin{lem}
\begin{enumerate}
\item $\Vcal_{{\sgot}}\subset (\ugot_{{\sgot}})^*$ est invariant sous l'application $-\sigma$.
\item La sous-variété $Y_{\sgot}$ est stable sous l'involution $\tau$.
\item L'intersection $Y_{\sgot}\cap\Zcal$ est égale à $\Tcal_{\tsgot}$.
\item L'intersection $U Y_{\sgot}\cap\Zcal$ est égale à $K\Tcal_{\tsgot}$.
\item $K\Tcal_{\tsgot}$ est un ouvert dense de $\Zcal$.
\end{enumerate}
\end{lem}

\begin{proof} Rappelons que $K'$ est la composante connexe du sous-groupe centralisateur $Z_K(\agot)$. Soit $w_0'$ le plus long élément du groupe de Weyl 
$W'=N_{K'}(T)/T$. Alors l'application linéaire $\sigma_+(\xi)=-w_0'\sigma(\xi)$ de $\tgot^*$ préserve la chambre de Weyl $\tgot^*_+$, et pour tout $\xi\in \tgot^*_+$, on a
 $-\sigma(U\xi)=U(\sigma_+(\xi))$. Soit $k'\in K'$ un représentant de $w_0'$. 

Remarquons que le sous-groupe $U_{\sgot}$ est stable sous $\sigma$ puisque $-\sigma$ fixe les éléments de $\sgot\cap(\tgot^*)^{-\sigma}$. 
Prenons $\eta\in \Vcal_{\sgot}$ : alors 
$\eta=g\xi$ où $g\in U_{\sgot}$ et $\xi\in \tgot^*_+$ satisfait $U_\xi\subset U_{\sgot}$. Nous avons 
$-\sigma(\eta)=\sigma(g)k'\sigma_+(\xi)$, où $\sigma(g)k'\in U_{\sgot}$ car $K'\subset U_{\sgot}$. Nous voyons maintenant que le sous-groupe stabilisateur 
$U_{\sigma_+(\xi)}$ est égal à $Ad(k')\circ\sigma(U_\xi)$ : comme $U_\xi\subset U_{\sgot}$, nous obtenons $U_{\sigma_+(\xi)}\subset U_{\sgot}$. 
Le premier point est prouvé et le second découle directement du premier.

L'inclusion $\Tcal_{\tsgot}\subset Y_{\sgot}\cap\Zcal$ est immédiate. Vérifions l'inclusion inverse. Soit $z\in Y_{\sgot}\cap\Zcal$ et 
$\eta=\Phi_\ugot(z)\in \Vcal_{\sgot}\cap (\ugot^*)^{-\sigma}$. En prenant la décomposition $\eta=g\xi$ comme précédemment, 
nous voyons que $\xi\in\tgot^*_+\cap(\tgot^*)^{-\sigma}$ et  $g\xi\in U\xi\cap(\ugot^*)^{-\sigma}=K\xi$ : il existe $k\in K$ tel que $g\xi=k\xi$ ou, en d'autres termes, $g^{-1}k\in U_\xi\subset U_{\sgot}$. 
Mais $g\in U_{\sgot}$ et donc $k\in U_{\sgot}\cap K=K_{\tsgot}$. À ce stade, nous savons que $\Phi_\pgot(k^{-1}z)=j^*(\xi)\in \agot^*_+$. 
D'une part, nous savons que $\xi\in\overline{\sgot}$ car $j^*(\xi)\in \Delta_\pgot(\Zcal)\subset \overline{\tsgot}$. D'autre part, nous savons que 
$U_\xi\subset U_{\sgot}$. Cela montre que $\xi$ appartient à $\sgot$, et donc que $k^{-1}z\in \Tcal_{\tsgot}$. Puisque $k\in K_{\tsgot}$, 
nous pouvons conclure que $z$ appartient à 
$\Tcal_{\tsgot}$.  Le troisième point est réglé.

Grâce au troisième point, nous savons que l'intersection $U Y_{\sgot}\cap\Zcal$ contient $K\Tcal_{\tsgot}$. Prouvons que 
$U Y_{\sgot}\cap\Zcal\subset K\Tcal_{\tsgot}$. Soit $(z,y,u)\in\Zcal\times Y_{\sgot}\times U$ tel que $z=uy$. Nous écrivons $\Phi_\ugot(y)=g\xi$,  
où $g\in U_{\sgot}$ et $\xi\in \tgot^*_+$ satisfait $U_\xi\subset U_{\sgot}$. Nous voyons alors que $\Phi_\ugot(z)=ug\xi$ appartient à $U\xi\cap(\ugot^*)^{-\sigma}=K\xi$ :
il existe $k\in K$ tel que $ug\xi=k\xi$, donc $k^{-1}ug\in U_{\sgot}$. Nous avons prouvé qu'il existe $g'\in U^{\sgot}$ tel que $u=kg'$. L'identité
$z=kg'y$ montre alors que $g'y=k^{-1}z\in Y_{\sgot}\cap\Zcal=\Tcal_{\tsgot}$. Nous avons prouvé que $z\in K\Tcal_{\tsgot}$.  
Finalement $K\Tcal_{\tsgot}=U Y_{\sgot}\cap \Zcal$ est ouvert de $\Zcal$ puisque $U Y_{\sgot}$ est ouvert dans $M$.  La densité de 
$K\Tcal_{{\tsgot}}$ dans $\Zcal$ est démontrée dans le lemme \ref{lem:regular-element}. 
\end{proof}

Nous pouvons maintenant terminer la preuve du théorème \ref{theo:principal-cross-section}. 

L'identité $Y_{\sgot}\cap\Zcal=\Tcal_{\tsgot}$ montre que $\Tcal_{\tsgot}$ correspond à l'union des composantes connexes de la sous-variété $(Y_{\sgot})^\tau$
contenue dans $\Zcal$. Par conséquent, $\Tcal_{\tsgot}$ est une sous-variété de $\Zcal$.

Le dernier point du lemme précédent montre que le difféomorphisme $U\times_{U_{\sgot}}Y_{\sgot}\overset{\sim}{\longrightarrow} U Y_{\sgot}$ induit 
le difféomorphisme $K\times_{K_{\tsgot}}\Tcal_{\tsgot}\overset{\sim}{\longrightarrow}K \Tcal_{\tsgot}$.
\end{proof}

Le polytope convexe fermé $\Delta_\pgot(\Zcal)$ génère un sous-espace affine $\Pi$ de $\R{\tsgot}$. La fonction $\Phi_\pgot$, lorsqu'elle est restreinte à la tranche $\Tcal_{\tsgot}$, définit une application $\Phi_\pgot:\Tcal_{\tsgot}\to \Pi$.

\begin{lem}\label{lem:regular-element}
Soit $x\in \Zcal$ et soit $\Vcal_x$ un voisinage $K$-invariant de $x$ dans $\Zcal$. Il existe $y\in \Vcal_x\cap \Tcal_{\tsgot}$ tel que $y$ soit un élément régulier de l'application $\Phi_\pgot: \Tcal_{\tsgot}\to \Pi$. Cela montre en particulier que $K\Tcal_{\tsgot}$ est dense dans $\Zcal$.
\end{lem}

\begin{proof} Quitte à composer par un élément du groupe $K$, nous pouvons supposer que $\lambda=\Phi_\pgot(x)$ appartient à la chambre de Weyl $\agot^*_+$. 
Le {\em théorème de convexité locale} de Sjamaar (voir \cite{Sjamaar98}[Theorem 6.5] et \cite{OSS}[Theorem 8.2]) nous dit que le polytope moment local 
$\Delta_\pgot(\Vcal_x):=\Phi_\pgot(\Vcal_x)\cap\agot^*_+$ est un voisinage 
de $\lambda$ dans $\Delta_\pgot(\Zcal)$. Ainsi, l'image $\Phi_\pgot(\Vcal_x\cap \Tcal_{\tsgot})$ contient un sous-ensemble ouvert de $\Pi$. Il s'ensuit qu'il existe $y\in\Vcal_x\cap \Tcal_{\tsgot}$ 
tel que $d\Phi_\pgot\vert_{y}:\T_{y}\Tcal_{\tsgot}\to \overrightarrow{\Pi}$ 
est surjective. 
\end{proof}

\section{Stabilisateur générique}\label{sec:stabilisateur-gen}

Le but de cette section est de prouver le résultat suivant

\begin{prop}\label{prop:minimal-stabilizer}
Il existe un sous-espace $\hgot\subset \pgot$ tel que 
\begin{enumerate}
\item $\forall x\in \Zcal, \exists k\in K$ tel que $\mathrm{Ad}(k)(\hgot)\subset \pgot_x$,
\item $\dim(\hgot)=\dim(\pgot_x)$ sur un sous-ensemble ouvert dense de $\Zcal$.
\end{enumerate}
Par conséquent, $\dim_{\pgot}(\Zcal)=\dim(\hgot)$.
\end{prop}

Le sous-espace $\hgot$ est appelé {\em le stabilisateur générique} de l'action infinitésimale $\pgot\circlearrowright\Zcal$.

\begin{coro}\label{coro:stabilisateur-min}
Soit $\hgot_o\subset \pgot$ un sous-espace tel que l'ensemble $K \Zcal^{\hgot_o}$ ait un intérieur non vide dans $\Zcal$. Alors 
$\dim_{\pgot}(\Zcal)\geq \dim(\hgot_o)$.
\end{coro}

Pour simplifier l'exposé, nous utilisons les identifications $\pgot^*\simeq\pgot$ et $\agot\simeq\agot^*$ données par le produit scalaire invariant $(-,-)$ sur $\ugot_\C$ (voir l'annexe). 
Considérons les décompositions orthogonales $\pgot=\agot\oplus\qgot$ et $\agot=\R{\tsgot}\oplus\R{\tsgot}^\perp$. Il en résulte que 
$\pgot_{\tsgot}=\agot\oplus\qgot_{\tsgot}$.

Le polytope convexe fermé $\Delta_\pgot(\Zcal)$ génère un sous-espace affine 
$\Pi$ de $\R{\tsgot}$. Soit $\overrightarrow{\Pi}^\perp$ l'orthogonal de $\overrightarrow{\Pi}$ dans $\R{\tsgot}$.

Pour tout $x\in \Zcal$, nous définissons $\agot_x=\{X\in\agot, X\cdot x=0\}$.

\begin{lem}\label{lem:agot-pi}
\begin{enumerate}
\item Pour tout $x\in \Tcal_{\tsgot}$, on a $\R{\tsgot}^\perp\oplus \qgot_{\tsgot}\subset \pgot_x\subset \pgot_{\tsgot}$.
\item Pour tout $x\in \Tcal_{\tsgot}$, on a $\pgot_x=\agot_x\oplus\qgot_{\tsgot}$ avec $\overrightarrow{\Pi}^\perp\oplus\R{\tsgot}^\perp\subset \agot_x$.
et l'égalité $\overrightarrow{\Pi}^\perp\oplus\R{\tsgot}^\perp= \agot_x$ est vérifiée sur un sous-ensemble ouvert dense de $\Tcal_{\tsgot}$.
\end{enumerate}
\end{lem}

{\em Preuve :} 
Pour tout $\beta\in \pgot$, le champ de vecteurs $z\in\Zcal\mapsto \beta\cdot z$ est le champ de vecteurs gradient de la fonction 
$\langle\Phi_\pgot,\beta\rangle :\Zcal\to \R$. Soit $x\in\Tcal_{\tsgot}$ :  alors $\beta\in \pgot_x$ si et seulement si la différentielle 
$d\langle\Phi_\pgot,\beta\rangle\vert_x :\T_x\Zcal\to \R$ est nulle. Grâce au deuxième point du théorème \ref{theo:principal-cross-section}, 
nous savons que $\T_x\Zcal=\T_x\Tcal_{\tsgot}+\kgot\cdot x$. Pour tout $X\in \kgot$, on a 
$$
d\langle\Phi_\pgot,\beta\rangle\vert_x(X\cdot x)=\langle\Phi_\pgot(x),[\beta,X]\rangle
$$
avec $\Phi_\pgot(x)\in\tsgot$. Ainsi, $d\langle\Phi_\pgot,\beta\rangle\vert_x $ s'annule sur $\kgot\cdot x$ si et seulement 
si $\beta\in\pgot_{\tsgot}$. À ce stade, nous savons que $\beta \in\pgot_x$ si et seulement si $\beta\in \pgot_{\tsgot}$ et 
la différentielle $d\langle\Phi_\pgot,\beta\rangle\vert_x :\T_x\Tcal_{\tsgot}\to \R$ est nulle. La fonction $\Phi_\pgot$, lorsqu'elle 
est restreinte à la sous-variété $\Tcal_{\tsgot}$, prend des valeurs dans $\Pi\subset \R\tsgot$. 
Si nous prenons $\beta\in \R{\tsgot}^\perp\oplus \qgot_{\tsgot}$, la fonction $\langle\Phi_\pgot,\beta\rangle$ 
est constante et égale à zéro sur $\Tcal_{\tsgot}$, donc $\R{\tsgot}^\perp\oplus \qgot_{\tsgot}\subset \pgot_x$ 
pour tout $x\in \Tcal_{\tsgot}$. Le premier point est démontré.

On a $\pgot_x=\agot_x \oplus\qgot_{\tsgot},\forall x\in  \Tcal_{\tsgot}$, et $\beta\in\agot$ appartient à $\agot_x$ 
si et seulement si la différentielle $d\langle\Phi_\pgot,\beta\rangle\vert_x :\T_x\Tcal_{\tsgot}\to \R$ est nulle.  Nous voyons que pour tout 
$\beta\in \overrightarrow{\Pi}^\perp\oplus\R{\tsgot}^\perp$, l'application $\langle\Phi_\pgot,\beta\rangle :\Tcal_{\tsgot}\to \R$ 
est constante : on obtient alors que $\overrightarrow{\Pi}^\perp\oplus\R{\tsgot}^\perp\subset \agot_x$, $\forall x\in \Tcal_{\tsgot}$. 

De plus, l'ensemble $\Tcal_{\tsgot}^0:=\{x\in \Tcal_{\tsgot}, \overrightarrow{\Pi}^\perp\oplus\R{\tsgot}^\perp= \agot_x\}$ coïncide avec 
l'ensemble des éléments réguliers de l'application $\Phi_\pgot: \Tcal_{\tsgot}\to \Pi$ : grâce au lemme \ref{lem:regular-element}, nous savons que $\Tcal_{\tsgot}^0$ est un 
sous-ensemble ouvert dense de la variété $\Tcal_{\tsgot}$. $\Box$

\medskip

Nous terminons maintenant la preuve de la proposition \ref{prop:minimal-stabilizer}. Considérons le sous-espace 
$\hgot= \overrightarrow{\Pi}^\perp\oplus\R{\tsgot}^\perp\oplus\qgot_{\tsgot}\subset\pgot$. Nous avons prouvé dans le lemme \ref{lem:agot-pi} que 
\begin{itemize}
\item $\forall x\in \Tcal_{\tsgot}$, $\hgot\subset \pgot_x$,
\item $\dim(\hgot)=\dim(\pgot_x)$ sur un sous-ensemble ouvert dense de $\Tcal_{\tsgot}$.
\end{itemize}
Comme $K\Tcal_{\tsgot}$ est dense dans $\Zcal$, nous voyons que  $\dim(\hgot)=\dim(\pgot_x)$ sur un sous-ensemble ouvert dense de $\Zcal$.

Soit $x\in\Zcal$. Un voisinage $U$-invariant $\Wcal_x$ de $Ux$ dans $M$ est difféomorphe à $U\times_{U_x}E$, où $E$ est un $U_x$-module complexe 
muni d'une involution $\tau_x$. Par conséquent, un voisinage $K$-invariant $\Vcal_x$ de $Kx$ dans $\Zcal$ est difféomorphe à $K\times_{K_x}E^{\tau_x}$. 
Il s'ensuit que pour tout $y\in \Vcal_x$, il existe $k\in K$ tel que $U_y\subset \mathrm{Ad}(k)(U_x)$, en particulier $\pgot_y\subset \mathrm{Ad}(k)(\pgot_x)$. 
D'après le lemme \ref{lem:regular-element}, il existe $y_o\in \Vcal_x\cap\Tcal_{\tsgot}$ tel que $y_o$ est un élément régulier 
de l'application $\Phi_\pgot: \Tcal_{\tsgot}\to \Pi$ : 
on obtient ainsi $\hgot=\pgot_{y_o}\subset \mathrm{Ad}(k)(\pgot_x)$ pour un certain $k\in K$.

\section{Preuve du th\'eor\`eme \ref{th:real-ressayre-pairs}}\label{sec:preuve-th-real-RP}

Soit $\Delta_\pgot(\Zcal)\subset\agot^*_{+}$ le polytope moment réel d'une composante connexe $\Zcal$ de la sous-variété $M^\tau$ (supposée non vide) d'une 
variété de Kähler $(U,\sigma)$-hamiltonienne  $(M,\Omega,\tau)$.
  
Nous définissons les sous-ensembles convexes suivants de la chambre $\agot^*_{+}$. 
\begin{itemize}
\item $\Delta_{\infrp}$ est l'ensemble des points $\xi\in\agot^*_+$ satisfaisant les inégalités $\langle \xi,\zeta\rangle\geq \langle \Phi_\pgot(\Ccal),\zeta\rangle$, 
pour toute \emph{paire de Ressayre infinitésimale réelle} $(\zeta,\Ccal)$ de $\Zcal$.
\item $\Delta_{\rp}$ est l'ensemble des points $\xi\in\agot^*_+$ satisfaisant les inégalités $\langle \xi,\zeta\rangle\geq \langle \Phi_\pgot(C),\zeta\rangle$,
 pour toute \emph{paire de Ressayre  réelle} $(\zeta,C)$ de $\Zcal$.
\end{itemize}

Si nous ne travaillons qu'avec des paires de Ressayre {\em régulières}, nous définissons de manière similaire les ensembles convexes $\Delta^{^{\reg}}_{\rp}$ et $\Delta^{^{\reg}}_{\infrp}$. 
Par définition, nous avons le diagramme commutatif, où toutes les applications sont des inclusions :
\begin{equation}\label{diagramme-1}
\xymatrix{
\Delta_{\infrp} \ar@{^{(}->}[d] \ar@{^{(}->}[r] & \Delta_{\rp} \ar@{^{(}->}[d]  \\
\Delta^{^{\reg}}_{\infrp} \ar@{^{(}->}[r]          & \Delta^{^{\reg}}_{\rp}\cdot
}
\end{equation}

Le théorème \ref{th:real-ressayre-pairs} est une conséquence du résultat suivant

\begin{theorem}\label{th:real-ressayre-pairs-preuve}
Nous avons les relations:
$$
\Delta_\pgot(\Zcal)=\Delta^{^{\reg}}_{\infrp}=\Delta^{^{\reg}}_{\rp}=\Delta_{\infrp}=\Delta_{\rp}.
$$
\end{theorem}

La preuve se divise en trois parties.

Dans les deux premières, nous prouvons les inclusions $ \Delta_{\rp}\subset\Delta_\pgot(\Zcal)\subset \Delta_{\infrp}$. Au moyen du diagramme \ref{diagramme-1}, 
on peut alors conclure que $\Delta_\pgot(\Zcal)=\Delta_{\infrp}=\Delta_{\rp}$.

Dans la dernière partie, nous prouvons l'inclusion $\Delta^{^{\reg}}_{\rp}\subset \Delta_\pgot(\Zcal)$. Puisque 
$\Delta_\pgot(\Zcal)=\Delta_{\infrp}\subset \Delta^{^{\reg}}_{\infrp} \subset\Delta_{\rp}^{^{\reg}}$, nous obtenons finalement que
 $\Delta_\pgot(\Zcal)=\Delta^{^{\reg}}_{\infrp}=\Delta^{^{\reg}}_{\rp}$. La preuve du théorème \ref{th:real-ressayre-pairs-preuve} sera complète.

\subsection{\'Etape 1 : $\Delta_{\rp}\subset\Delta_\pgot(\Zcal)$}

Soit $\txi_o\in\agot^*_{+}$ qui n'appartient pas à $\Delta_\pgot(\Zcal)$. Le but de cette section est de prouver 
que $\txi_o\notin \Delta_{\rp}$. En d'autres termes, nous allons montrer l'existence d'une Ressayre réelle $(\gamma,C)$ de $\Zcal$ telle 
tque $\langle \txi_o,\gamma\rangle < \langle \Phi_\pgot(C),\gamma\rangle$.

Soit $r>0$ la distance entre $\txi_o$ et $\Delta_\pgot(\Zcal)$, et soit $\txi$ un élément à l'intérieur de la chambre de Weyl $\agot^*_{+}$ 
tel que $\|\txi-\txi_o\|<\frac{r}{2}$ : ainsi, la distance entre $\txi$ et  $\Delta_\pgot(\Zcal)$ est strictement supérieure à $\frac{r}{2}$.

Soit $\txi'$ la projection orthogonale de $\txi$ sur $\Delta_\pgot(\Zcal)$ et soit $\zeta=\txi'-\txi\in\agot^*\simeq \agot$. 
Soit $C$ une composante connexe de $M^{\zeta}$ qui intersecte $\Zcal\cap\Phi_\pgot^{-1}(\txi')$.  
Grâce au théorème \ref{theo:construction-RP}, nous savons que $(\zeta, C)$ est une paire Ressayre réelle de  
$\Zcal$.

En utilisant le fait que $\zeta=\txi'-\txi$, nous calculons 
\begin{eqnarray*}
\langle \txi_o,\zeta\rangle- \langle \Phi_\pgot(\Ccal),\zeta\rangle&=&\langle \txi_o,\zeta\rangle- \langle \txi',\zeta\rangle\\
&=&\langle \txi_o-\txi,\zeta\rangle- \|\zeta\|^2\\
&\leq &- \|\zeta\|\left(\|\zeta\|- \|\txi_o-\txi\|\right)\\
&<&0\, .
\end{eqnarray*}
La dernière inégalité vient du fait que $\|\txi_o-\txi\|<\frac{r}{2}$, tandis que $\|\zeta\|>\frac{r}{2}$ puisque $\|\zeta\|$ représente la distance entre
 $\txi$ and $\Delta_\pgot(\Zcal)$.

\subsection{\'Etape 2 : $\Delta_\pgot(\Zcal)\subset \Delta_{\infrp}$}

Soit $\txi\in\agot^*_{+}$ appartenant à $\Delta_\pgot(\Zcal)$. Le but de cette section est de prouver le résultat suivant

\begin{prop}\label{prop:delta-p-delta-inf-RP}
L'inégalité 
$\langle \txi,\zeta\rangle \geq \langle \Phi_\pgot(\Ccal),\zeta\rangle$ est satisfaite pour toute paire de Ressayre infinitésimale réelle $(\zeta, \Ccal)$ de $\Zcal$. 
\end{prop}

\begin{proof}

Soit $\xi$ l'élément de $\tgot^*_+\cap(\tgot^*)^{-\sigma}$ tel que $\txi=j^*(\xi)$. Considérons la variété de Kähler $(U,\sigma)$-hamiltonienne   $N:=M\times (U\xi)^o\simeq M\times U_\C/P(-\xi)$ et la composante connexe $\Zcal_N:=\Zcal\times K\xi\simeq \Zcal\times G/P(-\xi)\cap G$ de sa partie réelle.

Soit $\Phi_\pgot^N: \Zcal_N\to\pgot^*$ la fonction gradient. Comme $0\in\Delta_\ugot(N)$, nous savons que la strate
$$
(\Zcal_N)_{\langle 0 \rangle}:=\Zcal_N\cap N_{\langle 0 \rangle}=\left\{n\in \Zcal_N,\, \overline{G\, n}\cap (\Phi_\pgot^N)^{-1}(0)\neq \emptyset\right\}
$$
est un ouvert dense $G$-invariant de $\Zcal_N$ (voir le lemme \ref{lem:Z-0}).

Soit $(\zeta, \Ccal)$ une paire de Ressayre infinitésimale réelle de $\Zcal$, et soit $\Ccal_N:= \Ccal\times G_\zeta / G_\zeta\cap P(-\xi)$ la 
composante connexe correspondante de $\Zcal_N^\zeta$. Remarquons que $\Ccal_N$ est invariant sous l'action du sous-groupe stabilisateur 
$G_\zeta$, et que la fonction $n\in\Ccal_N\mapsto \langle \Phi_\pgot(n),\zeta\rangle$ est constante, égale à $\langle \Phi_\pgot(\Ccal_N),\zeta\rangle$.

Soit $\Ccal^-_N:=\{n\in \Zcal_N,  \lim_{t\to\infty} e^{t\zeta} n\ \in \Ccal_N\}$ la sous-variété réelle de Bialynicki-Birula.

\begin{lem}\label{lem:KCinterior}
\begin{enumerate}
\item L'ensemble $G\Ccal^{-}_N$ a un intérieur non vide dans $\Zcal_N$.
 
\item $\Ccal^-_N\cap (\Zcal_N)_{\langle 0 \rangle}\neq \emptyset$.
\end{enumerate}
\end{lem}
\begin{proof}Le deuxième point découle du premier puisque $(\Zcal_N)_{\langle 0 \rangle}$ est un sous-ensemble dense $G$-invariant de $\Zcal_N$.

Soit $x\in \Ccal$ tel que $Y\in\ngot^{\zeta>0}\mapsto Y\cdot x\in (\T_x \Zcal)^{\zeta>0}$ est un isomorphisme. Nous allons montrer 
que $n=(x,[e])\in  \Ccal_N$ appartient à l'intérieur de $G\Ccal^-_N$ en vérifiant que $\ggot\cdot n + \T_n  \Ccal_N^-=\T_n \Zcal_N$. 
Puisque $\T_n \Ccal_N^-=(\T_n \Zcal_N)^{\zeta\leq 0}$, il suffit de s'assurer que $(\T_n \Zcal_N)^{\zeta>0}\subset \ggot\cdot n$. 
Nous avons la décomposition $(\T_n \Zcal_N)^{\zeta>0}=(\T_x \Zcal)^{\zeta>0}\oplus \ggot^{\gamma>0}\cdot [e]$. 
Ainsi, pour tout $v\in (\T_n \Zcal_N)^{\gamma>0}$, il existe $X\in\ggot^{\zeta>0}$ tel que $v-X\cdot(x,[e])\in (\T_x  \Zcal_N)^{\zeta>0}$.
Comme $\ngot^{\zeta>0}\simeq (\T_x  \Zcal)^{\zeta>0}$, il existe donc $Y\in \ngot^{\zeta>0}$ tel que $v-X\cdot(x,[e])=Y\cdot x$.
L'algèbre de Lie $\ngot$ est contenue dans l'algèbre de Lie du sous-groupe parabolique $P(-\xi)$ :
cela implique que $Y\cdot [e]=0$. Enfin, nous avons prouvé que $v=(X+Y)\cdot(x,[e])\in  \ggot\cdot n$. 
\end{proof}

\medskip

On peut maintenant terminer preuve de la proposition \ref{prop:delta-p-delta-inf-RP}.

Soit $n\in \Ccal^-_N\cap (\Zcal_N)_{\langle 0 \rangle}$, et soit $n_\zeta\in\Ccal_N$ la limite $\lim_{t\to\infty}e^{t\zeta} n$. 
Soit $P_G(\zeta):=\Pbb(\zeta)\cap G$ le sous-groupe parabolique de $G$ associé à $\zeta$ (voir (\ref{eq:parabolique-G})). 
Puisque $G=KP_G(\zeta)$, le fait que $n\in (\Zcal_N)_{\langle 0 \rangle}$ signifie que 
$\overline{P_G(\zeta)n}\cap(\Phi^N_\pgot)^{-1}(0)\neq \emptyset$. 
En d'autres termes, $\min_{z\in P_G(\zeta) n}\|\Phi^N_\pgot(z)\|=0$ ce qui implique 
 $0\geq \min_{z\in P_G(\zeta) n}\langle\Phi^N_\pgot(z),\gamma\rangle$.

Considérons maintenant la fonction $t\geq 0\mapsto \langle \Phi^N_\pgot(e^{t\zeta} z),\zeta\rangle$ associée à $z\in P_G(\zeta) n$.  
Puisque  $\frac{d}{dt}\langle\Phi^N_\pgot(e^{\zeta}z),\zeta\rangle= -\|\zeta_{\Zcal_N}\|^2(e^{t\zeta}z)\leq 0$, 
nous avons
\begin{equation}\label{eq:N-ss-gamma}
\langle\Phi^N_\pgot(z),\zeta\rangle \geq \langle\Phi^N_\pgot(e^{t\zeta}z),\zeta\rangle,\qquad \forall t\geq 0.
\end{equation}

Prenons $p\in P_G(\zeta)$ et $z=pn$. Alors, la limite $\lim_{t\to\infty}e^{t\zeta} z$ est égale à $g_\zeta n_\zeta\in \Ccal_N$ où 
$g_\zeta=\lim_{t\to\infty}e^{\zeta} p \, e^{-t\zeta}\in G_\zeta$. Si nous prenons la limite  dans (\ref{eq:N-ss-gamma}) lorsque $t\to\infty$, nous obtenons
$$
\langle\Phi^N_\pgot(z),\zeta\rangle\geq \langle \Phi^N_\pgot(g_\zeta n_\zeta),\zeta\rangle= \langle \Phi^N_\pgot(\Ccal_N),\zeta\rangle=
\langle \Phi_\pgot(\Ccal),\zeta\rangle-\langle \xi,\zeta\rangle,\quad \forall z\in P_G(\zeta) n.
$$
Nous obtenons finalement $0\geq \min_{z\in P_G(\zeta) n}\langle\Phi^N_\pgot(z),\zeta\rangle\geq \langle \Phi_\pgot(\Ccal),\zeta\rangle-\langle \xi,\zeta\rangle$.
\end{proof}

\subsection{\'Etape 3 : $\Delta^{^{\reg}}_{\rp}\subset \Delta_\pgot(\Zcal)$}\label{sec:step-3}

Le but de cette section est de prouver le théorème suivant 
\begin{theorem}
Soit $\xi\in\agot^*_{+}$ satisfaisant les inégalités $\langle \xi,\zeta\rangle\geq \langle \Phi_\pgot(C),\zeta\rangle$, 
pour toute paire de Ressayre réelle régulière $(\zeta,C)$ de $\Zcal$.
Alors $\xi\in\Delta_\pgot(\Zcal)$.
\end{theorem}

Notre argumentation est la suivante : nous allons montrer qu'il existe une collection $(\zeta_i, \Ccal_i)_{i\in I}$ de paires de Ressayre réelles régulières 
pour lesquelles nous avons 
$$
\bigcap_{i\in I}\Big\{\xi\in\agot^*_{+}, \langle \xi,\zeta_i\rangle\geq \langle \Phi(\Ccal_i),\zeta_i\rangle\Big\}=\Delta_\pgot(\Zcal).
$$
L'ensemble $I$ sera fini lorsque $\Zcal$ est compact.

Avant de commencer la description de la collection $(\zeta_i, \Ccal_i)_{i\in I}$, rappelons le fait suivant concernant les éléments admissibles (voir \S 4.2 dans \cite{pep-ressayre}).
\begin{rem}\label{lem:admissible}
Soit $\zeta\in\agot$ un élément rationnel tel que  $K\,\Zcal^\zeta=\Zcal$. Nous avons $\dim_\pgot(\Zcal^\zeta)=\dim_\pgot(\Zcal)$, donc il existe une composante connexe 
$C\subset M^\zeta$ telle que $\dim_\pgot(C\cap\Zcal)=\dim_\pgot(\Zcal)$.
\end{rem}

\subsubsection{La paire de Ressayre  $(\zeta_{\tsgot}, \Ccal_{\tsgot})$}

Soit $\tsgot$ la plus petite face de $\agot^*_+$ telle que $\Delta_\pgot(\Zcal)\subset \overline{\tsgot}$ (voir \S \ref{sec:principal-cross-section}).
Dans cette section, nous  montrons qu'une paire de Ressayre réelle régulière  caractérise le fait que $\Delta_\pgot(\Zcal)$ est contenu dans $\overline{\tsgot}$.

Notons $h_\alpha\in \agot$ la coracine associée à une racine $\alpha\in \Sigma$ :  $h_\alpha$ est l'élément rationnel de 
$[\ggot_\alpha,\ggot_{-\alpha}]\cap\agot$ satisfaisant $\langle \alpha, X\rangle=(h_\alpha,X)$, $\forall X\in\agot$ (voir l'annexe). 
Soit $\Sigma^+\subset \Sigma$ l'ensemble des racines positives associées au choix de la chambre de Weyl $\agot^*_+$ : 
$\xi\in \agot^*_+$ si et seulement si $\langle \xi, h_\alpha\rangle\geq 0$ pour tout $\alpha\in \Sigma^+$. Soit 
$\Sigma^+_{\tsgot}\subset \Sigma^+$ l'ensemble des racines positives orthogonales à $\tsgot$ :
 $\alpha\in \Sigma^+_{\tsgot}$ si $\langle \xi, h_\alpha\rangle = 0$ pour tout $\xi\in\tsgot$.
 
\begin{definition}
Considérons le vecteur rationnel 
$$
\zeta_{\tsgot}:= -\sum_{\alpha\in\Sigma^+_{\tsgot}}h_\alpha \ \in \agot.
$$
\end{definition}

\begin{lem}L'élément $\zeta_{\tsgot}$ satisfait les propriétés suivantes :
\begin{itemize}
\item $\langle\xi,\zeta_{\tsgot}\rangle\leq 0$ pour tout $\xi\in\agot^*_{+}$.
\item Pour tout $\xi\in\agot^*_{+}$, $\langle\xi,\zeta_{\tsgot}\rangle=0$ si et seulement si $\xi\in \bar{\tsgot}$.
\item $\zeta_{\tsgot}$  agit trivialement sur la coupure principale $\Tcal_{\tsgot}$.
\item $\langle\alpha,\zeta_{\tsgot}\rangle< 0$ pour tout $\alpha\in\Sigma^+_{\tsgot}$.
\end{itemize}
\end{lem}

\begin{proof} Les deux premiers points découlent du fait que $\langle\xi,h_\alpha\rangle\geq 0$ pour tout $\xi\in\agot^*_{+}$ 
et toute racine positive $\alpha$. Le troisième point est dû au fait que  $\zeta_{\tsgot}\in \R{\tsgot}^\perp$ 
(voir le lemme \ref{lem:agot-pi}). Pour le dernier point, voir l'annexe. 
\end{proof}


\begin{lem}\label{lem:gamma-s}Soit $C\subset M^{\zeta_{\tsgot}}$ telle que $C\cap \Tcal_{\tsgot}\neq\emptyset$.
\begin{itemize}
\item $\zeta_{\tsgot}$ est un élément admissible.
\item Pour tout $\xi\in\agot^*_+$, l'inégalité $\langle \xi,\zeta_{\tsgot}\rangle\geq \langle \Phi_\pgot(C),\zeta_{\tsgot}\rangle$
est équivalente à $\xi\in\bar{{\tsgot}}$.
\end{itemize}
\end{lem}

\begin{proof}
Nous savons que $\Tcal_{\tsgot}\subset\Zcal^{\zeta_{\tsgot}}$ et que $K\Tcal_{\tsgot}$ est dense dans $\Zcal$. Il s'ensuit que  
$K \Zcal^{\zeta_{\tsgot}}=\Zcal$, donc $\zeta_{\tsgot}$ est un élément admissible (voir Remarque \ref{lem:admissible}).
Nous considérons maintenant l'inégalité $\langle \xi,\zeta_{\tsgot}\rangle\geq \langle \Phi_\pgot(C),\zeta_{\tsgot}\rangle$ 
pour un élément $\xi\in\agot^*_{+}$. Nous remarquons tout d'abord que $\langle \Phi_\pgot(C),\zeta_{\tsgot}\rangle=0$, car 
$C\cap \Tcal_{\tsgot}\neq\emptyset$ et les deux premiers points du lemme précédent nous indiquent que 
$\langle \xi,\zeta_{\tsgot}\rangle\geq 0$ est équivalent à $\xi\in\bar{{\tsgot}}$.
\end{proof}

Prenons $x_o\in \Tcal_{\tsgot}$ tel que $\dim(\pgot_{x_o})=\dim_\pgot(\Zcal)$, et notons $C_{\tsgot}$ 
la composante connexe de $M^{\zeta_{\tsgot}}$ contenant $x_o$.

\begin{prop}\label{prop:gamma-tau}
$(\zeta_{\tsgot}, C_{\tsgot})$ est une paire de Ressayre réelle régulière de $\Zcal$.
\end{prop}

\begin{proof}En utilisant l'identification $\agot\simeq \agot^*$, nous voyons $\zeta_{\tsgot}$ comme un élément rationnel de 
${\tsgot}^*$ orthogonal à ${\tsgot}$. Notons $\txi'=\Phi_\pgot(x)\in\Delta_\pgot(\Zcal)\cap{\tsgot}$, définissons pour tout entier $n\geq 1$, 
l'élément $\txi(n):=\txi'-\frac{1}{n}\zeta_{\tsgot}$. 
Nous remarquons que pour $n$ suffisamment grand
\begin{enumerate}
\item $\txi(n)$ est un élément régulier de la chambre de Weyl $\agot^*_+$,
\item $\txi(n)\notin \Delta_\pgot(\Zcal)$,
\item $\txi'$ est égal à la projection orthogonale de $\txi(n)$ sur $\Delta_\pgot(\Zcal)$.
\end{enumerate}
Nous pouvons donc exploiter les résultats de \S \ref{sec:construction-RP} avec les éléments $\txi(n)$ pour $n>>1$. 
La proposition \ref{theo:construction-RP} et le lemme \ref{lem:gamma-s} nous indiquent que $(\gamma_{\tsgot}, C_{\tsgot})$ 
est une paire de Ressayre réelle. Celle-ci est régulière car il existe $x_o\in C_{\tsgot}$ tel que $\dim(\pgot_{x_o})=\dim_\pgot(\Zcal)$
\end{proof}

\subsubsection{Les paires de Ressayre $(\zeta^\pm_l, C^\pm_l)$}

Soit $\R{\tsgot}\subset\agot^*$ le sous-espace vectoriel rationnel généré par la face ${\tsgot}$. Le polytope convexe fermé $\Delta_\pgot(\Zcal)$ engendre 
un sous-espace affine $\Pi$ de $\R{\tsgot}$. Dans cette section, nous montrons qu'une famille finie de paires de Ressayre réelles décrit le fait que 
$\Delta_\pgot(\Zcal)$, considéré comme un sous-ensemble de $\R{\tsgot}$, est contenu dans le sous-espace affine $\Pi$.

Dans cette section, nous utiliserons les identifications $\pgot^*\simeq\pgot$ et $\agot\simeq\agot^*$ données par le produit scalaire invariant $(-,-)$ (voir l'annexe).

Nous commençons par les décompositions orthogonales $\pgot=\agot\oplus\qgot$ et $\agot=\R{\tsgot}\oplus\R{\tsgot}^\perp$. 
Il s'ensuit que $\pgot_{\tsgot}=\agot\oplus\qgot_{\tsgot}$. Soit $\overrightarrow{\Pi}^\perp$ l'orthogonal de $\overrightarrow{\Pi}$ dans $\R{\tsgot}$.
On a montré au lemme \ref{lem:agot-pi} que l'égalité $\agot_x=\overrightarrow{\Pi}^\perp\oplus\R{\tsgot}^\perp$ est vérifiée sur un sous-ensemble 
ouvert dense de $\Tcal_{\tsgot}$. Cela permet de voir que $\overrightarrow{\Pi}$ est un sous-espace vectoriel rationnel de $\agot$. 
Soit $(\eta_l)_{l\in L}$ une base rationnelle de $\overrightarrow{\Pi}^\perp$. Nous considérons alors les éléments rationnels
$$
\zeta^\pm_l:= \pm \eta_l +\zeta_{\tsgot}\ \in\ \overrightarrow{\Pi}^\perp\oplus\R{\tsgot}^\perp,\quad l\in L.
$$

Grâce au lemme \ref{lem:agot-pi}, nous savons que $\Tcal_{\tsgot}\subset \Zcal^{\zeta^\pm_l}$, $\forall l\in L$.  
Prenons $x_o\in \Tcal_{\tsgot}$ tel que $\dim(\pgot_{x_o})=\dim_\pgot(\Zcal)$, et notons $C_{l}^\pm$ 
la composante connexe de $M^{\zeta^\pm_l}$ contenant $x_o$.

\begin{lem}\label{lem:gamma-S}
L'ensemble des éléments $\xi\in\R{\tsgot}$ satisfaisant les inégalités
\begin{equation}\label{eq:gamma-S}
\langle \xi,\zeta_l^\pm\rangle\geq \langle \Phi_\pgot(C^\pm_l),\zeta^\pm_l\rangle, \quad \forall l\in L
\end{equation}
correspond au sous-espace affine $\Pi$.
\end{lem}

\begin{proof}L'élément $\txi'_o=\Phi_\pgot(x_o)$ appartient à $\Pi$. 
Puisque $\langle\xi,\zeta_{\tsgot}\rangle=0, \forall \xi\in\R{\tsgot}$, les inégalités (\ref{eq:gamma-S}) sont équivalentes à
$\pm\langle \xi-\txi'_o,\eta_l\rangle\geq 0,  \forall l\in L$ : en d'autres termes, $\xi-\txi'_o\in \overrightarrow{\Pi}$, donc
$\xi\in\Pi$.
\end{proof}

\begin{prop}\label{prop:gamma-S}
Pour tout $l\in L$, $(\zeta^\pm_l, C^\pm_l)$ est une paire de Ressayre réelle régulière de $\Zcal$.
\end{prop}

\begin{proof} La preuve suit les mêmes lignes que la preuve de la proposition \ref{prop:gamma-tau}.  L'élément $\txi'_o$ est 
contenu dans l'image $\Phi_\pgot(C^\pm_l)$. Nous considérons les éléments $\txi_l^\pm(n):=\txi'_o-\frac{1}{n}\zeta_l^\pm$ pour $n\geq 1$.
Nous remarquons que pour $n$ suffisamment grand

\begin{enumerate}
\item $\txi_l^\pm(n)$ est un élément régulier de la chambre $\agot^*_+$,
\item $\txi_l^\pm(n)\notin \Delta_\pgot(\Zcal)$,
\item $\txi'_o$ est égal à la projection orthogonale de $\xi_l^\pm(n)$ sur $\Delta_\pgot(\Zcal)$.
\end{enumerate}
Nous pouvons donc exploiter les résultats de \S \ref{sec:construction-RP} avec les éléments $\txi_l^\pm(n)$ pour $n>>1$. Proposition
\ref{theo:construction-RP} et le lemme \ref{lem:gamma-S} nous indiquent que $(\zeta^\pm_l, C^\pm_l)$ 
est une paire de Ressayre réelle $\forall l\in L$. Ces dernières sont régulières car 
$\dim_\pgot(\Zcal)=\dim_\pgot(C^\pm_l\cap \Zcal)$, $\forall l\in L$. 
\end{proof}

\subsubsection{Les paires de Ressayre $(\zeta_F, \Ccal_F)$}

Dans cette section, nous montrons que le polytope $\Delta_\pgot(\Zcal)$, considéré comme un sous-ensemble de l'espace affine $\Pi$, est l'intersection 
du cône $\Pi\cap\agot^*_{+}$ avec un ensemble de demi-espaces de $\Pi$ paramétrés par une famille de paires de Ressayre réelles régulières.

\begin{definition}
Une face ouverte $F$ de $\Delta_\pgot(\Zcal)$ est dite {\em non triviale} si $F\subset{\tsgot}$. 
Nous notons par $\Fcal(\Zcal)$ l'ensemble des faces ouvertes non triviales de $\Delta_\pgot(\Zcal)$.
\end{definition}

Soit $F\in \Fcal(\Zcal)$. Il existe\footnote{Ici, nous considérons $\overrightarrow{\Pi}$ comme un sous-espace de $\agot$ grâce à l'identification $\agot^*\simeq \agot$.} $\eta_F\in\overrightarrow{\Pi}$ tel que l'espace affine généré par $F$ est 
$\Pi_F=\{\xi\in \Pi,\langle\xi,\eta_F\rangle=\langle\xi_F,\eta_F\rangle\}$ pour tout $\xi_F\in F$. Le vecteur $\eta_F$ est choisi de telle sorte que
 $\Delta_\pgot(\Zcal)$ soit contenu dans le demi-espace $\{\xi\in \Pi,\langle\xi,\eta_F\rangle\geq\langle\xi_F,\eta_F\rangle\}$.

Par définition de l'ensemble $\Fcal(\Zcal)$, on a la description suivante du polytope $\Delta_\pgot(\Zcal)$ :
\begin{equation}\label{eq:description-delta-K-M}
\Delta_\pgot(\Zcal)= \bigcap_{F\in \Fcal(\Zcal)}\Big\{\xi\in \Pi,\langle\xi,\eta_F\rangle\geq\langle\xi_F,\eta_F\rangle\Big\}\hspace{2mm}\bigcap\hspace{2mm}\agot^*_{+}.
\end{equation}

La fonction $\langle\Phi_\pgot,\eta_F\rangle: \Tcal_{\tsgot}\to\R$ est localement constante sur $\Tcal_{\tsgot}^{\eta_F}$ et prend sa valeur minimale sur
 $\Phi^{-1}_\pgot(F)\subset \Tcal_{\tsgot}$, donc $\Phi^{-1}_\pgot(F)$ est un sous-ensemble ouvert de la sous-variété $\Tcal_{\tsgot}^{\eta_F}$. La fonction 
$\Phi_\pgot:\Phi^{-1}_\pgot(F)\to F$ est surjective, donc elle admet une valeur régulière $\txi_F'\in F$. Pour tout $x\in \Phi^{-1}_\pgot(\txi'_F)$, 
la différentielle $\T_{x}\Phi_\pgot: \T_{x}\Tcal_{\tsgot}^{\eta_F}\to \overrightarrow{F}$ est surjective, donc
\begin{equation}\label{eq:stabilizer-F}
\agot_{x}= \R\eta_F\oplus \overrightarrow{\Pi}^\perp\oplus\R{\tsgot}^\perp, \quad \forall x\in \Phi^{-1}_\pgot(\txi'_F).
\end{equation}
Comme les sous-espaces vectoriels $\agot_{x}$, $\overrightarrow{\Pi}^\perp$, $\R{\tsgot}^\perp$ sont rationnels, on voit que le vecteur $\eta_F$ peut être pris rationnel.

Nous considérons maintenant les éléments rationnels
$$
\zeta_F=\eta_F+\zeta_{\tsgot},\quad F\in\Fcal(\Zcal).
$$

Soit $C_F$ une composante connexe de $\Zcal^{\zeta_F}$ intersectant $\Phi^{-1}_\pgot(\txi'_F)$ (ce dernier vu comme sous-ensemble de 
$\Phi^{-1}_\pgot(F)\subset \Tcal_{\tsgot}$). L'identité
(\ref{eq:stabilizer-F}) et les résultats de la section \ref{sec:stabilisateur-gen} montrent que $\dim_\pgot(C_F)=\dim_\pgot(\Zcal)+1$.

\begin{prop}\label{prop:gamma-F}
\begin{itemize}
\item Pour tout $F\in \Fcal(\Zcal)$, le couple $(\zeta_F, C_F)$ est une paire de Ressayre réelle régulière sur $\Zcal$.
\item L'ensemble des éléments $\xi\in\Pi\cap \agot^*_{+}$ satisfaisant les inégalités
\begin{equation}\label{eq:gamma-F}
\langle \xi,\zeta_F\rangle\geq \langle \Phi_\pgot(C_F),\zeta_F\rangle, \quad \forall F\in \Fcal(\Zcal)
\end{equation}
correspond à $\Delta_\pgot(\Zcal)$.
\end{itemize}
\end{prop}

\begin{proof} La preuve suit les lignes de la preuve de la proposition \ref{prop:gamma-tau}. 
Considérons les éléments $\txi_F(n):=\txi'_F -\frac{1}{n}\zeta_F$ pour $n\geq 1$.
Nous remarquons que pour $n$ suffisamment grand
\begin{enumerate}
\item $\txi_F(n)$ est un élément régulier de la chambre de Weyl,
\item $\txi_F(n)\notin \Delta_\pgot(\Zcal)$,
\item $\txi'_F$ est la projection orthogonale de $\txi_F(n)$ sur $ \Delta_\pgot(\Zcal)$.
\end{enumerate}
Nous pouvons donc exploiter les résultats de \S \ref{sec:construction-RP} avec les éléments $\txi_F(n)$ pour $n>>1$. La proposition
\ref{theo:construction-RP} nous assure que $(\zeta_F, C_F)$ est une paire de Ressayre réelle sur $\Zcal$. Celles-ci sont régulières car 
les vecteurs $\eta_F$ sont rationnels et que $\dim_\pgot(C_F)=\dim_\pgot(\Zcal)+1$, $\forall F\in \Fcal(\Zcal)$. 

Remarquons tout d'abord que les relations $\langle\xi,\eta_F\rangle\geq\langle\xi_F,\eta_F\rangle$ et 
$\langle \xi,\zeta_F\rangle\geq \langle \Phi_\pgot(C_F),\zeta_F\rangle$ 
sont équivalentes pour tout $\xi\in\Pi$. Ainsi le second point est une conséquence de (\ref{eq:description-delta-K-M}). 
\end{proof}

\chapter{Calcul de Schubert}

Cette section est consacrée au calcul de Schubert \cite{Fulton97,Fulton13,Manivel-98}. Entre autre, nous expliquerons le lien entre les paires de Ressayre et 
la notion de Lévi-mobilité introduite par Belkale et Kumar \cite{BK06}. Dans le cadre avec involution, nous décrirons les cellules de Bruhat qui admettent un point fixe par rapport à l'involution anti-holomorphe.

\section{Cohomologie des vari\'et\'es de drapeaux}\label{sec:cohomology-flag}

Nous travaillons ici avec un groupe compact connexe $U$ muni d'un tore maximal $T$. On note $W=N(T)/T$ le groupe de Weyl. 
Le choix d'une chambre de Weyl $\tgot^*_+$ 
détermine un système de racines positives $\Rgot^+$ et un sous-groupe de Borel $B\subset U_\C$.

 \`A chaque $\gamma\in\tgot$, on associe le sous-groupe parabolique $P(\gamma)$ et la variété projective $\Fcal_\gamma=U_\C/P(\gamma)$ de dimension égale à 
$\dim_\C \Fcal_\gamma=\sharp\{\alpha\in\Rgot, (\alpha,\gamma)>0\}$.

Notons $W^\gamma\subset W$ le sous-groupe qui fixe $\gamma$. Grâce à la décomposition de Bruhat \cite{DKV83, Bump13}, on sait que les $B$-orbites sur 
$\Fcal_\gamma$ sont param\'etrées par le quotient $W/W^\gamma$:
\begin{equation}\label{eq:bruhat}
\Fcal_\gamma=\bigcup_{w\in W/W^\gamma}BwP(\gamma)/P(\gamma). 
\end{equation}
Chaque $B$-orbite $\Xgot^{o}_{w,\gamma}=BwP(\gamma)/P(\gamma)$ est une cellule, dite de Bruhat, qui est isomorphe \`a $\C^{d(w,\gamma)}$, avec 
\begin{equation}\label{eq:dim-cellule-schubert-1}
d(w,\gamma):= \dim_\C \Xgot^{o}_{w,\gamma}=\sharp\{\alpha\in\Rgot^+, (\alpha,w\gamma)>0\}.
\end{equation}

On associe \`a  tout $w\in W/W^\gamma$, la vari\'et\'e de Schubert $\Xgot_{w,\gamma}:=\overline{\Xgot^{o}_{w,\gamma}}$ et 
sa classe de cycle en cohomologie $[\Xgot_{w,\gamma}]\in H^{2c(w,\gamma)}(\Fcal_\gamma,\Z)$, où
\begin{equation}\label{eq:dim-cellule-schubert-2}
c(w,\gamma):=\dim_\C \Fcal_\gamma- \dim_\C \Xgot^{o}_{w,\gamma}=\sharp\{\alpha\in-\Rgot^+, (\alpha,w\gamma)>0\}.
\end{equation}
La décomposition cellulaire (\ref{eq:bruhat}) permet de montrer que la famille $[\Xgot_{w,\gamma}], w\in W/W^\gamma$ 
est une base du $\Z$-module $H^{*}(\Fcal_\gamma,\Z)$. En particulier, le groupe de cohomologie $H^{max}(\Fcal_\gamma,\Z)$ admet 
comme base la classe de cycle d'un point $[pt]$.

Soient $(w_1,\ldots,w_\ell)\in (W/W^\gamma)^\ell$ tels que $\sum_{i=1}^\ell c(w_i,\gamma)= \dim_\C \Fcal_\gamma$. Alors le produit \break
$[\Xgot_{w_1,\gamma}]\cdot \ldots\cdot[\Xgot_{w_\ell,\gamma}]$ appartient à $H^{max}(\Fcal_\gamma,\Z)$ : il existe 
$N(w_1,\ldots,w_\ell)\in\Z$ tels que 
$$
[\Xgot_{w_1,\gamma}]\cdot \ldots\cdot[\Xgot_{w_\ell,\gamma}]= N(w_1,\ldots,w_\ell) [pt].
$$

Nous rappelons le résultat clef suivant qui caractérise les nombres $N(w_1,\ldots,w_\ell)$ comme des nombres d'intersection. 
Celui-ci découle du théorème de Kleiman sur les intersections transverses (\cite[Proposition 3]{BK06}), et 
d'une propriété fondamentale en théorie de l'intersection, à savoir qu'une intersection génériquement transversale de sous-variétés 
de $\Fcal_\gamma$ représente le produit de leurs classes de cycles, car les multiplicités d'intersection sont égales 
à un (\cite[Remark 8.2]{Fulton13}).

\begin{prop}\label{prop:intersection-number-1}Soient $(w_1,\ldots,w_\ell)\in (W/W^\gamma)^\ell$ tels que 
$\sum_{i=1}^\ell c(w_i,\gamma)= \dim_\C \Fcal_\gamma$. Il existe un ouvert 
$\Ucal\subset U_\C^\ell$ Zariski dense, tel que pour tout $(g_1,\ldots,g_\ell)\in U$, $\bigcap_{i=1}^\ell  g_i\Xgot^{o}_{w_i,\gamma}$ est transverse à chaque point 
d'intersection, et 
$$
N(w_1,\ldots,w_\ell)=\sharp \bigcap_{i=1}^\ell  g_i\Xgot^{o}_{w_i,\gamma}=\sharp \bigcap_{i=1}^\ell  g_i\Xgot_{w_i,\gamma},\qquad \forall (g_1,\ldots,g_\ell)\in U.
$$
\end{prop}

La propriété précédente montre que les coefficients de structure de la $\Z$-algèbre $H^{*}(\Fcal_\gamma,\Z)$ sont des entiers naturels.

Nous pouvons généraliser ces résultats en considérant $U$ comme un sous-groupe fermé de $\tU$.  On choisit des tores maximaux 
$T\subset U$, et $\tT\subset \tU$ tels que $T\subset \tT$. \`A chaque $\gamma\in\tgot$, 
on associe les sous-groupe paraboliques $P(\gamma)\subset U_\C$ et $\tP(\gamma)\subset \tU_\C$, les variétés projectives $\Fcal_\gamma=U_\C/P(\gamma)$ 
et $\tFcal_\gamma=\tU_\C/\tP(\gamma)$, et l'immersion 
$$
\iota : \Fcal_\gamma\croc \tFcal_\gamma.
$$
Comme précédemment, le choix d'une chambre de Weyl $\ttgot^*_+$ détermine un système de racines positives $\tRgot^+$ et 
un sous-groupe de Borel $\tB\subset \tU_\C$. Notons $\tW:=N(\tT)/T$ le groupe de Weyl. Les $\tB$-orbites sur $\tFcal_\gamma$ sont 
param\'etrées par le quotient $\tW/\tW^\gamma$: 
$$
\tFcal_\gamma=\bigcup_{\tw\in \tW/\tW^\gamma}\tB\tw\tP(\gamma)/\tP(\gamma).
$$ 

Comme précedemment, on associe \`a  tout $\tw\in \tW/\tW^\gamma$, la cellule de Bruhat $\tXgot^{o}_{\tw,\gamma}:=\tB\tw\tP(\gamma)/\tP(\gamma)$, 
la vari\'et\'e de Schubert $\tXgot_{\tw,\gamma}:=\overline{\tXgot^{o}_{\tw,\gamma}}$ et 
sa classe de cycle en cohomologie $[\tXgot_{\tw,\gamma}]\in H^{2c(\tw,\gamma)}(\tFcal_\gamma,\Z)$. 
Notons $\iota^*: H^{*}(\tFcal_\gamma,\Z)\to H^{*}(\Fcal_\gamma,\Z)$ le morphisme induit par $\iota$. Nous avons 
l'extension suivante de la proposition \ref{prop:intersection-number-1}.

\begin{prop}\label{prop:intersection-number-2}
Soient $(w,\tw)\in W/W^\gamma\times\tW/\tW^\gamma$ tels que $d(w,\gamma)=c(\tw,\gamma)$. Il existe 
$N(w,\tw)\in\N$ caractérisé par l'une des conditions suivantes :
\begin{enumerate}
\item $[\Xgot_{w,\gamma}]\cdot \iota^*[\tXgot_{\tw,\gamma}]= N(w,\tw) [pt]$, dans  $H^{max}(\Fcal_\gamma,\Z)$. 
\item Il existe un ouvert \, $\Ucal\subset U_\C\times \tU_\C$ \, Zariski dense, tel que pour tout $(g,\tg)\in \Ucal$, l'intersection \break
$\iota(g\Xgot_{w,\gamma}^o)\bigcap \tg \Xgot_{\tw,\gamma}^o$ est transverse, et 
$$
N(w,\tw)=\sharp \iota(g\Xgot_{w,\gamma}^o)\bigcap\tg\Xgot_{\tw,\gamma}^o=\sharp\iota(g\Xgot_{w,\gamma})\bigcap \tg\Xgot_{\tw,\gamma}.
$$
\end{enumerate}
\end{prop}

\section{Cohomologie des grassmaniennes}\label{sec:calcul-schubert-grass}

Dans les sections suivantes, nous utiliserons les coefficients de Littlewood-Richardson pour paramétrer certaines inégalités. 
Rappelons leur définition. Une partition $\lambda$ est une suite décroissante d'entiers naturels, nulle à partir d'un certain rang: sa longueur, notée $\leng(\lambda)$, 
correspond au nombre d'éléments non-nuls de cette suite. Soient $\lambda,\mu$ et $\nu$ trois partitions de longueur inférieure à $r\geq 1$. Nous les associons 
aux représentations irréductibles $V_\lambda$, $V_\mu$ et $V_\nu$
du groupe unitaire $\upU_r$. Le coefficient de Littlewood-Richardson $\cc^\lambda_{\mu,\nu}$ peut être caractérisé par la relation
\begin{equation}\label{eq:coeff-LR}
\cc^\lambda_{\mu,\nu}=\dim \left[V_\lambda^*\otimes V_\mu\otimes V_\nu\right]^{U_r}.
\end{equation} 
Le point clef ici est que cette identité ne dépend pas de $r\geq \max\{\leng(\lambda),\leng(\mu),\leng(\nu)\}$. 

Voici une autre formulation de l'identité (\ref{eq:coeff-LR}). Soit $\bigwedge_r=\Z[x_1,\ldots,x_r]^{\Sgot_r}$ l'anneau des 
polynômes symétriques, à coefficients entiers, à $r$ variables. \`A chaque partition $\lambda$ de longueur inférieure à $r\geq 1$, 
on associe le polynôme de Schur $\mathbf{s}_\lambda\in \bigwedge_r$. Alors $(\mathbf{s}_\lambda)_{\leng(\lambda)\leq r}$ est une 
base du $\Z$-module $\bigwedge_r$, et les coefficients de Littlewood-Richardson $\cc^\lambda_{\mu,\nu}$ sont les coefficients de structure de l'anneau 
$\bigwedge_r$:
$$
\mathbf{s}_\mu\mathbf{s}_\nu=\sum_{\lambda, \leng(\lambda)\leq r} \cc^\lambda_{\mu,\nu}\ \mathbf{s}_\lambda.
$$

Nous allons maintenant rappeler comment les coefficients de Littlewood-Richardson apparaissent dans la calcul de Schubert des grassmaniennes.

Dans la suite, on note $[n]$ l'ensemble des entiers naturels $\{j\in\N, 1\leq n\leq k\}$. Rappelons que l'on a deux opérations sur $[n]$: 
$I\mapsto I^c=[n]-I$ et $I\mapsto I^o=\{n+1-i, i\in I\}$.

Soient $n>r\geq 1$. On désigne par $\G(r,n)$ la grassmanienne des sous-espaces vectoriels $F\subset \C^n$ de dimension $r$. 
Pour tout $i\in [n]$, on désigne par $\C^i\subset \C^n$ le sous-espace vectoriel engendré par les $i$-premiers vecteurs de la base canonique.
L'application $g\in GL_n(\C)\to g(\C^r)\in \G(r,n)$ définit un isomorphisme $GL_n(\C)/P\simeq \G(r,n)$ où $P$ est égal au sous-groupe parabolique 
$P(\gamma^r)$, avec 
\begin{equation}\label{eq:gamma-1-etage}
\gamma^r_o=\diag(\underbrace{0,\ldots,0}_{r\ times},\underbrace{1\ldots,1}_{n-r\ times}).
\end{equation}
Le groupe de Weyl $W$ est égal au groupe symétrique $\Sgot_n$, et l'application $\omega \in W\mapsto \omega([r])$ permet 
d'identifier le quotient $W/W^{\gamma^r}$ avec l'ensemble $\Pcal(r,n)$ formé des sous-ensembles $I\subset [n]$ de cardinal $r$.

 Lorsqu'une partition $\mu$ est incluse dans $r\times n-r$ si $n-r\geq \mu_1 \geq \cdots \geq \mu_r \geq 0$. Dans ce cas, on écrit 
$\mu\subset r\times n-r$.  \`A chaque $I=\{i_1<\cdots<i_r\}$ appartenant à $\Pcal(r,n)$, on associe les partitions contenues dans $r\times n-r$:
 \begin{itemize}
 \item $\mu(I)=(i_r-r\geq\cdots\geq i_1-1\geq 0)$,
 \item $\lambda(I)=(n-r+1-i_1\geq\cdots\geq n-i_r)$.
 \end{itemize}
Elles sont liées par la relation $\lambda(I)=\mu(I^o)$.

\`A chaque $I\in \Pcal(r,n)$, on associe la cellule de Bruhat  
$$
\Xgot^o_I:=\{F\in\G(r,n), F\cap \C^{i-1}\neq F\cap \C^{i}\ \mathrm{si\ et\ seulement\ si}\ i\in I\}
$$
et la variété de Schubert $\Xgot_I=\overline{\Xgot^o_I}$.  Faisons une remarque sur les notations utilisées : si $I=\omega([r])$, la variété de Schubert $\Xgot_I$ coincide avec 
la variété $\Xgot_{\omega,\gamma^r}$ introduite à la section \ref{sec:cohomology-flag}.

Un calcul classique \cite{Fulton97} montre que la dimension de $\Xgot^o_I$ est égale à $|\mu(I)| =\sum_{k=1}i_k-k$.
Comme $|\mu(I)| +|\lambda(I)|= r(n-r)=\dim_\C \G(r,n)$, la classe de cycle en cohomologie 
$[\Xgot_I]$ appartient à $H^{2|\lambda(I)|}(\G(r,n),\Z)$ pour tout $I\in \Pcal(r,n)$. Lorsque $I=[r]\subset [n]$, on a $\mu([r])=0$ et la cellule  
$\Xgot^o_{[r]}$ est réduite à un point. La classe $[\Xgot_{[r]}]$ est alors une base du $\Z$-module 
$H^{max}(\G(r,n),\Z)$ que l'on note $[pt]$. Dans la description des polytopes de Kirwan, les inégalités seront paramétrées par des familles $I_1,\ldots, I_s\in \Pcal(r,n)$ telles que 
le produit $\prod_{k=1}^s [\Xgot_{I_k}]$ est égal à $[pt]$: remarquons que c'est le cas seulement si $\sum_{k=1}^s |\mu(I_k)|= s\, r(n-r)$. On a une première propriété (voir \S 3.2.2 dans \cite{Manivel-98}).
\begin{lem} 
Soient $I,J\in\Pcal(r,n)$, tels que $|\mu(I)|+|\mu(J)|= 2r(n-r)$. Alors, les relations suivantes sont vérifiées dans $H^*(\G(r,n),\Z)$ : 
$$
[\Xgot_I]\cdot [\Xgot_{J}]=
\begin{cases}
[pt]\quad \mathrm{si} \quad J=I^o,\\
0\hspace{8mm} \mathrm{sinon}.
\end{cases}
$$
\end{lem}

Pour tout $I\in \Pcal(r,n)$, notons $\sigma_{\lambda(I)}$ la classe de cohomologie $[\Xgot_I]$. Alors 
$H^{*}(\G(r,n),\Z)=\oplus_{\lambda\subset r\times n-r} \ \Z\, \sigma_\lambda$. Rappelons le fait fondamental suivant (pour une preuve, voir \S 3.2.2 dans \cite{Manivel-98}).
\begin{theorem}\label{theo:phi-r}
Soit $n\geq r\geq 1$. L'application $\Phi_{n}:{\bigwedge}_r\longrightarrow H^{*}(\G(r,n),\Z)$ définie par les relations 
$$
\Phi_{n}(\mathbf{s}_\lambda)=
\begin{cases}
\sigma_\lambda\hspace{8mm} \mathrm{si}\quad \lambda_1\leq n,\\
0\hspace{1cm} \mathrm{si}\quad \lambda_1> n.
\end{cases}
$$
est un morphisme d'anneaux.
\end{theorem}

Grâce au théorème \ref{theo:phi-r}, on voit que les coefficients de Littlewood-Richardson admettent une caractérisation au moyen du calcul de Schubert 
(voir aussi \cite[Section 6]{Fulton00}).

\begin{prop}\label{prop:coeff-LR}
Soient $I,J,L\in \Pcal(r,n)$ tels que $|\mu(I)|+|\mu(J)|=|\mu(L)|$. Alors 
$$
[\Xgot_{I^o}]\cdot[\Xgot_{J^o}]\cdot[\Xgot_L]= \cc_{\mu(I),\mu(J)}^{\mu(L)} [pt].
$$
\end{prop}

L'application $I\mapsto I^{o,c}$ définit une bijection $\Pcal(r,n)\simeq \Pcal(n-r,n)$ telle que les classes $[\Xgot_{I}]\in H^{*}(\G(r,n),\Z)$ et 
$[\Xgot_{I^o}]\in H^{*}(\G(n-r,n),\Z)$ correspondent à travers l'isomorphisme $H^{*}(\G(r,n),\Z)\simeq H^{*}(\G(n-r,n),\Z)$ induit par le morphisme
$E\in \G(n-r,n)\mapsto E^\perp\in \G(r,n)$. La proposition \ref{prop:coeff-LR} entraine la propriété suivante.

\begin{coro}\label{coro:coeff-LR-sym}
Soient $I,J,L\in \Pcal(r,n)$ tels que $|\mu(I)|+|\mu(J)|=|\mu(L)|$. Alors $|\mu(I^{o,c})|+|\mu(J^{o,c})|=|\mu(L^{o,c})|$ et 
$\cc_{\mu(I),\mu(J)}^{\mu(L)} =\cc_{\mu(I^{o,c}),\mu(J^{o,c})}^{\mu(L^{o,c})}$. 
\end{coro}

Soient $n+m> n> r\geq 1$. On a un morphisme canonique $\iota_{m} : \G(r,n)\to \G(r,n+m)$ qui envoie $E\subset \C^n$ sur lui-même. Le théorème 
\ref{theo:phi-r} permet de vérifier le fait suivant: pour tout $I'\in \Pcal(r,n+m)$, l'image de la classe $[\Xgot_{I'}]\in H^{*}(\G(r,n+m),\Z)$ par $\iota_{m}^*$ est 
$$
\iota_{m}^*\left([\Xgot_{I'}]\right)=
\begin{cases}
0\hspace{2cm} \mathrm{si}\hspace{7mm} I'\cap [m]\neq\emptyset ,\\
[\Xgot_{I}]\hspace{1cm} \quad \mathrm{avec}\quad  I=\{i'-m,\, i'\in I' \}\quad \mathrm{si}\quad I'\cap [m]=\emptyset.
\end{cases}
$$

Nous terminons cette partie avec une dernière caractérisation des coefficients de Littlewood-Richardson. 
Soient $m>a\geq 1$ et $n>b\geq 1$. On consid\`ere le morphisme 
$$
\iota : \G(a,m)\times \G(b,n)\longrightarrow \G(a+b,m+n) 
$$
qui envoie $(E,F)$ sur $E\oplus F$. Le résultat suivant est classique (voir \cite{MacDonald} ou \cite[Lemma 2.12]{pep-toshi}).

\begin{prop}\label{prop:calcul-grassmanienne-a-b}
Soient $I\in\Pcal(a,m)$, $J\in\Pcal(b,n)$ et $L\in \Pcal(a+b,m+n)$ tels que $|\mu(I)|+|\mu(J)|=|\mu(L)|$. Alors
$$
[\Xgot_{I}]\times [\Xgot_{J}]\cdot \iota^*\left([\Xgot_{L^o}]\right)= \cc_{\mu(I),\mu(J)}^{\mu(L)} [pt]\times [pt].
$$
De plus $\cc_{\mu(I),\mu(J)}^{\mu(L)}=\cc_{\mu(I^{o,c}),\mu(J^{o,c})}^{\mu(L^{o,c})}$\ .
\end{prop}

\medskip

\begin{definition}Pour alléger nos notations, nous écrirons dans les sections suivantes 
$$
\cc_{I,J}^L\quad \in\quad  \N
$$
pour le coefficient de Littlewood-Richardson $\cc_{\mu(I),\mu(J)}^{\mu(L)}$ associé à trois sous-ensembles finis $I,J,L\subset \N-\{0\}$. 
Par définition, $\cc_{I,J}^L=0$ si $|\mu(I)|+|\mu(J)|\neq |\mu(L)|$.
\end{definition}

\section{La Lévi-mobilité}

Nous revenons au contexte de la section \ref{sec:cohomology-flag}. L'ouvert de Zariski $\Ucal\subset U_\C\times \tU_\C$ de la proposition \ref{prop:intersection-number-2} 
est invariant par translation par $U_\C$, donc il est de la forme $\Ucal=\{ (kg,k\tg), k\in U_\C, (g,\tg)\in \Ucal_0\}$ avec 
$\Ucal_0=\{(g,\tg)\in \Ucal, [e]\in \iota(g\Xgot_{w,\gamma}^o)\bigcap\tg\Xgot_{\tw,\gamma}^o\}$. 
Un calcul immédiat montre que $[e]\in \iota(g\Xgot_{w,\gamma}^o)\bigcap\tg\Xgot_{\tw,\gamma}^o$ si et seulement si il existe 
$(p,\tp)\in P(\gamma)\times \tP(\gamma)$ tels que 
$$
g\Xgot_{w,\gamma}^o=pw^{-1}\Xgot_{w,\gamma}^o\quad \mathrm{et}\quad \tg\Xgot_{\tw,\gamma}^o=\tp\tw^{-1}\tXgot_{\tw,\gamma}^o
$$
Ainsi, la proposition \ref{prop:intersection-number-2} peut être reformulée sous la forme suivante.

\begin{prop}\label{prop:intersection-number-3}
Soient $(w,\tw)\in W/W^\gamma\times\tW/\tW^\gamma$ tels que $d(w,\gamma)=c(\tw,\gamma)$. Les conditions 
suivantes sont équivalentes:
\begin{enumerate}
\item $[\Xgot_{w,\gamma}]\cdot \iota^*[\tXgot_{\tw,\gamma}]\neq 0$ dans  $H^{max}(\Fcal_\gamma,\Z)$. 
\item Pour tout $\tp\in \tP(\gamma)$ générique, l'intersection 
\begin{equation}\label{eq:intersection-X-tilde-X}
\tp\left(\iota(w^{-1}\Xgot_{w,\gamma}^o)\right)\bigcap \tw^{-1} \tXgot_{\tw,\gamma}^o
\end{equation} 
est transverse en $[e]$.
\end{enumerate}
\end{prop}

L'espace tangent à $w^{-1}\Xgot_{w,\gamma}^o$ en $[e]$ est égal au sous-espace vectoriel 
$$
T_{w,\gamma}:=\sum_{\alpha\in w^{-1}\Rgot^+, (\alpha,\gamma)>0}(\ugot_\C)_\alpha\quad \mbox{\Large$\subset$} \quad \ugot_\C^{\gamma>0}\simeq T_{[e]}\Fcal_\gamma.
$$
L'égalité $d(w,\gamma)=c(\tw,\gamma)$ signifie que la dimension de la cellule $w^{-1}\Xgot_{w,\gamma}^o$ est égale à la codimension de la cellule 
$\tw^{-1}\tXgot_{\tw,\gamma}^o$. Lorsque cette égalité est satisfaite, le produit $[\Xgot_{w,\gamma}]\cdot \iota^*[\tXgot_{\tw,\gamma}]$ appartient à 
$H^{max}(\Fcal_\gamma,\Z)$, et l'intersection (\ref{eq:intersection-X-tilde-X}) est transverse en $[e]$ si et seulement 
si on a la somme directe 
\begin{equation}\label{eq:transverse-w-gamma}
Ad(\tp)\left(T_{w,\gamma}\right)\bigoplus T_{\tw,\gamma}\bigoplus \tugot_\C^{\gamma\leq 0}=\tugot_\C
\end{equation}

Belkale et Kumar ont introduit la notion suivante \cite{BK06}.

\begin{definition} \label{def:LM}
Le couple $(w,\tw)\in W/W^\gamma\times\tW/\tW^\gamma$ est appelé Lévi-mobile, si, pour $\tilde{\ell}\in \tU^\gamma_\C$
 générique,  l'intersection 
$\tilde{\ell}\left(\iota(w^{-1}\Xgot_{w,\gamma}^o)\right)\bigcap \tw^{-1} \tXgot_{\tw,\gamma}^o$ est transverse en $[e]$.
\end{definition}

Dans le cas particulier où $U$ est inclus diagonalement dans $\tU:=U^s$ pour $s\geq 2$, la définition \ref{def:LM} prend la forme suivante.

\begin{definition} \label{def:LM-1}
 $(w_1,\ldots, w_{s+1})\in (W/W^\gamma)^{s+1}$ est appelé Lévi-mobile, si, pour 
$(\ell_1,\ldots, \ell_{s+1})\in (U^\gamma_\C)^{s+1}$ générique,  l'intersection 
\begin{equation}\label{eq:intersection-transverse}
\ell_1\left(w_1^{-1}\Xgot_{w_1,\gamma}^o\right)\bigcap \cdots \bigcap \ell_{s+1}\left(w_{s+1}^{-1}\Xgot_{w_{s+1},\gamma}^o\right)
\end{equation}
est transverse en $[e]$.
\end{definition}

Rappelons que l'intersection (\ref{eq:intersection-transverse}) est transverse si et seulement si la projection
$$
\ugot_\C^{\gamma>0}\longrightarrow \bigoplus_{i=1}^{s+1} \ugot_\C^{\gamma>0}/T_{w_i,\gamma}
$$
est surjective.

Le prochain résultat permettra de montrer que la Lévi-mobilité est caractérisé par un invariant numérique.

\begin{lem}\label{lem:levi-condition-numerique}
Soit $(w,\tw)\in W/W^\gamma\times\tW/\tW^\gamma$ tels que $d(w,\gamma)=c(\tw,\gamma)$. Si $(w,\tw)$ est  Lévi-mobile, alors 
\begin{equation}\label{eq:invariant-levi-mobile}
\tr(w\gamma \circlearrowright \ngot^{w\gamma>0})=\tr(\tw\gamma \circlearrowright \tngot_{-}^{\tw\gamma>0}).
\end{equation}
\end{lem}

\begin{proof}
Lorsque $(w,\tw)$ est  Lévi-mobile, l'identité 
$Ad(\tilde{\ell})\left(T_{w,\gamma}\right)\bigoplus T_{\tw,\gamma}\bigoplus \tugot_\C^{\gamma\leq 0}=\tugot_\C$, qui est valable pour 
$\tilde{\ell}\in \tU^\gamma_\C$ générique, implique que 
\begin{equation}\label{eq:trace-A}
\tr\left(\gamma \circlearrowright T_{w,\gamma}\right)+ \tr\left(\gamma \circlearrowright T_{\tw,\gamma}\right)
= \tr(\gamma \circlearrowright \tugot_\C^{\gamma> 0}).
\end{equation} 
Comme $Ad(w^{-1})\left( \ngot^{w\gamma>0}\right)=T_{w,\gamma}$ et  
$Ad(\tilde{w}^{-1})\left( \tngot^{\tw\gamma>0}\right)=T_{\tw,\gamma}$, on a  
\begin{equation}\label{eq:trace-B}
\tr\left(\gamma \circlearrowright T_{w,\gamma}\right)+\tr\left(\gamma \circlearrowright T_{\tw,\gamma}\right)=\tr(w\gamma \circlearrowright \ngot^{w\gamma>0})+ \tr\left(\tw\gamma \circlearrowright \tngot^{\tw\gamma>0}\right).
\end{equation}
Maintenant on utilise que 
\begin{equation}\label{eq:trace-C}
\tr\left(\gamma \circlearrowright \tugot_\C^{\gamma> 0}\right)=
\tr\left(\tw\gamma \circlearrowright \tugot_\C^{\tw\gamma> 0}\right)=\tr\left(\tw\gamma \circlearrowright \tngot^{\tw\gamma>0}\right)
+\tr\left(\tw\gamma \circlearrowright \tngot_-^{\tw\gamma>0}\right).
\end{equation}
Alors, (\ref{eq:invariant-levi-mobile}) et une conséquence de (\ref{eq:trace-A}), (\ref{eq:trace-B}) et (\ref{eq:trace-C}).
\end{proof}

\medskip

Nous montrons maintenant le lien entre la Lévi-mobilité et les paires de Ressayre. Pour cela, nous considérons $N:=\tU_C$ 
comme une variété de K\"ahler $\tU\times U$-hamiltonienne (voir l'exemple \ref{ex:U-tilde-C-Kahler}). A tout $\gamma\in\tgot$ 
et $\mathbf{w}:=(w,\tw)\in W/W^\gamma\times\tW/\tW^\gamma$, on associe le vecteur 
$\gamma_{\mathbf{w}}:=(\tw\gamma,w\gamma)$.

La sous-variété $N^{\gamma_{\mathbf{w}}}:=\{n\in N,\ \gamma_{\mathbf{w}}\cdot n= 0\}$ est connexe, égale à 
$C_{\gamma,\mathbf{w}}:=\tw \tU_\C^\gamma w^{-1}$. La sous-variété de Bia\l ynicki-Birula associée à 
$C_{\gamma,\mathbf{w}}$ est $C_{\gamma,\mathbf{w}}^-:= \tilde{w}\tP(\gamma)w^{-1}$
où $\tP(\gamma)$ est le sous-groupe parabolique de $\tU_\C$ associé à $\gamma$ (voir (\ref{eq:P-gamma})). Le sous-groupe parabolique de 
$\tU_\C\times U_\C$ associé au poids $\gamma_{\mathbf{w}}$ est 
$$
P(\gamma_{\mathbf{w}}):=\mathrm{Ad}(\tilde{w})(\tilde{P}(\gamma))\times \mathrm{Ad}(w)(P(\gamma))=\tP(\tw\gamma)\times P(w\gamma).
$$

Soient $B\subset U_\C$ et $\tB\subset \tU_\C$ les sous-groupes de Borel associés au choix de la chambre de Weyl $\tgot^*_{+}$
et $\ttgot^*_+$. Nous pouvons maintenant considérer la fonction holomorphe
$$
q_{\gamma,\mathbf{w}}: X_{\gamma,\mathbf{w}}:=(\tilde{B}\times B)\times_{(\tilde{B}\times B)\cap P(\gamma_{\mathbf{w}})} C_{\gamma,\mathbf{w}}^- \longrightarrow  \tU_\C
$$
qui envoie $[\tilde{b},b;\tilde{w}\tilde{p}w^{-1}]$ vers $\tilde{b}\tilde{w}\tilde{p}w^{-1}b^{-1}$. 

Pour tout $\tp\in\tP(\gamma)$, l'espace tangent $\T_{[\tilde{e},e;\tilde{w}\tilde{p}w^{-1}]}X_{\gamma,\mathbf{w}}$ s'identifie avec 
$\tngot^{\tw\gamma>0}\times \ngot^{w\gamma>0}\times \tugot_\C^{\gamma\leq 0}$, et l'application 
$\T_{[\tilde{e},e;\tilde{w}\tilde{p}w^{-1}]}q_{\gamma,\mathbf{w}}$ est un isomorphisme si et seulement si 
\begin{equation}\label{eq:application-tangente-q-w-gamma}
\tngot^{\tw\gamma>0}\bigoplus Ad(\tilde{w}\tilde{p}w^{-1})\left( \ngot^{w\gamma>0}\right)\bigoplus Ad(\tilde{w})\left(\tugot_\C^{\gamma\leq 0}\right)= \tugot_\C
\end{equation}
Comme $Ad(w^{-1})\left( \ngot^{w\gamma>0}\right)=T_{w,\gamma}$ et  $Ad(\tilde{w}^{-1})\left( \tngot^{\tw\gamma>0}\right)=T_{\tw,\gamma}$, on voit que 
les sommes directes (\ref{eq:application-tangente-q-w-gamma}) et (\ref{eq:transverse-w-gamma}) sont équivalentes.

\begin{prop}\label{prop:RP-exemple-fondamental-1}
Les conditions suivantes sont \'equivalentes:
\begin{enumerate}\setlength{\itemsep}{5pt}
\item $(\gamma_{\mathbf{w}},C_{\gamma,\mathbf{w}})$ est une paire de Ressayre algébrique de $N$.
\item  $(\gamma_{\mathbf{w}},C_{\gamma,\mathbf{w}})$ est une paire de Ressayre de $N$.
\item  $(\gamma,w,\tw)$ satisfait les points suivants :
\begin{itemize}
\item[$(A_1)$] $d(w,\gamma)=c(\tw,\gamma)$.
\item[$(A_2)$] $\tr(w\gamma \circlearrowright \ngot^{w\gamma>0})=\tr(\tw\gamma \circlearrowright \tngot_{-}^{\tw\gamma>0})$.
\item[$(A_3)$] Pour $\tg\in\tU_\C$ générique, l'intersection $\tg\iota(\Xgot^o_{w,\gamma})\bigcap \tXgot^o_{\tw,\gamma}$ est réduite à un singleton.
\end{itemize}
\item $(\gamma,w,\tw)$ satisfait les points suivants :
\begin{itemize}
\item[$(B_1)$] $[\Xgot_{w,\gamma}]\cdot \iota^*[\tXgot_{\tw,\gamma}]= [pt]$\ dans  $H^{max}(\Fcal_\gamma,\Z)$. 
\item[$(B_2)$] $(w,\tw)$ est Lévi-mobile.
\end{itemize}
\item $(\gamma,w,\tw)$ satisfait $(A_2)$ et $(B_1)$.
\end{enumerate}
\end{prop}

\begin{proof}
L'implication $1.\Longrightarrow 2.$ est immédiate.

D'après la proposition \ref{prop:caracteriser-RP}, $(\gamma_{\mathbf{w}},C_{\gamma,\mathbf{w}})$ est une paire de Ressayre si et seulement si elle satisfait trois conditions. 
La première condition est l'égalité 
$$
\dim \tngot^{\tw\gamma>0} + \dim \ngot^{w\gamma>0} = \mathrm{rank} \left(TN\vert_{C_{\gamma,\mathbf{w}}}\right)^{\gamma_{\mathbf{w}}>0}
$$
Un calcul direct montre que cette identité est équivalente à $(A_1)$.

La deuxième condition est l'égalité 
\begin{equation}\label{eq:trace-RP-exemple-fondamental}
\tr\left(\tw\gamma \circlearrowright \tngot^{\tw\gamma>0}\right)+ \tr\left(w\gamma \circlearrowright \ngot^{w\gamma>0}\right)
=\tr\left(\gamma_{\mathbf{w}} \circlearrowright \left(TN\vert_{C_{\gamma,\mathbf{w}}}\right)^{\gamma_{\mathbf{w}}>0}\right).
\end{equation}
Un calcul direct donne 
$$
\tr\left(\gamma_{\mathbf{w}} \circlearrowright \left(TN\vert_{C_{\gamma,\mathbf{w}}}\right)^{\gamma_{\mathbf{w}}>0}\right)=
\tr\left(\gamma \circlearrowright \tugot_\C^{\gamma>0}\right)=
 \tr\left(\tw\gamma \circlearrowright \tugot_\C^{\tw\gamma>0}\right).
$$
On voit alors que (\ref{eq:trace-RP-exemple-fondamental}) est équivalente à la relation $(A_2)$.

La dernière condition de  la proposition \ref{prop:caracteriser-RP} concerne 
$Z_{\gamma,\mathbf{w}}:=\left\{\tg\in\tU_\C, \ \sharp q_{\gamma,\mathbf{w}}^{-1}(\tg)\,=\,1\, \right\}$ : celui-ci doit
contenir un ouvert dense. Considérons l'application 
$$
\pi_{\gamma,\mathbf{w}}: X_{\gamma,\mathbf{w}} \longrightarrow  \tFcal_\gamma\times \Fcal_{\gamma}
$$
qui envoie $[\tilde{b},b;\tilde{w}\tilde{p}w^{-1}]$ vers $([\tilde{b}\tw], [bw])$. Il est facile de vérifier que 
$$
(q_{\gamma,\mathbf{w}}, \pi_{\gamma,\mathbf{w}}) : X_{\gamma,\mathbf{w}} \longrightarrow  \tU_\C\times \tFcal_\gamma\times \Fcal_{\gamma}
$$
est une application injective, donc l'application $\pi_{\gamma,\mathbf{w}}$ est injective sur chaque fibre $q_{\gamma,\mathbf{w}}^{-1}(\tg)$. Nous vérifions maintenant directement que 
\begin{equation}\label{eq:identification-fibre}
\pi_{\gamma,\mathbf{w}}\left(q_{\gamma,\mathbf{w}}^{-1}(\tg)\right)=\Big\{(\tilde{x},x)\in \tXgot^o_{\tw,\gamma}\times \Xgot^o_{w,\gamma}, \ \tg\iota (x)=\tilde{x}\Big\}\simeq 
\tg\iota(\Xgot^o_{w,\gamma})\bigcap \tXgot^o_{\tw,\gamma}.
\end{equation}
Nous avons montré que $Z_{\gamma,\mathbf{w}}$  correspond au sous-ensemble formé par les éléments $\tg\in\tU_\C$, tels que $\tg\iota(\Xgot^o_{w,\gamma})\bigcap \tXgot^o_{\tw,\gamma}$ est un singleton. On a ainsi montré l'équivalence $2.\Longleftrightarrow 3.$.

L'implication $4.\Longrightarrow 3.$ est la conséquence de deux faits. Tout d'abord, les conditions $(A_1)$ et $(A_3)$ sont équivalentes à $(B_1)$ (voir la proposition 
\ref{prop:intersection-number-2}), et ensuite la condition $(B_2)$ implique $(A_2)$ d'après le lemme \ref{lem:levi-condition-numerique}.

Pour l'implication $3.\Longrightarrow 4.$, il suffit d'expliquer pourquoi le point $3.$ implique que $(w,\tw)$ est Lévi-mobile. Si le point $3.$ est satisfait, 
alors $(\gamma_{\mathbf{w}},C_{\gamma,\mathbf{w}})$ est une paire de Ressayre, et cela implique qu'il 
existe $\tw\tilde{\ell}w\in C_{\gamma,\mathbf{w}}$ tel que l'application tangente $\T_{[\tilde{e},e;\tilde{w}\tilde{\ell}w^{-1}]}q_{\gamma,\mathbf{w}}$ 
est un isomorphisme. Dans la discussion précédant la proposition \ref{prop:RP-exemple-fondamental-1}, on a expliqué que ce dernier fait est 
équivalent à ce que $\tilde{\ell}$ satisfasse (\ref{eq:transverse-w-gamma}), c'est à dire que 
$(w,\tw)$ est Lévi-mobile.

L'équivalence $3.\Longleftrightarrow 5.$ est due au fait que les conditions $(A_1)$ et $(A_3)$ sont équivalentes à $(B_1)$.

Le point 4. implique qu'il existe un ouvert de Zariski $\Vcal\subset \tU_\C$ tel que 
$q_{\gamma,\mathbf{w}}: X_{\gamma,\mathbf{w}} \longrightarrow  \tU_\C$ définisse un difféomorphisme $q_{\gamma,\mathbf{w}}^{-1}(\Vcal)\simeq \Vcal$. 
Cela implique que $\deg(q_{\gamma,\mathbf{w}})=1$. D'autre part l'hypothèse de Lévi-mobilité implique que $\Vcal\cap C_{\gamma,\mathbf{w}}\neq\emptyset$.
Nous venons donc de vérifier que $4.\Longrightarrow 1.$
\end{proof}

\begin{rem}\label{rem:degre-cohomologie}
Lorsque $d(w,\gamma)=c(\tw,\gamma)$, l'identité (\ref{eq:identification-fibre}) montre que la cardinal générique de la fibre du morphisme $q_{\gamma,\mathbf{w}}$ est égal 
au cardinal générique des intersections $\tg\iota(\Xgot^o_{w,\gamma})\bigcap \tXgot^o_{\tw,\gamma}$. Cela montre que 
$q_{\gamma,\mathbf{w}}$ est un morphisme de type fini si et seulement si $[\Xgot_{w,\gamma}]\cdot \iota^*[\tXgot_{\tw,\gamma}]= N(w,\tw) [pt]$ avec $N(w,\tw)>0$, et 
dans ce cas le degré de $q_{\gamma,\mathbf{w}}$ est égal à $N(w,\tw)$.
\end{rem}

On a résultat similaire pour les paires de Ressayre infinitésimales (avec le m\^eme type de preuve).

\begin{prop}\label{prop:RP-exemple-fondamental-2}
Les conditions suivantes sont \'equivalentes:
\begin{enumerate}\setlength{\itemsep}{5pt}
\item $(\gamma_{\mathbf{w}},C_{\gamma,\mathbf{w}})$  est une paire de Ressayre de type fini.
\item $(\gamma_{\mathbf{w}},C_{\gamma,\mathbf{w}})$  est une paire de Ressayre infinitésimale.
\item  $(\gamma,w,\tw)$ satisfait les points suivants :
\begin{itemize}
\item[$(A_1)$] $d(w,\gamma)=c(\tw,\gamma)$.
\item[$(A_2)$] $\tr(w\gamma \circlearrowright \ngot^{w\gamma>0})=\tr(\tw\gamma \circlearrowright \tngot_{-}^{\tw\gamma>0})$.
\item[$(A'_3)$] Pour $\tg\in\tU_\C$ générique, l'intersection $\tg\iota(\Xgot^o_{w,\gamma})\bigcap \tXgot^o_{\tw,\gamma}$ est non-vide.
\end{itemize}
\item $(\gamma,w,\tw)$ satisfait les points suivants :
\begin{itemize}
\item[$(B'_1)$] $[\Xgot_{w,\gamma}]\cdot \iota^*[\tXgot_{\tw,\gamma}]= k[pt]$, avec $k\geq 1$,\, dans  \,$H^{max}(\Fcal_\gamma,\Z)$. 
\item[$(B_2)$] $(w,\tw)$ est Lévi-mobile.
\end{itemize}
\item  $(\gamma,w,\tw)$ satisfait $(A_1)$ et $(B_2)$
\item $(\gamma,w,\tw)$ satisfait $(A_2)$ et $(B'_1)$.
\end{enumerate}
\end{prop}

Dans le cas où $U$ est inclus diagonalement dans  $\tU=U^s$,  les propositions \ref{prop:RP-exemple-fondamental-1} et \ref{prop:RP-exemple-fondamental-2} 
fournissent le résultat suivant.

\begin{prop}\label{prop:RP-exemple-fondamental-3}
Soit $(w_1,\cdots, w_{s+1})\in (W/W^\gamma)^{s+1}$ tel que $\sum_{i=1}^{s+1} d(w_i,\gamma)= s\dim_\C(\Fcal_\gamma)$. 
Les conditions suivantes sont \'equivalentes:
\begin{enumerate}\setlength{\itemsep}{5pt}
\item $(w_1,\cdots, w_{s+1})$ est Lévi-mobile.
\item $(w_1,\cdots, w_{s+1})$ satisfait les points suivants :
\begin{itemize}\setlength{\itemsep}{5pt}
\item[$(B'_1)$] $[\Xgot_{w_1,\gamma}]\cdots [\Xgot_{w_{s+1},\gamma}]= k[pt]$, avec $k\geq 1$,\, dans  \,$H^{max}(\Fcal_\gamma,\Z)$. 
\item[$(A_2)$] $\sum_{i=1}^{s+1}\tr(w_i\gamma \circlearrowright \ngot^{w_i\gamma>0})= s\tr(\gamma \circlearrowright \ugot_\C^{\gamma>0})$.
\end{itemize}
\end{enumerate}
\end{prop}

Regardons maintenant le cas de la grassmanienne $\G(r,n)$ (voir la section \ref{sec:calcul-schubert-grass}). Les variétés de Schubert 
$\Xgot_{w,\gamma^r}$ sont désignées par $\Xgot_I$, où $I=w([r])$, et un calcul direct montre que 
$$
|\mu(I)|= d(w,\gamma^r)= \tr(w\gamma^r_o \circlearrowright \ngot^{w\gamma^r_o>0}),\qquad \mathrm{et}\qquad 
\tr(\gamma^r_o \circlearrowright \ugot_\C^{\gamma^r_o>0})=\dim_\C(\G(r,n)).
$$
On voit donc que, dans le cas de la grassmanienne, les conditions $\sum_{i=1}^{s+1} d(w_i,\gamma^r)= s\dim_\C(\G(r,n))$ 
et $(A_2)$ sont toutes deux équivalentes à 
$$
\sum_{i=1}^{s+1}|\mu(I_i)|= sr(n-r).
$$
Ainsi, pour le calcul de Schubert dans la grassmanienne $\G(r,n)$, la notion de Lévi-mobilité est vérifiée dès que la condition $(B'_1)$ est satisfaite.

\section{Le cadre avec involution}\label{sec:cadre-involution}

Nous revenons au cadre de la section \ref{sec:U-sigma-hamiltonien}, c'est à dire une involution $\sigma$ sur un groupe compact connexe $U$, que l'on étend 
en une involution \emph{anti-holomorphe} sur $U_\C$  (encore notée $\sigma$).

Soit $G$ la composante connexe du sous-groupe de $U_\C$ fix\'e par $\sigma$ : c'est un sous-groupe r\'eductif r\'eel d'algèbre de Lie 
$\ggot=\kgot\oplus \pgot$ o\`u $\kgot=\ugot^{\sigma}$  et $\pgot= i\ugot^{-\sigma}$. Nous d\'esignons par $K$ la composante 
connexe du sous-groupe $U^\sigma$: c'est un sous-groupe compact maximal de $G$, d'algèbre de Lie $\kgot$.

\subsection{Groupes de Weyl}

Nous choisissons un tore maximal $T\subset U$ invariants par $\sigma$, tel que le sous-espace $\agot:=\frac{1}{i}\tgot^{-\sigma}$ est ab\'elien maximal 
dans $\pgot$. Rappelons les liens entre le groupe de Weyl $W:=N_K(T)/T$, le groupe de Weyl restreint $W_\agot:=N_K(\agot)/Z_K(\agot)$, et  
l'involution $\sigma$ (voir l'annexe B dans \cite{OSS}). 

Soit $\zeta_o$ un élément régulier de $\agot$. Le sous-groupe $W^\sigma$ fixé par $\sigma$ est égal au sous-groupe normalisateur $N_{W}(\agot)$, 
le sous-groupe $W^{\zeta_o}$ coïncide avec le sous-groupe centralisateur $Z_{W}(\agot)$, et le groupe de Weyl restreint $W_{\agot}$ admet 
une identification canonique avec $N_{W}(\agot)/Z_{W}(\agot)$.

Considérons $\zeta\in\agot$ et les sous-groupes $W^\zeta\subset W$ et $W_\agot^\zeta\subset W_\agot$ fixant $\zeta$. Nous avons donc deux suites exactes 
$0\to W^{\zeta_o}\to W^{\sigma}\to W_{\agot}\to 0$ et $0\to W^{\zeta_o}\to W^\sigma\cap W^{\zeta}\to W_{\agot}^\zeta\to 0$ qui induisent un isomorphisme 
\begin{equation}\label{eq:Weyl-group-sigma-1}
W^\sigma/W^\sigma\cap W^{\zeta}\simeq W_{\agot}/W_{\agot}^\zeta.
\end{equation} 
L'autre isomorphisme 
\begin{equation}\label{eq:Weyl-group-sigma-2}
(W/W^{\zeta})^\sigma\simeq W_\agot/W_\agot^\zeta
\end{equation} 
découle du fait que $W/W^{\zeta}\simeq W\zeta=U\zeta\cap i\tgot$, et donc
$$
(W/W^{\zeta})^\sigma\simeq (W\zeta)^\sigma=U\zeta\cap i\tgot^{-\sigma}=K\zeta\cap i\tgot^{-\sigma}= W_\agot\zeta
\simeq W_\agot/W_\agot^\zeta.
$$

Le morphisme $W^\sigma/W^\sigma\cap W^{\zeta}\to W/W^\zeta$ est injectif et $W^\sigma$-équivariante. Sachant que 
$W^{\zeta_o}$ est un sous-groupe distingué de $W^\sigma$ contenu dans $W^{\zeta}$, on voit que l'action de $W^{\zeta_o}$ sur 
$W^\sigma/W^\sigma\cap W^{\zeta}$ est triviale. Ainsi, l'application canonique 
\begin{equation}\label{eq:Weyl-group-sigma-3}
W^\sigma/W^\sigma\cap W^{\zeta}\to W^{\zeta_o} \backslash W/W^\zeta
\end{equation} 
est injective. Si on combine ce fait à la relation (\ref{eq:Weyl-group-sigma-1}), nous voyons que $W_{\agot}/W_{\agot}^\zeta$ peut être considéré 
comme un sous-ensemble de $W^{\zeta_o} \backslash W/W^\zeta$. 

On peut résumer ce que l'on vient de faire avec le lemme suivant.
\begin{lem}\label{lem:groupes-weyl-sigma}
Pour tout $\zeta\in\agot$, l'ensemble
$$
(W/W^{\zeta})^\sigma\simeq W^\sigma/W^\sigma\cap W^{\zeta}\simeq W_\agot/W_\agot^\zeta
$$
est un sous-ensemble de $W^{\zeta_o} \backslash W/W^\zeta$.
\end{lem}

\subsection{Cellules de Bruhat, complexes et réelles}

Soit $\tgot_+\subset \tgot$ une chambre de Weyl telle que l'intersection $\agot_+:=\frac{1}{i}(\tgot^{-\sigma}\cap\tgot_+)$ 
paramètre les $K$-orbites dans $\pgot$. Notons $\Rgot\subset \tgot^*$ l'ensemble des racines relatives à l'action de $T$ sur $\ugot_\C$, et 
$\Sigma\subset \agot^*$ l'ensemble des racines (restreintes) relatives à l'action de $\agot$ sur $\ggot$. Notons $\Rgot^+\subset\Rgot$ et 
$\Sigma^+\subset\Sigma$ les systèmes de racines positives associés à la chambre de Weyl $\tgot_+$. Soit $B\subset U_\C$ le sous-groupe de Borel 
associé à $\Rgot^+$.

A tout $\zeta\in\agot$, nous associons, via (\ref{eq:P-gamma-reel}), le sous-groupe parabolique $\Pbb(\zeta)\subset U_\C$. Comme celui-ci est stable 
sous l'involution $\sigma$, la vari\'et\'e de drapeaux $\Fcal_\zeta= U_\C/\Pbb(\zeta)$ est munie de l'involution antiholomorphe 
$\tau(g\Pbb(\zeta)):=\sigma(g)\Pbb(\zeta)$.

L'intersection $P^\R(\zeta):= \Pbb(\zeta)\cap G\subset (\Pbb(\zeta))^\sigma$ définit un sous-groupe parabolique réel de $G$ (qui n'est pas nécessairement connexe). 
Expliquons pourquoi la sous-variété des point fixes $(\Fcal_\gamma)^\tau$ correspond à la variété de drapeaux réelle $\Fcal_\zeta^\R:=G/P^\R(\zeta)$. 
Comme la variété drapeau $\Fcal_\zeta= U_\C/\Pbb(\zeta)\simeq U/U^\zeta$, munie de l'involution $\tau$, s'identifie à l'orbite adjointe 
$U\zeta$, nous avons 
$$
(\Fcal_\zeta)^\tau\simeq U\zeta \cap i\ugot^{-\sigma}= K\zeta\simeq K/K_\zeta\simeq G/P^\R(\zeta)=\Fcal_\zeta^\R.
$$ 
Ici, le point crucial est l'égalité $U\zeta \cap i\ugot^{-\sigma}= K\zeta$ (voir la remarque \ref{rem:orbite-symetrique}).

Nous fixons un élément $\zeta_o$ à l'intérieur de la chambre de Weyl $\agot_+$ et nous considérons le sous-groupe parabolique 
$\Pbb:= \Pbb(-\zeta_o)\subset U_\C$. Alors $P^\R:=\Pbb\cap G$ est le sous-groupe parabolique minimal de $G$, avec algèbre de Lie 
$Z_{\kgot}(\agot)\oplus \agot\oplus \sum_{\beta\in\Sigma^+}\ggot_\beta$.

\medskip

Nous considérons plusieurs décompositions de Bruhat :
\begin{enumerate}\setlength{\itemsep}{8pt}
\item $\Fcal_\zeta=\bigcup_{w\in W/W^\zeta}B[w]$ relativement à l'action $B$ sur $\Fcal_\zeta$,
\item $\Fcal_\zeta=\bigcup_{u\in W^{\zeta_o}\setminus W/W^\zeta}\Pbb[u]$ par rapport à l'action $\Pbb$ sur $\Fcal_\zeta$,
\item $\Fcal^\R_\zeta=\bigcup_{v\in W_{\agot}/W_{\agot}^\zeta} P^\R[v]$ par rapport à l'action $P^\R$ sur $\Fcal_\zeta^\R$.
\end{enumerate}

Nous avons prouv\'e précédemment que $(W/W^{\zeta})^\sigma\simeq W_{\agot}/W_{\agot}^\zeta$ peut \^etre consid\'er\'e 
comme un sous-ensemble de $W^{\zeta_o} \backslash W/W^\zeta$. Le lemme suivant caract\'erise ce sous-ensemble 
en termes de l'intersection $(B[w])^\tau=B[w]\cap\Fcal^\R_\zeta$. Soit $N_\ggot\subset G$ 
le sous-groupe niltpotent d'algèbre de Lie $\oplus_{\beta\in\Sigma^+}\ggot_\beta$.

\begin{lem}\label{lem:schubert-reel} 
Soit $w\in W$. Les \'enonc\'es suivants sont \'equivalents
\begin{enumerate}\setlength{\itemsep}{10pt}
\item $(B[w])^{\tau}\neq\emptyset$.
\item $(\Pbb[w])^{\tau}\neq\emptyset$.
\item La classe de $w$ dans $W^{\zeta_o}\setminus W/W^\zeta$ est contenue dans $W^\sigma/W^\sigma\cap W^{\zeta}\simeq W_{\agot}/W_{\agot}^\zeta$.
\item La classe de $w$ dans $W/W^\zeta$ est contenue dans $(W/W^\zeta)^\sigma$.
\end{enumerate}
Si $u\in W^\sigma$, alors
\begin{enumerate}\setlength{\itemsep}{10pt}
\item $B[u]=\Pbb[u]$ et donc $\tau(B[u])=B[u]$
\item $(B[u])^\tau= P^\R[u]= N_\ggot[u]$.
\end{enumerate} 
\end{lem}

{\em Preuve :} L'implication {\em 1.} $\Longrightarrow$ {\em 2.} est \'evidente puisque $B[w]\subset \Pbb[w]$. Pour l'implication \break 
{\em 2.} $\Longrightarrow$ {\em 3.} nous consid\'erons 
la d\'ecomposition de Bruhat $\Fcal_\gamma^\R=\bigcup_{v\in W_{\agot}/W_{\agot}^\gamma}P^\R[v]$. Comme nous l'avons d\'ej\`a  expliqu\'e, 
chaque classe de $W_{\agot}/W_{\agot}^\gamma$ peut \^etre repr\'esent\'ee par un \'el\'ement de $W^\sigma$. 
Si $(\Pbb[w])^\tau\neq \emptyset$, il existe $v\in W^\sigma$ tel que $P^\R[v]\subset (\Pbb[w])^\tau$. Par cons\'equent, $\Pbb[v]=\Pbb[w]$, 
et donc $w=v$ dans $W^{\gamma_o} \backslash W/W^{\gamma}$. L'implication {\em 2.} $\Longrightarrow$ {\em 3.} est 
r\'esolue et la derni\`ere {\em 3.} $\Longrightarrow$ {\em 1.} est imm\'ediate. L'équivalence {\em 3.} $\Longleftrightarrow$ {\em 4.} est une conséquence du lemme 
\ref{lem:groupes-weyl-sigma}.

Consid\'erons maintenant $u\in W^\sigma$. Puisque $B\subset\Pbb$, la d\'ecomposition de Bruhat 1. nous dit que 
$\Pbb[u]=\cup_{w}B[w]$ o\`u l'union s'\'etend sur les $w\in W/W^\zeta$ tel que $w=u$ dans $W^{\zeta_o}\setminus W/W^\zeta$. 
Comme $\sigma(u)=u$, la derni\`ere relation implique que $w=u$ dans $W/W^\zeta$. Nous avons donc prouv\'e que $\Pbb[u]=B[u]$. 
La d\'ecomposition de Bruhat 3. montre maintenant que 
$(B[u])^\tau=\Pbb[u]\cap \Fcal^\R_\zeta=\cup_{v} P^\R[v]$ o\`u l'union se fait sur les $v\in W_{\agot}/W_{\agot}^\zeta\simeq W^\sigma/W^\sigma\cap W^{\zeta}$ 
tels que $P[v]\subset \Pbb[u]$. La derni\`ere inclusion implique que $v=u$ dans $W^{\gamma_o} \backslash W/W^{\gamma}$, et alors 
$v=u$ dans $W^\sigma/W^\sigma\cap W^{\gamma}$. Nous avons prouv\'e que $(\B[u])^\sigma= P^\R[u]$.

Nous avons $P^\R= N_\ggot A M$ où $M=Z_K(\agot)$. Puisque $u^{-1}A M u= A M\subset \Pbb(\zeta)$ pour tout $u\in W^\sigma$, nous voyons que 
$P^\R[u]=N_\ggot[u]$ lorsque $u\in W^\sigma$. $\Box$

\subsection{Théorème de Borel et Haefliger}

\`A chaque $w\in W/W^\zeta$, on associe la variété de Schubert (complexe) $\Xgot_{w,\zeta}=\overline{Bw\Pbb(\zeta)/\Pbb(\zeta)}\subset \Fcal_\zeta$ et sa classe de cycle 
$[\Xgot_{w,\zeta}]\in H^{2c(w,\zeta)}(\Fcal_\zeta,\Z)$.

Grâce à la décomposition de Bruhat $\Fcal_\zeta^\R=\bigcup_{v\in W_{\agot}/W_{\agot}^\zeta} P^\R [v]$, nous savons que l'on peut associer des classes de cycles 
$[\Xgot^\R_{v, \zeta}]\in H^{c(v,\zeta)}(\Fcal_\zeta^\R,\Z_2)$ aux variétés Schubert réeelles 
$$
\Xgot_{v,\zeta}^\R : =\overline{P^\R vP^{\R}(\zeta)/P^{\R}(\zeta)}\subset \Fcal_\zeta^\R.
$$
 De plus, la famille $[\Xgot^\R_{v, \zeta}],v\in W_{\agot}/W_{\agot}^\zeta$ définit une base de l'anneau de cohomologie $H^*(\Fcal_\zeta^\R,\Z_2)$ (voir \cite{Tak65} 
 et \cite{DKV83}).
 
 Lorsque $v\in (W/W^\zeta)^\sigma\simeq W_{\agot}/W_{\agot}^\zeta$, les variétés de Schubert réelles $\Xgot_{v,\zeta}^\R$ correspondent 
 à la partie réelle des variétés de Schubert complexes $\Xgot_{v,\zeta}$. Notons $n(\zeta)$ la dimension complexe de $\Fcal_\zeta$. 
 
 Dans ce cadre, un théorème de Borel et Haefliger \cite[Proposition 5.14]{BH61} 
 nous dit que la relation $[\Xgot_{v_1,\zeta}]\cdot\ldots\cdot[\Xgot_{w_{s+1},\zeta}]= [pt]$ dans $H^{2n(\zeta)} (\Fcal_\zeta,\Z)$ 
 implique la relation $[\Xgot^\R_{v_1,\zeta}]\cdot \ldots\cdot[\Xgot^\R_{v_{s+1},\zeta}]= [pt]$ dans $H^{n(\zeta)}(\Fcal_\zeta^\R,\Z_2)$.

\section{Variétés de drapeaux $\F(r,n-r; n)$ et sous-espaces vectoriels isotropes}\label{sec:drapeaux-2-etage}

Soient $n > b>a\geq 1$. Dans la suite, pour tout $X\subset [n]$, on note $\C^X:=\vect(e_k,k\in X)$. 

\begin{definition}
Désignons par $\F(a,b; n)$ la variété des drapeaux des séquences imbriquées de sous-espaces vectoriels 
$E\subset F\subset \C^{n}$ où $\dim E= a$ et $\dim F= b$. 
\end{definition}

L'application $g\in GL_{n}(\C)\mapsto \left(g(\C^{[a]})\subset g(\C^{[b]})\right)$ induit un isomorphisme 
$GL_{n}(\C)/P(\gamma^{a,b})\simeq \F(a,b; n)$ où 
\begin{equation}\label{eq:gamma-a-b}
\gamma^{a,b}_o=\diag(\underbrace{-1,\ldots,-1}_{a\ termes},\underbrace{0,\ldots,0}_{b-a\ termes},\underbrace{1,\ldots\ldots,1}_{n-b\ termes}).
\end{equation}

L'application $w\mapsto (w([a]),w([b])$ détermine une bijection entre $W/W^{\gamma^{a,b}}$ et l'ensemble $\Pcal(a,b;n)$ des sous-ensembles emboités 
$I\subset J\subset [n]$ avec $\sharp I=a$ et $\sharp J=b$.  Ainsi les cellules de Bruhat de $\F(a,b; n)$ sont de la forme 
$$
\Xgot_{I\subset J}^o:= B_n\cdot\left(\C^I\subset \C^J\right),
$$
avec $(I\subset J)\in \Pcal(a,b;n)$, et $B_n\subset GL_n(\C)$ est le sous-groupe de Borel des matrices triangulaires supérieures. 
La variété de Schubert correspondante est notée $\Xgot_{I\subset J}:=\overline{\Xgot_{I\subset J}^o}$.

\medskip

Fixons des entiers $p\geq q\geq r\geq 1$ tels que $p+q=n$.  Soit $\langle-,-\rangle$ le produit hermitien canonique sur $\C^n$, et soit $\langle-,-\rangle_{p,q}$ le produit hermitien de signature $(p,q)$, qui est défini par la relation 
$\langle x,y\rangle_{p,q}:= \langle x,\mathbf{J}_{p,q}y\rangle$ où
\begin{equation}\label{eq:J-p-q}
\mathbf{J}_{p,q}=
\begin{pmatrix}
0&0& J_q\\
0&I_{p-q} & 0\\
J_q& 0 & 0
\end{pmatrix}\quad \mathrm{et}\quad J_q=\begin{pmatrix}0&\cdots & 1\\ 0&\reflectbox{$\ddots$}& 0\\1& \cdots & 0\end{pmatrix}.
\end{equation} 
Pour tout sous-espace vectoriel $E\subset \C^n$, notons\footnote{L'orthogonal $E^{\Perp}$ dépend des variables $p,q$.}  $E^{\Perp}$ 
l'orthogonal de $E$ pour le produit hermitien $\langle-,-\rangle_{p,q}$. 

Nous travaillons avec l'involution anti-holomorphe 
$$
\tau_{p,q} : \F(r,n-r;n)\longrightarrow \F(r,n-r;n)
$$
qui envoie $E\subset F$ sur $F^{\Perp}\subset E^{\Perp}$.

On remarque que la variété des points fixes $\F(r,n-r;n)^{\tau_{p,q}}$ s'identife avec la sous-variété $\{E\in \G(r,n), E\subset E^{\Perp}\}$ 
des sous-espaces vectoriels isotropes. Nous allons maintenant déterminer les cellules de Bruhat qui admettent un point fixé par $\tau_{p,q}$.

\begin{lem}\label{lem:schubert-fixe-involution}
Pour tout $(I\subset J)\in \Pcal(r,n-r;n)$, $(\Xgot_{I\subset J}^o)^{\tau_{p,q}}\neq \emptyset$ si et seulement si 
$I\cap I^o=\emptyset$, $I\cap [q+1,\ldots,p]=\emptyset$ et $J= I^{o,c}$.
\end{lem}

Nous allons vérifier que ce résultat est un cas particulier du Lemme \ref{lem:schubert-reel}.

\medskip

Considérons l'involution $\sigma'_{p,q}(g)=I_{p,q}(g^*)^{-1}I_{p,q}$ sur $GL_n(\C)$, où $I_{p,q}$ est la matrice $\diag(I_p,-I_q)$. 
Alors $G=U(p,q)\subset GL_n(\C)$ est le sous-groupe fixé par $\sigma'_{p,q}$. La décomposition de Cartan de l'algèbre de Lie de 
$G$ est $\ggot=\kgot\oplus\pgot$ où $\kgot=\ugot(p)\times\ugot(q)$ et $\pgot\simeq M_{p,q}(\C)$ en tant que $U(p)\times U(q)$-module. 

Soit $\tgot'\subset \ugot(n)$ le tore maximal composé des matrices de la forme 
$$
i\,\begin{pmatrix} \diag(a)&  0& A(c)\\
0& \diag(b)& 0\\
 A(c^*)&  0& \diag(a^*)
\end{pmatrix}\quad\mathrm{avec}\quad 
A(c)=\begin{pmatrix}
0&\cdots & c_1\\
\vdots &\reflectbox{$\ddots$}& \vdots\\
c_q&\cdots & 0\\
\end{pmatrix}, 
$$
pour tout $(a,b,c)\in \R^q\times \R^{p-q}\times \R^q$.

Nous remarquons que $\tgot'$ est invariant sous l'involution $\sigma'_{p,q}$, et que $\agot'=\frac{1}{i} (\tgot')^{-\sigma'_{p,q}}$ est le sous-espace abélien maximal de 
$\pgot$ constitué de toutes les matrices de la forme
\begin{equation}\label{eq:maxabelian}
X(c):=\begin{pmatrix}0&  0& A(c)\\
0&  0& 0\\
A(c^*) & 0 & 0
\end{pmatrix},\qquad c\in\R^q.
\end{equation}

Soit $f_i:\agot'\to\R$ la forme linéaire dont la valeur sur la matrice $X(c)$ est $c_i$.
L'ensemble $\Sigma(\agot')$ des racines restreintes comprend tous les applications linéaires $\pm f_i\pm f_j$ avec $i\neq j$ et $\pm 2 f_i$ pour tous les $i$. 
De même, les $\pm f_i$ sont des racines restreintes si $p\neq q$ (voir \S 6 dans \cite{Knapp-book}). Nous pouvons donc choisir la chambre de Weyl restreinte suivante $\agot'_+:=\{X(c), c_1\geq c_2\geq \cdots\geq c_q\geq 0\}$. Dans ce cas, le groupe de Weyl restreint $W_{\agot'}$ est le produit semi-direct 
$\Sgot_q\rtimes \{1,-1\}^q$ où $\Sgot_q$ agit sur $\R^q$ en permutant les indices et $\{1,-1\}^q=\{\epsilon=(\epsilon_1,\cdots,\epsilon_q)\ \vert\  \epsilon_k=\pm 1\}$ agit par changements de signes : $\epsilon \cdot c= (\epsilon_1 c_1,\ldots, \epsilon_q c_q)$.

Soit $\tgot'_+\subset\tgot'$ une chambre de Weyl adaptée à l'involution\footnote{On a donc $\agot'_+=\frac{1}{i}(\tgot'_+\cap (\tgot')^{-\sigma_{p,q}})$.}, et $B'\subset GL_n(C)$ le sous-groupe de Borel associé. Soit $W'$ le groupe de Weyl associé à l'algèbre de Cartan $\tgot'$. Pour tout $r\in\{1,\ldots,q\}$, on note
$\zeta_r:=X(c_r)$ où $c_r=(\underbrace{-1,\ldots,-1}_{ r\ fois},0,\ldots,0)$, et on désigne par $\Fcal_{\zeta_r}:=GL_n(\C)/\Pbb(\zeta_r)$ la variété de drapeaux correspondante. 
Comme l'élément $\zeta_r$ est fixé par l'involution $\sigma'_{p,q}$, on peut définir une involution anti-holomorphe $\tau'_{p,q}$ sur la variété 
$\Fcal_{\zeta_r}:=GL_n(\C)/\Pbb(\zeta_r)$ par la relation
$\tau'_{p,q}(g \Pbb(\zeta_r)):=\sigma'_{p,q}(g)\Pbb(\zeta_r)$.

Pour vérifier que les données géométriques $\left(\Fcal_{\zeta_r}, \tau'_{p,q}\right)$ et $\left(\F(r,n-r;n),\tau_{p,q}\right)$ sont identiques, on utilise la conjugaison 
par l'élément d'ordre 2 
\begin{equation}\label{eq:theta}
\theta_{p,q}=\frac{1}{\sqrt{2}}
\begin{pmatrix}
I_q&0& J_q\\
0&\sqrt{2}I_{p-q} & 0\\
J_q& 0 & -I_q
\end{pmatrix}\ \in O(n).
\end{equation}
Pour simplifier nos notations, nous écrirons $\theta=\theta_{p,q}$. Voici quelques propriétés faciles à vérifier:

\begin{itemize}\setlength{\itemsep}{10pt}
\item $\tgot:=Ad(\theta)(\tgot')$ est le tore maximal des matrices diagonales. Plus précisément, l'image de la 
matrice $
\begin{pmatrix} D(a)&  0& A(c)\\
0& D(b)& 0\\
 A(c^*)&  0& D(a^*)
\end{pmatrix}
$
par $Ad(\theta)$ est égale à 
$
\begin{pmatrix} D(a+c)&  0&0\\
0& D(b)& 0\\
0&  0& D(a^*-c^*)
\end{pmatrix}.
$
\item $\agot:=Ad(\theta)(\agot')$ est formé des matrices $\diag(c_1,\cdots,c_q, 0\cdots, 0, -c_q\cdots, -c_1)$.
\item $\tgot_+:=Ad(\theta)(\tgot'_+)$ est l'ensemble des matrices $i\,\diag(x_1\geq\cdots\geq x_n)$ et 
$\agot_+:=Ad(\theta)(\agot'_+)$ est l'ensemble des matrices $\diag(c_1\geq\cdots\geq c_q\geq 0\cdots\geq 0\geq -c_q\cdots\geq -c_1)$.
\item Le groupe de Weyl $W'$ s'identifie à $W=\Sgot_n$ à travers $Ad(\theta)$.
\item Le groupe de Weyl restreint $W_{\agot'}$ s'identifie, à travers $Ad(\theta)$, au groupe 
$$
W_{\agot}= \Big\{w\in\Sgot_n, w(k^o)=w(k)^o, \forall k\in [q]\ \mathrm{et}\ w(k)=k, \forall k\in [q+1,\ldots,p]\Big\}.
$$
\item $B_n:= Ad(\theta)(B')\subset GL_n(\C)$ est le sous-groupe de Borel des matrices triangulaires supérieures.
\item Le sous-groupe parabolique $Ad(\theta)( \Pbb(\zeta_r))$ est \'egal au groupe parabolique $P(\gamma^{r,n-r})$ (voir (\ref{eq:gamma-a-b})). 
\item Comme $\mathbf{J}_{p,q}:=Ad(\theta)(I_{p,q})$, l'involution $Ad(\theta) \circ \sigma'_{p,q}\circ Ad(\theta)$ est égale à $\sigma_{p,q}(g):=\mathbf{J}_{p,q}(g^*)^{-1}\mathbf{J}_{p,q}$.

\end{itemize}

\medskip

Nous travaillons avec le difféomorphisme 
\begin{equation}\label{eq:diffeo-F-r-n-r}
\varphi : \Fcal_{\zeta_r}\to \F(r,n-r;n)
\end{equation}
défini par $\varphi(g \Pbb(\zeta_r))= Ad(\theta) (g)\cdot (\C^r\subset \C^{n-r})$. 
Grâce aux points précédents, on voit que le difféomorphisme (\ref{eq:diffeo-F-r-n-r}) identifie les variétés de Schubert 
$\Xgot_{w,\zeta_r}:=\overline{B' w\Pbb(\zeta_r)/\Pbb(\zeta_r)}$, $w\in W'/(W')^{\zeta_r}$ aux variétés de Schubert 
$\Xgot_{I\subset J},\sharp I= r,\sharp J= n-r$ définies précédemment.

De plus, on a la relation $\varphi\circ \tau'_{p,q}=\tau_{p,q}\circ\varphi$. Celle-ci permet de voir que nous sommes dans la situation traitée dans le lemme \ref{lem:schubert-reel}:  
l'ensemble $(\Xgot_{I\subset J}^o)^{\tau_{p,q}}$ est non-vide si et seulement si il existe $w\in W_\agot$ tel que $I=w([r])$ et $J=w([n-r])$. 
Cette dernière condition est équivalente aux relations: $I\cap I^o=\emptyset$, $I\cap [q+1,\ldots,p]=\emptyset$ et $J= I^{o,c}$. On a 
ainsi démontré le lemme \ref{lem:schubert-fixe-involution}.

\chapter{Formules multiplicatives}

L'objectif des deux prochains sections est de passer en revue certains des résultats obtenus par Ressayre et Richmond \cite{Res11b,Ric12,RR11}. 
Ce faisant, nous en améliorerons certains.

\section{Formules multiplicatives et paires de Ressayre}\label{sec:multiplicatif-ressayre-pair}

Soit $M$ une variété de K\"ahler $U$-hamiltonienne possédant une structure de variété algébrique. Soit $T\subset U$ un tore maximal et $\gamma,\hgamma\in\tgot$. 
Soit $\mathbb{T}\subset T$ le sous-tore engendré par $\gamma$ et $\hgamma$.

Soient $\hC$ une composante connexe de $M^{\mathbb{T}}$ et $C$ la composante connexe de $M^\gamma$ contenant $\hC$.
Nous considérons les variétés de Bialinicki-Birula suivantes\footnote{$\gamma_o:=\frac{1}{i}\gamma$ et $\hgamma_o:=\frac{1}{i}\hgamma$.}
$$
C^-:= \Big\{m\in M, \lim_{t\to\infty} \exp(t\gamma_o) m\ \in C \Big\}\quad \mathrm{et}\quad \hC^-:= \Big\{m\in M, \lim_{t\to\infty} \exp(t\hgamma_o) m\ \in \hC \Big\}.
$$

Nous travaillons avec les sous-groupes paraboliques $P(\gamma):=\{g\in U_\C, \ \lim_{z\to\infty}e^{t\gamma_o}ge^{-t\gamma_o}\ \mathrm{existe}\}$ et 
$P(\hgamma)$. Dans cette section, nous faisons les hypothèses suivantes :

\medskip

\begin{itemize}
\item[{\em (H1)}] $\hC^-\subset C^-$,
\item[{\em (H2)}] $P(\,\hgamma\,) \subset P(\gamma)$.
\end{itemize}

\medskip

On considère les projections $\pi: C^-\to C$ et $\hpi: \hC^-\to \hC$ définies par les relations $\pi(x)=\lim_{t\to\infty}\exp(t\gamma_o) x$ 
et $\hpi(y)=\lim_{t\to\infty}\exp(t\,\hgamma_o) y$.

\begin{prop}\label{prop:hypothesesH1H2}
\begin{enumerate}
\item L'hypthèse (H2) est équivalente au fait que $\ugot_\C^{\hgamma_o\leq 0}\subset \ugot_\C^{\gamma_o\leq 0}$.
\item L'hypthèse (H2) implique les relations suivantes:
\begin{enumerate}
\item  $\ugot_\C^{\hgamma_o}\subset \ugot_\C^{\gamma_o}$,
\item $\ugot_\C^{\hgamma_o>0}\cap\ugot_\C^{\gamma_o<0}=0$,
\item $U_\C^\gamma/(U_\C^\gamma\cap P(\,\hgamma\,))\simeq P(\gamma)/P(\,\hgamma\,)$.
\end{enumerate}
\item L'hypthèse (H1) est équivalente au fait que $(TM\vert_{\hC})^{\hgamma_o\leq  0}\subset (TM\vert_{\hC})^{\gamma_o\leq 0}$.
\item Pour tout $x\in \hC^-$, on a $\pi(x)\in \hC^-\cap C$ et $\hpi(\pi(x))=\hpi(x)$.
\item Supposons que $(TM\vert_{\hC})^{\hgamma_o< 0}\cap (TM\vert_{\hC})^{\gamma_o> 0}=0$. Alors $\forall x\in C^-$, on a 
$x\in \hC^-\Longleftrightarrow \pi(x)\in \hC^-\cap C$.
\end{enumerate}
\end{prop}

\begin{proof} Le point {\em 1.} découle du fait que $P(\hgamma)$ et $P(\gamma)$ sont des groupes  de Lie connexes ayant pour algèbres de Lie 
$\ugot_\C^{\hgamma_o\leq 0}$ et $\ugot_\C^{\gamma_o\leq 0}$.

Soit $\Rgot\subset\tgot^*$ l'ensemble des racines de $(U,T)$. On a les partitions $\Rgot=\Rgot^{\gamma\leq 0}\cup\Rgot^{\gamma> 0}$ et 
$\Rgot=\Rgot^{\hgamma\leq 0}\cup\Rgot^{\hgamma> 0}$. L'inclusion $\ugot_\C^{\hgamma_o\leq 0}\subset \ugot_\C^{\gamma_o\leq 0}$ signifie que 
$\Rgot^{\hgamma\leq 0}\subset \Rgot^{\gamma\leq 0}$. Alors $\Rgot^{\hgamma\geq 0}=-\Rgot^{\hgamma\leq 0}\subset -\Rgot^{\gamma\leq 0}
=\Rgot^{\gamma\geq 0}$ et donc $\Rgot^{\hgamma= 0}=\Rgot^{\hgamma\leq 0}\cap\Rgot^{\hgamma\geq 0} \subset \Rgot^{\gamma\leq 0}\cap\Rgot^{\gamma\geq 0}= \Rgot^{\gamma= 0}$. Ceci démontre le point {\em 2.a)}.

Les relations $\Rgot^{\hgamma\leq 0}\subset \Rgot^{\gamma\leq 0}$, $\Rgot^{\hgamma\geq 0}\subset \Rgot^{\gamma\geq 0}$ et 
$\Rgot^{\gamma< 0}\subset \Rgot^{\hgamma< 0}$ sont toutes équivalentes. Cela montre que 
$\Rgot^{\gamma< 0}\cap \Rgot^{\hgamma> 0}=\emptyset$, et le point  {\em 2.b)} en découle.

Les calculs précédents montrent que 
\begin{equation}\label{eq:Rgot-gamma}
\Rgot^{\gamma\leq 0}=\Rgot^{\hgamma\leq 0}\bigcup \Rgot^{\gamma = 0}\cap \Rgot^{\hgamma> 0}\qquad \mathrm{et}\qquad 
\Rgot^{\hgamma< 0}=\Rgot^{\gamma< 0}\bigcup \Rgot^{\hgamma < 0}\cap \Rgot^{\gamma=0}.
\end{equation} 
Cela implique que les espaces tangents en $[e]$  des variétés $U_\C^\gamma/U_\C^\gamma\cap P(\hgamma)$ et 
$P(\gamma)/P(\hgamma)$ coincident. Le point  {\em 2.c)} est ainsi démontré.

L'inclusion $\hC^-\subset C^-$ implique, au niveau des espaces tangents, que $\T_{\hx}\hC^-\subset \T_{\hx}C^-,\forall \hx\in\hC$, et donc 
$(\T_{\hx}M)^{\hgamma_o\leq 0}\subset (\T_{\hx}M)^{\gamma_o\leq 0}$. 

Réciproquement, considérons $x\in \hC^-$ tel que $\hat{x}:=\lim_{t\to+\infty}\exp(t\hgamma_o)x\in \hC\subset M^{\mathbb{T}}$ satisfait la relation 
$(T_{\hx}M)^{\hgamma_o\leq 0}\subset (T_{\hx}M)^{\gamma_o\leq 0}$. Si on montre que $x\in C^-$, la preuve du point {\em 3.} sera complète. 
Se faisant, on démontrera le point {\em 4}.

L'action de $\mathbb{T}$ sur $T_{\hx}M$ se diagonalise : on a la 
décomposition 
$T_{\hx}M=\oplus_{\beta}(T_{\hx}M)_\beta$ où pour tout $\beta\in Lie(\mathbb{T})^*$, on pose 
$$
(T_{\hx}M)_\beta:=\left\{v\in T_{\hx}M; \ X\cdot v =i\langle\beta ,X\rangle v,\forall X\in Lie(\mathbb{T})\right\}.
$$

Soit $\mathbb{T}_\C\subset U_\C$ la complexification du tore $\mathbb{T}$. Comme $\hx\in M^{\mathbb{T}_\C}$, 
un théorème de Koras \cite{Kor86, Sjamaar95} nous dit que l'action $\mathbb{T}_\C$ peut être linéarisée au voisinage de $\hx$. 
En d'autres termes, il existe un difféomorphisme holomorphe $\mathbb{T}_\C$-équivariant $\varphi : \Vcal \to \Ucal$ où $\Ucal$  
est un voisinage ouvert $\mathbb{T}_\C$-invariant de $0$ dans $\T_{\hx}M$ et $\Vcal$ est un voisinage ouvert $\mathbb{T}_\C$-invariant de $\hx$ dans $M$. 

Notons $(\Vcal\cap M^\gamma)_0$  la composante connexe de $\Vcal\cap M^\gamma$ 
contenant $\hx$. De la même manière, notons $(\Ucal\cap (\T_{\hx}M)^\gamma)_0$ 
la composante connexe de $\Ucal\cap (\T_{\hx}M)^\gamma$  contenant $0=\varphi(\hx)$. 
On a $(\Vcal\cap M^\gamma)_0\simeq (\Ucal\cap (\T_{\hx}M)^\gamma)_0$ à travers le difféomorphisme $\varphi$, et de plus on a 
$(\Vcal\cap M^\gamma)_0\subset \Vcal\cap C$.

Par définition, $\exp(t\,\hgamma_o)x$ tend vers $\hx$ lorsque $t$ tend vers l'infini. Alors $\exp(t\,\hgamma_o) x\in \Vcal$ lorsque $t$ est suffisamment grand. 
Comme $\Vcal$ est invariant par $\mathbb{T}_\C$, on voit que $x\in\Vcal$. Il nous reste à montrer les relations
$$
\pi(x):=\lim_{s\to+\infty} e^{s\gamma_o}x\ \in \  (\Vcal\cap M^\gamma)_0\qquad \mathrm{et}\qquad 
\lim_{t\to+\infty} e^{t\,\hgamma_o}\pi(x)=\hx,
$$
qui impliquent respectivement que $x\in C^-$ et $\pi(x)\in\hC^-$.

Choisissons un produit hermitien $\mathbb{T}$-invariant sur $T_{\hx}M$, et posons $B(\epsilon):=\{\|v\|\leq \epsilon\}\subset T_{\hx}M$ pour tout $\epsilon >0$. 
Fixons $\epsilon_o>0$ tel que $B(\epsilon_o)\subset \Ucal$. Comme $\lim_{t\to+\infty}e^{t\,\hgamma_o}\varphi(x)=0$, on a 
$$
\varphi(x)\in \bigoplus_{\langle\beta,\hgamma_o\rangle<0}(T_{\hx}M)_\beta\subset \bigoplus_{\langle\beta,\gamma_o\rangle\leq 0}(T_{\hx}M)_\beta.
$$
Soit $t_0>0$, tel que $e^{t_0\hgamma_o}\varphi(x)\in B(\epsilon_o)$. Alors 
\begin{equation}\label{eq:s-t-rho}
\rho e^{t\,\hgamma_o}e^{s\gamma_o}\varphi(x)\in B(\epsilon_o),\qquad \forall t\geq t_0,\forall s\geq 0,\forall \rho\in [0,1].
\end{equation}
Notons $v_x:=\lim_{s\to +\infty}e^{s\gamma_o}\varphi(x)\in B(\epsilon_o)\cap (\T_{\hx}M)^\gamma$. Grâce à (\ref{eq:s-t-rho}), on voit que 
$\rho v_x\in \Ucal\cap (\T_{\hx}M)^\gamma$ pour tout $\rho\in [0,1]$, ainsi $v_x\in (\Ucal\cap (\T_{\hx}M)^{\gamma})_0$, et donc 
$\pi(x)=\varphi^{-1}(v_x)\in (\Vcal\cap M^\gamma)_0$. 

Le vecteur $v_x$ appartient à $ B(\epsilon_o) \bigcap \bigoplus_{\langle\beta,\hgamma_o\rangle<0}(T_{\hx}M)_\beta$, et donc 
$\lim_{t\to+\infty} e^{t\hgamma_o}v_x=0$. Cela entraine que $\lim_{t\to+\infty} e^{t\,\hgamma_o}\pi(x)=\hx$.

Il nous reste à vérifier le dernier point. On a déjà montré que $\forall x\in C^-$, si $x\in\hC^-$ alors $\pi(x)\in \hC^-\cap C$. Montrons la réciproque lorsque 
$(TM\vert_{\hC})^{\hgamma_o< 0}\cap (TM\vert_{\hC})^{\gamma_o> 0}=0$.

Soit $x\in C^-$ tel que $\pi(x)\in \hC^-\cap C$. Notons $\hx=\hpi(\pi(x))\in \hC$. Comme précédemment, considérons un difféomorphisme holomorphe 
$\mathbb{T}_\C$-équivariant $\varphi : \Vcal \to \Ucal$ où $\Ucal$  est un voisinage ouvert $\mathbb{T}_\C$-invariant de $0$ dans 
$\T_{\hx}M$ et $\Vcal$ est un voisinage ouvert $\mathbb{T}_\C$-invariant de $\hx$ dans $M$. Comme $\hx$ appartient à 
$\overline{\mathbb{T}\pi(x)}\subset \overline{\mathbb{T}x}$, les éléments
$x$ et $\pi(x)$ appartient à $\Vcal$. Il s'ensuit que $\varphi(x)\in (T_{\hx}M)^{\gamma_o\leq 0}$ et 
$\varphi(\pi(x))=\lim_{t\to\infty}e^{t\gamma_o}\varphi(x)\in (T_{\hx}M)^{\hgamma_o <0}$. D'après nos hypothèses, on a 
$$
(T_{\hx}M)^{\gamma_o = 0} = (T_{\hx}M)^{\hgamma_o = 0} \bigoplus (T_{\hx}M)^{\gamma_o =0}\cap (T_{\hx}M)^{\hgamma_o >0}.
$$
et 
$$
(T_{\hx}M)^{\gamma_o \leq 0} = (T_{\hx}M)^{\hgamma_o \leq 0} \bigoplus (T_{\hx}M)^{\gamma_o =0}\cap (T_{\hx}M)^{\hgamma_o >0}.
$$
Ces deux relations impliquent que $(T_{\hx}M)^{\gamma_o < 0} \subset  (T_{\hx}M)^{\hgamma_o < 0}$. Ainsi on obtient 
$\varphi(x)-\varphi(\pi(x))\in (T_{\hx}M)^{\gamma_o < 0} \subset  (T_{\hx}M)^{\hgamma_o < 0}$ et $\varphi(\pi(x))\in (T_{\hx}M)^{\hgamma_o <0}$. 
On a donc montré que $\varphi(x)\in  (T_{\hx}M)^{\hgamma_o < 0}$. Cela signifie que $\lim_{t\to\infty}e^{t\hgamma_o}x=\hx\in \hC$, soit $x\in\hC^-$.
\end{proof}

\medskip

Notons $B^\gamma= B\cap U_\C^{\gamma}$. Dans le théorème suivant, qui raffine le résultat de \cite{Res11b}, nous comparons les différents morphismes :
$$
q_\gamma : B\times_{B\cap P(\gamma)} C^-\longrightarrow M,\qquad\mathrm{et} \qquad q_{\hgamma} : B\times_{B\cap P(\hgamma)} \hC^-\longrightarrow M,
$$
$$
q_\Xcal : \Xcal\longrightarrow C,
\qquad\mathrm{et} \qquad q_{\Xcal}^- : \Xcal^-\longrightarrow C^-.
$$
où 
$$
\Xcal^-=\left(B\cap P(\gamma)\right)\times_{B\cap P(\hgamma)} \hC^-\simeq B^\gamma\times_{B^\gamma\cap P(\hgamma)} \hC^-  
\qquad \mathrm{et}\qquad
\Xcal=B^\gamma\times_{B^\gamma\cap P(\hgamma)}  \hC^-\cap C.
$$

Comme $B\times_{B\cap P(\hgamma)} \hC^-$ est isomorphe à $B\times_{B\cap P(\gamma)} \Xcal^-$,  nous voyons que
\begin{equation}\label{eq:composition-q-gamma}
q_{\hgamma}= (Id\times q_{\Xcal}^-)\circ q_\gamma
\end{equation}
où $Id\times q_{\Xcal}^-: B\times_{B\cap P(\gamma)} \Xcal^-\to B\times_{B\cap P(\gamma)} C$ est le morphisme qui envoie $[b,x]$ sur $[b,q_{\Xcal}^-(x)]$.

\medskip

Le prochain théorème est une version améliorée d'un résultat obtenu par Ressayre \cite[Proposition 1]{Res11b}.

\begin{theorem}\label{th:paire-ressayre-multiplicatif}
$(\hC,\hgamma\,)$ est une paire de Ressayre de degré fini pour la $U_\C$-variété $M$ si et seulement si les conditions suivantes sont vérifiées
\begin{itemize}\setlength{\itemsep}{8pt}
\item[$i)$] $(C,\gamma)$ est une paire de Ressayre de degré fini pour la $U_\C$-variété $M$.
\item[$ii)$] $(\hC,\hgamma)$ est une paire de Ressayre de degré fini pour la $U^\gamma_\C$-variété $C$.
\item[$iii)$] $\tr\left(\hgamma\circlearrowright \ngot^{\gamma_o>0}\right)=\tr\left(\hgamma\circlearrowright (TM\vert_{\hC})^{\gamma_o>0}\right)$.
\end{itemize}
\medskip

Dans ce cas,  les degrés satisfont la relation multiplicative $\deg(q_{\hgamma})=\deg(q_{\Xcal}) \deg(q_\gamma)$.

\end{theorem}

\begin{proof}
Supposons que $(\hC,\hgamma\,)$ est une paire de Ressayre de degré fini pour la $U_\C$-variété $M$. Alors, il existe $\hx\in\hC$ tel que l'application linéaire 
\begin{equation}\label{eq:iso-tilde-condition-d}
X\in \ngot^{\hgamma_o>0}\longmapsto X\cdot \hx\in (T_{\hx} M)^{\hgamma_o>0}
\end{equation}
est un isomorphisme linéaire. Ainsi, 
$$
T_{\hx} M =\ngot^{\hgamma_o>0}\cdot \hx\bigoplus (T_{\hx} M)^{\hgamma_o\leq 0}\quad\mathrm{et}\quad \ngot^{\hgamma_o>0}\cap (\ugot_\C)_{\hx}=\{0\}.
$$
Puisque $(T_{\hx} M)^{\hgamma_o\leq 0}\subset (T_{\hx} M)^{\gamma_o\leq 0}$ et $(\ugot_\C)^{\gamma_o>0}\subset (\ugot_\C)^{\hgamma_o>0}$, on obtient
\begin{align*} 
T_{\hx} M &=  \ngot^{\gamma_o>0}\cdot \hx\bigoplus (T_{\hx} M)^{\gamma_o\leq 0} \\ 
T_{\hx} C &=  \ngot^{\hgamma_o>0,\gamma_o= 0}\cdot \hx\bigoplus T_{\hx} (\hC^{-}\cap C)
\end{align*}
ainsi que $\ngot^{\hgamma_o>0,\gamma_o= 0}\cap (\ugot_\C)_{\hx}=\ngot^{\gamma_o>0}\cap (\ugot_\C)_{\hx}=\{0\}$.
Cela montre que les applications linéaires 
\begin{align} 
X\in \ngot^{\gamma_o>0}& \longmapsto X\cdot \hx\in (T_{\hx} M)^{\gamma_o> 0}\label{eq:iso-condition-d-1}\\
Z\in \ngot^{\hgamma_o>0,\gamma_o= 0} &\longmapsto Z\cdot \hx\in (T_{\hx} C)^{\hgamma_o>0}\label{eq:iso-condition-d-3}
\end{align}
sont des isomorphismes linéaires $\mathbb{T}$-équivariants. L'égalités $iii)$ découle de l'isomorphisme (\ref{eq:iso-condition-d-1}).

D'après notre hypothèse, le morphisme $q_{\hgamma}$ est dominant, de degré fini noté $\deg(q_{\hgamma})$. L'isomorphisme (\ref{eq:iso-condition-d-1}) 
permet de voir que les variétés $M$, $B\times_{B\cap P(\gamma)} \Xcal^-$, et $B\times_{B\cap P(\gamma)} C$ sont de même dimension. Alors la relation 
(\ref{eq:composition-q-gamma}) montre que les morphismes $q_{\Xcal}^-$ et $q_\gamma$ sont dominants, et leurs degrés satisfont la relation multiplicative
\begin{equation}\label{eq:degre-multiplicatif}
\deg(q_{\hgamma})=\deg(Id\times q_{\Xcal}^-) \deg(q_\gamma)=\deg(q_{\Xcal}^-) \deg(q_\gamma).
\end{equation}
On a montré que le morphisme $q_\gamma$ est dominant: grace à (\ref{eq:iso-condition-d-1}), on peut conclure que $(C,\gamma)$ est une 
paire de Ressayre de degré fini pour la $U_\C$-variété $M$.

Comme  $\ngot^{\gamma_o>0}\cap\ngot^{\hgamma_o<0}=0$ d'après le point {\em 2.b)} de la proposition \ref{prop:hypothesesH1H2}, l'isomorphisme 
(\ref{eq:iso-condition-d-1}) montre que $(TM\vert_{\hC})^{\hgamma_o< 0}\cap (TM\vert_{\hC})^{\gamma_o> 0}=0$. Grâce au cinquième point de la proposition \ref{prop:hypothesesH1H2}, on sait alors que $\pi: \hC^-\to \hC^-\cap C$ est un sous-fibré de $\pi: C^-\to C$. Ainsi, on a le diagramme commutatif 
\begin{equation}\label{diagramme-commutatif}
\xymatrixcolsep{5pc}
\xymatrix{
\Xcal^-  \ar[d]^{Id\times\pi}  \ar[r]^{q^{-}_{\Xcal}} & C^- \ar[d]^{\pi} \\
\Xcal \ar[r] ^{q_{\Xcal}}         & C.
}
\end{equation}
et les fibres des morphismes $q^{-}_{\Xcal}$ et $q_{\Xcal}$ s'identifient au moyen de la projection $Id\times\pi$. Cela montre qu'il existe un ouvert de Zariski $\Vcal_\Xcal\subset C$ pour lequel $\sharp q_{\Xcal}^{-1}(y) = \deg(q^{-}_{\Xcal}),\forall y\in\Vcal_\Xcal$. Cela montre que $q_{\Xcal}$ est une application dominante 
de degré égal à $d(q^{-}_{\Xcal})$. Grâce à (\ref{eq:iso-condition-d-3}), on peut conclure que $(\hC,\hgamma)$ est une paire de Ressayre de degré fini pour la $U^\gamma_\C$-variété $C$. 

On a démontré les points $i)$, $ii)$, et $iii)$, ainsi que la relation multiplicative, lorsque $(\hC,\hgamma)$ est une paire de Ressayre de degré fini pour la $U_\C$-variété $M$. Démontrons maintenant l'implication réciproque.

Supposons que les points $i)$ et $ii)$ sont satisfaits. Le morphisme $q_\Xcal : \Xcal\longrightarrow C$ est dominant d'indice fini $\deg(q_\Xcal)$ et on sait d'autre part que 
$(TM\vert_{\hC})^{\hgamma_o< 0}\cap (TM\vert_{\hC})^{\gamma_o> 0}=0$ grâce à l'isomorphisme (\ref{eq:iso-condition-d-1}).  Cela montre que 
$\pi: \hC^-\to \hC^-\cap C$ est un sous-fibré de $\pi: C^-\to C$. Grâce au diagramme commutatif (\ref{diagramme-commutatif}), on voit alors que 
morphisme $q^-_\Xcal : \Xcal^-\longrightarrow C^-$ est aussi dominant, d'indice $\deg(q_\Xcal)$. La relation $q_{\hgamma}= (Id\times q_{\Xcal}^-)\circ q_\gamma$ 
montre que le morphisme $q_{\hgamma}$ est dominant d'indice égal à $\deg(q_\Xcal)\deg(q_\gamma)$. Pour pouvoir conclure que $(\hC,\hgamma)$ 
est une paire de Ressayre de degré fini pour la $U_\C$-variété $M$, il faut montrer que l'ouvert de Zariski $\widehat{\Omega}\subset \hC$ 
formé des points $\hx\in \hC$ pour lesquels (\ref{eq:iso-tilde-condition-d}) est un isomorphisme est non-vide.

Comme $(C,\gamma)$ est une paire de Ressayre, on sait que l'ouvert de Zariski  $\Omega\subset C$ formé des points $x\in C$ 
pour lesquels (\ref{eq:iso-condition-d-1}) est un isomorphisme est non-vide. De la même façon, l'ouvert de Zariski  $\Omega'\subset \hC$ 
formé des points $\hx\in \hC$ pour lesquels (\ref{eq:iso-condition-d-3}) est un isomorphisme est non-vide, car $(\hC,\hgamma)$ est une 
paire de Ressayre pour la $U^\gamma_\C$-variété $C$. Il est immédiat de voir que 
$$
\widehat{\Omega}= \Omega' \bigcap \Omega\cap \hC,
$$
ainsi $\widehat{\Omega}$ est non-vide si on montre que $\Omega\cap \hC$ est non-vide.

Considérons le fibré en droites $B^\gamma$-équivariant $\mathbb{L}\to C$ défini par la relation
$$
\mathbb{L}=\hom(\det(\ngot^{\gamma_o>0}), \det ((TM\vert_C)^{\gamma_o>0})).
$$
Soit $r\geq 0$ le rang de $\ngot^{\gamma_o>0}$. Nous considérons la section $B^\gamma$-équivariante $\theta\in H^0(C, \mathbb{L})$ définie par la relation
$$
\theta(x): X_1\wedge\cdots\wedge X_r\longmapsto X_1\cdot x\wedge\cdots\wedge X_r\cdot x.
$$
Ainsi $\Omega=\{x\in C, \theta(x)\neq 0\}$. L'application $q_\Xcal$ étant dominante, son image $B^\gamma(\hC^-\cap C)$ est dense dans $C$. Ainsi, 
la restriction $\theta\vert_{\hC^-\cap C}$ est une section non nulle $\mathbb{T}$-équivariante du fibré en droites  
$\mathbb{L}\vert_{\hC^-\cap C}\to \hC^-\cap C$. La sous-variété des points fixes $(\hC^-\cap C)^{\mathbb{T}}$ est égale à $\hC$, et un résultat classique 
(qui conduit au lemme \ref{lem:RP-fondamental}) assure que $\theta\vert_{\hC}$ est non nulle si et seulement si $\mathbb{T}$ agit trivialement sur le fibré en 
droites $\mathbb{L}\vert_{\hC}$ (voir \cite[Lemma 2.2]{pep-ressayre}).

Ici, $\gamma$ agit trivialement sur $\mathbb{L}\vert_{\hC}$ car $(C, \gamma)$ est une paire de Ressayre infinitésimale, et 
$\hgamma$ agit trivialement sur $\mathbb{L}\vert_{\hC}$ grâce à notre hypothèse $iii)$. On vient de vérifier que $\Omega\cap \hC=\{\hx\in\hC, \theta(\hx)\neq 0\}$ 
est non-vide.
\end{proof}

\section{Formules multiplicatives de Ressayre-Richmond}\label{sec:RR-mutliplicative}

Nous revenons au contexte de la section \ref{sec:cohomology-flag}. Le groupe de Weyl de $(U,T)$ est noté $W$, et le groupe de Weyl de $(U^\gamma,T)$ 
est égal au sous-groupe $W^\gamma$ fixant $\gamma\in\tgot$. Un résultat classique nous assure que chaque classe d'équivalence du quotient $W/W^\gamma$ 
possède un unique représentant de longueur minimale. Notons $W(\gamma)\subset W$ l'ensemble de ces représentants, de telle sorte que 
$W(\gamma)\simeq W/W^\gamma$.

Considérons deux vecteurs $\gamma,\hgamma\in\tgot$ vérifiant $P(\,\hgamma\,)\subset P(\gamma)$. 
Cette identité implique que $U^{\hgamma}\subset U^\gamma$ et donc que $W^{\hgamma}\subset W^\gamma$. 
Notons $W^\gamma(\,\hgamma\,)\subset W^\gamma$ l'ensemble des représentants de longueur minimale des classes d'équivalence du quotient 
$W^\gamma/W^{\hgamma}$. On voit aisément que $W^\gamma(\,\hgamma\,)$ coincide avec $W(\,\hgamma\,)\cap W^{\gamma}$.

Le résultat suivant est expliqué dans \cite[Lemma 2.2]{Ric12},

\begin{lem}\label{lem:richmond-lemma}
L'application $W(\gamma) \times W^\gamma(\,\hgamma\,) \to W(\,\hgamma\,)$ donnée par $(u,v)\mapsto uv$  est bien définie et bijective.
\end{lem}

Le lemme précédent montre que, pour tout $w\in W(\,\hgamma\,)$, il existe $u\in W(\gamma)$ et  $v\in W^\gamma(\,\hgamma\,)$, 
uniques, tels que $w=uv$. Nous supposerons cette relation entre $w$, $u$ et $v$ pour tout $w\in W(\,\hgamma\,)$.

Nous supposons maintenant que $U$ est un sous-groupe fermé de $\tU$.  On choisit des tores maximaux 
$T\subset U$, et $\tT\subset \tU$ tels que $T\subset \tT$. \`A chaque $\gamma\in\tgot$, 
on associe les sous-groupes paraboliques $\tP(\gamma)\subset \tU_\C$ et $P(\gamma)= \tP(\gamma)\cap U_\C$.

Dans cette section, nous fixons deux vecteurs $\gamma,\hgamma\in\tgot$ vérifiant 
\begin{equation}\label{eq:condition-gamma-hat}
\tP(\,\hgamma\,)\subset \tP(\gamma).
\end{equation} 
Cela implique que $P(\,\hgamma\,)$ est contenu dans $P(\gamma)$. Soit $\tW$ le groupe de Weyl de $(\tU,\tT)$. D'après le lemme \ref{lem:richmond-lemma}, 
pour tout $\tw\in \tW(\,\hgamma\,)$, il existe $\tu\in \tW(\gamma)$ et  $\tv\in \tW^\gamma(\,\hgamma\,)$, 
uniques, tels que $\tw=\tu\tv$.

Ici, nous travaillons avec trois immersions  entre des variétés de drapeaux :
$$
\iota_{\hgamma}:\Fcal_{\hgamma}\croc \tFcal_{\hgamma},\qquad
\iota_{\gamma}: \Fcal_{\gamma}\croc \tFcal_{\gamma},\qquad \mathrm{et}\qquad
\iota_{\hgamma}^\gamma:\Fcal_{\hgamma}^\gamma\croc \tFcal_{\hgamma}^\gamma,
$$
où $\Fcal_{\hgamma}^\gamma:=U_\C^\gamma/\left(U_\C^\gamma\cap P(\,\hgamma\,)\right)$ et 
$\tFcal_{\hgamma}^\gamma:=\tU_\C^\gamma/\left(\tU_\C^\gamma\cap \tP(\,\hgamma\,)\right)$.

\medskip

Le résultat suivant, qui une application du théorème \ref{th:paire-ressayre-multiplicatif}, améliore un résultat obtenu par Ressayre-Richmond 
\cite{Res11b,Ric12,RR11}.

\begin{theorem}\label{th:levi-mobile-multiplicatif}
Soient $\gamma,\hgamma\in\tgot$ vérifiant (\ref{eq:condition-gamma-hat}) et un couple 
$(w,\tw)\in W(\,\hgamma\,)\times\tW(\,\hgamma\,)$ satisfaisant la relation $d(w,\hgamma\,)=c(\tw,\hgamma\,)$. 

Alors, le couple 
$(w,\tw)$ est Lévi-mobile si et seulement si les conditions suivantes sont satisfaites :
\begin{enumerate}\setlength{\itemsep}{8pt}
\item $(u,\tu)\in W(\gamma)\times\tW(\gamma)$ est Lévi-mobile, et vérifie $d(u,\gamma)=c(\tu,\gamma)$.
\item $(v,\tv)\in W^\gamma(\,\hgamma\,)\times\tW^\gamma(\,\hgamma\,)$ est Lévi-mobile, et vérifie $d(v,\hgamma)=c(\tv,\hgamma)$.
\item $\tr\left(w\hgamma_o \circlearrowright \ngot^{w\gamma_o>0}\right)=\tr\left(\tw\hgamma_o \circlearrowright \tngot_{-}^{\tw\gamma_o>0}\right)$.
\end{enumerate}
Lorsque ces conditions sont satisfaites, on peut définir trois entiers naturels
\begin{align*}
[\Xgot_{w,\hgamma}]\cdot \iota_{\hgamma}^*[\tXgot_{\tw,\hgamma}]&= N(w,\tw) [pt]\quad \mathrm{dans}\quad H^{max}(\Fcal_{\hgamma},\Z),\\
[\Xgot_{u,\gamma}]\cdot \iota_{\gamma}^*[\tXgot_{\tu,\gamma}]&= N(u,\tu) [pt]\quad \mathrm{dans}\quad H^{max}(\Fcal_\gamma,\Z),\\
[\Xgot_{v,\hgamma}]\cdot (\iota_{\hgamma}^\gamma)^*[\tXgot_{\tv,\hgamma}]&= N(v,\tv) [pt]\quad \mathrm{dans}\quad H^{max}(\Fcal_{\hgamma}^\gamma,\Z),
\end{align*}
qui satisfont la relation $N(w,\tw)=N(u,\tu)N(v,\tv)$.
\end{theorem}

\begin{proof} Fixons quelques notations. Notons $G:=\tU_\C\times U_\C$  le groupe réductif avec sous-groupe de Borel $\B:=\tB\times B$.
\`A tout  couple $\mathbf{w}:=(w,\tw)\in W(\,\hgamma\,)\times\tW(\,\hgamma\,)$, on associe les vecteurs $\gamma_{\mathbf{e}}:=(\gamma,\gamma)$,
$\gamma_{\mathbf{w}}:=(\tw\gamma,w\gamma)$ et $\hgamma_{\mathbf{w}}:=(\tw\hgamma,w\hgamma)$, le groupe de Borel 
$\B_{\mathbf{w}}:=\tw^{-1}\tB\tw \times w^{-1}B w$, les sous-groupes paraboliques 
$$
\Pbb(\gamma_{\mathbf{w}}):= \tP(\tw\gamma)\times  P(w\gamma)\quad \mathrm{et}\quad\Pbb(\,\hgamma_{\mathbf{w}}):= \tP(\tw\,\hgamma)\times  P(w\hgamma\,)
$$ 
de $G$.

Nous considérons la variété de K\"ahler $\tU\times U$-hamiltonienne $N:=\tU_C$ (voir l'exemple \ref{ex:U-tilde-C-Kahler}).
La variété $N^{\hgamma_{\mathbf{w}}}:=\{n\in N,\ \hgamma_{\mathbf{w}}\cdot n= 0\}$ est connexe, égale à 
$\hC:=\tw \tU_\C^{\hgamma} w^{-1}$, et la sous-variété de Bia\l ynicki-Birula associée à $\hC$ est $\hC^-:= \tilde{w}\tP(\,\hgamma\,)w^{-1}$. De même, 
la sous-variété $N^{\gamma_{\mathbf{w}}}:=\{n\in N,\ \gamma_{\mathbf{w}}\cdot n= 0\}$ est connexe, égale à $C:=\tw \tU_\C^\gamma w^{-1}$, et la sous-variété de Bia\l ynicki-Birula associée est $C^-:= \tilde{w}\tP(\gamma)w^{-1}$.

Si on utilise les décompositions $w=uv$ et $\tw=\tu\tv$ du lemme \ref{lem:richmond-lemma}, on remarque que $\gamma_{\mathbf{w}}=\gamma_{\mathbf{u}}$ et donc que
$C:=\tu \tU_\C^\gamma u^{-1}$, et $C^-= \tu\tP(\gamma)u^{-1}$.

La $G$-variété $N$ satisfait les conditions de la section \ref{sec:multiplicatif-ressayre-pair}, puisque $\hC\subset C$, $\hC^-\subset C^-$, et 
$$
\Pbb(\,\hgamma_{\mathbf{w}})\subset \Pbb(\gamma_{\mathbf{w}})=\Pbb(\gamma_{\mathbf{u}}).
$$

D'après la proposition \ref{prop:RP-exemple-fondamental-2}, on sait que 
\begin{enumerate}
\item[(1)] $(\hgamma_{\mathbf{w}},\hC)$  est une $\B$-paire de Ressayre de type fini de la $G$-variété $N$ 
si et seulement si $(w,\tw)\in W(\,\hgamma\,)\times \tW(\,\hgamma\,)$ est Lévi-mobile, et vérifie 
$d(w,\hgamma)=c(\tw,\hgamma)$.
\item[(2)] $(\gamma_{\mathbf{u}},C)$  est une $\B$-paire de Ressayre de type fini de la $U_\C$-variété $N$ si et seulement si 
$(u,\tu)\in W(\gamma)\times \tW(\gamma)$ est Lévi-mobile, et vérifie 
$d(u,\gamma)=c(\tu,\gamma)$.
\end{enumerate}

Le point $(1)$ est satisfait lorsque le morphisme
$$
q_{\hgamma,\mathbf{w}}: \B\times_{\B\cap \Pbb(\,\hgamma_{\mathbf{w}})} \hC^- \longrightarrow  \tU_\C
$$
est de type fini et que la condition $(A_2)$ de la définition \ref{def:B-ressayre-pair-alg} est vérifiée. Si on modifie $q_{\hgamma,\mathbf{w}}$ 
à travers le difféomorphisme $g\in \tU_\C\mapsto \tw^{-1}g w\in \tU_\C$, on obtient le morphisme
$$
Q_{\hgamma,\mathbf{w}}: \B_{\mathbf{w}}\times_{\B_{\mathbf{w}}\cap \Pbb(\,\hgamma\,)} \tP(\,\hgamma\,) \longrightarrow  \tU_\C.
$$
Autrement dit, $(1)$ est équivalent $(1')$: $(\hgamma, \tU_\C^{\,\hgamma})$ est une $\B_{\mathbf{w}}$-paire de Ressayre de type 
fini de la $G$-variété $N$ (voir la remarque \ref{rem:B-paire}).

Exploitons maintenant le théorème \ref{th:paire-ressayre-multiplicatif}. Le point $(1')$ est satisfait si et seulement si les conditions suivantes sont vérifiées:
\begin{itemize}\setlength{\itemsep}{5pt}
\item[(3)] $(\gamma, \tU_\C^{\gamma})$ est une $\B_{\mathbf{w}}$-paire de Ressayre de degré fini pour la $G$-variété $N$.
\item[(4)] $(\hgamma, \tU_\C^{\,\hgamma})$ est une $\B_{\mathbf{w}}\cap G^\gamma$-paire de Ressayre de degré fini pour la $G^\gamma$-variété $\tU_\C^{\gamma}$.
\item[(5)] Condition numérique $iii)$.
\end{itemize}

Grâce à la remarque \ref{rem:B-paire}, on voit que le point $(3)$ est équivalent au point $(2)$. D'autre part, une vérification élémentaire 
montre que la condition numérique $iii)$ du théorème théorème \ref{th:paire-ressayre-multiplicatif} correspond à $\tr\left(w\hgamma_o \circlearrowright \ngot^{w\gamma_o>0}\right)=\tr\left(\tw\hgamma_o \circlearrowright \tngot_{-}^{\tw\gamma_o>0}\right)$. Le point $(4)$ est satisfait lorsque le morphisme
$$
r_{\hgamma,\mathbf{w}}: \B_{\mathbf{w}}\cap G^\gamma\times_{\B_{\mathbf{w}}\cap G^\gamma\cap\Pbb(\,\hgamma\,)} 
\left(\tP(\,\hgamma\,) \cap \tU_\C^\gamma\right)\longrightarrow  \tU_\C^\gamma.
$$
est de type fini et que la condition $(A_2)$ de la définition \ref{def:B-ressayre-pair-alg} est vérifiée.

Comme $\B_{\mathbf{w}}\cap G^\gamma=\tv^{-1}\tB^\gamma\tv\times v^{-1}B^\gamma v$, le morphisme $r_{\hgamma,\mathbf{w}}$ 
peut être modifié à travers le difféomorphisme $g\in \tU_\C^\gamma\mapsto \tv g v^{-1}\in \tU_\C^\gamma$, pour obtenir  le morphisme
$$
R_{\hgamma,\mathbf{v}}: \B\cap G^\gamma\times_{\B\cap G^\gamma\cap\Pbb(\gamma_{\mathbf{v}})} 
\left(\tv\tP(\,\hgamma\,)v^{-1} \cap \tU_\C^\gamma\right)\longrightarrow  \tU_\C^\gamma.
$$
Ces considérations montrent que le point $(4)$ est équivalent au fait que $(v,\tv)\in W^\gamma(\,\hgamma\,)\times\tW^\gamma(\,\hgamma\,)$ est Lévi-mobile, et vérifie $d(v,\hgamma)=c(\tv,\hgamma)$.

Lorsque les conditions 1., 2. et 3. du théorème \ref{th:levi-mobile-multiplicatif} sont satisfaites, les applications $q_{\hgamma,\mathbf{w}}$, $R_{\hgamma,\mathbf{v}}$ et 
$q_{\gamma,\mathbf{u}}: \B\times_{\B\cap \Pbb(\gamma_{\mathbf{u}})} C^- \longrightarrow  \tU_\C$ sont de degrés finis, et de plus 
$\deg(q_{\hgamma,\mathbf{w}})=\deg(q_{\gamma,\mathbf{u}})\deg(R_{\hgamma,\mathbf{v}})$. Cette dernière relation implique que $N(w,\tw)=N(u,\tu)N(v,\tv)$
(voir la remarque \ref{rem:degre-cohomologie}). La preuve du théorème est complète.
\end{proof}

\medskip

Réécrivons le théorème \ref{th:levi-mobile-multiplicatif} lorsque $U$ est inclus diagonalement dans $\tU:=U^s$. Dans ce cas, le résultat a été obtenu par Richmond 
\cite[Section 1.1]{Ric12} et Ressayre \cite[Theorem 2]{Res11a}.

\begin{theorem}\label{th:levi-mobile-multiplicatif-2}
Le $(s+1)$-uplet $w_\bullet:=(w_1,\ldots,w_{s+1})\in W(\,\hgamma\,)^{s+1}$ est Lévi-mobile, et vérife $\sum_{i=1}^{s+1} d(w_i,\hgamma)= s\dim_\C(\Fcal_{\hgamma})$ 
si et seulement si
\begin{enumerate}\setlength{\itemsep}{8pt}
\item $u_\bullet:=(u_1,\ldots,u_{s+1})\in W(\gamma)^{s+1}$ est Lévi-mobile, et vérife $\sum_{i=1}^{s+1} d(u_i,\gamma)= s\dim_\C(\Fcal_\gamma)$.
\item $v_\bullet:=(v_1,\ldots,v_{s+1})\in (W^\gamma(\,\hgamma\,))^{s+1}$ est Lévi-mobile, et vérife 
$\sum_{i=1}^{s+1} d(v_i,\hgamma)= s\dim_\C(\Fcal^\gamma_{\hgamma})$.
\item $\sum_{i=1}^{s+1}\tr\left(w_i\hgamma_o \circlearrowright \ngot^{w_i\gamma_o>0}\right)=s\tr\left(\hgamma_o \circlearrowright \ugot_{\C}^{\gamma_o>0}\right)$.
\end{enumerate}

Lorsque ces conditions sont satisfaites, on peut définir trois entiers naturels
\begin{align*}
[\Xgot_{w_1,\hgamma}]\cdots [\Xgot_{w_{s+1},\hgamma}]&= N(w_\bullet) [pt]\quad \mathrm{dans}\quad H^{max}(\Fcal_{\hgamma},\Z),\\
[\Xgot_{u_1,\gamma}]\cdots [\Xgot_{u_{s+1},\gamma}]&= N(u_\bullet) [pt]\quad \mathrm{dans}\quad H^{max}(\Fcal_\gamma,\Z),\\
[\Xgot_{v_1,\hgamma}]\cdots [\Xgot_{v_{s+1},\hgamma}]&= N(v_\bullet) [pt]\quad \mathrm{dans}\quad H^{max}(\Fcal_{\hgamma}^\gamma,\Z),
\end{align*}
qui satisfont la relation $N(w_\bullet)=N(u_\bullet)N(v_\bullet)$.
\end{theorem}

\section{Quelques exemples}\label{sec:application-multiplicatif}

\subsection{Première application des formules mutliplicatives }\label{sec:application-multiplicatif-1}

Considérons le contexte où $U=\GL_n(\C)$ est inclus diagonalement dans $\tU:=\GL_n(\C)^s$, avec $s>1$.

Soient $n > b>a\geq 1$. Nous travaillons avec les vecteur $\gamma:=\gamma^{b}$ et $\hgamma:=\gamma^{a,b}$ définis respectivement en 
(\ref{eq:gamma-1-etage}) et (\ref{eq:gamma-a-b}). Ceux-ci satisfont la condition (\ref{eq:condition-gamma-hat}) de la section précédente. Explicitons la bijection $W(\gamma) \times W^\gamma(\,\hgamma\,) \simeq W(\,\hgamma\,)$ dans ce contexte.

Les applications $v\mapsto v([a])$ et $u\mapsto u([b])$ identifient respectivement $W^\gamma(\,\hgamma\,)$ à $\Pcal(a,b)$, et $W(\gamma)$ avec $\Pcal(b,n)$. 
Ainsi le lemme \ref{lem:richmond-lemma} nous fournit une bijection canonique $(\bullet)\quad \Pcal(a,b;n)\simeq \Pcal(b;n)\times \Pcal(a;b)$.

\begin{definition}\label{def:natural}
\begin{enumerate}
\item Pour des sous-ensembles finis $X,Y\subset\N$, on pose $ \sharp X < Y := \mathrm{Cardinal}\{(x,y)\in X\times Y, x<y\}$. Idem pour $\sharp X \leq Y $.
\item Pour des sous ensembles $A\subset B\subset \N-\{0\}$, on pose $A\natural B:=\{\sharp B\leq x , x\in A\}\in \Pcal(\sharp A, \sharp B)$.
\end{enumerate}
\end{definition}

\begin{lem}\label{lem:bijection-lemme-richmond-exemple}
L'image de $(I\subset J)\in \Pcal(a,b;n)$ à travers la bijection $(\bullet)$ est le couple 
$(J, I\natural J)$.
\end{lem}

\begin{proof}Soit $w\in W(\,\hgamma\,)$ tel que $I=w([a])$ et $J=w([b])$. Soit $(u,v)\in W(\gamma) \times W^\gamma(\,\hgamma\,)$ tel que 
$w=uv$. Comme $v([b])=[b]$, on a $u([b])=w([b])=J$, i.e. $J\in  \Pcal(b;n)$ correspond à $v\in W(\gamma)$ à travers la bijection $(\bullet)$.

L'application $v$ détermine une permutation de $[b]$, l'application $u$ réalise une bijection croissante entre $[b]$ et $J$, et nous avons 
$u(v(\ell))=w(\ell)\in I$ pour tout $\ell\in [a]$. Ainsi, si $x=u(v(\ell))\in I$, on a $v(\ell)=\sharp J\leq x$. Cela montre que 
$v([a])$ est égal à $I\natural J:=\{\sharp J\leq x , x\in I\}$, i.e. $I\natural J\in  \Pcal(a;b)$ correspond à $v\in W^\gamma(\,\hgamma\,)$ 
à travers la bijection $(\bullet)$.
\end{proof}

Nous utilisons le théorème \ref{th:levi-mobile-multiplicatif-2} dans le contexte où les variétés de drapeaux $\Fcal_{\hgamma}=\GL_n(\C)/P(\,\hgamma\,)$, 
$\Fcal_{\gamma}=\GL_n(\C)/P(\gamma)$, et $\Fcal_{\hgamma}^\gamma=\GL_n(\C)^\gamma/\GL_n(\C)^\gamma\cap P(\,\hgamma\,)$ sont respectivement 
égales à $\Fcal_{\hgamma}=\F(a,b; n)$, $\Fcal_{\gamma}=\G(b; n)$, et $\Fcal_{\hgamma}^\gamma=\G(a;b)$.

Les calculs élémentaires suivants sont laissés  à la discrétion du lecteur. Ici, $\ngot\subset\glgot_n(\C)$ désigne la sous-algèbre nilpotente 
formées des matrices triangulaires supérieures.

\begin{lem}\label{lem:calcul-dimensions}

Soient $w\in \Sgot_n$, $I=w([a])$ et $J=w([b])$.

\begin{itemize}\setlength{\itemsep}{10pt}
\item $\dim_\C \F(a,b; n) = a(b-a)+b(n-b)$.
\item $\dim_\C \Xgot_{I\subset J}^o= |\mu(I)|+|\mu(J)|- \sharp J^c <I= |\mu(J)|+|\mu(I\natural J)|$.
\item $\tr\left(w\gamma^{b}_o \circlearrowright \ngot^{w\gamma^{b}_o>0}\right)=\sharp J^c <J= |\mu(J)| $.
\item $\tr\left(w\gamma^{a,b}_o \circlearrowright \ngot^{w\gamma^{a,b}_o>0}\right)=|\mu(I)|+|\mu(J)|$.
\item $\tr\left(w\gamma^{a,b}_o \circlearrowright \ngot^{w\gamma^{b}_o>0}\right)=|\mu(J)| +  \sharp J^c <I= |\mu(I)|+|\mu(J)|-|\mu(I\natural J)|$.
\item $\tr\left(\gamma_o^{a,b} \circlearrowright \glgot_n(\C)^{\gamma^b_o>0}\right)= (n-b)(a+b)$.
\item $\tr\left(\gamma_o^{a,b} \circlearrowright \glgot_n(\C)^{\gamma^{a,b}_o>0}\right)= a(n-a)+b(n-b)$.
\end{itemize}
\end{lem}

Soit $(I_k\subset J_k)_{1\leq k\leq s+1}$ une collection d'éléments de $\Pcal(a,b;n)$. D'après les calculs menés au lemme \ref{lem:calcul-dimensions}, on sait que le 
produit  $\prod_{k=1}^{s+1} [\Xgot_{I_k\subset J_k}]$ appartient à $H^{max}(\F(a,b; n),\Z)$ si 
\begin{equation}\label{eq:degre-max-F-a-b}
\sum_{k=1}^{s+1}|\mu(J_k)|+|\mu(I_k\natural J_k)| =  s\,\Big(a(b-a)+b(n-b)\Big).
\end{equation}
Une collection de sous-ensembles $(I_k\subset J_k)_{1\leq k\leq s}$ de $\Pcal(a,b;n)$ est dite Lévi-mobile si 
la famille $(w_k)_{1\leq k\leq s}\in (\Sgot_n)^s$ vérifiant $\varphi(w_k)=(I_k,J_k)$ est Lévi-mobile.

Le résultat suivant, qui est une conséquence du théorème \ref{th:levi-mobile-multiplicatif-2}, compare les produits des classes 
$[\Xgot_{I\subset J}]\in H^*(\F(a,b;n),\Z)$, $[\Xgot_{J}]\in H^*(\G(b;n),\Z)$ et $[\Xgot_{I\natural J}]\in H^*(\G(a;b),\Z)$.

\begin{prop} \label{prop:levi-mobile-double-grass}
Soit $(I_k\subset J_k)_{1\leq k\leq s+1}$ une collection d'éléments de $\Pcal(a,b;n)$. Les propriétés suivantes sont équivalentes:
\begin{enumerate}\setlength{\itemsep}{8pt}
\item $\prod_{k=1}^{s+1} [\Xgot_{I_k\subset J_k}]= \ell[pt]$, avec $\ell\geq 1$, et $(I_k\subset J_k)_{1\leq k\leq s+1}$ est Lévi-mobile.
\item $\prod_{k=1}^{s+1} [\Xgot_{I_k\subset J_k}]= \ell[pt]$, avec $\ell\geq 1$, et $\sum_{k=1}^{s+1}|\mu(I_k)|+|\mu(J_k)|= s\left(a(n-a)+b(n-b)\right)$.
\item La collection $(I_k\subset J_k)_{1\leq k\leq s+1}$ vérifie les conditions suivantes :
\begin{enumerate}\setlength{\itemsep}{8pt}
\item $\prod_{k=1}^{s+1} [\Xgot_{J_k}]= \ell'[pt]$, avec $\ell'\geq 1$, dans $H^{max}(\G(b, n),\Z)$.
\item $\prod_{k=1}^{s+1} [\Xgot_{I_k\natural J_k}]= \ell''[pt]$, avec $\ell''\geq 1$, dans $H^{max}(\G(a, b),\Z)$.
\item $\sum_{k=1}^{s+1} |\mu(I_k)|= s\, a(n-a)$.
\end{enumerate}
\end{enumerate}
Lorsque ces conditions sont satisfaites on a $\ell=\ell'\ell''$.
\end{prop}

\begin{proof} L'équivalence entre 1. et 2. est une conséquence de la proposition \ref{prop:RP-exemple-fondamental-3} 
et des calculs effectués au lemme \ref{lem:calcul-dimensions}. La relation du point 3. de la proposition \ref{prop:RP-exemple-fondamental-3}, donne ici
\begin{equation}\label{eq:preuve-prop-levi-grass-1}
\sum_{k=1}^{s+1}|\mu(I_k)|+|\mu(J_k)|-|\mu(I_k\natural J_k)| =  s\,\Big((a+b)(n-b)\Big).
\end{equation}
D'autre part, les relations $\prod_{k=1}^{s+1} [\Xgot_{J_k}]= \ell'[pt]$ et $\prod_{k=1}^{s+1} [\Xgot_{I_k\natural J_k}]= \ell''[pt]$ fournissent 
des égalités de dimensions suivantes
\begin{equation}\label{eq:preuve-prop-levi-grass-2}
\sum_{k=1}^{s+1}|\mu(J_k)|=  s\,\Big(b(n-b)\Big)\quad\mathrm{et}\quad \sum_{k=1}^{s+1}|\mu(I_k\natural J_k)|=  s\,\Big(a(b-a)\Big).
\end{equation}
Finalement, les relations (\ref{eq:preuve-prop-levi-grass-2}) montre que (\ref{eq:preuve-prop-levi-grass-1}) est équivalent à 
$\sum_{k=1}^{s+1} |\mu(I_k)|= s\, a(n-a)$. Cela termine la preuve du fait que la proposition \ref{prop:levi-mobile-double-grass} 
découle du théorème \ref{th:levi-mobile-multiplicatif-2}.
\end{proof}

\medskip

Nous terminons cette partie, en nous plaçant dans le cadre où $a+b=n$ et $a<\tfrac{n}{2}$. Chaque $I\subset [n]$ de cardinal $a$ vérifiant\footnote{$I^o=\{n+1-i, i\in I\}$  et $I^{o,c}=[n]-I^o$}
$I\cap I^o=\emptyset$ définit la variété de Schubert $\Xgot_{I\subset I^{o,c}}$ dans $\F(a,n-a; n)$. Dans la prochaine proposition, 
nous adaptons la proposition \ref{prop:levi-mobile-double-grass} à ce type de variétés de Schubert.

\begin{prop} \label{prop:levi-mobile-double-grass-theta}
Soit $(I_k)_{1\leq k\leq s+1}$ une collection de sous-ensembles de cardinal $a$ de $[n]$ vérifiant $I_k\cap I_k^o=\emptyset,\forall k$.
Alors, les propriétés suivantes sont équivalentes:
\begin{enumerate}\setlength{\itemsep}{8pt}
\item $\prod_{k=1}^{s+1} [\Xgot_{I_k\subset I_k^{o,c}}]= \ell[pt]$ avec $\ell\geq 1$ et $(I_k\subset I_L^{o,c})_{1\leq k\leq s+1}$ est Lévi-mobile.
\item $\prod_{k=1}^{s+1} [\Xgot_{I_k\subset I_k^{o,c}}]= \ell[pt]$ avec $\ell\geq 1$ et $\sum_{k=1}^{s+1} |\mu(I_k)|= s\, a(n-a)$.
\item La collection $(I_k)_{1\leq k\leq s+1}$ vérifie les conditions suivantes :
\begin{enumerate}\setlength{\itemsep}{8pt}
\item $\prod_{k=1}^{s+1} [\Xgot_{I_k}]= \ell'[pt]$, avec $\ell'\geq 1$, dans $H^{max}(\G(a, n),\Z)$.
\item $\prod_{k=1}^{s+1} [\Xgot_{I_k\natural I_k^{o,c}}]= \ell''[pt]$, avec $\ell''\geq 1$, dans $H^{max}(\G(a, n-a),\Z)$.
\end{enumerate}
\end{enumerate}
Lorsque ces conditions sont satisfaites on a $\ell=\ell'\ell''$.
\end{prop}

\begin{proof}
Considérons le produit bilinéaire $(x,y)=\sum_{k=1}^n x_k y_k$ sur $\C^n$. Pour tout sous-espace vectoriel $E\subset \C^n$, notons $E^\perp$ l'orthogonal 
de $E$ par rapport à $(-,-)$. L'application $E\mapsto E^\perp$ définit un difféomorphisme $j: \G(n-a, n)\to \G(a, n)$ tel que 
$$
j^*([\Xgot_{I}])=[\Xgot_{{I}^{o,c}}]\quad \mathrm{et} \quad |\mu(I)|=|\mu(I^{o,c})|
$$
pour tout $I\in\Pcal(a,n)$. Ainsi, la relation $\prod_{k=1}^{s+1} [\Xgot_{I_k}]= \ell'[pt]$ dans $H^{max}(\G(a, n),\Z)$ est équivalente à la relation 
$\prod_{k=1}^{s+1} [\Xgot_{I_L^{o,c}}]= \ell'[pt]$ dans $H^{max}(\G(n-a, n),\Z)$. On voit d'autre part, que l'identité 
$\sum_{k=1}^{s+1} |\mu(I_k)|= s\, a(n-a)$ est une conséquence du point $3.$ $a)$. On a ainsi vérifié que la proposition 
\ref{prop:levi-mobile-double-grass-theta} est une conséquence de la proposition \ref{prop:levi-mobile-double-grass}.
\end{proof}

\subsection{Deuxième application des formules multiplicatives }\label{sec:application-multiplicatif-2}

Soient $n,n'>1$. Considérons le cas où $U_\C=\GL_n(\C)\times \GL_{n'}(\C)$ est inclus diagonalement dans $\tU_\C:=\GL_{n+n'}(\C)$. 
Le tore maximal $T\subset\tU:=\upU_{n+n'}$ des matrices diagonales est aussi un tore maximal de $U=\upU_{n}\times\upU_{n'}$.
Le groupe de Weyl de $(U,T)$ est égal à $W:=\Sgot_{n}\times \Sgot_{n'}$, tandis que celui de $(\tU,T)$ est égal à $\tW:=\Sgot_{n+n'}$.

\subsubsection{Premier cadre }\label{sec:application-multiplicatif-2-1}

Soient $1\leq a<b<n$ et $1\leq a'<b'<n'$. Dans cette partie, on travaille avec les vecteurs de l'algèbre de Lie $\tgot\simeq \R^n\times\R^{n'}$ suivants :
$$
\hgamma:=(\gamma^{a,b},\gamma^{a',b'})\quad \mathrm{et}\quad \gamma:=(\gamma^{b},\gamma^{b'}).
$$
Notons $w_{\mathrm{I}}\in \Sgot_{n+n'}$ la permutation définie par : $w_{\mathrm{I}}(k)=k+n-a$ si $a+1\leq k\leq a+b'$, $w_{\mathrm{I}}(k)=k-n+a$ si 
$a+b'+1\leq k\leq n+b'$, et $w_{\mathrm{I}}(k)=k$ dans les autres cas. On constate que 
$$
\hgamma=\Ad(w_{\mathrm{I}})\left(\gamma^{a+a',b+b'}\right)\quad \mathrm{et}\quad \gamma=\Ad(w_{\mathrm{I}})\left(\gamma^{b+b'}\right).
$$
Le morphisme $\Ad(w_{\mathrm{I}})$ permet d'identifier $\tW(\,\hgamma\,)$, $\tW(\gamma)$ et $\tW^\gamma(\,\hgamma\,)$, respectivement à 
$\tW(\gamma^{a+a',b+b'})$, $\tW(\gamma^{b+b'})$ et $\tW^{\gamma^{b+b'}}(\gamma^{a+a',b+b'})$. Grâce au 
lemme \ref{lem:bijection-lemme-richmond-exemple}, nous avons un premier résultat.

\begin{lem}\label{lem:bijection-lemme-richmond-exemple-2}
\begin{enumerate}
\item La bijection $W(\,\hgamma\,)\simeq W(\gamma) \times W^\gamma(\,\hgamma\,) $ correspond à l'application 
$$
\Pcal(a,b;n)\times\Pcal(a',b';n') \longrightarrow \Pcal(b;n)\times \Pcal(b';n')\times \Pcal(a;b)\times \Pcal(a';b')
$$
qui envoie $(I\subset J,I'\subset J')$ sur $(J,J',I\natural J, I'\natural J')$.
\item La bijection $\tW(\,\hgamma\,)\simeq \tW(\gamma) \times \tW^\gamma(\,\hgamma\,) $ correspond, à travers $\Ad(w_I)$, à l'application 
$$
\Pcal(a+a',b+b';n+n')\longrightarrow \Pcal(b+b';n+n')\times  \Pcal(a+a';b+b')
$$
qui envoie $(I''\subset J'')$ sur $(J'',I''\natural J'')$.
\end{enumerate}
\end{lem}

Maintenant, nous travaillons avec les morphismes 
$$
\iota_\gamma: \Fcal_\gamma\to \tFcal_{\gamma},\quad \iota_{\,\hgamma\,}: \Fcal_{\,\hgamma\,}\to \tFcal_{\,\hgamma\,},\quad \mathrm{et}\quad
\iota_{\,\hgamma\,}^\gamma: \Fcal_{\,\hgamma\,}^\gamma\to \tFcal_{\,\hgamma\,}^\gamma.
$$
Nous utilisons les isomorphismes canoniques $\Fcal_\gamma\simeq \G(b;n)\times\G(b';n')$, $\Fcal_{\,\hgamma\,}\simeq \F(a,b;n)\times \F(a',b';n)$, 
et $\Fcal_{\,\hgamma\,}^\gamma\simeq \G(a;b)\times\G(a';b')$. On voit maintenant que l'application $g\mapsto g w_I,\GL_{n+n'}(\C)\to \GL_{n+n'}(\C)$ 
induit les isomorphismes
\begin{align*}
\tFcal_{\,\hgamma\,}&:=\GL_{n+n'}(\C)/\tP(\,\hgamma\,)\hspace{4mm}\simeq \hspace{3mm}\GL_{n+n'}(\C)/\tP(\gamma^{a+a',b+b'})
\hspace{3mm}\simeq \hspace{3mm} \F(a+a',b+b';n+n'),\\
\tFcal_{\gamma}&:=\GL_{n+n'}(\C)/\tP(\gamma)\hspace{5mm}\simeq\hspace{4mm} \GL_{n+n'}(\C)/\tP(\gamma^{b+b'})
\hspace{4mm}\simeq\hspace{4mm}  \G(b+b';n+n'),\\
\tFcal_{\,\hgamma\,}^\gamma&:=\GL_{n+n'}(\C)^\gamma/\tP(\,\hgamma\,)^\gamma\hspace{2mm}\simeq \hspace{2mm}\GL_{n+n'}(\C)^{\gamma^{b+b'}}/\tP(\gamma^{a+a',b+b'})^{\gamma^{b+b'}}\hspace{2mm}\simeq  \hspace{2mm}\G(a+a';b+b').
\end{align*}

Au travers des isomorphismes précédents, on vérifie aisément que les morphismes $\iota_\gamma$, $\iota_{\,\hgamma\,}$ et $\iota_{\,\hgamma\,}^\gamma$ sont définis de la manière suivante:
\begin{itemize}
\item $\iota_\gamma:\G(b;n)\times\G(b';n')\to \G(b+b';n+n')$ envoie $(E,E')$ sur $E\oplus E'$,
\item $\iota_{\,\hgamma\,}:\F(a,b;n)\times\F(a',b';n')\to \F(a+a',b+b';n+n')$ envoie $(E\subset F,E'\subset F')$ sur $E\oplus E'\subset F\oplus F'$,
\item $\iota_{\,\hgamma\,}^\gamma:\G(a;b)\times\G(a';b')\to \G(a+a';b+b')$ envoie $(H,H')$ sur $H\oplus H'$,
\end{itemize}

\medskip

Fixons $(I\subset J)\in \Pcal(a,b;n)$, $(I'\subset J')\in \Pcal(a',b';n')$ et $(I''\subset J'')\in \Pcal(a+a',b+b';n+n')$. On associe à cette donnée les 
classes de cohomologie suivantes: $[\Xgot_{I\subset J}]\in H^*(\F(a,b;n),\Z)$, $[\Xgot_{I'\subset J'}]\in H^*(\F(a',b';n'),\Z)$ et 
$[\Xgot_{I''\subset J''}]\in H^*(\F(a+a',b+b';n+n'),\Z)$. Grâce aux calculs effectués au lemme \ref{lem:calcul-dimensions}, on voit que 
le produit 
$$
[\Xgot_{I\subset J}]\times [\Xgot_{I'\subset J'}]\cdot \iota_{\,\hgamma\,}^*\left([\Xgot_{I''\subset J''}]\right)
$$
appartient à $H^{max}(\F(a,b;n),\Z)\times H^{max}(\F(a',b';n),\Z)$ si 
\begin{equation}\label{eq:dimension-application-2}
|\mu(J)|+|\mu(J')|+|\mu(J'')|+|\mu(I\natural J)|+|\mu(I'\natural J')|+|\mu(I''\natural J'')| =\dim_\C\F(a+a',b+b';n+n').
\end{equation}

Le prochain résultat est une application du théorème \ref{th:levi-mobile-multiplicatif}.

\begin{prop} \label{prop:levi-mobile-application-2}
Les propriétés suivantes sont équivalentes:
\begin{enumerate}\setlength{\itemsep}{8pt}
\item $[\Xgot_{I\subset J}]\times [\Xgot_{I'\subset J'}]\cdot \iota_{\,\hgamma\,}^*\left([\Xgot_{I''\subset J''}]\right)= \ell[pt]$, avec $\ell\geq 1$, et 
$(I\subset J,I'\subset J',I''\subset J'')$ est Lévi-mobile.
\item $[\Xgot_{I\subset J}]\times [\Xgot_{I'\subset J'}]\cdot \iota_{\,\hgamma\,}^*\left([\Xgot_{I''\subset J''}]\right)= \ell[pt]$, avec $\ell\geq 1$, et 
\begin{equation}\label{eq:levi-application-2}
|\mu(I)|+|\mu(I')|+|\mu(I'')|+|\mu(J)|+|\mu(J')|+|\mu(J'')|= (a+a')(n+n'-a-a')+(b+b')(n+n'-b-b').
\end{equation}
\item $(I\subset J,I'\subset J',I''\subset J'')$ satisfait les conditions suivantes:
\begin{enumerate}\setlength{\itemsep}{8pt}
\item $[\Xgot_{J}]\times [\Xgot_{J'}]\cdot \iota_{\gamma}^*\left([\Xgot_{J''}]\right)= \ell'[pt]$, avec $\ell'\geq 1$, dans $H^{max}(\G(b, n))\times H^{max}(\G(b', n'))$.
\item $[\Xgot_{I\natural J}]\times [\Xgot_{I'\natural J'}]\cdot (\iota_{\,\hgamma\,}^\gamma)^*\left([\Xgot_{I''\natural J''}]\right)= \ell''[pt]$, avec $\ell''\geq 1$, 
dans $H^{max}(\G(a,b))\times H^{max}(\G(a', b'))$.
\item $|\mu(I)|+|\mu(I')|+|\mu(I'')|= (a+a')(n+n'-a-a')$.
\end{enumerate}
\item  $(I\subset J,I'\subset J',I''\subset J'')$ satisfait les conditions suivantes:
\begin{enumerate}\setlength{\itemsep}{8pt}
\item $\cc^{(J'')^o}_{J,J'}= \ell'\geq 1$.
\item $\cc^{(I''\natural J'')^o}_{I\natural J,I'\natural J'}= \ell''\geq 1$.
\item $|\mu(I)|+|\mu(I')|+|\mu(I'')|= (a+a')(n+n'-a-a')$.
\end{enumerate}
\end{enumerate}
Lorsque ces conditions sont satisfaites on a $\ell=\ell'\ell''$.
\end{prop}

\begin{proof}L'équivalence $1.\Longleftrightarrow 2.$ est une conséquence de la proposition \ref{prop:RP-exemple-fondamental-1}. Grâce aux calculs du lemme 
\ref{lem:calcul-dimensions}, on voit que la condition $(A_2)$ $\tr(w\,\hgamma\, \circlearrowright \ngot^{w\,\hgamma>0})=\tr(\tw\,\hgamma \circlearrowright \tngot_{-}^{\tw\,\hgamma>0})$ 
est équivalente à (\ref{eq:levi-application-2}).

L'équivalence $2.\Longleftrightarrow 3.$ est une conséquence de la proposition \ref{th:levi-mobile-multiplicatif}. Grâce aux calculs du lemme 
\ref{lem:calcul-dimensions}, on voit que la condition $\tr\left(w\hgamma \circlearrowright \ngot^{w\gamma>0}\right)=\tr\left(\tw\hgamma \circlearrowright \tngot_{-}^{\tw\gamma>0}\right)$ 
est équivalente à 
$$
|\mu(I)|+|\mu(I')|+|\mu(I'')|+|\mu(J)|+|\mu(J')|+|\mu(J'')| -|\mu(I\natural J)|-|\mu(I'\natural J')|-|\mu(I''\natural J'')|=(n+n'-b-b')(a+a'+b+b').
$$
Cette dernière relation, combinée avec (\ref{eq:dimension-application-2}) et (\ref{eq:levi-application-2}), donne
\begin{enumerate}
\item[i)] $|\mu(J)|+|\mu(J')|+|\mu(J'')|=(b+b')(n+n'-b-b')$,
\item[ii)] $|\mu(I\natural J)|+|\mu(I'\natural J')|+|\mu(I''\natural J'')|=(a+a')(b+b'-a-a')$,
\item[iii)] $|\mu(I)|+|\mu(I')|+|\mu(I'')|= (a+a')(n+n'-a-a')$.
\end{enumerate}
On remarque que i) et ii) sont les relations de dimensions associées aux identités 3. a) et 3. b). Finalement, on voit que les relations i), ii) et iii) impliquent 
(\ref{eq:dimension-application-2}) et (\ref{eq:levi-application-2}). L'équivalence $2.\Longleftrightarrow 3.$ est donc démontrée.

L'équivalence $3.\Longleftrightarrow 4.$ découle de la proposition \ref{prop:calcul-grassmanienne-a-b}.
\end{proof}

\subsubsection{Second cadre }\label{sec:application-multiplicatif-2-2}

Soient $1\leq a<n$ et $1\leq a'<n'$. Dans cette deuxième partie, on travaille avec les vecteurs de l'algèbre de Lie $\tgot\simeq \R^n\times\R^{n'}$ suivants :
$$
\hgamma:=(\underbrace{-1,\ldots,-1}_{a\ termes},\underbrace{0,\ldots,0}_{n-a\ termes},\underbrace{0,\ldots,0}_{n'-a'\ termes},\underbrace{1,\ldots,1}_{a'\ termes}) \quad 
\mathrm{et}\quad \gamma:=(\underbrace{0,\ldots\ldots,0}_{n+n'-a'\ termes},\underbrace{1,\ldots,1}_{a'\ termes}).
$$

L'application $(w_1,w_2)\mapsto (w_1([a]),w_2([n'-a']))$ permet d'identifier $W(\,\hgamma\,)$, $W(\gamma)$ et $W^\gamma(\,\hgamma\,)$ 
respectivement à $\Pcal(a,n)\times\Pcal(n'-a',n')$,  $\Pcal(a,n)$ et $\Pcal(n'-a',n')$. Grâce au lemme \ref{lem:bijection-lemme-richmond-exemple}, 
nous voyons que la bijection $\tW(\,\hgamma\,)\simeq \tW(\gamma) \times \tW^\gamma(\,\hgamma\,)$ 
correspond à l'application $\Pcal(a,n+n'-a';n+n')\longrightarrow \Pcal(n+n'-a';n+n')\times  \Pcal(a;n+n'-a')$ qui envoie $(J\subset L)$ sur $(L,J\natural L)$.

Considérons les morphismes $\iota_\gamma: \Fcal_\gamma\to \tFcal_{\gamma}$,\ $\iota_{\,\hgamma\,}: \Fcal_{\,\hgamma\,}\to \tFcal_{\,\hgamma\,}$\ et \ 
$\iota_{\,\hgamma\,}^\gamma: \Fcal_{\,\hgamma\,}^\gamma\to \tFcal_{\,\hgamma\,}^\gamma$. \`A travers les isomorphismes canoniques 
\begin{align*}
    \Fcal_\gamma&\simeq \G(n'-a';n'),\hspace{20mm}\Fcal_{\,\hgamma\,}\simeq \G(a;n)\times \G(n'-a';n'),\qquad  
\Fcal_{\,\hgamma\,}^\gamma\simeq \G(a;n),   \\
 \tFcal_\gamma&\simeq \G(n+n'-a';n+n'),\qquad \tFcal_{\,\hgamma\,}\simeq \F(a,n+n'-a';n+n'),\qquad  
\Fcal_{\,\hgamma\,}^\gamma\simeq \G(a;n+n'-a'),
    \end{align*}
les morphismes $\iota_\gamma$, $\iota_{\,\hgamma\,}$ et $\iota_{\,\hgamma\,}^\gamma$ sont définis de la manière suivante:
\begin{itemize}
\item $\iota_\gamma:\G(n'-a';n')\to \G(n+n'-a';n+n')$ envoie $E$ sur $\C^n\oplus E$,
\item $\iota_{\,\hgamma\,}:\G(a;n)\times \G(n'-a';n')\to \F(a,n+n'-a';n+n')$ envoie $(F,E)$ sur $F\subset \C^n\oplus E$,
\item $\iota_{\,\hgamma\,}^\gamma:\G(a;n)\to \G(a;n+n'-a')$ envoie $F$ sur $F$.
\end{itemize}

\medskip

Fixons $I\in \Pcal(a;n)$, $I'\in \Pcal(n-a';n')$ et $(J\subset L)\in \Pcal(a,n+n'-a';n+n')$. On associe à cette donnée les 
classes de cohomologie suivantes: $[\Xgot_{I}]\in H^*(\G(a;n))$, $[\Xgot_{I'}]\in H^*(\F(n'-a';n'))$ et 
$[\Xgot_{J\subset J}]\in H^*(\F(a,n+n'-a';n+n'))$. Grâce aux calculs effectués au lemme \ref{lem:calcul-dimensions}, on voit que 
le produit 
$$
[\Xgot_{I}]\times [\Xgot_{I'}]\cdot \iota_{\,\hgamma\,}^*\left([\Xgot_{J\subset L}]\right)
$$
appartient à $H^{max}(\G(a;n))\times H^{max}(\G(n'-a';n'))$ ssi $|\mu(I)|+|\mu(I')|+|\mu(L)|+|\mu(J\natural L)|$ $=$ \break $\dim_\C \F(a,n+n'-a';n+n')$.

Le prochain résultat est une application du théorème \ref{th:levi-mobile-multiplicatif}.

\begin{prop} \label{prop:levi-mobile-application-3}
Soient $I\in \Pcal(a;n)$, $I'\in \Pcal(n'-a';n')$ et $(J\subset L)\in \Pcal(a,n+n'-a';n+n')$. 
Les propriétés suivantes sont équivalentes:
\begin{enumerate}\setlength{\itemsep}{8pt}
\item $[\Xgot_{I}]\times [\Xgot_{I'}]\cdot \iota_{\,\hgamma\,}^*\left([\Xgot_{J\subset L}]\right)= \ell[pt]$, avec $\ell\geq 1$, et 
$(I,I',J\subset L)$ est Lévi-mobile.
\item $[\Xgot_{I}]\times [\Xgot_{I'}]\cdot \iota_{\,\hgamma\,}^*\left([\Xgot_{J\subset L}]\right)= \ell[pt]$, avec $\ell\geq 1$, et 
\begin{equation}\label{eq:levi-application-3}
|\mu(I')|+|\mu(L)|+|\mu(J)|-|\mu(J\natural L)|= a'(n+n'-a') + aa'.
\end{equation}
\item $(I,I',J\subset L)$ satisfait les conditions suivantes:
\begin{enumerate}\setlength{\itemsep}{8pt}
\item $[\Xgot_{I'}]\cdot \iota_{\gamma}^*\left([\Xgot_{L}]\right)= \ell'[pt]$, avec $\ell'\geq 1$, dans $H^{max}(\G(n'-a', n'))$.
\item $[\Xgot_{I}]\cdot (\iota_{\,\hgamma\,}^\gamma)^*\left( [\Xgot_{J\natural L}]\right)= \ell''[pt]$, avec $\ell''\geq 1$, 
dans $H^{max}(\G(a;n))$.
\item $|\mu(J)|-|\mu(J\natural L)|= aa'$.
\end{enumerate}
\item  $(I,I',J\subset L)$ satisfait les conditions suivantes:
\begin{enumerate}\setlength{\itemsep}{8pt}
\item $L=\{n'+1-i',\ i'\in I'\}\cup \{k\in [n+n'], \ k\geq n'+1\}$.
\item $J\natural L=\{ n+n'+1-a' -i, i\in I\}$.
\item $|\mu(J)|-|\mu(J\natural L)|= aa'$.
\end{enumerate}
\end{enumerate}
Lorsque ces conditions sont satisfaites, on a $\ell=\ell'=\ell''=1$.
\end{prop}

\begin{proof}La preuve, qui est identique à celle de la proposition \ref{prop:levi-mobile-application-2}, est laissée à la discrétion du lecteur.
\end{proof}

\begin{coro}\label{coro:pour-B-n}
Considérons le cas particulier où $a=a'=1$ et $n=n'$. Choisissons $I=\{i\}\subset [n]$, $I'=\{i'\}^c\subset [n]$, $J=\{j\}\subset [2n]$ et $L=\{\ell\}^c\subset [2n]$. Alors 
$[\Xgot_{\{i\}}]\times [\Xgot_{\{i'\}^c}]\cdot \iota_{\,\hgamma\,}^*\left([\Xgot_{\{j\}\subset \{\ell\}^c}]\right)= k[pt]$, avec $k\geq 1$ et 
$(\{i\},\{i'\}^c,\{j\}\subset \{\ell\}^c)$ est Lévi-mobile 
si et seulement si $j+i= 2n+1$, et $\ell + i'=n+1$. Lorsque ces conditions sont satisfaites, on a $k=1$.
\end{coro}

\chapter{Le c\^one $\LR(U,\tU)$}\label{sec:LR-U-tilde-U}

Soient $\iota:U\croc \tU$ deux groupes de Lie compacts connexes. Nous choisissons un produit scalaire invariant $(-,-)$ sur l'algèbre de Lie $\tugot$ de $\tU$, 
et nous notons par $\pi:\tugot\to\ugot$ la projection orthogonale.

Sélectionnons des tores maximaux $T$ dans $U$ et $\tT$ dans $\tU$ tels que $T\subset \tT$, et des 
chambres de Weyl $\tgot_+\subset\tgot$ et $\ttgot_+\subset\ttgot$, où $\tgot$ et $\ttgot$ désignent les algèbres de Lie de $T$, resp. $\tT$.
 
Dans le prochain chapitre, nous rappelons la description du cône suivant donnée dans \cite{Berenstein-Sjamaar,Ressayre10}:
$$
\LR(U,\tU)=\left\{(\txi,\xi)\in \ttgot_+\times\tgot_+,\ U\xi\subset \pi\big(\,\tU\txi\,\big)\right\}.
$$
Nous verrons que celle ci est une conséquence du théorème \ref{th:infinitesimal-ressayre-pairs}. 

\section{Description g\'en\'erale}

Considérons le réseau $\wedge:=\frac{1}{2\pi}\ker(\exp : \tgot\to T)$ et les groupes de Weyl $\tW=N_{\tU}(\tT)/\tT$ et $W=N_{U}(T)/T$. 
Nous désignons par $w_o\in W$ l'élément le plus long. Rappelons qu'n vecteur $\gamma\in \tgot$ est dit {\em rationnel} 
s'il appartient au $\Q$-espace vectoriel $\tgot_\Q$ engendré par $\wedge$. Nous allons voir que le cône 
$\LR(U,\tU)$ est complètement décrit par des inégalités de la forme 
$$
(\tilde{\xi},\tilde{w}\gamma)\geq (\xi,w_ow\gamma)
$$
avec $\gamma$ rationnel anti-dominant et $(w,\tw)\in W\times \tW$.

\subsubsection{Éléments admissibles}

Soit $\Rgot(\tugot/\ugot)\subset \tgot^*$ l'ensemble des poids relatifs à l'action de $T$ sur $(\tugot/\ugot)\otimes\C$. Si $\gamma\in\tgot$, nous notons par 
$\Rgot(\tugot/\ugot)\cap \gamma^\perp$ le sous-ensemble des poids s'annulant contre $\gamma$.

\begin{definition}\label{admissible-exemple-1} Un élément rationnel $\gamma\in \tgot$ est dit {\em admissible} lorsque
\begin{equation}\label{eq:condition-gamma}
\vect\big(\Rgot(\tugot/\ugot)\cap \gamma^\perp\big)=\vect\big(\Rgot(\tugot/\ugot)\big)\cap \gamma^\perp.
\end{equation}
\end{definition}

\medskip

\begin{lem}\label{lem:R-tugot-ugot}
L'ensemble $\Rgot(\tugot/\ugot)$ engendre $\tgot^*$ si et seulement si aucun idéal non-nul de $\tugot$ n'est contenu dans $\ugot$. Dans ce cas, $\gamma$ est admissible si 
l'hyperplan $\gamma^\perp\subset\tgot^*$ est engendré par des éléments de $\Rgot(\tugot/\ugot)$.
\end{lem}

\begin{proof} L'ensemble $\bgot:=\{X\in \ugot, [X,\tugot]\subset\ugot\}$ est un idéal de $\ugot$, donc on a la décomposition orthogonale $\ugot=\bgot\oplus \bgot^\perp$ où 
$\bgot^\perp$ est aussi un idéal de $\ugot$. Comme $[\bgot,\bgot^\perp]=0$, on remarque que $([X,Y],Z)=0$ pour tout $(X,Y,Z)\in \bgot\times\tugot\times\bgot^\perp$. Ainsi 
$[\bgot,\tugot]\subset\bgot$, autrement dit $\bgot$ est un idéal de $\tugot$ contenu dans $\ugot$. D'autre part, si $\cgot\subset \ugot$ est un idéal de $\tugot$, 
on a par définition $\cgot\subset \bgot$. On voit maintenant que les énoncés suivants sont équivalents:
\begin{itemize}
\item Aucun idéal non-nul de $\tugot$ n'est contenu dans $\ugot$. 
\item L'idéal $\bgot$ est réduit à $\{0\}$.
\item $\bgot\cap\tgot=\{0\}$.
\item $\Rgot(\tugot/\ugot)$ engendre $\tgot^*$.
\end{itemize}
\end{proof}

Faisons le lien entre les notions {\em d'admissibilité} introduites dans les définitions \ref{def:admissible-U-action} et \ref{admissible-exemple-1}. 
Pour cela, on reprend le cas de la vari\'et\'e de K\"{a}hler $\tU\times U$-Hamiltonienne $N:=\tU_\C$ (voir l'exemple \ref{ex:U-tilde-C-Kahler}). 
Pour $(\tgamma,\gamma)\in\ttgot\times \tgot$, notons $N^{(\tgamma,\gamma)}$ la sous-variété des points fixé par l'action infinitésimale de $(\tgamma,\gamma)$. 
Une preuve des résultats suivants est donnée dans \cite[Section 6.1]{pep-ressayre}.

\begin{lem}\label{lem:point-fixe-cotangent}
Soit $(\tgamma,\gamma)\in\ttgot\times \tgot$.
\begin{enumerate}
\item $N^{(\tgamma,\gamma)}\neq\emptyset$ si et seulement si il existe $\tw\in\tW$ tel que $\tgamma=\tw\gamma$.
\item Pour tout $\tw\in\tW$, le couple $(\tw\gamma,\gamma)$ est admissible pour l'action $\tU\times U\circlearrowright N$ 
si et seulement si $\gamma$ est rationnel et satisfait la relation (\ref{eq:condition-gamma}).
\end{enumerate}
\end{lem}

\begin{lem}
Supposons qu'aucun idéal non nul de $\ugot$ ne soit un idéal de $\tugot$. Alors l'action $\tU\times U\circlearrowright \tU_\C$ admet 
un nombre fini\footnote{Modulo les identifications $(\tw q\gamma,w q \gamma) \simeq (\tw\gamma,w\gamma)$ pour tout $q\in\Q^{>0}$.}
d'éléments admissibles de la forme $(\tw\gamma,w\gamma)$ avec $\gamma\in\tgot$ anti-dominant, et $(\tw,w)\in \tW\times W$.
\end{lem}

\subsubsection{Trace polarisée}

Soient $\Rgot(\ugot)$ et $\Rgot(\tugot)$ les ensembles de racines associés aux algèbres de Lie $\ugot$ et $\tugot$. Le choix des chambres de Weyl $\tgot_{+}$
et $\ttgot_{+}$  définit des sous-ensembles de racines positives $\Rgot^+\subset\Rgot(\ugot)$ et $\widetilde{\Rgot}^+\subset\Rgot(\tugot)$.

Pour $(\gamma,w,\tw)\in \tgot \times W\times \tW$, la relation \ $(A_2)$\ 
$\tr(w\gamma \circlearrowright \ngot^{w\gamma>0})=\tr(\tw\gamma \circlearrowright \tngot_{-}^{\tw\gamma>0})$  
de la proposition \ref{prop:RP-exemple-fondamental-1} est équivalente à 
\begin{equation}\label{eq:trace-condition-w-tilde}
\sum_{\stackrel{\alpha\in\Rgot^+}{\langle\alpha,w\gamma\rangle> 0}}\langle\alpha,w\gamma\rangle+
\sum_{\stackrel{\tilde{\alpha}\in\widetilde{\Rgot}^+}{\langle\tilde{\alpha},\tilde{w}\gamma\rangle< 0}}\langle\tilde{\alpha},\tilde{w}\gamma\rangle=0
\end{equation}

\subsubsection{Description of $\LR(U,\tU)$}\label{sec:LR-U-tildeU}

Soit $\gamma\in\tgot$ admissible et $\mathbf{w}:=(w,\tw)\in W/W^\gamma\times\tW/\tW^\gamma$. 
La sous-variété $N^{(\tw\gamma,w\gamma)}$ est connexe, égale à 
$C_{\gamma,\mathbf{w}}:=\tw \tU_\C^\gamma w^{-1}$. \`A la proposition \ref{prop:RP-exemple-fondamental-1}, nous avons montré que 
$(\gamma_{\mathbf{w}},C_{\gamma,\mathbf{w}})$ est une paire de Ressayre de $\tU\times U\circlearrowright  N$ si et seulement si
$[\Xgot_{w,\gamma}]\cdot \iota^*[\tXgot_{\tw,\gamma}]= [pt]$\ dans  $H^{max}(\Fcal_\gamma,\Z)$ et l'identité (\ref{eq:trace-condition-w-tilde}) est satisfaite.

On sait que $(\txi,\xi)\in \LR(U,\tU)$ si et seulement si $(\txi,-w_o\xi)$ appartient au polytope de Kirwan $\Delta_{\tugot\times\ugot}(T^*\tU)$. Ainsi, la description de 
$\Delta_{\tugot\times\ugot}(T^*\tU)$ donnée par le théorème \ref{th:infinitesimal-ressayre-pairs} fournit la description suivante du cône $\LR(U,\tU)$.

\begin{theorem}\label{theo:LR-general} Soit $(\xi,\txi)\in\tgot_{+}\times \ttgot_{+}$. Nous avons  $U\xi\subset \pi\big(\,\tU\txi\,\big)$ si et seulement si 
\begin{equation}\label{eq:inequality-polytope}
\langle \tilde{\xi},\tilde{w}\gamma\rangle\geq \langle \xi,w_ow\gamma\rangle
\end{equation}
pour tout $(\gamma,w,\tilde{w})\in\tgot \times W\times \tilde{W}$ satisfaisant les propriétés suivantes :
\begin{enumerate}\setlength{\itemsep}{8pt}
\item[a)] $\gamma$ est admissible antidominant.
\item[b)] $[\Xgot_{w,\gamma}]\cdot \iota^*([\tilde{\Xgot}_{\tilde{w},\gamma}])= [pt]$ dans $H^*(\Fcal_\gamma,\Z)$.
\item[c)] L'identité (\ref{eq:trace-condition-w-tilde}) est vérifiée.
\end{enumerate}
Le résultat reste valable si l'on remplace  $b)$ par la condition plus faible
$$
b') \qquad [\Xgot_{\gamma}]\cdot \iota^*([\tilde{\Xgot}_{\tilde{w},\gamma}])= \ell[pt], \quad\ell\geq 1,\qquad \mathrm{dans}\quad H^*(\Fcal_\gamma,\Z).
$$
\end{theorem}

De nombreuses personnes ont contribué au théorème \ref{theo:LR-general}. La première contribution a été apportée par Klyachko \cite{Klyachko}, puis affinée par 
Belkale \cite{Belkale06}, dans le cas de $\SL_n\croc (\SL_n)^s$. Le cas $U_\C\croc (U_\C)^s$ a été traité par Belkale-Kumar \cite{BK06} et par Kapovich-Leeb-Millson 
\cite{KLM-memoir-08}. Enfin, Berenstein-Sjamaar \cite{Berenstein-Sjamaar} et Ressayre \cite{Ressayre10, Res11a}  ont étudié le cas général. 
Ressayre \cite{Ressayre10} a également prouvé l'irréductibilité de la liste d'inégalités fournit par le théorème \ref{theo:LR-general},  
lorsque aucun idéal non-nul de $\tugot$ n'est contenu dans $\ugot$.

Nous renvoyons le lecteur aux articles de synthèse \cite{Fulton00,Brion-bourbaki,Kumar14}  pour plus de détails.

\section{Exemple: $\horn(n)$}\label{sec:horn-n}

Voici les notations utilisées dans les deux prochaines sections: 

\begin{itemize}
\item $\R^\ell_{+}$ est l'ensemble des suites $x=(x_1\geq \cdots \geq x_\ell)$ de nombres réels.
\item Pour tout entier $\ell\geq 1$, soit $[\ell]$ l'ensemble $\{1,\ldots,\ell\}$.
\item Si $x\in\R^\ell$ et $A\subset [\ell]$, on écrit $|x|_A=\sum_{a\in A}x_a$ et $|x|=\sum_{i=1}^\ell x_i$.
\item Si $I=\{i_1<\cdots< i_p\}$ est une suite croissante d'entiers strictement positifs, on pose $\mu(I)=(i_p-p\geq \cdots \geq i_1-1\geq 0)$.
\item $\Pcal(r,n)$ est l'ensemble des sous-ensembles de cardinal $r$ de $[n]$. 
\item $\e(X)\in\R^n_+$ désigne le vecteur des valeurs propres d'une matrice hermitienne $X\in\herm(n)$.
\end{itemize}

\medskip

Rappelons que $\horn(n)\subset (\R^\ell_{+})^3$ désigne le cône formé des triplets $(\e(X),\e(Y),\e(X+Y))$, où $X,Y\in \herm(n)$. Sachant que l'algèbre de 
Lie du groupe $\upU_n$ s'identifie canoniquement avec $\herm(n)$, on remarque que $\horn(n)$ correspond au cône $\LR(\upU_n, \upU_n\times \upU_n)$.

\begin{definition}
Pour tout $1\leq r<n$, $\LR^n_r$ désigne l'ensemble des triplets $(I,J,L)\in (\Pcal(r,n))^3$ tels que $(\mu(I),\mu(J),\mu(L))\in\horn(r)$.
\end{definition}

Rappelons le lien entre les coefficients de Littlewood-Richardson et le cône $\horn(n)$. Ce résultat découle du théorème de Kempf-Ness \cite{KN78}, qui établit que 
le quotient GIT d'un groupe de Lie complexe $U_\C$ sur une variété projective lisse $\Xcal$ est homéomorphe au quotient symplectique de $\Xcal$ par 
le sous-groupe compact maximal $U$.

\begin{prop}[Kempf-Ness]\label{prop:kempf-ness} Soient $\lambda,\mu$ et $\nu$ trois partitions de longueur inférieure à $n\geq 1$. Alors
\begin{equation}\label{eq:kempf-ness}
\exists k\geq 1,\ \cc^{k\lambda}_{k\mu,k\nu}\neq 0\quad\Longleftrightarrow\quad (\mu,\nu,\lambda)\in \horn(n).
\end{equation}
\end{prop}

Le théorème de saturation de Knutson et Tao \cite{Knutson-Tao-99} est le résultat clef pour obtenir une description récursive de $\horn(n)$.
\begin{theorem}[Knutson-Tao]\label{theo:saturation} Soient $\lambda,\mu$ et $\nu$ trois partitions de longueur inférieure à $n\geq 1$. Alors
\begin{equation}\label{eq:saturation}
\cc^\lambda_{\mu,\nu}\neq 0\quad\Longleftrightarrow\quad \exists k\geq 1,\ \cc^{k\lambda}_{k\mu,k\nu}\neq 0.
\end{equation}
Ainsi, $\cc^\lambda_{\mu,\nu}\neq 0$ si et seulement si $(\mu,\nu,\lambda)\in \horn(n)$.
\end{theorem}

Grâce à la proposition \ref{prop:coeff-LR} et aux équivalences (\ref{eq:kempf-ness}) et (\ref{eq:saturation}), il est  facile de vérifier que le prochain théorème est une 
application du théorème \ref{theo:LR-general} au cas où $U=\upU_n$ s'injecte diagonalement dans $\tU:=\upU_n\times \upU_n$.
Cette description récursive de $\horn(n)$ a été conjecturée par Horn \cite{Horn} et initialement prouvée par une combinaison des travaux de 
Klyachko \cite{Klyachko} et Knutson-Tao \cite{Knutson-Tao-99}. 

P. Belkale a ensuite donné une autre preuve des inégalités de Horn et de la propriété de saturation \cite{Belkale06}, en géométrisant la relation classique 
entre la théorie des invariants et la théorie de l'intersection des variétés de Schubert. Pour une belle exposition de la méthode géométrique de Belkale pour les inégalités de Horn, voir \cite{Berline-Vergne-Walter18}.

\begin{theorem}\label{theo:horn} Le triplet $(x,y,z)\in (\R^n_+)^3$ appartient à $\horn(n)$ si et seulement si les conditions suivantes sont vérifiées :
\begin{itemize}
\item $|x|+|y|=|z|$,
\item $|x|_I+|y|_J\geq |z|_L$, pour tout $r<n$ et tout $(I,J,L)\in \LR^n_r$.
\end{itemize}
\end{theorem}

\begin{rem}
Dans le théorème \ref{theo:horn}, nous pouvons renforcer la deuxième condition en exigeant que $c^L_{I,J}=1$ \cite{Belkale01,Belkale06}. 
Dans ce cas, l'ensemble d'inégalités obtenues est minimal \cite{KTW04}.
\end{rem}

\begin{exemple}
Voici quelques inégalités qui apparaissent dans la liste du théorème \ref{theo:horn} .
\begin{itemize}
\item Inégalités de Weyl (1912): pour $i,j\in [n]$ tel que $i+j-1\leq n$ on a 
$$
 \e_i(A)+\e_j(B)\geq \e_{i+j-1}(A+B),\qquad \forall A,B\in \herm(n).
$$
\item Inégalités de Lidskii-Wielandt (1950): pour tout $I\in\Pcal(r,n)$, on a 
$$
\sum_{k=1}^r\e_{i_k}(A) +\sum_{k=1}^r \e_{k}(B)\geq \sum_{k=1}^r\e_{i_k}(A+B),\qquad \forall A,B\in \herm(n).
$$
\item Inégalités de Thompson-Freede (1971):  pour tout $I,J\in\Pcal(r,n)$, tel que $i_r+j_r-r\leq n$, on a   
$$
\sum_{k=1}^r\e_{i_k}(A) +\sum_{k=1}^r \e_{j_k}(B)\geq \sum_{k=1}^r\e_{i_k+j_k-k}(A+B),\qquad \forall A,B\in \herm(n).
$$
\end{itemize}
\end{exemple}

\section{Exemple: $\LR(m,n)$}\label{sec:LR-m-n}

Soient $m,n\geq 1$. Écrivons une matrice hermitienne $X\in \herm(m+n)$ par blocs 
$X=\begin{pmatrix}
X_{I}& *\\
*& X_{II}
\end{pmatrix}$
où $X_{I}\in \herm(m)$ et $X_{II}\in \herm(n)$. Dans cette section, nous nous intéressons au cône 
$$
\LR(m,n):=\Big\{\left(\e(X),\e(X_{I}),\e(X_{II})\right); \quad X\in \herm(m+n)\Big\}.
$$
Ce dernier correspond au cône $\LR(\upU_{m}\times \upU_n, \upU_{m+n})$.

Grâce au théorème \ref{theo:LR-general}, nous obtenons la description récursive suivante des cônes $\LR(m,n)$. 
Les détails de la preuve sont donnés dans \cite[Section 2.5]{pep-toshi}.

\begin{theorem}\label{theo:LR-m-n} Le triplet $(z,x,y)\in \R^{m+n}_+\times \R^{m}_+\times \R^{n}_+$ appartient à $\LR(m,n)$ si et seulement si les conditions suivantes sont vérifiées :
\begin{itemize}\setlength{\itemsep}{8pt}
\item $|z|=|x|+|y|$,
\item $z_{n+k}\leq x_k\leq z_k$, $\forall k\in [m]$,
\item $z_{m+\ell}\leq y_\ell\leq z_\ell$, $\forall \ell\in [n]$,
\item $|z|_L\geq |x|_I + |y|_J$, pour tout triplet $I,J,L$ satisfaisant :
\begin{enumerate}\setlength{\itemsep}{8pt}
\item $L\subset [m+n]$, $I\subset [m]$ et $J\subset [n]$ sont des sous-ensembles stricts,
\item $\sharp L =\sharp I +\sharp J$,
\item Le coefficient de Littlewood-Richardson $c^{L}_{I,J}$ est non-nul.
\end{enumerate}
\end{itemize}
De plus, la condition $c^{L}_{I,J}\neq 0$ est équivalente à $(\mu(L),\mu(I),\mu(J))\in \LR(u,v)$, où $u=\sharp I$ et $v=\sharp J$.
\end{theorem}

\begin{rem}
Dans le théorème \ref{theo:LR-m-n}, nous pouvons renforcer la condition 3. en exigeant que \break $c^{L}_{I,J}=1$. 
Dans ce cas, la liste d'inégalités obtenues est minimale \cite{Ressayre10}.
\end{rem}

\begin{exemple}\label{ex:LR-2-2} $(z,x,y)\in\R^4_{+}\times\R^2_{+} \times \R^{2}_{+}$ appartient à $\LR(2,2)$ si et seulement si $|z|=|x|+|y|$ et 
\begin{equation*}
\boxed{
\begin{array}{c}
z_1\geq x_1 \geq  z_3\\
z_1\geq y_1 \geq  z_3\\
z_2\geq x_2\geq z_4\\
z_2\geq y_2\geq z_4\\
z_1+z_2\geq  x_1+y_1\\
z_2+z_3 \geq  x_2+y_2\\
z_1+z_3 \geq \max(x_1+y_2,x_2+y_1)
\end{array}
}
\end{equation*}

\end{exemple}

\chapter[Représentations isotropes des espaces symétriques]{Cônes convexes associés aux représentations isotropes des espaces symétriques}

Cette section traite des propriétés de convexité associées aux représentations isotropes des espaces symétriques. 
Soit $G/K$ un espace symétrique riemannien de type non compact et soit $\pgot:=\T_e G/K$ la représentation isotrope 
du groupe de Lie compact $K$ (que nous supposons connexe). 
Les orbites de $K$ dans $\pgot$ sont paramétrées par un cône fermé $\agot_+$ contenu dans un sous-espace abélien maximal $\agot\subset \pgot$.

Supposons que $G/K\croc \tG/\tK$ soit un plongement d'espaces symétriques riemannien de type non compact. 
Ainsi, $K$ est un sous-groupe fermé de $\tK$ et nous avons une projection orthogonale $K$-équivariante 
$\pi :\tpgot\to \pgot$. L'objectif principal de cette section est la description du cône suivant
\begin{equation}\label{eq:cone-reel}
\horn_\pgot(K,\tK):=\Big\{(\txi,\xi)\in \tagot_+\times\agot,\ K \xi \subset \pi(\tK\txi)\Big\}.
\end{equation}

\section{Description du cône $\horn_\pgot(K,\tK)$}

On travaille dans le cadre où $\iota:G\croc \tG\subset \GL_N(\R)$ sont deux groupes de Lie réductifs connexes admettant une complexification 
$\iota_\C:G_\C\croc \tG_\C\subset \GL_{N}(\C)$. C'est par exemple le cas lorsque $G$ et $\tG$ sont semi-simples (voir \cite{Knapp-book}, \S VII.1). Notons
\begin{itemize}
\item $K=G\cap \SO_N(\R)$ et $\tK=\tG\cap \SO_N(\R)$ les sous-groupes compacts maximaux de $G$ et $\tG$. Leurs algèbres de Lie sont notées par 
$\iota:\kgot\croc\tkgot$.
\item $U= G_\C\cap U_{N}$ et  $\tU=\tG_\C\cap U_N$ les sous-groupes compacts maximaux de $G_\C$ et $\tG_\C$. Leurs algèbres de Lie sont notées par 
$\iota:\ugot\croc\tugot$.
\item $\theta(X):=-X^*$ l'involution de Cartan sur $\glgot_N(\C)$.
\end{itemize}

\begin{rem}
Notons que les groupes réductifs complexes $G_\C\croc \tG_\C$ sont les complexification des groupes compacts $U\croc \tU$. Dans le reste de cette section, 
nous utiliserons cette notation à la place de la notation $U_\C\croc \tU_\C$ utilisée dans les sections précédentes.
\end{rem}

La conjugaison $g\mapsto \overline{g}$ sur $\GL_N(\C)$ définit une involution anti-holomorphe 
$\sigma$ sur $G_\C\croc \tG_\C$ et sur $U\croc \tU$ telle que $G,\tG,K$ et $\tK$ sont respectivement égaux aux composantes connexes de 
$(G_\C)^\sigma$, $(\tG_\C)^\sigma$, $U^\sigma$ et $\tU^\sigma$. 

Considérons les décompositions de Cartan, $\ggot=\kgot\oplus \pgot$ et $\tggot=\tkgot\oplus \tpgot$, de $G$ et $\tG$. Au niveau des algèbres de Lie, nous avons $\tugot=\tkgot\oplus i \tpgot$ et $\ugot=\kgot\oplus i \pgot$, où $\ugot^{-\sigma}=i\pgot$ et $\tugot^{-\sigma}=i\tpgot$.

Soit $\pi:\tggot_\C\to\ggot_\C$ la projection orthogonale par rapport au produit hermitien $(X,Y)=\Re\tr(XY^*)$ sur $\glgot_N(\C)$.

\subsection{Idéal $I_{\ggot,\tggot}$ et éléments admissibles}

\begin{definition}
Nous posons $I_{\ggot,\tggot}:=\Big\{X\in\ggot, [X,\tggot]\subset \ggot\Big\}$
\end{definition}

Rappelons que $\Sigma(\tggot/\ggot)\subset\agot^*$ désigne l'ensemble des poids non-nuls de l'action $\agot\circlearrowright \tggot/\ggot$. 
D'autre part, $Z(\tggot)$ désigne le centre de l'algèbre de Lie $\tggot$. Nous commençons avec quelques propriétés élémentaires.

\begin{lem}
\begin{enumerate}
\item $I_{\ggot,\tggot}$ est le plus grand idéal de $\tggot$ contenu dans $\ggot$.
\item $Z(\tggot)\cap\ggot\subset I_{\ggot,\tggot}$.
\item L'orthognal, pour la dualité, de  $\Sigma(\tggot/\ggot)$ est égal à $I_{\ggot,\tggot}\cap\agot$.
\item $\Sigma(\tggot/\ggot)$ engendre $\agot^*$ si seulement si $I_{\ggot,\tggot}\cap\pgot=0$.
\end{enumerate}
\end{lem}

\begin{proof}
$I_{\ggot,\tggot}:=\{X\in\ggot, [X,\tggot]\subset \ggot\}$ est, par définition, un idéal de $\ggot$ stable sous l'involution de Cartan $\theta$. 
On a la décomposition orthogonale $\ggot=I_{\ggot,\tggot}\oplus (I_{\ggot,\tggot})^\perp$ où 
$(I_{\ggot,\tggot})^\perp:=\{X\in\ggot,\ (X,Y)=0,\ \forall Y\in I_{\ggot,\tggot}\}$ est aussi un idéal de $\ggot$. 
Comme $[I_{\ggot,\tggot},(I_{\ggot,\tggot})^\perp]=0$, on remarque que $([\tilde{Y},X],Z)=(\tilde{Y},[\theta(X),Z])=0$ pour tout 
$(X,\tilde{Y},Z)\in I_{\ggot,\tggot}\times\tggot\times(I_{\ggot,\tggot})^\perp$. Ainsi 
$[\tggot, I_{\ggot,\tggot}]\subset I_{\ggot,\tggot}$:  autrement dit $I_{\ggot,\tggot}$ est un idéal de $\tggot$ contenu dans $\ggot$. 
D'autre part, si $\bgot\subset \ggot$ est un idéal de $\tggot$, on a par définition $\bgot\subset I_{\ggot,\tggot}$. Le premier point est démontré et le second est évident.

Le troisième point est clair car, pour tout $\xi\in\agot$, la condition ``$\langle\beta,\xi\rangle=0,\ \forall \beta\in \Sigma(\tggot/\ggot)$'' 
signifie que $\xi\in I_{\ggot,\tggot}\cap\agot$. Finalement, d'après le troisième point, $\Sigma(\tggot/\ggot)$ engendre $\agot^*$ 
si et seulement si $I_{\ggot,\tggot}\cap\agot = 0$, et cette dernière égalité 
est équivalente à $I_{\ggot,\tggot}\cap\pgot=0$.
\end{proof}

Dans cette section, nous utiliserons l'hypothèse suivante:
\begin{hyp}\label{hypothese-G}
L'idéal $I_{\ggot,\tggot}$ est égal à $Z(\tggot)\cap\ggot$.
\end{hyp}

Lorsque cette hypothèse est satisfaite, on voit que $I_{\ggot,\tggot}\cap\agot=I_{\ggot,\tggot}\cap\pgot=Z(\tggot)\cap\agot$.

\medskip

Un élément $\zeta\in \agot$ est dit {\em rationnel} si les valeurs propres de l'endomorphisme symétrique $\ad(\zeta):\tggot\to\tggot$ sont rationnelles.
Nous notons par $\Sigma(\tggot/\ggot)\cap \zeta^\perp$ le sous-ensemble des poids s'annulant contre $\zeta$.

\begin{definition}\label{def:admissible-reel}
Un élément $\zeta\in \agot$ est {\em admissible} s'il est {\em rationnel} et si 
\begin{equation}\label{eq:condition-zeta}
\vect\big(\Sigma(\tggot/\ggot)\cap \zeta^\perp\big)=\vect\big(\Sigma(\tggot/\ggot)\big)\cap \zeta^\perp.
\end{equation}
\end{definition}

\subsubsection{Calcul de Schubert}

Nous choisissons des tores maximaux $T\subset U$ et $\tT\subset\tU$ invariants par $\sigma$, tels que les sous-espaces correspondants 
$\agot:=\frac{1}{i}\tgot^{-\sigma}\subset\tagot:=\frac{1}{i}\ttgot^{-\sigma}$ sont ab\'eliens maximaux dans $\pgot\subset\tpgot$. Notons 
$W=N_U(T)/T$, resp. $\tW=N_{\tU}(\tT)/\tT$, les groupes de Weyl, et $W_\agot=N_K(\agot)/Z_K(\agot)$, resp. $\tW_{\tagot}=N_{\tK}(\tagot)/Z_{\tK}(\tagot)$, 
les groupes de Weyl restreints.

Nous choisissons 
des chambres de Weyl $\tgot_+$, resp. $\ttgot_+$,  de telle manière que $\agot_+:=\frac{1}{i}(\tgot^{-\sigma}\cap\tgot_+)$, resp. 
$\tagot_+:=\frac{1}{i}(\ttgot^{-\sigma}\cap\ttgot_+)$, sont des chambres de Weyl restreintes. Notons $B\subset G_\C$, resp. $\tB\subset \tG_\C$, 
les sous-groupes de Borel associés à notre choix de chambres de Weyl $\tgot_+$, resp. $\ttgot_+$.

A tout $\zeta\in\agot$, nous associons via (\ref{eq:P-gamma-reel}) les sous-groupes paraboliques $\Pbb(\zeta)\subset \tPbb(\zeta)$ et 
nous d\'efinissons les vari\'et\'es drapeaux $\Fcal_\zeta= G_\C/\Pbb(\zeta)$ et $\tFcal_\zeta= \tG_\C/\tPbb(\zeta)$. Comme 
$\Pbb(\zeta)=G_\C\cap \tPbb(\zeta)$, nous avons un morphisme canonique $\iota:\Fcal(\zeta)\croc \tFcal(\zeta)$. Nous d\'enotons 
$\iota^* : H^{*}(\tFcal(\zeta),\Z)\to  H^{*}(\tFcal(\zeta),\Z)$ le tiré en arrière en cohomologie.

Grâce aux lemmes \ref{lem:groupes-weyl-sigma} et \ref{lem:schubert-reel}, nous pouvons attacher une cellule de Bruhat 
$$
\Xgot^{o}_{w,\zeta}:=Bw\Pbb(\zeta)/\Pbb(\zeta)= \Pbb w\Pbb(\zeta)/\Pbb(\zeta)\quad \subset\quad \Fcal_\zeta
$$
\`a  tout \'el\'ement $w\in W_\agot/W_\agot^\zeta\simeq (W/W^\zeta)^\sigma$. Ici, $\Pbb$ désigne le sous-groupe parabolique $\Pbb(-\zeta_o)$ associé à 
un élément régulier $\zeta_o\in\agot_+$. On associe \`a  tout $w\in W_\agot/W_\agot^\zeta$, la vari\'et\'e de Schubert
$\Xgot_{w,\zeta}:=\overline{\Xgot^{o}_{w,\zeta}}$ et sa classe de cycle en cohomologie $[\Xgot_{w,\zeta}]\in H^{*}(\Fcal(\zeta),\Z)$.

\subsubsection{R\'esultat principal}

Soit $w_0\in W_\agot$ l'unique \'el\'ement tel que $w_0(\agot_+)=-\agot_+$. 
Le choix de la chambre de Weyl $\agot_+$ d\'efinit un syst\`eme de racines positives $\Sigma^+$, et nous d\'esignons par 
$\Ngot=\sum_{\beta\in\Sigma^+}\ggot_\beta$ la sous-alg\`ebre nilpotente r\'eelle correspondante de $\ggot$. Idem pour la définition de
$\tNgot\subset\tggot$.

Voici l'un des résultats principaux de cette monographie.

\begin{theorem}\label{theo:main-reel}  Soit $G\subset \tG$ deux groupes réductifs réels linéaires satisfaisant l'hypothèse \ref{hypothese-G}. 

Alors, pour tout $(\tilde{\xi},\xi)\in\tagot_+\times \agot_+$, on a $K\xi\subset \pi\big(\tK\txi\big)$ si et seulement si\footnote{ Ici $(-,-)$ désigne la restriction à $\tagot$ du produit scalaire invariant sur $\tugot_\C$.}
\begin{equation}\label{eq:inegalite-main-th}
(\tilde{\xi},\tilde{w}\zeta)\geq (\xi,w_0w\zeta)
\end{equation}
pour tout \'el\'ement rationnel antidominant $\zeta\in -\agot_+$ et tout $ (w,\tilde{w})\in W_\agot/W_\agot^\zeta\times 
W_{\tagot}/W_{\tagot}^\zeta$ satisfaisant les propri\'et\'es suivantes :
\begin{enumerate}\setlength{\itemsep}{8pt}
\item[a)] $\vect\big(\Sigma(\tggot/\ggot)\cap \zeta^\perp\big)=\vect\big(\Sigma(\tggot/\ggot)\big)\cap \zeta^\perp$.
\item[b)] $[\Xgot_{w,\zeta}]\cdot \iota^*([\tilde{\Xgot}_{\tilde{w},\zeta}])= [pt]$ \quad dans \quad $H^{max}(\Fcal_\zeta,\Z)$.
\item[c)] $\tr(w\zeta \circlearrowright \Ngot^{w\zeta>0})+\tr(\tw\zeta \circlearrowright \tNgot^{\tw\zeta>0})=\tr(\zeta \circlearrowright \tggot^{\zeta>0})$.
\end{enumerate}
Le r\'esultat tient toujours si l'on remplace la condition $b)$ par la condition plus faible 
$$
b')\qquad [\Xgot_{w,\zeta}]\cdot \iota^*([\tilde{\Xgot}_{\tilde{w},\zeta}])=\ell [pt],\quad \mathrm{avec}\ \ell\geq 1,\quad \mathrm{dans}\ H^{max}(\Fcal_\zeta,\Z).
$$
\end{theorem}

\medskip

Dans un travail à venir avec Nicolas Ressayre, nous aborderons la conjecture suivante.

\medskip

\textbf{Conjecture 1:} La liste d'inégalités (\ref{eq:inegalite-main-th}) paramétrées par les conditions $a)$, $b)$ et $c)$ est minimale lorsque aucun idéal non-nul de 
$\tggot$ n'est contenu dans $\ggot$, c'est à dire lorsque $I_{\ggot,\tggot}=0$.

\section{Comparaison avec un résultat de Kapovich-Leeb-Millson}\label{sec:KLM}

Supposons que $\tG=G^s$ pour $s\geq 2$. Ici, $\tpgot=\pgot^s$, $\pi(X_1,\cdots,X_s)=\sum_{j=1}^s X_j$, 
et l'ensemble $\Sigma(\tggot/\ggot)$ correspond à l'ensemble $\Sigma$  des poids relatifs à l'action $\agot$ sur $\ggot$. Le groupe de Weyl restreint 
$W_{\tagot}$ est égal à $W_{\agot}^s$.

Dans ce cadre, le théorème \ref{theo:main-reel}  devient 

\begin{theorem}\label{theo:delta-exemple-KLM}  Soit $(\xi_0,\xi_1,\cdots,\xi_s)\in(\agot_+)^{s+1}$. 
Nous avons  $K\xi_0\subset\sum_{i=1}^s K\xi_i$ si et seulement si\footnote{ Ici $(-,-)$ désigne la restriction à $\agot$ du produit scalaire invariant sur $\ugot_\C$.} 
$$
\sum_{i=1}^s ( \xi_i,w_i\zeta)\geq ( \xi_0,w_0 w\zeta)
$$
pour tout élément antidominant rationnel $\zeta\in-\agot_+$ et tout $(w,w_1,\cdots,w_s)\in  
(W_\agot/W_\agot^\zeta)^{s+1}$ satisfaisant les propriétés suivantes :
\begin{enumerate}\setlength{\itemsep}{8pt}
\item[a)] $\vect\big(\Sigma\cap \zeta^\perp\big)=\vect\big(\Sigma\big)\cap \zeta^\perp$.
\item[b)] $[\Xgot_{w,\zeta}]\cdot [\Xgot_{w_1,\zeta}]\cdot\ldots\cdot[\Xgot_{w_s,\zeta}]= [pt]$ dans $H^*(\Fcal_\zeta,\Z)$.
\item[c)] $\tr(w\zeta \circlearrowright \Ngot^{w\zeta>0})+\sum_{i=1}^s\tr(w_i\zeta \circlearrowright \Ngot^{w_i\zeta>0})=s\, \tr(\zeta \circlearrowright \ggot^{\zeta>0})$.
\end{enumerate}
\end{theorem}

\medskip

On reprend les éléments de la section \ref{sec:cadre-involution}. La sous-variété de $\Fcal_\zeta$ fixée par l'involution 
anti-holomorphe est difféomorphe à la variété de drapeaux réels $\Fcal_\zeta^\R:=G/P^\R(\zeta)$.  La décomposition de Bruhat 
$\Fcal_\gamma^\R=\bigcup_{w\in W_{\agot}/W_{\agot}^\zeta} P^\R[w]$ permet de montrer que les classes de cycles 
$[\Xgot^\R_{w,\gamma}], w\in W_{\agot}/W_{\agot}^\zeta$ associées aux variétés Schubert réelles 
$\Xgot_{w,\gamma}^\R:=\overline{P^\R[w]}$ définissent une base de la cohomologie $H^*(\Fcal_\zeta^\R,\Z_2)$ à coefficients dans $\Z_2$ (voir \cite{Tak65} 
 et \cite{DKV83}).

Comme les variétés de Schubert réelles $\Xgot_{v,\zeta}^\R$ correspondent à la partie réelle des variétés de Schubert complexes $\Xgot_{v,\zeta}$, 
un résultat de Borel et Haefliger \cite[Proposition 5.14]{BH61} nous dit que la relation
 $$
 b)\qquad  [\Xgot_{w,\zeta}]\cdot [\Xgot_{w_1,\zeta}]\cdot\ldots\cdot[\Xgot_{w_{s+1},\zeta}]= [pt]\quad \mathrm{dans}\quad H^{max} (\Fcal_\zeta,\Z).
 $$ 
 implique 
 $$
 b^{\R})\qquad[\Xgot^\R_{w,\zeta}]\cdot  [\Xgot^\R_{w_1,\zeta}]\cdot \ldots\cdot[\Xgot^\R_{w_{s+1},\zeta}]= [pt]\quad  \mathrm{dans}\quad 
 H^{max}(\Fcal_\zeta^\R,\Z_2).
 $$

La description de $\horn_{\pgot}(K^s,K)$ obtenue par Kapovich-Leeb-Millson \cite{KLM-memoir-08} était en termes d'éléments
$(\zeta,w,w_1,\cdots,w_s)\in\agot\times W_\agot^{s+1}$ vérifiant les conditions $a)$ et $b^{\R})$. 
Notre description est donc plus précise, d'abord en ajoutant la condition de Lévi-mobilité $c)$, puis en prenant la condition affinée $b)$.

\section{Preuve du théorème \ref{theo:main-reel}}

Nous allons expliquer comment le théorème \ref{theo:main-reel} est une conséquence du 
théorème \ref{th:real-ressayre-pairs} appliqué à la vari\'et\'e de K\"{a}hler $(\tU\times U,\sigma)$-hamiltonienne $\tU_\C$. On travaille ici avec  $\Zcal=\tG$ qui est 
une composante connexe de $(\tU_\C)^\tau$.

\subsection{Eléments admissibles } 

Nous considérons $\tG$ comme une sous-variété de $\tU_\C$ équipée de l'action $\tG\times G$ suivante : $(\tg,g)\cdot \tilde{x}=\tg\tilde{x} g^{-1}$.
Soit $(\tzeta,\zeta)$ un élément de  $\tagot\times \agot$. Dans le lemme suivant, nous décrivons la variété $\tG^{(\tzeta,\zeta)}$ des zéros du champ de vecteurs engendré par
$(\tzeta,\zeta)$. Dans la suite, nous notons $\tG_\zeta=\{\tg\in\tG, \tg\zeta=\zeta\}$ le sous-groupe stabilisateur de $\zeta\in\agot$.

\begin{lem}\label{lem:point-fixe}
Soit $(\tzeta,\zeta)\in\tagot\times \agot$.
\begin{itemize}\setlength{\itemsep}{8pt}
\item $\tG^{(\tzeta,\zeta)}\neq\emptyset$ si et seulement si  $\tzeta\notin \tW_\agot \zeta$.
\item Si $\tzeta=\tw \zeta$ avec $\tw\in \tW_\agot$, alors $\tG^{(\tzeta,\zeta)}=\tw\cdot \tG_\zeta$.
\end{itemize}
\end{lem}

\begin{proof} L'ensemble $\tG^{(\tzeta,\zeta)}$ est non vide si et seulement si $\tzeta$ appartient à l'orbite adjointe $\tG \zeta$. Le lemme découle alors du fait que 
l'intersection $\tG \zeta\cap \tagot$ est égale à $W_{\tagot} \zeta$.  
\end{proof}

\medskip

Rappelons la definition \ref{def:reel-admissible} dans notre situation : $(\tzeta,\zeta)$ est dit {\em admissible} par rapport à l'action $\tG\times G\circlearrowright \tG$ si 
$(\tzeta,\zeta)$ est rationnel, et si $\dim_{\tpgot\times\pgot}(\tG^{(\tzeta,\zeta)})-\dim_{\tpgot\times\pgot}(\tG)\in\{0,1\}$.

Le but principal de cette section est de prouver le résultat suivant.

\begin{prop}\label{prop:admissible-G-tilde}
Soit $G\subset \tG$ deux groupes réductifs réels linéaires satisfaisant l'hypothèse \ref{hypothese-G}. 

Pour tout $\tw\in\tW$, le couple $(\tw\zeta,\zeta)$ est admissible pour l'action $\tG\times G\circlearrowright \tG$ 
si et seulement si $\zeta$ est rationnel et satisfait la relation (\ref{eq:condition-zeta}).
\end{prop}

La preuve de la proposition \ref{prop:admissible-G-tilde} est une conséquence du lemme suivant.

\begin{lem}Soit $\zeta\in\agot$ et $\tzeta_o\in\tagot$ un élément régulier.
\begin{enumerate}\setlength{\itemsep}{8pt}
\item Nous avons
$\dim_{\tpgot\times\pgot}\tG=\dim_{\pgot}\tK\tzeta_o$\  et \ $\dim_{\tpgot\times\pgot}\tG^{(\tilde{w}\zeta,\zeta)}=\dim_{\pgot_\zeta}\tK_\zeta\tzeta_o$.
\item Sous l'hypothèse \ref{hypothese-G}, nous avons
\begin{itemize}\setlength{\itemsep}{8pt}
\item $\dim_{\tpgot\times\pgot}\tG=\dim Z(\tggot)\cap\agot$.
\item $\dim_{\pgot_\zeta}\tK_\zeta\tzeta_o-\dim_{\pgot}\tK\tzeta_o\in\{0,1\}$ si et seulement si 
$$
\vect\big(\Sigma(\tggot/\ggot)\cap \zeta^\perp\big)=\vect\big(\Sigma(\tggot/\ggot)\big)\cap \zeta^\perp.
$$
\end{itemize}
\end{enumerate}
\end{lem}
\begin{proof} Pour tout $\tg\in \tG$, le stabilisateur infinitésimal $(\tpgot\times\pgot)_{\tg}:=\{(\tX,X)\in \tpgot\times\pgot, (\tX,X)\cdot\tg=0\}$ 
admet une identification canonique avec $\Ad(\tg)(\pgot)\cap\tpgot$. Si nous écrivons 
$\tg=\tk e^{\tY}$, avec $\tk\in\tK$ et $\tY\in\tpgot$, nous voyons que $\Ad(\tg)(\pgot)\cap\tpgot\simeq\Ad(e^{\tY})(\pgot)\cap\tpgot=\pgot\cap\tpgot_{\tY}$.
Ainsi, $\dim_{\tpgot\times\pgot}\tG=\min_{\tY\in\tpgot}\dim(\pgot\cap\tpgot_{\tY})$. Si nous écrivons $\tY=\tk\cdot\tzeta$, avec $\tzeta\in\tagot$, nous avons
$\Ad(\tk)(\tagot)\subset \tk\cdot\tpgot_{\tzeta}=\tpgot_{\tY}$, et l'égalité $\Ad(\tk)(\tagot)=\tpgot_{\tY}$ est vérifiée lorsque $\tY$ est régulier. Donc
$\dim_{\tpgot\times\pgot}\tG=\min_{\tk\in\tK}\dim(\pgot\cap \Ad(\tk)(\tagot))$. D'autre part, puisque $\tzeta_o\in\tagot$ est un élément régulier, il est immédiat que
 $\dim_{\pgot}\tK\tzeta_o=\min_{\tk\in\tK}\dim(\pgot\cap\Ad(\tk)(\tagot))$. Nous obtenons finalement que $\dim_{\tpgot\times\pgot}\tG=\dim_{\pgot}\tK\tzeta_o$.
 
Puisque $\dim_{\tpgot\times\pgot}\tG^{(\tilde{w}\zeta,\zeta)}= \dim_{\tpgot\times\pgot}\tG_{\zeta}$, l'identité $\dim_{\tpgot\times\pgot}\tG^{(\tilde{w}\zeta,\zeta)}=
\dim_{\pgot_\zeta}\tK_\zeta\txi_o$ admet la même preuve que précédemment. On voit d'abord que
$$
\dim_{\tpgot\times\pgot}\tG_{\zeta}=\min_{\tk\in\tK_\zeta}\dim(\pgot\cap\Ad(\tk)(\tagot))=\min_{\tk\in\tK_\zeta}\dim(\pgot_\zeta\cap\Ad(\tk)(\tagot))
$$
puis on utilise que $\min_{\tk\in\tK_\zeta}\dim(\pgot_\zeta\cap\Ad(\tk))=\dim_{\pgot_\zeta}\tK_\zeta\tzeta_o$.
Le premier point est démontré.

Rappelons que l'orbite $\tK\tzeta_o$ est la partie réelle de l'orbite coadjointe $\tU\tzeta_o$. Grâce à la proposition \ref{prop:minimal-stabilizer}, nous savons 
qu'il existe un sous-espace $\hgot\subset \pgot$ tel que  
\begin{equation}\label{eq:stabilisateur-generique}
\forall \txi\in \tK\tzeta_o,\quad \exists k\in K\quad \mathrm{tel\ que}\quad \Ad(k)(\hgot)\subset \pgot_{\txi},
\end{equation} 
et $\dim(\hgot)=\dim(\pgot_{\txi})$ sur un sous-ensemble ouvert dense de $\tK\tzeta_o$ : $\hgot$ est le stabilisateur générique pour l'action infinitésimale de $\pgot$ sur 
$\tK\tzeta_o$ et  $\dim_{\pgot}(\tK\tzeta_o)=\dim(\hgot)$. 

Vérifions que sous l'hypothèse \ref{hypothese-G},  nous avons $\hgot=Z(\tggot)\cap\pgot$.

Si $\xi'\in\hgot$, la relation (\ref{eq:stabilisateur-generique}) implique que $K(\tK\tzeta_o)^{\xi'}=\tK\tzeta_o$. Écrivons $\xi'=\Ad(k')(\xi)$ avec 
$\xi\in\agot$ et $k'\in K$. Puisque $(\tK\tzeta_o)^{\xi}=\cup_{\tw\in W_{\tagot}}\tK_\xi\tw\tzeta_o$, la relation $K(\tK\tzeta_o)^{\xi'}=\tK\tzeta_o$ 
implique qu'un sous-ensemble $K\tK_\xi\tw\tzeta_o\subset \tK\tzeta_o$ a un intérieur non vide. Soit $g=k\tk\tw\in K\tK_\xi\tw$ tel que $g\tzeta_o$ appartienne à l'intérieur 
de $K\tK_\xi\tw\tzeta_o$ : cela implique l'identité 
\begin{equation}\label{eq:submersion}
\kgot+\Ad(k)(\tkgot_\xi)+\Ad(g)(\tkgot_{\tzeta_o})=\tkgot.
\end{equation}
Mais $\Ad(g)(\tkgot_{\tzeta_o})=\Ad(k\tk)(\tkgot_{\tzeta_o})\subset\Ad(k)(\tkgot_{\xi})$, donc (\ref{eq:submersion}) est équivalent à 
$\kgot+\tkgot_\xi=\tkgot$. Maintenant, un calcul standard montre que 
$$
\kgot+\tkgot_\xi=\tkgot\Longleftrightarrow [\xi,\tkgot]\subset\pgot\Longleftrightarrow [\xi,\tpgot]\subset\kgot\Longleftrightarrow [\xi,\tggot]\subset\ggot
\Longleftrightarrow [\xi',\tggot]\subset\ggot.
$$
Nous avons prouvé que tout élément $\xi'\in\hgot\subset\pgot$ appartient à l'idéal $I_{\ggot,\tggot}:=\{X\in\ggot, [X,\tggot]\subset \ggot\}$. 
Grâce à l'hypothèse \ref{hypothese-G}, nous savons que $I_{\ggot,\tggot}=Z(\tggot)\cap\ggot$, et donc que $\hgot\subset Z(\tggot)\cap\agot$.
D'autre part, il est immédiat que $Z(\tggot)\cap\agot\subset \pgot_{\txi}$ pour tout $\txi\in\tK\tzeta_o$. Nous avons prouvé que l'égalité 
$Z(\tggot)\cap\pgot=\pgot_{\txi}$ est vérifiée sur un sous-ensemble ouvert de $\tK\tzeta_o$. La première partie du deuxième point est démontrée.

Concentrons-nous sur la dernière partie du deuxième point. Soit $\zeta\in\agot$.

Si $\zeta\in Z(\tggot)$, alors $\dim_{\pgot_{\zeta}}\tK_{\zeta}\tzeta_o=
\dim_{\pgot}\tK\tzeta_o$ et $\vect\big(\Sigma(\tggot/\ggot)\cap \zeta^\perp\big)=\vect\big(\Sigma(\tggot/\ggot)\big)\cap \zeta^\perp=
\vect\big(\Sigma(\tggot/\ggot)\big)$.

Supposons maintenant que $\zeta\notin Z(\tggot)$. Nous voyons que $Z(\tggot)\cap\agot\oplus \R\zeta\subset (\pgot_\zeta)_{\txi}$ pour tout 
$\txi\in \tK_\zeta\tzeta_o$, d'où $\dim_{\pgot_\zeta}\tK_\zeta\tzeta_o-\dim_{\pgot}\tK\tzeta_o=1$ si et seulement si $Z(\tggot)\cap\agot\oplus \R\zeta$ 
est le stabilisateur générique pour l'action infinitésimale de $\pgot_\zeta$ sur 
$\tK_\zeta\tzeta_o$.

Supposons que $\vect\big(\Sigma(\tggot/\ggot)\cap \zeta^\perp\big)=\vect\big(\Sigma(\tggot/\ggot)\big)\cap \zeta^\perp$. Pour tout $\xi\in\agot$,
 nous avons les équivalences 
\begin{eqnarray*}
[\xi,\tggot_{\zeta}]\subset \ggot_{\zeta}\Longleftrightarrow \langle\beta,\xi\rangle=0,\forall \beta\in \Sigma(\tggot/\ggot)\cap \zeta^\perp
&\Longleftrightarrow&\xi\in \vect\Big(\Sigma(\tggot/\ggot)\cap \zeta^\perp\Big)^\perp\\
&\Longleftrightarrow&\xi\in \left(\vect\big(\Sigma(\tggot/\ggot)\big)\cap \zeta^\perp\right)^\perp\\
&\Longleftrightarrow&\xi\in \vect\big(\Sigma(\tggot/\ggot))^\perp+ \R\zeta\\
&\Longleftrightarrow&\xi\in Z(\tggot)\cap\agot\oplus\R \zeta.
\end{eqnarray*}
Soit $\hgot_\zeta\subset\pgot_\zeta$ le stabilisateur générique pour l'action infinitésimale de $\pgot_\zeta$ sur $\tK_\zeta\tzeta_o$. 
Les mêmes arguments que ceux utilisés précédemment montrent que $[\xi,\tggot_\zeta]\subset \ggot_\zeta$ pour tout $\xi\in \hgot_\zeta$. 
Cela montre que $\hgot_\zeta$ est égal à $Z(\tggot)\cap\pgot\oplus \R\zeta$.

Supposons maintenant que $\vect\big(\Sigma(\tggot/\ggot)\cap \zeta^\perp\big)\neq\vect\big(\Sigma(\tggot/\ggot)\big)\cap \zeta^\perp$. 
Cela signifie qu'il existe $\xi\notin Z(\tggot)\cap\agot\oplus \R\zeta=\vect\big(\Sigma(\tggot/\ggot))^\perp+ \R\zeta$ tel que 
$\langle\beta,\xi\rangle=0,\forall \beta\in \Sigma(\tggot/\ggot)\cap \zeta^\perp$, c'est-à-dire $[\xi,\tggot_\zeta]\subset \ggot_\zeta$. 
Considérons le sous-espace $\hgot_o=Z(\tggot)\cap\pgot\oplus \R\zeta\oplus\R\xi$. Comme nous l'avons montré précédemment, 
l'inclusion $[\xi,\tggot_\zeta]\subset \ggot_\zeta$ montre que 
l'ensemble $L_\zeta(\tK_\zeta\tzeta_o)^{\hgot_o}=L_\zeta(\tK_\zeta\tzeta_o)^{\xi}$ a un intérieur non vide. 
Grâce au lemme \ref{coro:stabilisateur-min}, nous pouvons conclure que $\dim_{\pgot_\zeta}\tK_\zeta\tzeta_o\geq \dim(\hgot_o)$ 
et donc $\dim_{\pgot_\zeta}\tK_\zeta\tzeta_o-\dim_{\pgot}\tK\tzeta_o\geq 2$.
 
La dernière partie du deuxième point est démontrée. 
\end{proof}

\subsection{Fin de la preuve} 

On suppose ici que $G\subset \tG$ satisfont l'hypothèse \ref{hypothese-G}. La restriction de l'application moment 
$\Phi_{\tugot\times\ugot}: \tG_\C\longrightarrow \tugot^*\times\ugot^*$ à la sous-variété $\tG$ détermine\footnote{Nous utilisons les identifications 
$\tugot\simeq\tugot^*$ et $\ugot\simeq\ugot^*$ données par le produit hermitien sur $\glgot_N(\C)$.} une application moment {\em réelle} 
\begin{eqnarray*}
\Phi_{\tpgot\times\pgot} : \tG\simeq \tK\times\tpgot&\longrightarrow & \tpgot\times\pgot\\
(\tk,\tX)&\longmapsto & (\tk\tX,-\pi(\tX)).
\end{eqnarray*}

Nous posons $\Delta_{\tpgot\times\pgot}(\tG):=\mathrm{Image}(\Phi_{\tpgot\times\pgot})\cap \tagot_+\times\agot$. Par définition, le couple $(\txi,\xi)\in \tagot_+\times\agot$ 
appartient à $\Delta_{\tpgot\times\pgot}(\tG)$ si et seulement si $-K\xi\subset \pi(\tK\txi)$.

On sait, d'après la proposition \ref{prop:admissible-G-tilde}, que tous les éléments admissibles
pour l'action de $\tG\times G$ sur $\tG$ sont de la forme $\zeta_{\mathbf{w}}:=(\tilde{w}\zeta,w\zeta)$
avec $\mathbf{w}:=(w,\tilde{w})\in W_\agot/W_\agot^\zeta\times W_{\tagot}/W_{\tagot}^\zeta$ et $\zeta\in\agot$ rationnel 
anti-dominant satisfaisant la relation (\ref{eq:condition-zeta}).

La sous-variété complexe de $\tU_\C$ fixée par $\zeta_{\mathbf{w}}$ est $C_{\zeta,\mathbf{w}}:=\tilde{w}\tU^\zeta_\C w^{-1}$ et la partie réelle contenue dans $\tG$ est 
$C_{\zeta,\mathbf{w}}\cap \tG=\tilde{w}\tG_\zeta w^{-1}$. D'autre part, la fonction $\tg\in \tG\mapsto (\Phi_{\tpgot\times\pgot}(\tg),\zeta_{\mathbf{w}})$ est constante, égale à 
$0$ sur $C_{\zeta,\mathbf{w}}\cap \tG$.

Ainsi le théorème \ref{th:real-ressayre-pairs} permet de voir que  $(\txi,\xi)\in \Delta_{\tpgot\times\pgot}(\tG)$ si et seulement si 
\begin{equation}\label{eq:inequality-polytope-reel}
(\tilde{\xi},\tilde{w}\zeta)+(\xi,w\zeta)\geq 0
\end{equation}
pour tout $(\zeta,w,\tw)$ tel que $(\zeta_{\mathbf{w}},C_{\zeta,\mathbf{w}})$ est un paire de Ressayre réelle régulière de $\tG$.

La proposition qui suit décrit les paires de Ressayre réelles de $\tG$.

\begin{prop}
Les conditions suivantes sont équivalentes
\begin{enumerate}
\item $(\zeta_{\mathbf{w}},C_{\zeta,\mathbf{w}})$ est un paire de Ressayre réelle de $\tG$.
\item $(i\zeta_{\mathbf{w}},C_{\zeta,\mathbf{w}})$ est une paire de Ressayre de $\tG_\C$.
\item $(\zeta,w,\tilde{w})$ satisfait les conditions  
\begin{align*}
(A_2) \qquad& \tr(w\zeta \circlearrowright \Ngot^{w\zeta>0})+\tr(\tw\zeta \circlearrowright \tNgot^{\tw\zeta>0})=\tr(\zeta \circlearrowright \tggot^{\zeta>0}),\\
(B_1') \qquad& [\Xgot_{w,\zeta}]\cdot \iota^*[\tXgot_{\tw,\zeta}]= [pt]\quad \mathrm{dans}\quad H^{max}(\Fcal_\zeta,\Z).\\
\end{align*}
\end{enumerate}
\end{prop}

\begin{proof} L'équivalence entre les deux premiers points est une conséquence de la propriété \ref{prop:caracteriser-RP-reel}, et l'équivalence 
$2.\Leftrightarrow 3.$ est démontrée dans la proposition \ref{prop:RP-exemple-fondamental-2}. Ici, on utilise le fait que la relation $(A_2)$
est équivalente à 
$tr(w\zeta \circlearrowright \ngot^{w\zeta>0})+\tr(\tw\zeta \circlearrowright \tngot^{\tw\zeta>0})=\tr(\zeta \circlearrowright \tggot_\C^{\zeta>0})$ 
(voir la remarque \ref{rem:trace-reelle-versus-complexe}).
\end{proof}

\medskip

Nous devons apporter quelques petites modifications au résultat précédent afin d'obtenir le théorème \ref{theo:main-reel}. 
L'élément le plus long $w_0\in W$, qui est invariant par $\sigma$, peut être considéré comme l'élément unique de $W_\agot$ tel que $w_0(\agot_+)=-\agot_+$. 
Ainsi tout élément $\xi\in \agot_+$ peut s'écrire $\xi=-w_0(\xi')$ avec $\xi'\in\agot_+$, et dans ce cas 
\begin{itemize}\setlength{\itemsep}{8pt}
\item $K\xi'\subset \pi\big(\tK\txi\big)$ est équivalent à $-K\xi\subset \pi\big(\tK\txi\big)$.
\item (\ref{eq:inequality-polytope-reel}) est équivalent à $(\tilde{\xi},\tilde{w}\zeta)\geq (\xi',w_0w\zeta)$.
\end{itemize}

\medskip

On a terminé la preuve du fait que le théorème \ref{theo:main-reel} est une application du théorème \ref{th:real-ressayre-pairs}.

\section{Le cas où $\tG\simeq G_\C$} \label{sec:delta-U-sigma}

Considérons un groupe compact connexe $U$ muni d'un involution $\sigma$. Notons 
$\pi_{-}:  \ugot\to \ugot^{-\sigma}$ la projection associée à la décomposition de l'algèbre de Lie 
$\ugot=\ugot^{\sigma}\oplus\ugot^{-\sigma}$. Soit $K$ la composante connexe de l'identité du sous-groupe $U^\sigma$.

Soit $T\subset U$ un tore maximal adapté à $\sigma$, i.e. $T$ est stable par $\sigma$ et $\tgot^{-\sigma}$ est de dimension maximale. 
On choisit une chambre de Weyl $\tgot_+\subset\tgot$ de telle manière à ce que $\tgot_+\cap\tgot^{-\sigma}$ paramètre les orbites de 
$K$ dans $\ugot^{-\sigma}$. Dans cette section, on s'intéresse au cône 
$$
\Delta(U,\sigma):=\left\{(x,y)\in  \tgot_+\times (\tgot_+\cap\tgot^{-\sigma}), Ky\subset \pi_{-}(Ux)\right\}.
$$

Nous commençons par expliquer en quoi le cône $\Delta(U,\sigma)$ est un cas particulier des cônes $\horn_\pgot(K,\tK)$ étudiés précédemment.

On d\'esigne encore par $\sigma$ l'involution anti-holomorphe sur $U_\C$ d\'efinie par (\ref{eq:sigma-anti-holomorphe}). 
Soit $G$ la composante connexe de l'identité du sous-groupe ferm\'e de $U_\C$ fix\'e par $\sigma$ : c'est un sous-groupe r\'eductif 
réel qui est stable sous l'involution de Cartan. Son algèbre de Lie est égale à $\ggot=\kgot\oplus \pgot$ où 
$\kgot=\ugot^{\sigma}$ et $\pgot=i \ugot^{-\sigma}$.

On considère maintenant le groupe  $\tU_\C:=U_\C\times U_\C$, muni de l'involution anti-holomorphe 
$$
\tilde{\sigma}(g_1,g_2)=\left(\sigma(g_2),\sigma(g_1)\right),\qquad \forall g_1,g_2\in U_\C.
$$

Si $\iota: U_\C\croc \tU_\C$ désigne le morphisme diagonal, on a $\iota\circ \sigma =\tilde{\sigma}\circ\iota$. 
Le sous-groupe de $\tU_\C$ fixé par $\tilde{\sigma}$ est $\tG:=\{(g,\sigma(g)), g\in U_\C\}$. Son algèbre de Lie est égale à $\tggot=\tkgot\oplus \tpgot$ où 
$\tkgot=\{(X,\sigma(X)), X\in\ugot\}$ et $\tpgot:=\{(Y,\sigma(Y)), Y\in i\ugot\}$. Le sous-groupe de $\tU$ fixé par $\tilde{\sigma}$ est $\tK:=\{(k,\sigma(k)), k\in U\}$.

Le tore $\tT:=T\times T\subset \tU$ est adapté à $\tilde{\sigma}$. Le sous-espace $\ttgot^{-\tilde{\sigma}}$ 
est formé des éléments $(Y,-\sigma(Y)), Y\in \tgot$. On choisit pour $(\tU,\tT)$ la chambre de Weyl
\begin{equation}\label{eq:borel-sigma-adapte}
\ttgot_+:=\tgot_+\times -\sigma(\tgot_+),
\end{equation}
de telle manière à ce que $\frac{1}{i}(\ttgot_+\cap\ttgot^{-\sigma})$ soit égal à la chambre de Weyl restreinte $\tagot_+:=\{(\txi,\sigma(\txi)), \txi\in \frac{1}{i}\tgot_+\}$.

La projection orthogonale $\pi: \tugot_\C\to \ugot_\C$ associée au morphisme $\iota: \ugot_\C\croc \tugot_\C$ est définie par la relation
$\pi(X,Y)=\frac{1}{2}(X+Y)$. Les isomorphismes $X\in \ugot\mapsto (iX,\sigma(iX))\in\tpgot$ et $Y\in\ugot^{-\sigma}\mapsto iY\in \pgot$ s'insèrent 
dans le diagramme 
commutatif suivant
\begin{equation}\label{eq:commutatif-delta-U-sigma}
\xymatrixcolsep{5pc}
\xymatrix{
\ugot  \ar[d]  \ar[r]^{\pi_{-}} & \ugot^{-\sigma} \ar[d] \\
\tpgot \ar[r] ^{\pi}         & \pgot.
}
\end{equation}

\begin{definition}\label{def:x-y-prime}
Pour tout $(x,y)\in \tgot_+\times (\tgot_+\cap\tgot^{-\sigma})$, posons $x'=\tfrac{1}{i}x$ et $y'=\tfrac{1}{i}y\in\agot_+$: ainsi $(x',\sigma(x'))\in\tagot_+$. 
\end{definition}

Grâce au diagramme commutatif (\ref{eq:commutatif-delta-U-sigma}), on voit que pour tout $(x,y)\in  \tgot_+\times (\tgot_+\cap\tgot^{-\sigma})$, la relation $Ky\subset \pi_{-}(Ux)$ est équivalente à 
\begin{equation}\label{eq:definition-delta-U-sigma}
Ky'\subset \pi\Big(\tK(x',\sigma(x'))\Big).
\end{equation}

On vient de montrer le fait suivant. 

\begin{lem}
L'application $(x,y)\in \tgot_+\times (\tgot_+\cap\tgot^{-\sigma})\mapsto (x',\sigma(x'),y')\in\tagot_+\times\agot_+$ détermine un isomorphisme
entre $\Delta(U,\sigma)$ et le cône $\horn_\pgot(K,\tK)$ qui est définit au moyen des données géométriques suivantes: 
$$
K=(U^{\sigma})_0,\quad \tK\simeq U,\quad \pgot\simeq \ugot^{-\sigma},\quad \mathrm{et}\quad \tpgot\simeq \ugot.
$$
\end{lem}

Dans un premier temps, nous allons décrire $\Delta(U,\sigma)$ au moyen du théorème d'O'Shea-Sjamaar. Notons $\sigma_+: \tgot_+\to\tgot_+$ 
l'involution de la chambre de Weyl qui est définie par la relation  $-\sigma(Ux)=U\sigma_+(x), \forall x\in \tgot_+$ (on décrira plus précisément 
$\sigma_+$ au lemme \ref{lem:sigma-Borel}).

Rappelons que le cône $\horn(U)$ désigne l'ensemble des triplets $(a,b,c)\in(\tgot_+)^3$ vérifiant $Uc\subset Ua+Ub$.

\begin{prop}\label{prop:horn-U-sigma}Pour tout $(x,y)\in  \tgot_+\times (\tgot_+\cap\tgot^{-\sigma})$, on a l'équivalence
$$
(x,y)\in\Delta(U,\sigma)\quad \Longleftrightarrow \quad (x,\sigma_+(x), 2y)\in\horn(U).
$$
\end{prop}
\begin{proof}On sait déjà que la condition $(x,y)\in \Delta(U,\sigma)$ est équivalente à (\ref{eq:definition-delta-U-sigma}), et grâce au théorème d'OShea-Sjamaar 
cette dernière est équivalente à $Uy'\subset \pi\Big(U\times U(x',\sigma(x'))\Big)$, c'est à dire
$$
2Uy\subset Ux +U\sigma_+(x).
$$
\end{proof}

Nous allons maintenant donner une description plus précise du cône convexe $\Delta(U,\sigma)$ au moyen du théorème \ref{theo:main-reel}.

Comme $\ggot$ est une forme réelle de l'algèbre de Lie complexe $\tggot$, l'ensemble des racines par rapport à l'action de $\agot$ sur $\tggot/\ggot$ est égal au 
système de racines restreint $\Sigma_\agot\subset\agot^*$ de l'algèbre de Lie $\ggot$. De plus, le plus grand idéal de $\tggot$ contenu dans $\ggot$ est égal au centre de $\ggot$, ainsi l'hypothèse \ref{hypothese-G} est satisfaite.

Soit $B\subset U_\C$ le sous-groupe de Borel associé à la chambre de Weyl $\tgot_+$.
Soit $W$ le groupe de Weyl de $(U,T)$ : notons $w_0$ son plus long élément. Soit $W'\subset W$ le sous-groupe qui centralise $\tgot^{-\sigma}$ : 
notons $w'_0$ son plus long élément. Rappelons que $W'$ est contenu dans $W^\sigma$, car il est égal au groupe de Weyl du sous-groupe centralisateur 
$K'=Z_{K}(\tgot^{-\sigma})$ (voir l'Annexe dans \cite{OSS}).

\begin{lem}\label{lem:sigma-Borel}
\begin{enumerate}
\item La chambre de Weyl $\sigma(\tgot_+)$ est égale à $\Ad(w_0')(\tgot_+)$.
\item Le sous-groupe de Borel associé à la chambre de Weyl $\sigma(\tgot_+)$ est $\Ad(w_0')(B)$.
\item $\sigma_+(x)=-\Ad(w_0')\sigma(x),\forall x\in\tgot_+$.
\end{enumerate}
\end{lem}

\begin{proof}Soit $\Rgot$ le système de racines de $(U,T)$, et $\Rgot_+$ le système de racines positives. Alors
$\tgot_+=\{X\in\tgot, \langle \alpha,X\rangle\geq 0, \forall \alpha\in\Rgot_+\}$, et donc 
$-\sigma(\tgot_+)=\{X\in\tgot, \langle \alpha,X\rangle\geq 0, \forall \alpha\in-\sigma(\Rgot_+)\}$.
Soit $X_o$ un élément contenu dans l'intérieur de la chambre $\tgot^{-\sigma}\cap \tgot_+$. On a alors la décomposition 
$\Rgot_+=\Rgot'_+\cup \Rgot''_+$, où 
$$
\Rgot'_+=\Rgot_+\cap \{\sigma(\alpha)=\alpha\}\qquad \mathrm{et}\qquad \Rgot''_+= \{\alpha\in\Rgot, \langle \alpha,X_o\rangle > 0\}.
$$
Cela permet de voir que $-\sigma(\Rgot'_+)=-\Rgot'_+$ et $-\sigma(\Rgot''_+)=\Rgot''_+$. On a, d'autre part, 
$\Ad(w'_0)(\Rgot'_+)=-\Rgot'_+$ et $\Ad(w'_0)(\Rgot''_+)=\Rgot''_+$. On montre ainsi que $-\sigma(\tgot_+)=\Ad(w_0')(\tgot_+)$.  
Le premier point est démontré, et les deux autres découlent du premier. 
\end{proof}

Voici la description que l'on obtient de $\Delta(U,\sigma)$. Rappelons que $\ngot=\sum_{\alpha\in\Rgot_+}(\ugot_\C)_{\alpha}$.

\begin{theorem}\label{theo-delta-U-sigma} Un élément $(x,y)\in  \tgot_+\times (\tgot_+\cap\tgot^{-\sigma})$ appartient à $\Delta(U,\sigma)$ si et seulement 
si\footnote{On utilise les notations de la définition \ref{def:x-y-prime}.} 
$$
(x',w\zeta)\geq (y',w_0w_1\zeta)
$$
pour tout $(\zeta,w,w_1)\in -\agot_+\times W/W^\zeta\times W_\agot/W_\agot^\zeta$ satisfaisant les conditions suivantes
\begin{enumerate}
\item $\gamma$ est rationnel  et $\vect\big(\Sigma_\agot\cap \zeta^\perp\big)=\vect\big(\Sigma_\agot\big)\cap \zeta^\perp$,
\item $[\Xgot_{w,\zeta}]\cdot [\Xgot_{w_0'\sigma(w),\zeta}]\cdot[\Xgot_{w_1,\zeta}]= [pt]\quad \mathrm{dans}\quad H^{max} (\Fcal_\zeta,\Z)$.
\item $2\tr(w\zeta \circlearrowright \ngot^{w\zeta>0})+\tr(w_1\zeta \circlearrowright \ngot^{w_1\zeta>0})= 2\tr(\zeta \circlearrowright \ugot_\C^{\zeta>0})$.
\end{enumerate}
\end{theorem}

\begin{proof}Utilisons le théorème \ref{theo:main-reel} pour déterminer à quelles conditions $(x',\sigma(x'),y')$ appartient à $\horn_\pgot(K,\tK)$. Pour cela, on 
passe en revue les données nécessaires à nos calculs:
\begin{enumerate}
\item $W$ est le groupe de Weyl de $(U,T)$ et $W_\agot\simeq N_W(\agot)/Z_W(\agot)=W^\sigma/W'$ est le groupe de Weyl restreint.
\item $\tW=W\times W$ est le groupe de Weyl de $(\tU,\tT)$ et le groupe de Weyl restreint correspondant est 
$\tW_{\tagot}=\tW^{\tilde{\sigma}}=\{(w,\sigma(w)), w\in W\}$.
\item Le sous-groupe de Borel associé \`a $\tgot_+$ est noté $B$, et le sous-groupe de Borel associé \`a $\ttgot_+=\tgot_+\times -\sigma(\tgot_+)$ est 
$B\times \Ad(w_0')(B)$.
\item On travaille avec un élément $\zeta\in\agot$ admissible, i.e. $\vect\big(\Sigma_\agot\cap \zeta^\perp\big)=\vect\big(\Sigma_\agot\big)\cap \zeta^\perp$, et nous notons 
$\Pbb(\zeta)\subset U_\C$ le sous-groupe parabolique correspondant.
\item Le morphisme de groupe $\iota: U_\C\croc \tU_\C$ induit le morphisme diagonal $\iota : \Fcal_{\zeta}\croc \Fcal_{\zeta}\times \Fcal_{\zeta}$,
où $\Fcal_{\zeta}:= U_\C/\Pbb(\zeta)$.
\item \`A $w_1\in W_\agot$, on associe la variété de Schubert $\Xgot_{w_1,\zeta}:=\overline{Bw_1\Pbb(\zeta)/\Pbb(\zeta)}\subset \Fcal_{\zeta}$.
\item \`A $\tw=(w,\sigma(w))\in \tW_{\tagot}$, on associe la variété de Schubert 
$$
\tXgot_{\tw,\zeta}:=\overline{Bw\Pbb(\zeta)/\Pbb(\zeta)}\times \overline{\Ad(w_0')(B)\sigma(w)\Pbb(\zeta)/\Pbb(\zeta)}\quad \subset \quad\Fcal_{\zeta}\times \Fcal_{\zeta}.
$$
\end{enumerate}

On vérifie maintenant que l'inégalité (\ref{eq:inegalite-main-th}) associée à $(\zeta,w,w_1)\in -\agot_+\times W/W^\zeta\times W_\agot/W_\agot^\zeta$ est équivalente à $(x',w\zeta)\geq (y',w_0w_1\zeta)$ et que la condition c) du théorème \ref{theo:main-reel} est équivalente à 
$2\tr(w\zeta \circlearrowright \ngot^{w\zeta>0})+\tr(w_1\zeta \circlearrowright \ngot^{w_1\zeta>0})= 2\tr(\zeta \circlearrowright \ugot_\C^{\zeta>0})$.

Comme 
$$
\iota^*([\tilde{\Xgot}_{\tw,\zeta}])=[\overline{Bw\Pbb(\zeta)/\Pbb(\zeta)}]\cdot [\overline{\Ad(w_0')(B)\sigma(w)\Pbb(\zeta)/\Pbb(\zeta)}]=
[\Xgot_{w,\zeta}]\cdot [\Xgot_{w_0'\sigma(w),\zeta}]
$$
la condition cohomologique $[\Xgot_{w_1,\zeta}]\cdot \iota^*([\tilde{\Xgot}_{\tw,\zeta}])= [pt]$ du théorème \ref{theo:main-reel}  est équivalente ici à 
$[\Xgot_{w,\zeta}]\cdot [\Xgot_{w_0'\sigma(w),\zeta}]\cdot[\Xgot_{w_1,\zeta}]= [pt]$. 
\end{proof}

\section{Exemples} 

Au moyen du calcul de Schubert, nous avons décrit le cône $\horn_\pgot(K,\tK)$ associé à la donnée de groupes réductifs réels linéaires $G\subset \tG$.
Dans les sections qui suivent, nous allons expliciter ce cône convexe polyhédral dans les exemples suivants :  

\begin{center}
\renewcommand{\arraystretch}{2}
\begin{tabular}{|c|c|c|c|}
\hline
 &$\horn_\pgot(K,\tK)$ &$G$ & $\tG$  \\ \hline
 \ (1)\ & $\LR(U,\tU)$                  &    $ U_\C$          &    $\tU_\C$ \\ \hline
 \ (2)\ & $\horn(n)$                  &    $\GL_n(\C)$          &    $\GL_n(\C)\times \GL_n(\C)$  \\ \hline
  \ (3)\ & $\LR(m,n)$                  &    $\GL_m(\C)\times\GL_n(\C)$          &    $\GL_{m+n}(\C)$ \\ \hline
 \ (4)\ & $\Ecal_{\mathrm{I}}(n)$      &    $\GL_n(\R)_0$       &    $\GL_n(\C)$\\ \hline
 \ (5)\ & $\Ecal_{\mathrm{II}}(n)$     &    $\GL_n(\C)$           &    $\GL_{2n}(\R)_0$  \\ \hline
 \ (6)\ & $\sing(p,q)$         &     $\upU(p,q)$              &     $\upU(p,q)\times \upU(p,q)$  \\ \hline
 \ (7)\ & $\Acal(p,q)$             &       $\upU(p,q)$ & $\GL_{p+q}(\C)$ \\ \hline
 \ (8)\ & $\Bcal(n)$             &       $\GL_n(\C)$ & $\upU(n,n)$ \\ \hline
 \ (9)\ & $\Scal(p,q)$              &     $\upU(p,q)\times \upU(q,p)$ & $\upU(p+q,p+q)$ \\ \hline
\ (10)\ & $\Tcal(p,q)$             &      $\upU(p,p)\times \upU(q,q)$ & $\upU(p+q,p+q)$ \\ \hline

\end{tabular}
\end{center}

\bigskip

Les cas classiques (1), (2) et (3) ont déjà été traités à la section \ref{sec:LR-U-tilde-U}, et les exemples (4) et (7) font parties des sous-cas considérés à la section
\ref{sec:delta-U-sigma}.

\chapter{Cônes associés à des valeurs propres}

Rappelons que l'on peut associer, à trois sous-ensembles finis $I,J,L\subset \N-\{0\}$, le coefficient de Littlewood-Richardson 
$\cc_{I,J}^L\in  \N$. Voir la section \ref{sec:calcul-schubert-grass}.

\section{Les cônes $\Ecal_{\mathrm{I}}(n)$ }

Pour $n\geq 1$, on désigne par $\sym(n)$ l'espace vectoriel des matrices symétriques réelles $n\times n$. 
Nous considérons l'application $\Re : \herm(n)\to \sym(n)$ qui associe à une matrice hermitienne sa partie réelle.
Le spectre d'une matrice $X\in  \herm(n)$ est noté par $\e(X)=(\e_1\geq \cdots\geq \e_n)\in\R^n_+$.

L'objectif principal de cette section est l'étude du cône
$$
\Ecal_{\mathrm{I}}(n)=\Big\{(\e(X), \e(\Re(X))), \ X\in \herm(n)\Big\}\subset \R^n_+\times \R^n_{+}.
$$

Nous obtenons la description récursive suivante.


\begin{theorem}\label{theo:E-I}
Un élément $(x,y)\in\R^n_+\times \R^n_{+}$ appartient à $\Ecal_{\mathrm{I}}(n)$ si et seulement si les conditions suivantes sont satisfaites:
\begin{enumerate}
\item $|x|=|y|$,
\item $|x|_{I}\geq |y|_J$  pour tout $I,J\in\Pcal(r,n)$ tels que
\begin{equation}\label{eq:cone-E-I}
(2\mu(I),\mu(J))\in \Ecal_{\mathrm{I}}(r).
\end{equation}
\end{enumerate}
De plus, la condition $(2\mu(I),\mu(J))\in \Ecal_{\mathrm{I}}(r)$ est équivalente à $\cc^{J}_{I,I}\neq 0$.
\end{theorem}

\begin{rem}
Dans le théorème précédent, on peut remplacer la condition (\ref{eq:cone-E-I}) par $\cc^{J}_{I,I}=1$.
\end{rem}

\begin{exemple}
\begin{itemize}
\item $\Ecal_{\mathrm{I}}(1)=\{x=y\}\subset \R^2$.
\item $(x,y)\in(\R^2_{+})^2$ appartient à $\Ecal_{\mathrm{I}}(2)$ si et seulement si $x_1+x_2=y_1+y_2$ et $x_1\geq y_1$.
\item $(x,y)\in(\R^3_{+})^2$ appartient à $\Ecal_{\mathrm{I}}(3)$ si et seulement si les relations suivantes sont vérifiées
\begin{equation}\label{eq:E-I-3}
\boxed{
\begin{array}{rcl}
x_1+x_2+x_3 & = & y_1+y_2+y_3\\
x_1 &\geq & y_1 \\
x_2 &\geq & y_3 \\
x_3 &\leq & y_3 \\
x_2 &\leq & y_1 \\
\end{array}
}
\end{equation}
\item $(x,y)\in(\R^4_{+})^2$ appartient à $\Ecal_{\mathrm{I}}(4)$ si et seulement si les relations suivantes sont vérifiées
\begin{equation*}
\boxed{
\begin{array}{rcl}
x_1+x_2+x_3 +x_4& = & y_1+y_2+y_3+y_4\\
x_1 &\geq & y_1 \\
x_2 &\geq & y_3 \\
x_1+x_2& \geq & y_1+y_2\\
x_1+x_3 & \geq & y_2+y_3\\
x_1+x_3 & \geq & y_1+y_4\\
x_1+x_4& \geq& y_3+y_4\\
x_2+x_3 & \geq & y_3+y_4\\
x_4 &\leq & y_4 \\
x_3 &\leq & y_2 \\
\end{array}
}
\end{equation*}
\end{itemize}
\end{exemple}

\bigskip

Le reste de cette section est consacré à la preuve du théorème \ref{theo:E-I}.

\bigskip

Nous allons appliquer les résultats de la section \ref{sec:delta-U-sigma} à la situtation suivante: les groupes $U:=\upU_n\subset U_\C:=\GL_n(\C)$ 
sont munis de l'involution $\sigma(g)=\overline{g}$. Ici, la projection $\pi_{-}:\ugot\to \ugot^{-\sigma}$ correspond à $\Re : \herm(n)\to \sym(n)$ à travers 
les identifications $\ugot\simeq  \herm(n)$ et $ \ugot^{-\sigma}\simeq\sym(n)$ données par l'application $x\mapsto x'=\frac{1}{i}x$ (voir la définition \ref{def:x-y-prime}).

Nous obtenons une première description du cône $\Ecal_{\mathrm{I}}(n)$.

\begin{coro}
\begin{enumerate}
\item $\forall (x,y)\in\R^n_+\times \R^n_{+}$, $(x,y)\in\Ecal_{\mathrm{I}}(n)\ \Longleftrightarrow\ (x,x,2y)\in \horn(n)$.
\item Pour toute partitions $\lambda,\mu$ de longueur au plus $n$, $(2\lambda,\mu)\in\Ecal_{\mathrm{I}}(n)\ \Longleftrightarrow\ c^{\mu}_{\lambda,\lambda}\neq 0$.
\end{enumerate}
\end{coro}

\begin{proof}  Le premier point est une conséquence de la proposition \ref{prop:horn-U-sigma}, car dans notre contexte l'application 
$\sigma_+:\tgot_+\to\tgot_+$ est égale à l'identité. Le second point découle du premier et du fait que 
$c^{\mu}_{\lambda,\lambda}\neq 0$ si et seulement si $(\lambda,\lambda,\mu)\in \horn(n)$ (voir le théorème \ref{theo:saturation}). 
\end{proof}

\medskip

Si nous utilisons les inégalités décrivant $\horn(n)$ (voir section \ref{sec:horn-n}), nous obtenons que $(x,y)\in\Ecal_{\mathrm{I}}(n)$ si et seulement si 
$|x|=|y|$ et $|x|_{I_1}+|x|_{I_2}\geq 2 |y|_{J}$ pour tout $r<n$ et tout $(I_1,I_2,J)\in \LR^n_r$. Dans le cas où $I_1=I_2=I$, l'inégalité devient $|x|_I\geq  |y|_{J}$ 
et la condition $(I,I,J)\in \LR^n_r$ est équivalente à $(2\mu(I),\mu(J))\in \Ecal_{\mathrm{I}}(r)$.

\medskip

Nous allons maintenant montrer que les inégalités $|x|_{I_1}+|x|_{I_2}\geq 2 |y|_{J}$ avec $I_1\neq I_2$ ne sont pas nécessaires pour 
décrire $\Ecal_{\mathrm{I}}(n)$. 

\begin{rem}
Cette propriété de symétrie dans un système d'équations a été obtenue dans un cadre plus général dans des travaux récents de A. Médoc \cite{Medoc24}.
\end{rem}

\begin{exemple}Prenons le cas $n=3$. Alors, pour $(x,y)\in\R^3_+\times \R^3_{+}$, $(x,x,2y)\in \horn(3)$ si et seulement si le système (\ref{eq:E-I-3}) et le système 
\begin{equation}\label{eq:E-I-3-bis}
\boxed{
\begin{array}{rrcrl}
x_1+x_2  & \geq & 2y_2 & \geq & x_2 + x_3\\
2 y_1& \geq &x_1 +x_3 &\geq &  2y_3  \\
\end{array}
}
\end{equation}
sont satisfaits. Exercice intéressant : vérifier que le système (\ref{eq:E-I-3}) $\cup$ (\ref{eq:E-I-3-bis})  est équivalent au sytème (\ref{eq:E-I-3}) pour tout 
$(x,y)\in\R^3_+\times \R^3_{+}$.
\end{exemple}

Nous travaillons avec les sous-algèbre abélienne maximale $\agot=\{\diag(x), x\in \R^n\}$, et le système de racines restreint est 
$\Sigma(\ggot)=\{e_i-e_j, 1\leq i\neq j\leq n\}$. Le cône $\agot_+=\{\diag(x), x\in \R^n_{+}\}$ est une chambre de Weyl restreinte par rapport à l'action de $U$ sur $\herm(n)$. 
Le groupe de Weyl restreint $W_{\agot}$ est égal au groupe de Weyl $W=\Sgot_n$.

\begin{lem}\label{lem:admissible-eigenvalue-1}
Les éléments {\em admissibles} sont de la forme $t\,w\zeta_r + t' \zeta_0$ où $t\in \Q^{\geq 0}$, $t'\in \Q$, $w\in W_\agot$, $\zeta_0=\diag(1,\ldots,1)$ et 
$\zeta_r:=\diag(c_r)$, avec $c_r=(\underbrace{-1,\ldots,-1}_{ r\ fois},0,\ldots,0)$ pour certains $r\in [n]$.
\end{lem}

\begin{proof}Considérons un vecteur rationnel $\zeta\in\agot\simeq\R^n$ tel que $(\star)\ \mathrm{Vect}\big(\Sigma(\ggot)\cap \zeta^\perp\big)=
\mathrm{Vect}\big(\Sigma(\ggot)\big)\cap \zeta^\perp$. Modulo l'action du groupe de Weyl $W_\agot$, nous pouvons supposer que 
$\zeta=(\zeta_1\geq\cdots\geq\zeta_n)$ est dominant. Puisque l'ensemble $\Sigma(\ggot)$ est généré par les racines simples 
$S(\ggot)=\{e_{i}-e_{i+1}, 1\leq i\leq n-1\}$, nous voyons que pour $\zeta$ dominant, la condition $(\star)$ équivaut à demander que 
toutes les formes linéaires de $S(\ggot)$ sauf une doivent s'annuler par rapport à $\zeta$. Cette dernière condition implique que $\zeta$ est de la forme 
$t\,\zeta_r + t' \zeta_0$.
\end{proof}

Les \'el\'ements admissibles $\pm \zeta_0$ fournissent les in\'egalit\'es $|x|\geq |y|$ et $-|x|\geq -|y|$, soit $|x| = |y|$.

On travaille maintenant avec l'élément admissibles $\zeta_r$. L'involution $\sigma$ agit triviallement sur $W$, le groupe $W'$ est trivial, et donc $w_0'=Id$. Finalement, le théorème \ref{theo-delta-U-sigma} nous assure que $(x,y)\in \Ecal_{\mathrm{I}}(n)$ si et seulement si $|x|=|y|$ et
\begin{equation}\label{eq:inegalite-E-1-n}
(x,w\zeta_r)\geq (y,w_0w_1\zeta_r)
\end{equation}
pour tout $1\leq r<n$, et tout $(w,w_1)\in W/W^{\zeta_r}\times W/W^{\zeta_r}$ satisfaisant les conditions suivantes
\begin{enumerate}
\item[a)] $[\Xgot_{w,\zeta}]\cdot [\Xgot_{w,\zeta}]\cdot[\Xgot_{w_1,\zeta}]= [pt]\quad \mathrm{dans}\quad H^{max} (\Fcal_\zeta,\Z)$.
\item[b)] $2\tr(w\zeta \circlearrowright \ngot^{w\zeta>0})+\tr(w_1\zeta \circlearrowright \ngot^{w_1\zeta>0})= 2\tr(\zeta \circlearrowright \ugot_\C^{\zeta>0})$.
\end{enumerate}

Si on pose $I=w([r])$ et $J=w_1([r])$, l'inégalité (\ref{eq:inegalite-E-1-n}) devient $|x|_I\leq |y|_{J^o}$, la condition a) est équivalente à $\cc_{I^o, I^o}^J=1$, 
et la condition b) est impliquée par a). En se servant du fait que $\cc_{I^o, I^o}^J=\cc_{I^c, I^c}^{J^{o,c}}$, on obtient finalement que 
$(x,y)\in \Ecal_{\mathrm{I}}(n)$ si et seulement si $|x|=|y|$ et $|x|_{I}\geq |y|_J$  pour tout $I,J\in\Pcal(r,n)$ tels que $\cc^{J}_{I,I}=1$. 
La preuve du théorème \ref{theo:E-I} est complète.

\section{Les cônes $\Ecal_{\mathrm{II}}(n)$}

Considérons le morphisme d'algèbres de Lie $\iota: \glgot_n(\C)\to \glgot_{2n}(\R)$ qui envoie une matrice $X=A+iB$ vers 
$$
\iota(X)= \begin{pmatrix} A& -B \\B & A\end{pmatrix}.
$$
Ce morphisme induit une application linéaire $\iota : \herm(n)\to \sym(2n)$ et la projection orthogonale correspondante
$\pi: \sym(2n)\to \herm(n)$ est définie par la relation : 
$$
\pi(A)= \frac{X+Y}{2}+ i \frac{Z-{}^t Z}{2}\quad si \quad A=\begin{pmatrix} X& {}^t Z \\Z & Y\end{pmatrix}.
$$

Le but de cette section est l'étude du cône
$$
\Ecal_{\mathrm{II}}(n)=\Big\{(\e(A), \e(\pi(A))), \ A\in \sym(2n)\Big\}\subset \R^{2n}_+\times \R^n_{+}.
$$

Dans \cite{Chenciner}, Chenciner a mis en évidence le lien entre l'ensemble des spectres du moment cinétique d'un mouvement d'équilibre relatif d'une 
configuration centrale à $N$ corps et les sous-ensembles $\Delta_\lambda:= \Ecal_{\mathrm{II}}(n)\ \cap\  \{\lambda\}\times \R^n_{+}$. 
Ensuite, Chenciner et P\'erez ont montré que $\Delta_\lambda$ est un polytope convexe pour tout $\lambda\in \R^{2n}_+$ \cite{Chenciner-Perez}. 
Heckman et Zhao ont été les premiers à comprendre que cette propriété de convexité découle du  théorème d'O'Shea-Sjamaar \cite{Heckman-Zhao}.

Voici une description récursive que nous obtenons pour le cône $\Ecal_{\mathrm{II}}(n)$.

\begin{theorem}\label{theo:E-II} 
Un élément $(x,y)\in\R^{2n}_+\times \R^n_{+}$ appartient à $\Ecal_{\mathrm{II}}(n)$ si et seulement si  $|x|=2|y|$ et si $|x|_{I}\geq 2|y|_J$  est vérifié 
pour tout couple $(I,J)$ satisfaisant les conditions suivantes : il existe $r<n$ tel que
\begin{enumerate}
\item $I\subset [2n]$ est de cardinal $2r$,
\item $J\subset [n]$ est de cardinal $r$,
\item $(\mu(I),\mu(J))\in \Ecal_{\mathrm{II}}(r)$.
\end{enumerate}
De plus, la condition $(\mu(I),\mu(J))\in \Ecal_{\mathrm{II}}(r)$ est équivalente à $\cc^{I}_{J,J}\neq 0$.
\end{theorem}

\begin{rem}
Dans le théorème précédent, on peut remplacer la condition $(\mu(I),\mu(J))\in \Ecal_{\mathrm{II}}(r)$ par $\cc^{I}_{J,J}=1$.
\end{rem}

\begin{exemple}
\begin{itemize}
\item $\Ecal_{\mathrm{II}}(1)=\{(x,y)\in \R^2_+\times\R,\ x_1+x_2=2y\}$.
\item $(x,y)\in\R^4_{+}\times\R^2_+$ appartient à $\Ecal_{\mathrm{II}}(2)$ si et seulement si $x_1+x_2+x_3+x_4=2(y_1+y_2)$ et 
\begin{equation*}
\boxed{
\begin{array}{rcl}
x_1+x_2&\geq & 2y_1\\
x_2+x_3 &\geq & 2y_2\\
x_1+x_4 &\geq & 2y_2
\end{array}
}
\end{equation*}
\item $(x,y)\in\R^6_{+}\times\R^3_+$ appartient à $\Ecal_{\mathrm{II}}(3)$ si et seulement si $\sum_{i=1}^6 x_i=2(y_1+y_2+y_3)$ et 
\begin{equation*}
\boxed{
\begin{array}{rrrrl}
x_1+x_2&\geq& 2y_1&\geq&\max(x_1+x_6,x_2+x_5,x_3+x_4)\\
\min(x_2+x_3,x_1+x_4)&\geq& 2y_2&\geq&\max(x_4+x_5,x_3+x_6)\\
\min(x_1+x_6,x_2+x_5,x_3+x_4)&\geq& 2y_3&\geq&x_5+x_6\\
\end{array}
}
\end{equation*}
\end{itemize}
\end{exemple}

\subsection{Description au moyen du cône $\LR(n,n)$}

Considérons le morphisme $\iota: G:= \GL_n(\C)\croc \tG:= \GL_{2n}(\R)_0$. Les sous-groupes compacts maximaux sont $K=\upU_n\croc \tK= \SO_{2n}$. 
Nous avons les décompositions de Cartan $\ggot=\kgot\oplus\pgot$ avec $\pgot=\herm(n)$ et 
$\tggot=\tkgot\oplus\tpgot$ avec $\tpgot=\sym(2n)$.

La complexification du groupe $G$ est $G_\C:= \GL_n(\C)\times \GL_n(\C)$. L'inclusion $G\croc G_\C$ est donnée par l'application 
$g\mapsto (g,\overline{g})$ et l'involution antiholomorphe $\sigma$ sur $G_\C$ est définie par $\sigma(g_1,g_2)= (\overline{g_2},\overline{g_1})$, 
de sorte que $G$ correspond à l'ensemble des points fixes de $\sigma$.

La complexification du groupe $\tG$ est $\tG_\C:= \GL_{2n}(\C)$. L'inclusion $\tG\croc \tG_\C$ est donnée par 
$g\mapsto (g,g)$. L'involution antiholomorphe $\widetilde{\sigma}$ sur $\tG_\C$ est définie par $\widetilde{\sigma}(g)= \overline{g}$, de sorte que 
 $\tG$ correspond à la composante connexe de l'ensemble des points fixes de $\widetilde{\sigma}$.

Le morphisme $\iota: G\croc \tG$ admet une complexification $\iota_\C: G_\C\croc \tG_\C$ définie par
\begin{equation}\label{eq:iota-eigenvalue-2}
\iota_\C(g_1,g_2)=\frac{1}{2}\begin{pmatrix} g_1+g_2& i(g_1-g_2) \\-i(g_1-g_2) & g_1+g_2\end{pmatrix}= P\begin{pmatrix} g_1& 0 \\ 0& g_2\end{pmatrix}P^{-1},
\end{equation}
avec $P=\frac{1}{\sqrt{2}}\begin{pmatrix} I_n& -i I_n \\ -i I_n & Id_n\end{pmatrix}\in \upU_{2n}$.
Remarquez que $\iota_\C\circ \sigma= \widetilde{\sigma}\circ \iota_\C$.

Les groupes $U=\upU_n\times \upU_n$ et $\tU= \upU_{2n}$  sont respectivement des sous-groupes compacts maximaux de $G_\C$ et $\tG_\C$.
Les sous-groupes de points fixes des involutions sont $K=U^\sigma$ et $\tK=(\tU^{\tilde{\sigma}})_0$.
 
Au niveau de l'algèbre de Lie, nous avons une projection orthogonale $\pi : \glgot_{2n}(\C)\longrightarrow \glgot_n(\C)\times \glgot_{n}(\C)$ définie par 
$$
\pi\begin{pmatrix} X&Z \\ T & Y\end{pmatrix}=\tfrac{1}{2}\Big(X+Y-i(Z-T),X+Y+i(Z-T)\Big).
$$

Considérons une autre application $\pi_0 : \glgot_{2n}(\C)\longrightarrow \glgot_n(\C)\times \glgot_{n}(\C)$ définie par 
$\pi_0\begin{pmatrix} X&Z \\ T & Y\end{pmatrix}=(X,Y)$. On vérifie facilement que $\pi=\pi_0\circ \Ad(P^{-1})$.

\medskip

Rappelons la définition du cône $\LR(n,n)$ (voir section \ref{sec:LR-m-n}).

\begin{definition}
Soient $x\in\R^{2n}_+$ et $(y,z)\in(\R^n_+)^2$. Alors $(x,y,z)\in\LR(n,n)$ si et seulement si 
$\upU_{n}\cdot\diag(y)\times \upU_{n}\cdot\diag(z)\subset \pi_0\left(\upU_{2n}\cdot\diag(x)\right)$.
\end{definition}

\medskip

Voici une première description du cône $\Ecal_{\mathrm{II}}(n)$.

\begin{prop}\label{prop:OS-eigenvalue-2} Soit $(x,y)\in\R^{2n}_+\times\R^n_+$.
Les affirmations suivantes sont équivalentes :
\begin{enumerate}
\item $(x,y,y)\in\LR(n,n)$ ;
\item $(x,y)\in \Ecal_{\mathrm{II}}(n)$.
\end{enumerate}
\end{prop}

\begin{proof}
Cela découle de la proposition \ref{prop:OSS-orbites} et du fait que 
$\pi_0\left(\upU_{2n}\cdot\diag(x)\right)=\pi\left(\upU_{2n}\cdot\diag(x)\right)$. 
\end{proof}

Si nous utilisons les inégalités décrivant $\LR(n,n)$ (voir section \ref{sec:LR-m-n}), nous obtenons que $(x,y)\in\Ecal_{\mathrm{II}}(n)$ si et seulement si 
les conditions suivantes sont vérifiées : $|x|=2|y|$ et 
\begin{enumerate}\setlength{\itemsep}{8pt}
\item[i)] $x_{n+k}\leq y_k\leq x_k$, $\forall k\in [n]$,
\item[ii)] $|x|_L\geq |y|_I + |y|_J$, pour tout triplet $L\subset [2n]$, $I,J\subset [n]$,  satisfaisant $\sharp L =\sharp I +\sharp J$ et $\cc^L_{I,J}\neq 0$.
\end{enumerate}

Dans les prochaines sections, nous verrons que les inégalités $|x|_L\geq |y|_I + |y|_J$, avec $I\neq J$, ainsi que celles de i),  ne sont pas nécessaires pour 
décrire $\Ecal_{\mathrm{II}}(n)$.

\medskip
 
Le résultat suivant, qui découle du théorème \ref{theo:saturation}, est essentiel pour obtenir une description récursive de $\Ecal_{\mathrm{II}}(n)$.

\begin{lem}\label{prop:eigenvalue-2-saturation} Soient $\lambda,\mu$ deux partitions telles que $\leng(\lambda)\leq 2n$ et  $\leng(\mu)\leq n$.
Alors $(\lambda,\mu,\mu)\in \LR(n,n)$ si et seulement si  $c^{\lambda}_{\mu,\mu}\neq 0$.
\end{lem}

\subsection{Preuve du théorème \ref{theo:E-II} }

\subsubsection{Éléments admissibles}

Nous travaillons avec les sous-algèbres abéliennes maximales $\agot\subset \pgot=\herm(n)$, $\tagot\subset \tpgot=\sym(2n)$ définies par
$\agot=\{\diag(x), x\in \R^n\}$, et $\tagot=\{\diag(z), z\in \R^{2n}\}$. La fonction $\iota: \agot\to \tagot$ est définie par $\iota(\diag(x))=\diag(x,x)$. 
L'ensemble des racines relatives à l'action de $\agot$ sur $\tggot/\ggot$ est $\Sigma(\tggot/\ggot):=\{\epsilon_i-\epsilon_j, 1\leq i\neq j\leq n\}$.
 
Le cône $\agot_+=\{\diag(x), x\in \R^n_{+}\}$ est une chambre de Weyl restreinte par rapport à l'action de $K$ sur $\pgot$. 
Le groupe de Weyl restreint $W_{\agot}$ est naturellement isomorphe au groupe des permutations $\Sgot_n$. 
Comme au lemme \ref{lem:admissible-eigenvalue-1}, on voit que les éléments {\em admissibles} sont ici de la forme $t\,w\zeta_r + t' \zeta_0$.

Considérons le tore maximal $T\times T\subset U$ où $T\subset \upU_n$ est le tore maximal des matrices diagonales. 
Alors l'application $X\mapsto (iX,iX)$ détermine un isomorphisme $\agot\simeq(\tgot\times \tgot)^{-\sigma}$. Nous choisissons la chambre de Weyl 
$(\tgot\times \tgot)_+=\{(\diag(ix_1),\diag(ix_2)), x_1,x_2\in\R^n_+\}$ de sorte que le sous-groupe de Borel correspondant de 
$\GL_n(\C)\times \GL_n(\C)$ soit le produit $B_n\times B_n$.

Soit $\tT\subset \tU$ le tore maximal des matrices diagonales. Ici, $\ttgot^{-\sigma}=\ttgot=i\tagot$, et le groupe de Weyl restreint $W_{\tagot}$ 
est naturellement isomorphe au groupe de permutations $\Sgot_{2n}$. Nous considérons la chambre de Weyl $\ttgot_+=\{\diag(iz) , z\in\R^{2n}_+\}$ : 
le sous-groupe de Borel correspondant $B_{2n}\subset\GL_{2n}(\C)$ est formé par les matrices triangulaires supérieures.

\subsubsection{Conditions cohomologiques}

Les éléments $\pm \zeta_0$ sont des éléments admissibles qui satisfont de manière évidente la condition cohomologique. Dans ce cas, les inégalités correspondantes $\pm(x,\zeta_0)\geq \pm(y,w_0\zeta_0)$ sont équivalentes à $|x|=2|y|$.

Nous travaillons maintenant avec l'élément admissible $\zeta_r$. 
La variété des drapeaux $\GL_n(\C)/P(\zeta_r)$ admet une identification naturelle avec la  grassmannienne $\G(r,n)$.
 Soit $\tzeta_r=\begin{pmatrix} \zeta_r&0 \\ 0 & \zeta_r\end{pmatrix}$ et considérons le sous-groupe parabolique $P(\tzeta_r)\subset \GL_{2n}(\C)$ 
 qui lui est associé. On associe à un sous-espace vectoriel $V\subset \C^n$ deux sous-espaces vectoriels $V'=\{(v,0), v\in V\}$ et $V''=\{(0,v), v\in V\}$ de  
 $\C^n\times \C^n=\C^{2n}$. L'application $g P(\tzeta_r)\mapsto g\left((\C^r)'\oplus(\C^r)''\right)$ définit un isomorphisme entre la variété des drapeaux 
 $GL_{2n}(\C)/P(\tzeta_r)$ et  la grassmannienne $\G(2r,2n)$.
 
Le morphisme $\iota_\C : \GL_n(\C)\times \GL_n(\C) \to \GL_{2n}(\C)$ induit l'immersion
$$
\iota_r =  \G(r,n)\times\G(r,n)\to  \G(2r,2n)
$$ 
définie par $\iota_r(V)= P(V'\oplus V'')$. Ici, $P$ est la matrice unitaire utilisée dans (\ref{eq:iota-eigenvalue-2}).

 Soit $(w,\tw)\in \Sgot_n\times \Sgot_{2n}$. Les variétés de Schubert correspondantes sont 
 $\Xgot_{w,r}= \overline{B_n [w]}\times \overline{B_n [w]}\subset \G(r,n)\times\G(r,n)$ et 
 $\tXgot_{\tw,r}= \overline{B_{2n} [\tw]}\subset \G(2r,2n)$. 
 
 On associe à $(w,\tw)\in \Sgot_n\times \Sgot_{2n}$ les sous-ensembles $I=\tw\left([r]\cup\{n+k, k\in[r]\}\right)\subset [2n]$ et $J=w([r])\subset [n]$. 
 Le résultat qui suit est une conséquence de la proposition \ref{prop:calcul-grassmanienne-a-b}.

\begin{lem} Soit $\ell\in\N-\{0\}$. La condition $[\Xgot_{w,r}]\cdot \iota_r^*([\tilde{\Xgot}_{\tw,r}])= \ell [pt]$  est équivalente à 
$c^{I^o}_{J,J}=\ell$.
\end{lem}

\subsubsection{Inégalités linéaires}

Utilisons l'identification $\agot_+\simeq\R^n_+$ et $\tagot_+\simeq\R^{2n}_+$ . Si $(x,y)\in\tagot_+\times \agot_+$, 
les inégalités $(x,\tilde{w}\tzeta_r)\geq (y,w_0w\zeta_r)$, $-|x|_I\geq -2|y|_{J^o}$, et $|x|_{I^c}\geq 2|y|_{J^{o,c}}$ 
sont équivalentes, étant donné que $|x|=2|y|$.

Finalement, le théorème \ref{theo:main-reel} dit que $(x,y)\in\tagot_+\times \agot_+$ appartient à $\Ecal_{\mathrm{II}}(n)$ si et seulement si $|x|= 2|y|$ et 
$|x|_{I^c}\geq 2|y|_{J^{o,c}}$ pour tout 
couple de sous-ensembles $I,J$ tels que
\begin{enumerate}
\item $I\subset [2n]$ est de cardinal $2r$,
\item $J\subset [n]$ est de cardinal $r$,
\item $c^{I^o}_{J,J}\neq 0$.
\end{enumerate}
Comme $c^{I^c}_{J^{o,c},J^{o,c}}=c^{I^o}_{J,J}$ (voir la proposition \ref{prop:calcul-grassmanienne-a-b}), la preuve du théorème \ref{theo:E-II} est terminée.

\chapter{Cônes de Horn singulier}

Soit $A$ une matrice rectangulaire complexe, de taille $m\times n$, et soit $A^*$ la transposée conjuguée complexe de $A$.
Soit $\e_1(AA^*)\geq \cdots\geq \e_m(AA^*)\geq 0$ les valeurs propres de la matrice semi-définie positive $AA^*$. Remarquons que $\e_k(AA^*)=0$ 
lorsque $k>\ell:=\inf\{m,n\}$.

Les valeurs singulières de la matrice $A$ sont les coordonnées du vecteur 
$$
\s(A):=\left(\sqrt{\e_1(AA^*)},\ldots,\sqrt{\e_\ell(AA^*)}\right)\in \R^\ell_{++}.
$$

Considérons l'action canonique du groupe unitaire $U_m\times U_n$ sur $M_{m,n}(\C)$ : $(g,h)\cdot X= gXh^{-1}$, $\forall (g,h)\in U_m\times U_n$. 
L'application des valeurs singulières $\s : M_{m,n}(\C)\to \R^\ell_{++}$ induit une application bijective  \
$M_{m,n}(\C)/U_m\times U_n\stackrel{\sim}{\longrightarrow}\R^\ell_{++}$.

Dans le reste de cette section on fixe $p\geq q\geq 1$ et on pose $n=p+q$. Si $A$ est une matrice complexe $p\times q$, nous notons 
$\s(A)=(\s_1(A)\geq \cdots \geq\s_q(A)\geq 0)$ son spectre singulier.

\begin{definition}
Le cône de Horn singulier, noté  $\sing(p,q)\subset (\R_{++}^q)^3$, est défini comme l'ensemble des triplets 
$(\s(A),\s(B),\s(A+B))$ où $A,B$ parcourent les matrices complexes de taille 
$p\times q$.
\end{definition}

\begin{rem}
Comme $\s(A+B)=\s(-(A+B))$, le cône $\sing(p,q)$ peut être défini comme l'ensemble des triplets 
$(\s(A),\s(B),\s(C))$ où $A,B, C$ vérifient $A+B+C=0$. Cela permet de voir que $\sing(p,q)$ est stable par rapport 
à l'action de permutation de $\Sgot_3$ sur $(\R_{++}^q)^3$.
\end{rem}

Nous expliquons maintenant la description de $\sing(p,q)$ que nous obtenons en appliquant le théorème \ref{theo:main-reel} . 

\begin{definition}Pour tout $r\in [q]$, soit 
$\Bcal(r,p,q)\subset \Pcal(r,n)$ l'ensemble des sous-ensembles $I\subset [n]$ de cardinal $r$ satisfaisant
$I\cap I^o=\emptyset$ et $I\cap \{q+1,\ldots, p\}=\emptyset$. 
\end{definition}

On voit aisément que tous les éléments de $\Bcal(r,p,q)$ sont de la forme $I= I_+\cup I_{-}^o$ 
où $I_+,I_{-}$ sont deux sous-ensembles disjoints de $[q]$ tels que $\sharp I_+ + \sharp I_{-}= r$.
 
\`A chaque triplet $I,J,L\in\Bcal(r,p,q)$, on associe l'inégalité
$$
(\star)_{I,J,L}\hspace{2cm} |\,x\,|_{I_{-}}+|\,y\,|_{J_{-}}  +|\, z\,|_{L_{-}}   \geq  |\,x\,|_{I_+}+|\,y\,|_{J_+} +|\, z\,|_{L_+},\quad (x,y,z)\in\R^q.
$$

Rappelons que pour des sous ensembles non-vides $A\subset B\subset \N$, on pose $A\natural B:=\{\sharp B\leq x , x\in A\}$. 
Si  $A\subset B\subset [n]$, alors $A\natural B\subset [m]$ avec $m=\sharp B$, et on vérifie facilement que\footnote{Ici $A^o=\{n+1-x;\ x\in A\}$ et 
$(A\natural B)^o=\{m+1 - z;\ z\in A\natural B\}$.}   $(A\natural B)^o= A^o\natural B^o$.

\medskip

Voici le principal résultat de cette section.

\begin{theorem}\label{theo:singular-p-q}
Un élément $(x,y,z)\in(\R^q_{++})^3$ appartient à $\sing(p,q)$ si et seulement si l'inégalité $(\star)_{I,J,L}$ est satisfaite pour 
 tout $r\in [q]$, pour tout $I,J,L\in\Bcal(r,p,q)$ vérifiant les deux conditions
\begin{equation}\label{eq:singular-1}
(I^o,J^o,L)\in \LR_r^n,
\end{equation}
\begin{equation}\label{eq:singular-2}
(I^o\natural I^{c},J^o\natural J^{c},L\natural L^{o,c})\in\LR_r^{n-r}.
\end{equation}
\end{theorem}

On notera que les conditions (\ref{eq:singular-1}) et (\ref{eq:singular-2}) sont symétriques en $I,J,L$.

\medskip

Nous montrerons aussi que $\sing(p,q)$ admet une description intermédiare qui répond (en partie) à une conjecture de A. Buch (voir \cite{Fulton00}, \S 5).

\begin{theorem}\label{theo:singular-p-q-faible}
Un élément $(x,y,z)\in(\R^q_{++})^3$ appartient à $\sing(p,q)$ si et seulement si $(\star)_{I,J,L}$ est vérifié 
pour tout $r\in [q]$, pour tout $I,J,L\in\Bcal(r,p,q)$ tels que $(I^o,J^o,L)\in \LR_r^n$.
\end{theorem}

\begin{rem}\label{rem:singular-p-q-version-2}
Les théorèmes \ref{theo:singular-p-q} et  \ref{theo:singular-p-q-faible} sont toujours valables si nous renforçons les conditions en exigeant que 
$$
\cc^{L}_{I^o,J^o}\ =\ 1 \qquad \mathrm{et}\qquad \cc_{I^o\natural I^{c},J^o\natural J^{c}}^{L\natural L^{o,c}}\ =\ 1.
$$
\end{rem}

\section[$\sing(p,q)$ versus $\horn(n)$]{Description de $\sing(p,q)$ au moyen de $\horn(n)$}\label{sec:horn-versus-sing}

Il existe un lien naturel entre $\sing(p,q)$ et $\horn(n)$ car, pour toute matrice complexe A de taille $p\times q$, le spectre de la matrice hermitienne $n\times n$ 
\begin{equation}\label{eq:morphisme-h}
\widehat{A}:=\begin{pmatrix}
0 & A\\
A^* &0
\end{pmatrix}
\end{equation}
est égal à $(\s_1(A)\geq \cdots \geq\s_q(A)\geq 0\geq\cdots\geq 0\geq -\s_q(A)\geq \cdots \geq -\s_1(A))$. 

\medskip

\begin{definition}\label{def:x-chapeau}
À tout $x=(x_1,\ldots, x_q)\in \R^q$, nous associons 
le vecteur 
$$\widehat{x}=(x_1,\ldots, x_q,0,\ldots,0,- x_q,\ldots,- x_1)\in\R^n.
$$ 
\end{definition}

\medskip

\begin{prop}\label{prop:sing-versus-horn}
Pour tout $(x,y,z)\in(\R_{++}^q)^3$, nous avons l'équivalence:
$$
(x,y,z)\in \sing(p,q)\Longleftrightarrow (\,\widehat{x},\,\widehat{y}\,,\,\widehat{z}\ )\in \horn(n).
$$
\end{prop}
\begin{proof}Le sens $\Longrightarrow$ est immédiat. Soient $(x,y,z)\in \sing(p,q)$:  il existe $A,B\in M_{p,q}(\C)$ tel que $x=\s(A)$, $y=\s(B)$, et $z:=\s(A+B)$. 
On voit alors que $(\widehat{x},\widehat{y},\widehat{z}\,)\in \horn(n)$ car  $\widehat{x}=\e(\,\widehat{A}\,)$, $\widehat{y}=\e(\,\widehat{B}\,)$, et $\widehat{z}=
\e(\,\widehat{A}+\widehat{B}\,)$.

Expliquons maintenant pourquoi le sens $\Longleftarrow$ est une conséquence du théorème d'O'Shea-Sjamaar. 
Nous travaillons avec l'involution $\sigma_{p,q}(g)=I_{p,q}(g^*)^{-1}I_{p,q}$ sur $GL_n(\C)$, où $I_{p,q}$ est la matrice diagonale $\diag(I_p,-I_q)$. 
Soit $U_\C=GL_n(\C)$ plongé diagonalement diagonalement dans $\tU_\C=GL_n(\C)\times GL_n(\C)$. Nous considérons sur 
$\tU_\C$ l'involution  $\tilde{\sigma}_{p,q}(g_1,g_2)=(\sigma_{p,q}(g_1),\sigma_{p,q}(g_2))$. La projection orthogonale $\pi: \tugot_\C\to \ugot_\C$ est définie par 
$\pi(X,Y)=\frac{1}{2}(X+Y)$. 

Nous voyons alors que $G=U(p,q)$ et $\tG=G\times G$ sont les sous-groupes fixés par $\sigma_{p,q}$ et $\tilde{\sigma}_{p,q}$. Le sous-groupe compact 
$K:=U(p,q)\cap\upU_n$ est canoniquement isomorphe à $\upU_p\times\upU_q$.

Les sous-espaces vectoriels $(\ugot_\C)^{-\sigma_{p,q}}$  et $(\tugot_\C)^{-\tilde{\sigma}_{p,q}}$ sont respectivement égaux à $\{i h(A), A\in M_{p,q}(\C)\}$ et à 
$(\ugot_\C)^{-\sigma_{p,q}}\times (\ugot_\C)^{-\sigma_{p,q}}$. Soit $A,B,C\in M_{p,q}(\C)$. Le théorème d'O'Shea-Sjamaar permet de voir que 
les conditions suivantes sont équivalentes
\begin{enumerate}
\item $\upU_n\cdot\, \widehat{C}\subset \pi\Big(\upU_n\cdot\, \widehat{A}+\upU_n\cdot\, \widehat{B}\,\Big)$,
\item $K\cdot \widehat{C}\subset \pi\Big(K\cdot \widehat{A}+K\cdot \widehat{B}\,\Big)$.
\end{enumerate}
La première condition est équivalente à $(\e(\,\widehat{A}\,),\e(\,\widehat{B}\,),\e(\,\widehat{C}\,))\in\horn(n)$ tandis que la seconde signifie que 
$(\s(A),\s(B),\s(C))\in \sing(p,q)$. Comme
$\e(\,\widehat{X}\,)=\widehat{\s(X)},\forall X\in M_{p,q}(\C)$, la preuve est complète.
\end{proof}

\medskip

\begin{coro}
Les éléments $(x,y,z)$ de $\sing(p,q)$ satisfont les inégalités $(\star)_{I,J,L}$  pour 
 tout $r\in [q]$, pour tout $I,J,K\in\Bcal(r,p,q)$ vérifiant $(I^o,J^o,L)\in \LR_r^n$. Ainsi, le théorème \ref{theo:singular-p-q-faible} est une 
 conséquence du théorème \ref{theo:singular-p-q}.
\end{coro}

\begin{proof} Soient $I,J,L\in\Bcal(r,p,q)$ satisfaisant $(I^o,J^o,L)\in \LR_r^n$ et $(x,y,z)\in\sing(p,q)$. D'après la proposition précédente et la description de 
$\horn(n)$ obtenu dans le théorème \ref{theo:horn}, on a
$$
|\,\widehat{x}\,|_{I^o}+|\,\widehat{y}\,|_{J^o}\geq |\,\widehat{z}\, |_{L},
$$
et cette identité est équivalente à $(\star)_{I,J,L}$.
\end{proof}

\medskip

Le problème principal avec la description de $\sing(p,q)$ donnée dans la proposition \ref{prop:sing-versus-horn} est que la plupart des inégalités de $\horn(n)$ 
sont redondantes lorsqu'elles sont restreintes aux triplets $(\,\widehat{x},\,\widehat{y}\,,\,\widehat{z}\ )$.
Considérons le cas de $\sing(3,3)$. La proposition \ref{prop:sing-versus-horn} nous dit que $\sing(3,3)$ s'identifie à l'intersection 
$\horn(6)\cap V^3$ où $V=\{(a,b,c,-c,-b,-a)\}\subset \R^6$. Grâce à des calculs informatiques, nous savons que $\horn(6)\subset (\R^6)^3$ 
est décrit par une liste minimale de 536 inégalités. \`A la section \ref{sec:example-horn-1-2-3}, nous montrerons que $\sing(3,3)$ est décrit par $96$ inégalités.

\section{Inégalités classiques}\label{sec:WLT-inequalities}

Le but de cette section est d'expliciter quelques exemples de triplets $(I,J,L)$ qui satisfont les deux conditions du théorème \ref{theo:singular-p-q}. 
Nous retrouverons ainsi  les inégalités classiques de Weyl, Lidskii-Wielandt et Thompson-Freede.

\subsubsection*{Inégalités de Weyl}

Pour tout $m\geq 1$, l'ensemble $\LR^m_1$ est l'ensemble des triplets $(i,j,k)\in[m]^3$ tels que $i+j=k+1$.

Soient $i,j\in [q]$ tels que $i+j-1\leq q$. Considérons le triplet suivant de $\Pcal(1,p,q)$ : $I=\{n+1-i\}$, $J=\{n+1-j\}$, et $K=\{i+j-1\}$. Alors 
$(\mu(I^o),\mu(J^o),\mu(K))\in \LR^n_1$. On vérifie maintenant que 
$$
I^o\natural I^{c}=\{i\},\quad J^o\natural J^{c}=\{j\},\quad \mathrm{et}\quad K\natural L^{o,c}=\{i+j-1\}
$$
et donc $(I^o\natural I^{c},J^o\natural J^{c},K\natural L^{o,c})\in \LR^{n-1}_1$.

Ainsi $(I,J,K)$ vérife les deux conditions du théorème \ref{theo:singular-p-q}, et les inégalités correspondantes sont les {\em inégalités de Weyl} 
$$
\s_i(A)+\s_j(B)\geq \s_{i+j-1}(A+B),\qquad \forall A,B\in M_{p,q}(\C).
$$

\subsubsection*{Inégalités de Lidskii-Wielandt}

Pour tout $r\in[q]$, l'ensemble $I^n_r:=\{n-r+1,\ldots,n\}$ appartient à $\Bcal(r,p,q)$. Comme $\mu((I^n_r)^o)=0\in\Z^r$, le triplet 
$(I^n_r,J, J^o)$ satisfait la consition (\ref{eq:singular-1}) pour tout $J\in \Bcal(r,p,q)$. Comme $(I^n_r)^o\natural (I^n_r)^c =[r]$, on voit que 
$(I^n_r,J, J^o)$ satisfait aussi la consition (\ref{eq:singular-2}) pour tout $J\in \Bcal(r,p,q)$.

Tout $J\in \Bcal(r,p,q)$ s'écrit $J=J_+\cup J_{-}^o$ où $J_{-}, J_{+}$ sont deux sous-ensembles disjoints de $[q]$ tels que $\sharp J_{-}+\sharp J_{+}=r$. 
L' inégalité correspondante au triplet  $(I^n_r,J, J^o)$ fait partie des {\em inégalités de Lidskii-Wielandt}:
$$
\sum_{i=1}^r\s_i(A)+ \sum_{j\in J_-}\s_i(B)+\sum_{\ell\in J_+}s_\ell(A+B)
\geq 
 \sum_{\ell\in J_+}\s_\ell(B)+ \sum_{j\in J_-}\s_\ell(A+B) ,\qquad \forall A,B\in M_{p,q}(\C).
$$

\subsubsection*{Inégalités de Thompson-Freede}

On travaille ici avec des sous-ensembles $I,J\subset [q]$ de cardinal $r$ tels que $\max I + \max J\leq q+r$. On peut alors définir le sous ensemble 
$K=\{i_k+j_k-k,\ k\in[r]\}\subset [q]$. On laisse le soin au lecteur de vérifier que le triplet $(I^o, J^o, K)$ vérifie les deux conditions du théorème \ref{theo:singular-p-q}. 
Les inégalités correspondantes sont celles de Thompson-Freede: 
$$
\sum_{k=1}^r \s_{i_k}(A)+ \sum_{k=1}^r \s_{j_k}(B) \geq \sum_{k=1}^r \s_{i_k+j_k - k}(A+B),\qquad \forall A,B\in M_{p,q}(\C).
$$

\section{Preuve du théorème \ref{theo:singular-p-q}}\label{sec:preuve-theo-sing-p-q}

Revenons au contexte de la section \ref{sec:drapeaux-2-etage}. On travaille avec les involutions $\sigma_{p,q}(g)=\mathbf{J}_{p,q}(g^*)^{-1}\mathbf{J}_{p,q}$ sur $U_\C:=\GL_n(\C)$ et $\tilde{\sigma}_{p,q}(g_1,g_2)=(\sigma_{p,q}(g_1),\sigma_{p,q}(g_2))$ sur $\tU_\C:=\GL_n(\C)\times \GL_n(\C)$. La matrice 
$\mathbf{J}_{p,q}$ est définie par (\ref{eq:J-p-q}).

Alors $G=\Ad(\theta)(U(p,q))\subset GL_n(\C)$ est le sous-groupe fixé par $\sigma_{p,q}$, où $\theta\in\SO_n$ est définie par (\ref{eq:theta}). 
La décomposition de Cartan de l'algèbre de Lie de 
$G$ est $\ggot=\kgot\oplus\pgot$ où 
$$
\begin{array}{rrl}
\kgot  & = & \Ad(\theta)\left(\left\{\begin{pmatrix} X&  0\\ 0& Y \end{pmatrix},\ (X,Y)\in \ugot(p)\times\ugot(q)\right\}\right)\\
\pgot&= & \Ad(\theta)\left(\left\{\begin{pmatrix} 0&  X\\ X^*& 0 \end{pmatrix},\ X\in M_{p,q}(\C)\right\}\right).
\end{array}
$$
Le tore maximal $T\subset \upU_n$, formé des matrices diagonales, est stable sous l'involutions $\sigma_{p,q}$, et 
$\agot=\frac{1}{i}\tgot^{-\sigma_{p,q}}$ est un sous-espace abélien maximal de $\pgot$. Toutes les matrices de $\agot$ sont de la forme 
$\zeta(x):=\diag(x_1,\cdots,x_q, 0\cdots, 0, -x_q\cdots, -x_1),x\in\R^q$.

On remarque que pour $(x,y,z)\in \R^q_{++}$, on a l'équivalence
$$
(x,y,z)\in\sing(p,q)\Longleftrightarrow K\cdot\zeta(z)\subset K\cdot\zeta(x)+K\cdot\zeta(y),
$$
où $K\simeq \upU_p\times\upU_q$ est le sous-groupe compact maximal de $G$, d'algèbre de Lie $\kgot$.

L'ensemble $\Sigma(\ggot)$ des racines restreintes comprend tous les applications linéaires $\pm f_i\pm f_j$ avec $i\neq j$ et $\pm 2 f_i$ pour tous les $i$. 
De même, les $\pm f_i$ sont des racines restreintes si $p\neq q$ (voir \S 6 dans \cite{Knapp-book}). Nous pouvons donc choisir la chambre de Weyl restreinte suivante
$\agot_+:=\{\zeta(x), x_1\geq x_2\geq \cdots\geq x_q\geq 0\}$.

Le groupe de Weyl restreint $W_{\agot}$ s'identifie à un sous-groupe de $\Sgot_n$:
$$
W_{\agot}= \Big\{w\in\Sgot_n, w(k^o)=w(k)^o, \forall k\in [q]\ \mathrm{et}\ w(k)=k, \forall k\in [q+1,\ldots,p]\Big\}.
$$
Pour tout $r\in[q]$, on pose $\zeta_r= \diag(\underbrace{-1,\ldots,-1}_{ r\ fois},0,\ldots,0, \underbrace{1,\ldots,1}_{ r\ fois})\in\agot$.

\begin{lem}\label{lem:admissible-u-p-q}
Les éléments admissibles sont de la forme $t\,w\, \zeta_r$ avec $(t,w,r)\in \Q^{>0}\times W_{\agot}\times [q]$.
\end{lem}

\begin{proof}Considérons un vecteur rationnel $\zeta(x)\in\agot$ tel que $(\star)\ \mathrm{Vect}\big(\Sigma(\ggot)\cap \zeta(x)^\perp\big)=
\mathrm{Vect}\big(\Sigma(\ggot)\big)\cap \zeta(x)^\perp$. Modulo l'action du groupe de Weyl $W_\agot$, nous pouvons supposer que 
$x=(x_1\leq\cdots\leq x_q\leq 0)$. Puisque les racines restreintes positives s'expriment comme une combinaison lin\'eaire à 
coefficients positifs des \'el\'ements de $S(\ggot):=\{f_{i}-f_{i+1}, 1\leq i\leq n-1\}\cup \{f_n\}$, nous voyons que pour $\zeta(x)$ anti-dominant, 
la condition $(\star)$ équivaut à demander que toutes les formes linéaires de $S(\ggot)$ sauf une doivent s'annuler par rapport à $\zeta(x)$. 
Cette dernière condition implique que $\zeta(x)$ est de la forme $t\,\zeta_r$, avec $t\in  \Q^{>0}$ et $r\in [q]$.
\end{proof}

\medskip

Terminons la preuve en appliquant le théorème \ref{eq:inegalite-main-th} à notre situation: $(x,y,z)\in \sing(p,q)$ 
si et seulement si l'inégalité 
\begin{equation}\label{eq:preuve-sing-p-q}
(\zeta(x),w_1\zeta_r)+(\zeta(y),w_2\zeta_r)\geq (\zeta(z),w_0w_3\zeta_r)
\end{equation}
est satisfaite, pour tout $r\in [q]$, et pour tout $w_1,w_2,w_3\in W_{\agot}/W_{\agot}^{\zeta_r}$ tels que 
\begin{enumerate}\setlength{\itemsep}{8pt}
\item[a)] $[\Xgot_{w_1,\zeta_r}]\cdot[\Xgot_{w_2,\zeta_r}]\cdot[\Xgot_{w_3,\zeta_r}]=\ell [pt], \ell\geq 1$ dans $H^{max}( \Fcal_{\zeta_r},\Z)$,
\item[b)] $(w_1,w_2,w_3)$ est Levi-mobile.
\end{enumerate}

Utilisons le difféomorphisme $\varphi : \Fcal_{\zeta_r}\to \F(r,n-r;n)$ défini par $\varphi(g \Pbb(\zeta_r))= g\cdot (\C^r\subset \C^{n-r})$. 
Comme l'application $w\mapsto I:=w([r])$ identifie le quotient $W_{\agot}/W_{\agot}^{\zeta_r}$ avec $\Bcal(r,p,q)$, on voit que $\varphi$ identifie 
les variétés de Schubert $\Xgot_{w,\zeta_r}, w\in W_{\agot}/W_{\agot}^{\zeta_r}$ aux variétés de Schubert 
$\Xgot_{I\subset I^o},I\in \Bcal(r,p,q)$ définies à la section \ref{sec:drapeaux-2-etage}.

Grâce à la proposition \ref{prop:levi-mobile-double-grass-theta}, nous savons que les conditions a) et b) sont équivalentes à
\begin{enumerate}\setlength{\itemsep}{8pt}
\item[a')] $[\Xgot_{I}]\cdot[\Xgot_{J}]\cdot[\Xgot_{L}]=\ell [pt]$, avec $\ell'\geq 1$ dans $H^{max}(\G(r, n),\Z)$,
\item[b')] $[\Xgot_{I\natural I^{o,c}}]\cdot[\Xgot_{J\natural J^{o,c}}]\cdot[\Xgot_{L\natural L^{o,c}}]= \ell''[pt]$, avec $\ell''\geq 1$, dans $H^{max}(\G(r, n-r),\Z)$,
\end{enumerate}
avec $I:=w_1([r])$, $J:=w_2([r])$, et $L:=w_3([r])$. De plus, les conditions a') et b') sont équivalentes à 
(\ref{eq:singular-1}) et (\ref{eq:singular-2}).

Finalement, un calcul élémentaire montre que l'inégalité (\ref{eq:preuve-sing-p-q}) est équivalente à $(\star)_{I,J,L}$. 
La preuve du théorème \ref{theo:singular-p-q} est complète.

\section{Le cone convexe $\sing(\infty,q)$}\label{sec:preuve-singular-horn}

Dans cette section, nous étudions le comportement des cônes $\sing(p,q)$ lorsque $p$ varie. Pour ce faire, nous allons changer de notations et paramétrer 
les inégalités $(\star)_{I,J,L}$ au moyen des sous-ensembles $I_\pm,J_\pm,L_\pm\subset [q]$.

\subsection{Ensemble $\Ecal_{p,q}$}\label{sec:E-p-q}

Nous commençons par quelques définitions.

\begin{definition} Soit $p\geq q$.
\begin{enumerate}\setlength{\itemsep}{8pt}
\item Un sous-ensemble $X_\bullet\subset [q]$ est dit polarisé, s'il admet une partition $X_\bullet=X_+\coprod X_-$.
\item \`A un ensemble polarisé $X_\bullet\subset [q]$ de cardinal $r$, on associe 
\begin{itemize}\setlength{\itemsep}{8pt}
\item $X^p_\bullet=X_+\coprod \{p+q+1-x,x\in X_-\}\ \in\  \Bcal(r,p,q)$,
\item $\widetilde{X}^p_\bullet= X^p_\bullet \natural (X^p_\bullet)^{o}\  \in\  \Pcal(r,p+q-r)$.
\end{itemize}
\end{enumerate}
\end{definition}

Notons $\Ecal_q$ l'ensemble formé des triplets $(I_\bullet,J_\bullet,L_\bullet)$ de sous-ensembles polarisés de $[q]$ de même cardinal. \`A chaque 
$(I_\bullet,J_\bullet,L_\bullet)\in\Ecal_q$, on associe l'inégalité
$$
(\star)_{I_\bullet,J_\bullet,L_\bullet} :\quad  \qquad    |\,x\,|_{I_{-}}+|\,y\,|_{J_{-}} +|\, z\,|_{L_{-}}\geq |\,x\,|_{I_+}+|\,y\,|_{J_+} +|\, z\,|_{L_+},\qquad x,y,z\in \mathbb{R}^q.
$$

Un triplet $(\star)_{I_\bullet,J_\bullet,L_\bullet}$ est dit {\em régulier} 
lorsque\footnote{Cette relation est équivalente à $\sharp I_+ + \sharp J_+ + \sharp L_+ =r$ où $r=\sharp I_\bullet=\sharp J_\bullet=\sharp L_\bullet$.} 
$\sharp I_- +  \sharp J_- + \sharp L_- = 2(\sharp I_+ + \sharp J_+ + \sharp L_+)$.

\begin{definition}\label{def:E-p-q}
Notons $\Ecal_{p,q}\subset\Ecal_q$ l'ensemble des triplets $(I_\bullet,J_\bullet,L_\bullet)$ satisfaisant les conditions du
Théorème \ref{theo:singular-p-q}, c'est à dire 
\begin{itemize}\setlength{\itemsep}{8pt}
\item[1.] $\sharp I_\bullet=\sharp J_\bullet=\sharp L_\bullet=r\in [q]$,
\item[2.] $((I^p_\bullet)^o,(J^p_\bullet)^o,L^p_\bullet)\in \LR^{p+q}_r$,
\item[3.] $((\widetilde{I}^p_\bullet)^o,(\widetilde{J}^p_\bullet)^o,\widetilde{L}^p_\bullet)\in \LR^{p+q-r}_r$.
\end{itemize}
\end{definition}

On peut reformuler le théorème \ref{theo:singular-p-q} de la manière suivante : $(x,y,z)\in(\R^q_{++})^3$ appartient à $\sing(p,q)$ si et seulement si
$(\star)_{I_\bullet,J_\bullet,L_\bullet}$ est vérifiée pour tout triplet $(I_\bullet, J_\bullet,L_\bullet)\in \Ecal_{p,q}$.

\medskip

Nous allons maintenant étudier le comportement des partitions $\mu(X^p_\bullet)$ et $\mu(\widetilde{X}^p_\bullet)$ lorsque $p$ varie. 
Nous introduisons d'autres notations.

\begin{definition} 
\begin{enumerate}
\item Pour une partition $\mu=(\mu_1,\ldots,\mu_r)$, on pose $|\mu|=\sum_{i=1}^r \mu_i$.
\item Pour un sous-ensemble $X\subset \N-\{0\}$, on pose $|X|=\sum_{x\in X} x$
\item Soit $r\in [q]$. Pour tout $0\leq \alpha \leq r$, nous notons par $1_\alpha^r=(1,\ldots,1,0,\ldots,0)\in \mathbb{R}^r$ le vecteur où $1$ apparaît $\alpha$ fois.
\end{enumerate}
\end{definition}

\begin{lem}\label{lem:lambda-X}Soit $X_\bullet\subset [q]$ un sous-ensemble polarisé de cardinal $r\geq 1$.
\begin{enumerate}
\item Pour tout $p\geq q$, on a
\begin{equation}
\begin{array}{rrl}
|\mu(X^p_\bullet)|  & = &  |X_+| - |X_-|+ (p+q+1)\sharp X_- -\tfrac{r(r+1)}{2}\label{eq:codim-X}\\
|\mu(\widetilde{X}^p_\bullet)| &= & |\mu(X^p_\bullet)|-(\sharp X_-)^2 -2\sharp X_{-}< X_{+}.
\end{array}
\end{equation}
\item Pour tout $p'\geq p\geq q$, nous avons
$$
\mu(X^{p'}_\bullet)-\mu(X^p_\bullet)= (p'-p) 1_{\sharp X_-}^r\qquad \mathrm{et}\qquad
\mu(\widetilde{X}^{p'}_\bullet)-\mu(\widetilde{X}^p_\bullet)= (p'-p) 1_{\sharp X_-}^r.
$$
\end{enumerate}
\end{lem}

\begin{proof} Prouvons le premier point. Un calcul direct donne $|X^p_\bullet|=|X_+|-|X_-| + (p+q+1)\sharp X_-$, puis, comme 
$|\mu(X^p_\bullet)|=|X^p_\bullet| -\tfrac{r(r+1)}{2}$, nous obtenons la première relation de (\ref{eq:codim-X}). La deuxième relation 
utilise le fait que $|\mu(I\natural J)|=|\mu(I)|-\sharp J^c<I$ (voir le lemme \ref{lem:calcul-dimensions}).

Le deuxième point, qui peut être prouvé par calcul direct, est laissé au lecteur. 
\end{proof}

Les conditions {\em 2.} et {\em 3.} de la définition \ref{def:E-p-q} impliquent  les relations : $|\mu((I^p_\bullet)^o)|+|\mu((J^p_\bullet)^o)|=|\mu(L^p_\bullet)|$ et 
$|\mu((\widetilde{I}^p_\bullet)^o)| + |\mu((\widetilde{J}^p_\bullet)^o)| = |\mu(\widetilde{L}^p_\bullet)|$. Nous terminons cette partie en 
calculant précisément ces deux relations.

\begin{coro}\label{coro:inequality-p-q}
Si $(I_\bullet, J_\bullet,L_\bullet)\in \Ecal_{p,q}$, alors
\begin{enumerate}
\item[(C1)] $\sharp I_\bullet =\sharp J_\bullet = \sharp L_\bullet =r\in [q]$,
\item[(C2)] $|I_+| + |J_+| + | L_+| - (|I_-| + |J_-| + L_-|)+\tfrac{r(r+1)}{2}=(p+q+1)(\sharp I_+ + \sharp J_+ + \sharp L_+ -r)$,
\item[(C3)] $(\sharp I_+)^2 + (\sharp J_+)^2 + (\sharp L_+)^2 + 2\big(\sharp I_{+}< I_{-} +\sharp J_{+}< J_{-} +\sharp L_{+}< L_{-}\big)= r^2$.
\end{enumerate} 
\end{coro}

\subsection{Les cônes $\sing(p,q)\subset \sing(p+1,q)\subset\sing(\infty,q) $}

Soit $p\geq q$. À toute matrice $X\in M_{p,q}(\C)$, nous associons 
$$
X^\vee=\begin{pmatrix}X\\0\cdots 0\end{pmatrix}\in M_{p+1,q}(\C).
$$
On voit immédiatement que $X$ et $X^\vee$ ont les mêmes valeurs singulières : $\s(X)=\s(X^\vee)$ pour tout $X\in M_{p,q}(\C)$.
Grâce à ces considérations élémentaires, nous voyons que si $(\s(A),\s(B),\s(A+B))\in \sing(p,q)$, alors 
$(\s(A),\s(B),\s(A+B))=(\s(A^\vee),\s(B^\vee),\s(A^\vee + B^\vee))$ appartient à $\sing(p+1,q)$. Par conséquent, 
$\sing(p,q)\subset  \sing(p+1,q)$ pour tout $p\geq q$.

Cela nous amène à définir le cône convexe suivant
$$
\sing(\infty,q)=\bigcup_{p\geq q}\sing(p,q).
$$

Nous commençons par une observation fondamentale.
\begin{lem}
Il existe  $p_o\geq q$, tel que $\sing(\infty,q)=\sing(p_o,q)$. En d'autres termes, \break 
$\sing(p,q)=\sing(p+1,q)$ pour tout $p\geq p_o$. 
\end{lem}
{\em Preuve :} Grâce au théorème \ref{theo:singular-p-q} , nous savons que pour tout $p\geq q$, le cône convexe $\sing(p,q)$ est déterminé par un 
système fini d'inégalités paramétré par $\Ecal_{p,q}\subset \Ecal_q$. Comme $\Ecal_q$ est fini, il existe $\Rcal\subset \Ecal_q$ tel que 
$\N(\Rcal):=\{p\geq q, \Ecal_{p,q}=\Rcal\}$ est infini.
 
Prenons $p_o=\inf \N(\Rcal)$. Vérifions que $\sing(\infty,q)\subset\sing(p_o,q)$ : comme l'autre inclusion $\sing(p_o,q)\subset\sing(\infty,q)$ est évidente, la preuve 
sera terminée. Soit $x\in \sing(\infty,q)$ : il existe $p\geq q$ tel que $x\in\sing(p,q)$. Comme $\N(\Rcal)$ est infini, il existe $p'\in \N(\Rcal)$ tel que 
$p'\geq p$. Il s'ensuit que $x\in\sing(p,q)\subset \sing(p',q)=\sing(p_o,q)$. $\Box$

\medskip

Considérons  le sous-ensemble $\Ecal_{p,q}^{\mathrm{min}}\subset \Ecal_{p,q}$ associé au système minimal d'inégalités décrivant le cône convexe $\sing(p,q)$. 
Nous caractérisons le cône $\sing(\infty,q)$ de la manière suivante.

\begin{prop}\label{prop:inequality-infty-q}
L'identité $\sing(\infty,q)=\sing(p,q)$ est vraie si et seulement si toutes les triplets de $\Ecal_{p,q}^{\mathrm{min}}$ sont réguliers.
\end{prop}

\begin{proof}
Soit $p\geq q$ tel que $\sing(\infty,q)=\sing(p,q)$. Cela signifie que $\Ecal_{p',q}^{\mathrm{min}}=\Ecal_{p,q}^{\mathrm{min}}$ pour tout $p'>p$. 
Soit $(I_\bullet,J_\bullet,L_\bullet) \in \Ecal_{p,q}^{\mathrm{min}}$. Le corollaire \ref{coro:inequality-p-q} nous dit que 
$$
(C2)\qquad |I_+| + |J_+| + | L_+| - (|I_-| + |J_-| + L_-|)+\tfrac{r(r+1)}{2}=(p+q+1)(\sharp I_+ + \sharp J_+ + \sharp L_+ -r).
$$
Rappelons que  $(I_\bullet,J_\bullet,L_\bullet)$ est {\em régulier} si et seulement si $\sharp I_+ + \sharp J_+ + \sharp L_+ -r=0$. Par conséquent, si 
$(I_\bullet,J_\bullet,L_\bullet)$ n'est pas régulier, l'identité {\em (C2)} ne vaut plus lorsque l'on remplace $p$ par $p'>p$. Cela implique que
 $(I_\bullet,J_\bullet,L_\bullet)\notin \Ecal_{p',q}$ pour tout $p'>p$. Cela contredit le fait que $(I_\bullet,J_\bullet,L_\bullet)\in\Ecal_{p',q}^{\mathrm{min}}$.

Supposons maintenant que tous les éléments de $\Ecal_{p,q}^{\mathrm{min}}$ soient réguliers. Pour prouver que $\sing(\infty,q)$ coïncide avec 
$\sing(p,q)$, il suffit de montrer que $\Ecal_{p,q}^{\mathrm{min}}\subset \Ecal_{p',q}$ pour tout $p'\geq p$. 
Soit $(I_\bullet,J_\bullet,L_\bullet)\in \Ecal_{p,q}^{\mathrm{min}}$. Par définition, 
nous avons 
\begin{itemize}
\item[1.] $\sharp I_\bullet=\sharp J_\bullet=\sharp L_\bullet=r\in [q]$,
\item[2.] $\sharp I_+ +\sharp J_+ +\sharp L_+=r$,
\item[3.]$\Lambda^{p}:=(\mu((I^p_\bullet)^o),\mu((J^p_\bullet)^o),\mu(L^p_\bullet))\in\horn(r)$,
\item[4.] $\widetilde{\Lambda}^p:=(\mu((\widetilde{I}^p_\bullet)^o),\mu((\widetilde{J}^p_\bullet)^o),\mu(\widetilde{L}^p_\bullet))\in\horn(r)$.
\end{itemize}
L'identité {\em 2.} découle de notre hypothèse selon laquelle tous les éléments de $\Ecal_{p,q}^{\mathrm{min}}$ sont réguliers.

Notre objectif est atteint si nous montrons que $\Lambda^{p'}$ et $\widetilde{\Lambda}^{p'}$ appartiennent à  $\horn(r)$ pour tout $p'>p$.
Nous avons vu au lemme \ref{lem:lambda-X} que pour tout sous-ensemble polarisé $X_\bullet\subset [q]$ de cardinal $r\geq 1$, nous avons 
$\mu(X^{p'}_\bullet)-\mu(X^p_\bullet)= \mu(\widetilde{X}^{p'}_\bullet)-\mu(\widetilde{X}^p_\bullet)= (p'-p) 1_{\sharp X_-}^r$, $\forall p'\geq p$.
 
Cela conduit aux relations
$$
\Lambda^{p'}-\Lambda^{p}=\widetilde{\Lambda}^{p'}-\widetilde{\Lambda}^p=
(p'-p)\left(1^r_{\sharp I_+},1^r_{\sharp J_+},1^r_{\sharp L_+}-1^{r}\right),\qquad \forall p'\geq p.
$$
Il est maintenant facile de vérifier que $(1^r_{\sharp I_+},1^r_{\sharp J_+},1^r_{\sharp L_+}-1^{r})\in \horn(r)$ 
puisque $\sharp I_+ +\sharp J_+ +\sharp L_+=r$. En utilisant le fait que $\horn(r)$ est un cône convexe, nous pouvons conclure 
que $\Lambda^{p'}=\Lambda^p+(p'-p)(1^r_{\sharp I_+},1^r_{\sharp J_+},1^r_{\sharp L_+}-1^{r})$ et  
$\widetilde{\Lambda}^{p'}=\widetilde{\Lambda}^p+(p'-p)(1^r_{\sharp I_+},1^r_{\sharp J_+},1^r_{\sharp L_+}-1^{r})$ 
appartiennent à $\horn(r)$ pour tout $p'\geq p$. 
\end{proof}

\section{$\sing(p,1)$, $\sing(p,2)$ et $\sing(p,3)$}\label{sec:example-horn-1-2-3}

Dans les exemples suivants, la notation $X_\bullet=\{i_1^\epsilon,\ldots, i_r^\epsilon\}$ signifie que $X_+=\{i_k^\epsilon, \epsilon=+\}$ et \break 
$X_-=\{i_k^\epsilon, \epsilon=-\}$.
On utilise dans cette section la notation $\delta_{X_\bullet}:= \sharp X_{+}< X_{-}$.

\begin{exemple}\label{ex:singular-p-1}
Soit $p\geq q=1$. Ainsi, $(a,b,c)\in(\R_+)^3$ appartient à $\sing(p,1)$ si et seulement si les inégalités de Weyl sont vérifiées :
$a+b\geq c$, $a+c\geq b$, et $b+c\geq a$.
\end{exemple}

\begin{proof} Soient $I_\bullet,J_\bullet,L_\bullet\subset [1]$ satisfaisant les relations  du théorème  \ref{theo:singular-p-q}. Si nous utilisons la relation {\em (C3)} du corollaire 
\ref{coro:inequality-p-q}, nous obtenons que $(\sharp I_+)^2 + (\sharp J_+)^2 + (\sharp L_+)^2=1$. À une permutation près, nous devons avoir $I_\bullet=J_\bullet=\{1^-\}$ et
$L_\bullet=\{1^+\}$. Ce cas correspond à l'inégalité de Weyl $a+b\geq c$. 
\end{proof}

\medskip

\begin{exemple}\label{ex:singular-p-2}
Soit $p\geq q=2$. Alors, $(a,b,c)\in(\R^2_{++})^3$ appartient à $\sing(p,2)$ si et seulement si les 18 inégalités suivantes sont vérifiées
\begin{enumerate}
\item les inégalités de Weyl
\begin{itemize}
\item $a_1+ b_1\geq  c_1$ \quad (et 2 permutations),
\item $a_1+ b_2\geq  c_2$ \quad (et 5 permutations),
\end{itemize}
\item les inégalités de Lidskii-Wielandt
\begin{itemize}
\item $a_1+a_2+b_1+b_2\geq c_1+c_2$\quad (et 2 permutations),
\item $a_1 + a_2 + b_1 + c_2 \geq b_2 + c_1 $\quad (et 5 permutations).
\end{itemize}
\end{enumerate}
\end{exemple}

\medskip

\begin{proof} Soient $I_\bullet,J_\bullet,L_\bullet\subset [2]$ satisfaisant les relations  du théorème  \ref{theo:singular-p-q}. 
Soit $r:=\sharp I_\bullet=\sharp J_\bullet=\sharp L_\bullet$.

Premier cas : $r=1$.  La relation {\em (C3)} du corollaire \ref{coro:inequality-p-q} donne $(\sharp I_+)^2 + (\sharp J_+)^2 + (\sharp L_+)^2=1$. À une permutation près, 
nous avons $I_{\bullet}=\{i^-\}$, $J_{\bullet}=\{j^-\}$, $L_{\bullet}=\{L^+\}$ avec $i,j,k\in\{1,2\}$.   La relation {\em (C2)} du corollaire \ref{coro:inequality-p-q} 
impose que $k=i+j-1$, donc le triplet $(i,j,k)$ appartient à $(1,1,1)$, $(1,2,2)$ et $(2,1,2)$. Tous ces cas correspondent aux inégalités de Weyl : $a_1+ b_1\geq  c_1$, $a_1+ b_2\geq  c_2$, et $a_2+ b_1\geq  c_2$.

Deuxième cas : $r=2$. Ici, la relation {\em (C3)} devient  
$$
(\sharp I_+)^2 + (\sharp J_+)^2 + (\sharp L_+)^2 + 2(\delta_{I_\bullet}+\delta_{J_\bullet}+\delta_{L_\bullet})=4.
$$
À une permutation près, on peut supposer que $\sharp I_+\leq \sharp J_+\leq\sharp L_+$. Il existe deux possibilités :
\begin{itemize}
\item $(\sharp I_+, \sharp J_+,\sharp L_+)=(0,0,2)$. Ainsi, $I_\bullet=J_\bullet=\{1^-,2^-\}$ et $L_\bullet=\{1^+,2^+\}$. Ce cas correspond à l'inégalité de Lidskii-Wielandt 
$a_1+a_2+b_1+b_2\geq c_1+c_2$.
\item $(\sharp I_+, \sharp J_+,\sharp L_+)=(0,1,1)$. À une permutation près, nous avons $\delta_{L_\bullet}=1$ et $\delta_{J_\bullet}=0$. Ici, $I_\bullet=\{1^-,2^-\}$,  
$J_\bullet=\{1^-,2^+\}$ et $L_\bullet=\{1^+,2^-\}$. Cette situation correspond à l'inégalité de Lidskii-Wielandt $a_1+a_2+b_1-b_2\geq c_1-c_2$. 
\end{itemize}
\end{proof}

\medskip

\begin{exemple}\label{ex:singular-3-3}
Un élément $(a,b,c)\in(\R^3_{++})^3$ appartient à $\sing(3,3)$ si et seulement si les 87 inégalités suivantes sont vérifiées :
\begin{enumerate}
\item les inégalités de Weyl
\begin{itemize}
\item $a_1+ b_1\geq  c_1$ \quad (et 2 permutations),
\item $a_1+ b_2\geq  c_2$ \quad (et 5 permutations),
\item $a_2 + b_2\geq c_3$ \quad (et 2 permutations),
\item $a_1 + b_3\geq c_3$ \quad (et 5 permutations).
\end{itemize}
\item les inégalités de Lidskii-Wielandt
\begin{itemize}
\item $a_1 + a_2 + b_1 + b_2\geq c_1+c_2$ \quad (et 2 permutations),
\item $a_1 + a_2 + b_1 + b_3\geq c_1+c_3$ \quad (et 5 permutations),
\item $a_1 + a_2 + b_2 + b_3\geq c_2+c_3$ \quad (et 5 permutations),
\item $a_1 + a_2 + a_3 + b_1 + b_2 + b_3 \geq c_1 + c_2 + c_3$ \quad (et 2 permutations).
\item $a_1 + a_2 + b_1 + c_2 \geq b_2+ c_1 $ \quad (et 5 permutations),
\item $a_1 + a_2 + b_1 + c_3 \geq b_3+ c_1 $ \quad (et 5 permutations),
\item $a_1 + a_2 + b_2 + c_3 \geq b_3+ c_2 $ \quad (et 5 permutations),
\item $a_1 + a_2 + a_3 + b_1 + b_2 +c_3  \geq   b_3 + c_1 + c_2 $ \quad (et 5 permutations),
\item $a_1 + a_2 + a_3 + b_1  + b_3 + c_2 \geq  b_2 + c_1 + c_3$ \quad (et 5 permutations),
\item $a_1 + a_2 + a_3  + b_2 + b_3 + c_1 \geq  b_1 + c_2 + c_3$ \quad (et 5 permutations).
\end{itemize}
\item autres inégalités
\begin{itemize}
\item $a_1 + a_3 + b_1 + b_3\geq c_2+c_3$ \quad (et 2 permutations),
\item $a_1 + a_3 + b_1 + c_3\geq b_3+ c_2$ \quad (et 5 permutations),
\item $a_1 + a_2  + b_1  + b_3 + c_2 + c_3\geq a_3 + b_2 + c_1$ \quad (et 5 permutations).
\end{itemize}
\end{enumerate}
\end{exemple}

\begin{proof}
Soient $I_\bullet,J_\bullet,L_\bullet\subset [3]$ satisfaisant les relations  du théorème \ref{theo:singular-p-q}. Soit $r:=\sharp I_\bullet=\sharp J_\bullet=\sharp L_\bullet$.

\underline{Premier cas : $r=1$.}

Nous obtenons ici les inégalités de Weyl $a_i+b_j\geq c_{i+j-1}$ (à permutations près).

\underline{Deuxième cas : $r=2$.}

Utilisons la relation
$(\sharp I_+)^2 + (\sharp J_+)^2 + (\sharp L_+)^2 + 2(\delta_{I_\bullet}+\delta_{J_\bullet}+\delta_{L_\bullet})=4$.
À une permutation près, nous pouvons supposer que $\sharp I_+\leq \sharp J_+\leq\sharp L_+$. Il existe deux possibilités :
\begin{itemize}
\item $(\sharp I_+, \sharp J_+,\sharp L_+)=(0,0,2)$. Ici, nous avons $I_+=J_+=L_-=\emptyset$. La relation {\em (C3)} du corollaire \ref{coro:inequality-p-q} donne 
$|I_-|+|J_-|=|L_+|+3$. Sachant que $|I_-|,|J_-|, |L_+|\in\{3,4,5\}$, nous obtenons quelques cas liés aux inégalités de Lidskii-Wielandt:
\begin{enumerate}
\item $(|I_-|,|J_-|, |L_+|)=(3,3,3)$ :  $a_1 + a_2 + b_1 + b_2\geq c_1+c_2$,
\item $(|I_-|,|J_-|, |L_+|)=(3,4,4)\ ou\  (4,3,4)$ : $a_1 + a_2 + b_1 + b_3\geq c_1+c_3$ ou $a_1 + a_3 + b_1 + b_2\geq c_1+c_3$,
\item $(|I_-|,|J_-|, |L_+|)=(3,5,5)\ ou \ (5,3,5)$ : $a_1 + a_2 + b_2 + b_3\geq c_2+c_3$ ou $a_2 + a_3 + b_1 + b_2\geq c_2+c_3$.
\end{enumerate}
Le seul cas restant est $(|I_-|,|J_-|, |L_+|)=(4,4,5)$. Ici, $I_\bullet=J_\bullet=\{1^-,3^-\}$ et $L_\bullet=\{2^+,3^+\}$. Un petit calcul montre 
que les conditions de Horn sont satisfaites dans ce cas. Nous obtenons l'inégalité $a_1 + a_3 + b_1 + b_3\geq c_2+c_3$.

\item $(\sharp I_+, \sharp J_+,\sharp L_+)=(0,1,1)$. À une permutation près, nous pouvons supposer que $\delta_{L_\bullet}=1$ et $\delta_{J_\bullet}=0$. Ainsi, 
$I_\bullet=\{i_1^-<i_2^-\}$, $I_\bullet=\{j_1^-<j_2^+\}$ et $L_\bullet=\{L_1^+<L_2^-\}$. La relation {\em (C2)} du corollaire \ref{coro:inequality-p-q} donne 
$\underbrace{i_1^- + i_2^-}_\alpha + \underbrace{j_1^- + L_1^-}_\beta  =\underbrace{j_2^+ + L_1^+}_\gamma  + 3$.

Sachant que $\alpha,\beta,\gamma\in\{3,4,5\}$, nous obtenons 
\begin{enumerate}
\item $(\alpha,\beta,\gamma)= (3,3,3)$, $(3,4,4)$, ou $(3,5,5)$. Cela nous donne les inégalités de Lidskii-Wielandt.
\item $(\alpha,\beta,\gamma)= (4,3,4)$, ou $(5,3,5)$. Aucune solution.
\item $(\alpha,\beta,\gamma)= (4,4,5)$. Ici, $I_\bullet=\{1^-, 3^-\}$, $I_\bullet=\{1^-, 3^+\}$ et $L_\bullet=\{2^+,3^-\}$. Un petit calcul 
montre que les conditions de Horn sont satisfaites dans ce cas. Cela correspond à l'inégalité $a_1 + a_3 + b_1 - b_3\geq c_2 - c_3$.

\end{enumerate}
\end{itemize}

\medskip

\underline{Dernier cas : $r=3$.}

Commençons par la relation {\em (C3)} :
$(\sharp I_+)^2 + (\sharp J_+)^2 + (\sharp L_+)^2 + 2(\delta_{I_\pm}+\delta_{J_\pm}+\delta_{L_\pm})=9$.
À permutation près, on peut supposer que $\sharp I_+\leq \sharp J_+\leq\sharp L_+$. Tout d'abord, listons le triplet $(\sharp I_+\leq \sharp J_+\leq\sharp L_+)$
 qui ne peut pas satisfaire {\em (C3)} : $(0,0,0)$ ; $(0,0,1)$ ; $(0,1,1)$ ; $(0,0,2)$ ; $(1,1,2)$ ; $(0,2,2)$ ; $(2,2,2)$. Il reste quatre cas :
\begin{enumerate}
\item $(\sharp I_+,\sharp J_+ ,\sharp L_+)=(0,0,3)$. Ici, $I_\bullet=J_\bullet= \{1^-, 2^-, 3^-\}$ et $L_\bullet= \{1^+, 2^+, 3^+\}$. Nous obtenons l' inégalité de Lidskii-Wielandt
$a_1 + a_2 + a_3 + b_1 + b_2 + b_3\geq c_1+ c_2 + c_3$.
\item $(\sharp I_+,\sharp J_+ ,\sharp L_+)=(0,1,2)$. Nous devons avoir $\delta_{J_\bullet}+\delta_{L_\bullet}=2$, d'où trois possibilités
\begin{itemize}
\item[a)] $(\delta_{J_\bullet},\delta_{L_\bullet})=(0,2)$ : ici $I_\bullet= \{1^-, 2^-, 3^-\}$, $J_\bullet= \{1^-, 2^-, 3^+\}$ et $L_\bullet= \{1^+, 2^+, 3^-\}$. Nous obtenons l'inégalité signée 
de Lidskii-Wielandt $a_1 + a_2 + a_3 + b_1 + b_2 - b_3\geq c_1+ c_2 - c_3$.
\item[b)] $(\delta_{J_\bullet},\delta_{L_\bullet})=(1,1)$ : ici $I_\bullet= \{1^-, 2^-, 3^-\}$, $J_\bullet= \{1^-, 2^+, 3^-\}$ et $L_\bullet= \{1^+, 2^-, 3^+\}$. Nous obtenons l'inégalité signée 
de Lidskii-Wielandt $a_1 + a_2 + a_3 + b_1 - b_2 + b_3\geq c_1- c_2 + c_3$.
\item[c)] $(\delta_{J_\bullet},\delta_{L_\bullet})=(2,0)$ : ici, $I_\bullet= \{1^-, 2^-, 3^-\}$, $J_\bullet= \{1^+, 2^-, 3^-\}$ et $L_\bullet= \{1^-, 2^+, 3^+\}$. Nous obtenons l'inégalité signée 
de Lidskii-Wielandt $a_1 + a_2 + a_3 - b_1 + b_2 + b_3\geq - c_1+ c_2 + c_3$.
\end{itemize}
\item $(\sharp I_+,\sharp J_+ ,\sharp L_+)=(1,2,2)$. Nous devons avoir $\delta_{I_\bullet}=\delta_{J_\bullet}=\delta_{L_\bullet}=0$, donc $I_\bullet= \{1^-, 2^-, 3^+\}$, $J_\bullet= \{1^-, 2^+, 3^+\}$ et $L_\bullet= \{1^-, 2^+, 3^+\}$. Les partitions associées sont $\lambda(I^3_\bullet)=(1,0,0)$ et $\lambda(J^3_\bullet)=\lambda(L^3_\bullet)=(2,2,0)$. Comme $((1,0,0),(2,2,0),(-1,-1,-3))$ n'appartient pas à $\horn(3)$, ce triplet 
$(I_\bullet,J_\bullet,L_\bullet)$ ne satisfait pas les conditions du théorème \ref{theo:singular-p-q}.
\item $(\sharp I_+,\sharp J_+ ,\sharp L_+)=(1,1,1)$. Nous devons avoir $\delta_{I_\bullet}+\delta_{J_\bullet}+\delta_{L_\bullet}=3$. À permutation près, deux situations sont possibles
\begin{itemize}
\item[a)] $\delta_{I_\bullet}=\delta_{J_\bullet}=\delta_{L_\bullet}=1$ : ici $I_\bullet=J_\bullet= L_\bullet=\{1^-, 2^+, 3^-\}$. La partition correspondante est $\mu=(2,1,0)$, et on vérifie que 
$(\mu,\mu,\mu-3\cdot 1^3)\in \horn(3)$. Ainsi, ce triplet satisfait les conditions du théorème \ref{theo:singular-p-q}. L'inégalité correspondante, 
$$
a_1+a_3+b_1+b_3+c_1+c_3\geq a_2+b_2+c_2,
$$
est néanmoins redondante puisqu'elle découle de l'inégalité $x_1\geq x_2$ satisfaite par $a,b,c$.

\item[b)] $\delta_{I_\bullet}=0$, $\delta_{J_\bullet}=1$ et $\delta_{L_\bullet}=2$ :  ici $I_\bullet= \{1^-, 2^-, 3^+\}$, $J_\bullet= \{1^-, 2^+, 3^-\}$ et $L_\bullet= \{1^+, 2^-, 3^-\}$. Les partitions correspondantes sont 
$\lambda=\lambda(I^3_\bullet)=(1,0,0)$, $\mu=\lambda(J^3_\bullet)=(2,1,0)$ et $\nu=\lambda(L^3_\bullet)=(3,1,1)$. On vérifie  que 
$(\lambda,\mu,\nu-3\cdot 1^3)\in \horn(3)$, donc le triplet $(I_\bullet,J_\bullet,L_\bullet)$ satisfait les conditions du théorème \ref{theo:singular-p-q}. L'inégalité correspondante est $a_1+a_2+b_1+b_3+c_2+c_3\geq a_3+b_2+c_1$.

\end{itemize}

\end{enumerate}

\end{proof}

\begin{rem}
Le cas où $I_\bullet,J_\bullet, K_\bullet$ sont tous égaux à $A_\bullet=\{1^-, 2^+, 3^-\}$ donne l'exemple d'un triplet satisfaisant les conditions de Horn mais donnant une inégalité redondante. Ce phénomène de redondance s'explique par le fait que le triplet $(A_\bullet ,A_\bullet,A_\bullet)$ ne satisfait pas les conditions cohomologiques de la remarque \ref{rem:singular-p-q-version-2} : on a $[\Xgot_{A^3_\bullet}]\cdot [\Xgot_{A^3_\bullet}]\cdot [\Xgot_{A^3_\bullet}]=2 [pt]$  dans $H^*(\G(3,6),\Z)$ (voir \cite[exemple 4.2.1]{Berline-Vergne-Walter18} ).
\end{rem}

\begin{coro}
Nous avons $\sing(p,3)=\sing(3,3)$ pour tout $p\geq 3$.
\end{coro}

\begin{proof}Dans l'exemple \ref{ex:singular-3-3}, nous avons montré que $\sing(3,3)$ est décrit par des inégalités associées à des triplets {\em réguliers}. 
Grâce à la proposition \ref{prop:inequality-infty-q}, on peut conclure que $\sing(p,3)=\sing(3,3)$ pour tout $p\geq 3$.
\end{proof}

\section{Le cone convexe $\horn(\SO_{2q+1})$}\label{sec:horn-so-impair}

Rappelons la définition du cône de Horn du groupe $\SO_{2q+1}=\{g\in \SL_{2q+1}(\R), {}^tgg=Id\}$. À $x\in \R^{q}$, nous associons la matrice
$$
R_x:=\begin{pmatrix}
  \begin{pmatrix}
0 & -x_1 \\
x_1 & 0
\end{pmatrix}&  & & & \\
& \ddots & & &\\
& &
 \begin{pmatrix}
  0 & -x_q \\
  x_q & 0
\end{pmatrix} & \\
& & & 0\\
\end{pmatrix}\in\sogot_{2q+1}.
$$

Toute matrice $X\in \sogot_{2q+1}$ appartient à une orbite adjointe $\Ocal_x=\SO_{2q+1}\cdot\, R_x$ pour un unique $x\in \R^{q}_{++}$. 

\medskip

\begin{definition}
$\horn(\SO_{2q+1})$ est l'ensemble des triplets $(x,y,z)\in  (\R^{q}_{++})^3$ tels que $\Ocal_z\subset \Ocal_x+ \Ocal_y$.
\end{definition}

\medskip

Rappelons un résultat de Belkale-Kumar \cite{BK10} qui relie les cônes $\horn(\SO_{2q+1})$ et $\horn(2q+1)$. 
Soit $x\in\R^q\mapsto \widehat{x}=(x_1,\ldots,x_q,0,-x_q,\ldots,-x_1)\in\R^{2q+1}$. Notons que $\widehat{x}\in \R^{2q+1}_{+}$ si et seulement si 
$x\in\R^q_{++}$.

\medskip

\begin{theorem}[Belkale-Kumar]\label{theo:BK-SO-U} Pour tout $x,y,z\in \R^q_{++}$, nous avons
$$(x,y,z)\in \horn(\SO_{2q+1})\Longleftrightarrow  (\,\widehat{x}\,,\,\widehat{y}\,,\,\widehat{z}\ )\in \horn(2q+1).$$
\end{theorem} 

\begin{coro}
Les cônes $\horn(\SO_{2q+1})$ et $\sing(q+1,q)$ sont égaux.
\end{coro}

\begin{proof}
Cela découle du théorème \ref{theo:BK-SO-U}, et de  l'équivalence 
$$
(x,y,z)\in \sing(q+1,q)\Longleftrightarrow (\,\widehat{x},\,\widehat{y}\,,\,\widehat{z}\ )\in \horn(2q+1).
$$
démontrée à la proposition \ref{prop:sing-versus-horn}.
\end{proof}

La description des cônes $\sing(p,q)$ obtenue au théorème \ref{theo:singular-p-q} fournit une description de \break $\horn(\SO_{2q+1})$ 
qui ne fait pas intervenir le calcul de Schubert dans les grassmaniennes orthogonales. Rappelons que 
$\Bcal(r,q+1,q)$ désigne l'ensemble de tous les sous-ensembles $I\subset [2q+1]$ de cardinal $r$, vérifiant 
$I\cap I^o=\emptyset$ et $q+1\notin I$. Le résultat qui suit a été obtenu par Ressayre \cite{ressayre-ABC} par d'autres moyens.

\medskip

\begin{theorem}[Ressayre]\label{theo:horn-so-impair}
Un élément $(x,y,z)\in(\R^q_{++})^3$ appartient à $\horn(\SO_{2q+1})$ si et seulement si l'inégalité 
$$
 |\,x\,|_{I_{-}}+|\,y\,|_{J_{-}}  +|\, z\,|_{L_{-}}   \geq  |\,x\,|_{I_+}+|\,y\,|_{J_+} +|\, z\,|_{L_+}
$$
est satisfaite pour tout $r\in [q]$, et pour tout $I,J,L\in\Bcal(r,q+1,q)$ vérifiant les deux conditions

\medskip

\begin{itemize}\setlength{\itemsep}{8pt}
\item $(I^o,J^o,L)\in \LR_r^{2q+1}$,
\item $(I^o\natural I^{c},J^o\natural J^{c},L\natural L^{o,c})\in\LR_r^{2q+1-r}$.
\end{itemize}
\end{theorem}

\chapter{Valeurs propres versus valeurs singuli\`eres}

\section{Les c\^ones $\Acal(p,q)$}

Soit $p\geq q\geq 1$ et $n=p+q$. Nous considérons l'application $\pi_{p,q} : \herm(n)\to M_{p,q}(\C)$ qui associe à une matrice hermitienne 
$X$, son bloc hors diagonale $\pi_{p,q}(X)\in M_{p,q}(\C)$:
$$
X:=\begin{pmatrix}
 * & \pi_{p,q}(X)\\
 *& * 
\end{pmatrix}.
$$

Afin de comparer le spectre de $X$ avec le spectre singulier de $\pi_{p,q}(X)$, on considère le cône
$$
\Acal(p,q)=\Big\{(\e(X), \s(\pi_{p,q}(X))), \ X\in \herm(n)\Big\}\subset \R^n_+\times \R^q_{++}.
$$

Voici la description que nous obtenons du cône  $\Acal(p,q)$.

\begin{theorem}\label{theo:A-p-q} Soient $p\geq q\geq 1$ et $n=p+q$.
 
Un élément $(x,y)\in\R^n_+\times \R^q_{++}$ appartient à $\Acal(p,q)$ si et seulement si
$$
(\ast)_{I,J,L}\qquad\qquad\qquad |x|_I- | x |_{J^o}\geq 2\left(|y|_{L\cap[q]}-|y|_{L^o\cap[q]}\right)
$$
pour tout triplet  $(I,J,L)$ de sous-ensembles de $[n]$ satisfaisant les conditions suivantes :
\begin{enumerate}
\item[1)] $\sharp I= \sharp J= \sharp L=r \leq q$,
\item[2)] $I\cap J^o=\emptyset$, 
\item[3)] $L\cap L^o=\emptyset$ et $L\cap \{q+1,\ldots,p\}=\emptyset$,
\item[4)] $(I, J,L)\in \LR^n_r$,
\item[5)] $(I\natural J^{o,c}, J\natural I^{o,c}, L\natural L^{o,c})\in \LR^{n-r}_r$.
\end{enumerate}

Le résultat reste valable si l'on remplace  4) et 5) par les conditions plus restrictives suivantes :
\begin{enumerate}
\item[4')] $\cc^L_{I,J}=1$
\item[5')] $\cc^{L\natural L^{o,c}}_{I\natural J^{o,c}, J\natural I^{o,c}}=1$.
\end{enumerate}

\end{theorem}

\medskip

Il est intéressant de noter qu'une application directe du théorème d'O'Shea-Sjamaar donne une autre description du cône $\Acal(p,q)$ (voir \cite{pep-toshi}).

\begin{prop}\label{prop:A-p-q-OS} Soient $p\geq q\geq 1$ et $n=p+q$.
Un élément $(x,y)\in\R^n_+\times \R^q_{++}$ appartient à $\Acal(p,q)$ si et seulement si l'inégalité $(\ast)_{I,J,L}$ est vérifiée 
pour tout triplet  $(I,J,L)$ de sous-ensembles de $[n]$ satisfaisant les conditions suivantes :
\begin{enumerate}
\item[1')] $\sharp I= \sharp J= \sharp L=r\leq  \frac{n}{2}$,
\item[4)] $(I, J,L)\in \LR^n_r$.
\end{enumerate}
\end{prop}

Même si la description donnée dans la proposition \ref{prop:A-p-q-OS} est moins précise que celle du théorème \ref{theo:A-p-q}, elle permet de voir que 
le résultat du théorème \ref{theo:A-p-q} reste valable si l'on ne travaille qu'avec les quatre premières conditions.

\subsection{La description de Fomin-Fulton-Li-Poon}

 Lorsque $I,J,L$ sont des sous-ensembles de $[q]$, les sous-ensembles $I\natural J^{o,c}, J\natural I^{o,c}, L\natural L^{o,c}$ sont respectivement égaux à  
 $I$, $J$ et $L$. Dans ce cas, les conditions $I\cap J^o=\emptyset$ et $L\cap L^o=\emptyset$ sont automatiquement satisfaites, et les conditions 4) et 5)  
 du théorème \ref{theo:A-p-q} sont identiques.

Fomin-Fulton-Li-Poon ont montré, à l'aide d'une méthode ad hoc, le fait remarquable suivant : pour décrire $\Acal(p,q)$, il suffit de considérer les inégalités 
$(\ast)_{I,J,L}$ lorsque $I,J,L$ sont des sous-ensembles de $[q]$ \cite{FFLP}. 

\begin{theorem}[Fomin-Fulton-Li-Poon]\label{theo:FFLP} Un élément $(x,y)\in\R^n_+\times \R^q_{++}$ appartient à $\Acal(p,q)$ si et seulement si 
$$
|x|_I-|x|_{J^o}\geq 2|y|_{L}
$$
pour tout $r\leq q$ et tout $(I,J,L)\in \LR^q_r$.
\end{theorem}

Dans un travail à venir avec Nicolas Ressayre, nous aborderons la conjecture suivante.

\medskip

\textbf{Conjecture 2:} 

\begin{itemize}
\item Si $I,J,L$ sont des sous-ensembles de $[n]$ satisfaisant les conditions 1), 2), 3), 4') et 5') du théorème \ref{theo:A-p-q}, alors $I,J,L$ sont des sous-ensembles de $[q]$.
\item La liste d'inégalités $(\ast)_{I,J,L}$, avec $I,J,L\subset [q]$ de même cardinal et satisfaisant $\cc^L_{I,J}=1$, est une liste minimale pour décrire $\Acal(p,q)$.
\end{itemize}

\subsection{Preuve du theorème \ref{theo:A-p-q}}

Nous allons appliquer les résultats de la section \ref{sec:delta-U-sigma} à la situtation suivante: les groupes $U:=\upU_n\subset U_\C:=\GL_n(\C)$ 
sont munis de l'involution\footnote{Ici, la matrice $\mathbf{J}_{p,q}$ est définie par (\ref{eq:J-p-q}).}  $\sigma_{p,q}(g)=\mathbf{J}_{p,q}(g^*)^{-1}\mathbf{J}_{p,q}$. 
Notons $\pi_{-}:\ugot_n\to \ugot_n^{-\sigma_{p,q}}$ la projection associée à la décomposition $\ugot_n=\ugot_n^{\sigma_{p,q}}\oplus \ugot_n^{-\sigma_{p,q}}$.

Le groupe 
$K:=U^{\sigma_{p,q}}$ est formé des matrices de la forme
$$
\Ad(\theta)\left(\begin{pmatrix} g&  0\\ 0& h \end{pmatrix}\right),\qquad (g,h)\in \upU_p\times\upU_q,
$$
et le sous-espace $\ugot_n^{-\sigma_{p,q}}$ est formé des matrices de la forme
$$
\mathrm{j}(Y):=\Ad(\theta)\left(\begin{pmatrix} 0 & iY\\ iY^*& 0\end{pmatrix}\right),\qquad Y\in M_{p,q}(\C).
$$
La matrice orthogonale $\theta$ a été introduite à la section \ref{sec:drapeaux-2-etage}(voit (\ref{eq:theta})).

On reprend une notation utilisée à la section \ref{sec:horn-versus-sing}: à tout $y=(y_1,\ldots, y_q)\in \R^q$, nous associons 
le vecteur $\widehat{y}=(y_1,\ldots, y_q,0,\ldots,0,- y_q,\ldots,- y_1)\in\R^n$.

Le tore maximal $T\subset U$ des matrices diagonales est adapté à l'involution $\sigma_{p,q}$: $T$ est stable par rapport à $\sigma_{p,q}$ 
et le sous-espace $\tgot^{-\sigma_{p,q}}$, qui est formé des matrices $i\, \diag(\,\widehat{y}\,),y\in \R^q$, est de dimension maximale. 
On travaille avec la chambre de Weyl $\tgot_+:=\{i\,\diag(x), x\in\R^n_+\}$: ainsi 
$\tgot^{-\sigma_{p,q}}\cap \tgot_+=\{i\, \diag(\,\widehat{y}\,),y\in \R^q_{++}\}$ paramètre les $K$-orbites dans $\ugot_n^{-\sigma_{p,q}}$.

\begin{lem}\label{lem:equivalence-A-p-q}
Pour $(x,y)\in\R^n_+\times \R^q_{++}$, les conditions suivantes sont équivalentes
\begin{enumerate}
\item $(x,y)\in\Acal(p,q)$,
\item $\diag(i\ \widehat{y}\,)\ \in\  \pi_{-}\left(\upU_n\cdot \diag(ix)\right)$
\item $(i\,\diag(x),i\,\diag(\,\widehat{y}\,))\ \in\ \Delta(\upU_n,\sigma_{p,q})$.
\end{enumerate}
\end{lem}
\begin{proof} Un calcul direct montre que 
$$
\pi_{-}(iX)=\Ad(\theta)\Big(\mathrm{j}\left(\pi_{p,q}(\Ad(\theta)(X))\right)\Big),\qquad \forall X\in \herm(n).
$$ 
Cela permet de voir que la condition {\em 2.} est équivalente au fait que la matrice 
$$
N(y):=
\begin{pmatrix}
0&\cdots & y_1\\
\vdots &\reflectbox{$\ddots$}& \vdots\\
y_q&\cdots & 0\\
0&\cdots & 0\\
\vdots &\ddots& \vdots\\
0&\cdots & 0\\
\end{pmatrix}\in M_{p,q}(\C)
$$
appartienne à $\pi_{p,q}\left(\upU_n\cdot \diag(x)\right)$. Comme le spectre singulier de $N(y)$ est égal à $y$, on voit que $N(y)\in \pi_{p,q}\left(\upU_n\cdot \diag(x)\right)$ 
si et seulement si $(x,y)\in\Acal(p,q)$. L'équivalence {\em 1.}  $\Leftrightarrow$  {\em 2.} est démontrée et 
l'équivalence {\em 2.} $\Leftrightarrow$ {\em 3.} correspond à la définition du cône $\Delta(\upU_n,\sigma_{p,q})$ (voir section \ref{sec:delta-U-sigma}).
\end{proof}

\medskip

Proposons une première description du cône $\Acal(p,q)\simeq\Delta(\upU_n,\sigma_{p,q})$ qui utilise la proposition \ref{prop:horn-U-sigma}.

\medskip

Le groupe fixé par $\sigma_{p,q}$ est $G=\Ad(\theta)(U(p,q))$. La décomposition de Cartan de l'algèbre de Lie de $G$ est $\ggot=\kgot\oplus\pgot$ où 
$$
\begin{array}{rrl}
\kgot  & = & \Ad(\theta)\left(\left\{\begin{pmatrix} X&  0 \\ 0& Y \end{pmatrix},\ (X,Y)\in \ugot(p)\times\ugot(q)\right\}\right)\\
\pgot  & = & \Ad(\theta)\left(\left\{\begin{pmatrix} 0&  X \\ X^*& 0 \end{pmatrix},\ X\in M_{p,q}(\C)\right\}\right).
\end{array}
$$
Le tore maximal $T\subset \upU_n$, formé des matrices diagonales, est stable sous l'involutions $\sigma_{p,q}$, et 
$\agot=\frac{1}{i}\tgot^{-\sigma_{p,q}}=\{\diag(\,\widehat{y}\,), y\in\R^q\}$ est un sous-espace abélien maximal de $\pgot$.

Considérons l'involution $\sigma_+: \tgot_+\to\tgot_+$ définie par la relation  $-\sigma_{p,q}(\upU_n \xi)=\upU_n\sigma_+(\xi), \forall \xi\in \tgot_+$. 
Un calcul élémentaire montre que $\sigma_+(\diag(ix))=\diag(i x^\vee)$, où l'on note $x^\vee=(-x_n,\ldots,-x_1)$ pour tout $x=(x_1,\ldots,x_n)\in\R^n$. 

Le lemme \ref{lem:equivalence-A-p-q} et la proposition \ref{prop:horn-U-sigma} fournissent une première description de $\Acal(p,q)$.

\begin{prop}Pour tout $(x,y)\in  \R^n_+\times \R^q_{++}$, on a l'équivalence
$$
(x,y)\in\Acal(p,q)\quad \Longleftrightarrow \quad (x,x^\vee, 2\,\widehat{y}\,)\in\horn(n).
$$
\end{prop}

\begin{rem}
On remarque que le résultat précédent correspond à la proposition \ref{prop:A-p-q-OS}.
\end{rem}

Pour obtenir une description plus fine de $\Acal(p,q)\simeq\Delta(\upU_n,\sigma_{p,q})$, on va maintenant utiliser le théorème \ref{theo-delta-U-sigma}.
On a vu à la section \ref{sec:preuve-theo-sing-p-q}, les faits suivants:
\begin{itemize}\setlength{\itemsep}{8pt}
\item l'ensemble $\Sigma_\agot$ des racines restreintes comprend tous les applications linéaires 
$\pm f_i\pm f_j$ avec $i\neq j$ et $\pm 2 f_i$ pour tous les $i$. De même, les $\pm f_i$ sont des racines restreintes si $p\neq q$. 
Nous pouvons donc choisir la chambre de Weyl restreinte suivante $\agot_+:=\{\zeta(x), x_1\geq x_2\geq \cdots\geq x_q\geq 0\}$. 
\item Le groupe de Weyl restreint $W_{\agot}$ s'identifie à un sous-groupe du groupe de Weyl $W=\Sgot_n$:
$$
W_{\agot}= \Big\{w\in\Sgot_n, w(k^o)=w(k)^o, \forall k\in [q]\ \mathrm{et}\ w(k)=k, \forall k\in [q+1,\ldots,p]\Big\}.
$$
\item Les éléments admissibles sont de la forme $t\,w\, \zeta_r$ avec $(t,w,r)\in \Q^{>0}\times W_{\agot}\times [q]$,  
et où  $\zeta_r= \diag(\underbrace{-1,\ldots,-1}_{ r\ fois},0,\ldots,0, \underbrace{1,\ldots,1}_{ r\ fois})\in\agot$ 
(voir le lemme \ref{lem:admissible-u-p-q}).
\item La variété  de drapeaux $\GL_n(\C)/\Pbb(\zeta_r)$ s'identifie canoniquement à la variété $\mathbb{F}(r,n-r,n)$.
\end{itemize}

\medskip

Le sous-groupe $W'=Z_W(\agot)$ est à égal \`a $\{w\in \Sgot_n, \ w(k)=k,\forall k\notin\{q+1,\ldots,p\}\}$, et l'élément le plus long de $W'$ est représenté 
par la matrice de permutation
$$
w_0':=
\begin{pmatrix} I_q&  0& 0\\
0& J_{p-q}& 0\\
 0&  0& I_q
\end{pmatrix}
\qquad \mathrm{et} \qquad
J_{p-q}=\begin{pmatrix}
0&\cdots & 1\\
\vdots &\reflectbox{$\ddots$}& \vdots\\
1&\cdots & 0\\
\end{pmatrix}\in\GL_{p-q}(\C).
$$

D'après le théorème \ref{theo-delta-U-sigma}, un élément $(x,y)\in  \R^n_+\times \R^q_{++}$ appartient à $\Acal(p,q)$ si et seulement si
$$
(x,w\zeta_r)\geq (\,\widehat{y},w_0w_1\zeta_r)
$$
pour tout $r\in[q]$ et pour tout $(w,w_1)\in \times W/W^{\zeta_r}\times W_\agot/W_\agot^{\zeta_r}$ satisfaisant les conditions suivantes
\begin{enumerate}
\item[(a)] $[\Xgot_{w,\zeta_r}]\cdot [\Xgot_{w_0'\sigma_{p,q}(w),\zeta_r}]\cdot [\Xgot_{w_1,\zeta_r}]= [pt]\quad \mathrm{dans}\quad H^{max} (\Fcal_{\zeta_r},\Z)$.
\item[(b)] $2\tr(w\zeta_r \circlearrowright \ngot^{w\zeta_r>0})+\tr(w_1\zeta_r \circlearrowright \ngot^{w_1\zeta_r>0})= 2\tr(\zeta \circlearrowright \glgot_n(\C)^{\zeta_r>0})$.
\end{enumerate}

\`A travers l'identification $\GL_n(\C)/\Pbb(\zeta_r)\simeq\mathbb{F}(r,n-r,n)$, les classes $[\Xgot_{w,\zeta_r}]$ et $[\Xgot_{w_1,\zeta_r}]$ correspondent 
respectivement à $[\Xgot_{I\subset J}]$ et $[\Xgot_{K\subset L^{o,c}}]$. Ici $I=w([r])$, $J=w([n-r])$, et $K=w_1([r])$. On se sert ici du fait que 
$w_1\in W_\agot$ implique la relation $w_1([n-r])=L^{o,c}$.

L'élément $w_0'\sigma_{p,q}(w)\in W$ est égal à $w_0w w_{p,q}$ où $w_{p,q}$ est représenté par la matrice de permutations $\mathbf{J}_{p,q}$. Alors 
$w_0w w_{p,q}([r])=w_0w([n-r]^c)=J^{o,c}$ et $w_0w w_{p,q}([n-r])=w_0 w w_{p,q}([r]^{o,c})=w_0w([r]^{c})=I^{o,c}$. Finalement, la relation $(a)$ signifie que 
$$
[\Xgot_{J^{o,c}\subset I^{o,c}}]\cdot [\Xgot_{I\subset J}]\cdot [\Xgot_{L\subset L^{o,c}}]= [pt]\quad \mathrm{dans}\quad H^{max} (\mathbb{F}(r,n-r,n),\Z),
$$
tandis que $(b)$ signifie que le triplet $\left(J^{o,c}\subset I^{o,c},I\subset J, L\subset L^{o,c}\right)$ est Lévi-mobile. D'après la proposition 
\ref{prop:levi-mobile-double-grass}, on sait que les conditions $(a)$ et $(b)$ sont satisfaites si et seulement si les trois propriétés suivantes sont vraies:

\begin{enumerate}\setlength{\itemsep}{8pt}
\item $[\Xgot_{I^{o,c}}]\cdot [\Xgot_{J}]\cdot [\Xgot_{L^{o,c}}]= [pt]$,  dans $H^{max}(\G(n-r, n),\Z)$.
\item $[\Xgot_{J^{o,c}\natural I^{o,c}}]\cdot [\Xgot_{I\natural J}]\cdot [\Xgot_{L\natural L^{o,c}}]= [pt]$,  dans $H^{max}(\G(r, n-r),\Z)$.
\item $ |\mu(J^{o,c})|+  |\mu(I)| +|\mu(L)|= 2r(n-r)$.
\end{enumerate}

Le calcul des dimensions dans l'identité $1.$ donne $3.$, et de plus $1.$ est équivalent à 
$[\Xgot_{I}]\cdot [\Xgot_{J^{o,c}}]\cdot [\Xgot_{L}]= [pt]$,  dans $H^{max}(\G(r, n),\Z)$, c'est à dire $\cc_{I^o,J^c}^{L}=1$. De même, 
$2.$ signifie que $\cc_{J^{c}\natural I^{c},I^o\natural J^o}^{L\natural L^{o,c}}=1$

Un dernier calcul permet de voir que $(x,w\zeta_r)\geq (\,\widehat{y},w_0w_1\zeta_r)$ correspond à 
$$
|x|_{J^c}-|x|_{I}\geq 2\left( |y|_{L\cap [q]}-|y|_{L^o\cap [q]}\right).
$$
La preuve du théorème \ref{theo:A-p-q} est complète.

\section{Les cônes $\Bcal(n)$}

Soit $n\geq 1$. Nous considérons l'application $\pi_{n} : M_{n,n}(\C)\to \herm(n)$  définie par $\pi_n(X)=\frac{1}{2}(X+X^*)$. 
Afin de comparer le spectre singulier de $X$ avec le spectre de $\pi_n(X)$, on considère le cône
$$
\Bcal(n)=\Big\{(\s(X), \e(\pi_n(X))), \ X\in M_{n,n}(\C)\Big\}\subset \R^n_{++}\times \R^n_{+}.
$$

\begin{exemple}
$\Bcal(1)$ est l'ensemble des couples $(x,y)\in \R_+\times \R$ tels que $-x\leq y\leq x$.
\end{exemple}

Commençons par donner la description de $\Bcal(n)$ obtenue avec le théorème d'O'Shea-Sjamaar. 

\begin{prop}\label{prop:horn-B-n}Pour tout $(x,y)\in\R^n_{++}\times \R^n_{+}$, on a l'équivalence
$$
(x,y)\in \Bcal(n)\quad \Longleftrightarrow\quad (\,\widehat{x}\, ,y, y^\vee)\in\LR(n,n).
$$
\end{prop}

La description du cône $\LR(n,n)$ obtenue au théorème \ref{theo:LR-m-n}, permet de voir que $(x,y)\in \Bcal(n)$ si et seulement si les conditions suivantes sont satisfaites
\begin{enumerate}\setlength{\itemsep}{8pt}
\item[i)] $-x_{n+1-k}\leq y_k\leq x_k$, $\forall k\in [n]$,
\item[ii)] $|x|_{L\cap [n]}-|x|_{L^o\cap [n]}\geq |y|_I - |y|_{J^o}$, pour tout triplet de sous ensembles stricts $L\subset [2n]$, $I,J\subset [n]$, satisfaisant :
$\sharp L =\sharp I +\sharp J$, et $\cc^{L}_{I,J}=1$.
\end{enumerate}

\begin{rem}Les inégalités
$$
-\s_{n+1-k}(X)\leq\e_k\left(\frac{X+X^*}{2}\right)\leq \s_k(X),\quad \forall k\in [n], \quad \forall X\in M_{n,n}(\C)
$$
remontent aux travaux de Fan, Hoffman et Thompson \cite{Fan-Hoffman,Fan74,Thompson75-2}.
\end{rem}

\begin{exemple}
Grâce à la description de $\LR(2,2)$ obtenue dans l'exemple \ref{ex:LR-2-2}, on voit $\Bcal(2)$ est  l'ensemble des couples $(x,y)\in \R^2_{++}\times \R^2_+$ 
vérifiant
\begin{equation*}
\boxed{
\begin{array}{c}
x_1\geq y_1 \geq  -x_2\\
x_2\geq y_2 \geq  -x_1\\
x_1+x_2\geq  y_1-y_2
\end{array}
}
\end{equation*}

\end{exemple}

\medskip

Voici une description plus précise de $\Bcal(n)$ obtenue avec le théorème \ref{theo:main-reel}.

\begin{theorem}\label{theo:pep-B-n}Un élément  $(x,y)\in\R^n_{++}\times \R^n_{+}$ appartient à $\Bcal(n)$ si et seulement si 
\begin{enumerate}\setlength{\itemsep}{8pt}
\item $-x_{n+1-k}\leq y_k\leq x_k$, $\forall k\in [n]$,
\item $|x|_{L\cap [n]}-|x|_{L^o\cap [n]}\geq |y|_I - |y|_{J^o}$, pour tout triplet de sous ensembles stricts $L\subset [2n]$, $I,J\subset [n]$, satisfaisant :
\begin{enumerate}
\item $L\cap L^o=\emptyset$\quad  et\quad  $I\cap J^o=\emptyset$,
\item $\sharp I +\sharp J=\sharp L \leq n$,
\item $\cc^{L}_{I,J}=1$,
\item $\cc^{L\natural L^{o,c}}_{I\natural J^{o,c},J\natural I^{o,c}}=1$.
\end{enumerate}
\end{enumerate}
\end{theorem}

\begin{rem}
Le théorème précédent est encore vrai avec les conditions plus faibles : $\cc^{L}_{I,J}\neq 0$ et $\cc^{L\natural L^{o,c}}_{I\natural J^{o,c},J\natural I^{o,c}}\neq 0$. 
La proposition \ref{prop:horn-B-n} permet de voir que le résultat du théorème \ref{theo:pep-B-n} est encore valable 
si on ne considère pas la condition {\em 2.}  {\em d)}.
\end{rem}

\subsection{Preuve de la proposition \ref{prop:horn-B-n}}

Fixons quelques matrices avec lesquelles on va travailler:
$$
\Upsilon_n=\begin{pmatrix} 0&  I_n\\ I_n& 0\end{pmatrix},\quad 
\Omega_n:=\frac{1}{\sqrt{2}}\begin{pmatrix} I_n& I_n\\ -I_n& I_n\end{pmatrix}, \quad
I_{n,n}:=\begin{pmatrix} -I_n&  0\\ 0& I_n\end{pmatrix},\quad 
\Lambda_n:=\begin{pmatrix} I_n&  0\\ 0& J_n\end{pmatrix}, \quad
\Theta_n:=\begin{pmatrix} 0&  J_n\\ J_n& 0\end{pmatrix}.
$$

Considérons le groupe $\tU_\C:=\GL_{2n}(\C)$ muni de l'involution $\tilde{\sigma}(g)=\Upsilon_n (g^*)^{-1}\Upsilon_n$. 
Notons $\tG\subset \GL_{2n}(\C)$ le sous-groupe fixé par $\tilde{\sigma}$, et 
$\tK:=\tG\cap\upU_{2n}$ son sous-groupe compact maximal.  Soit $\tggot=\tkgot\oplus\tpgot$ la décomposition de Cartan de l'algèbre de Lie 
de $\tG$. Comme $\Upsilon_n =\Omega_n I_{n,n}\Omega_n^{-1}$, on voit 
que $\tG=\Ad(\Omega_n)\left(U(n,n)\right)$. Ainsi  
$$
\begin{array}{rrl}
\tkgot  & = & \Ad(\Omega_n)\left(\left\{\begin{pmatrix} X&  0 \\ 0& Y \end{pmatrix},\ (X,Y)\in \ugot_n\times\ugot_n\right\}\right),\\
\tpgot  & = & \Ad(\Omega_n)\left(\left\{\begin{pmatrix} 0&  X \\ X^*& 0 \end{pmatrix},\ X\in M_{n,n}(\C)\right\}\right)=
\left\{\begin{pmatrix} \frac{X+X^*}{2}&  \frac{X-X^*}{2} \\ \frac{-X+X^*}{2}& \frac{-X-X^*}{2} \end{pmatrix},\ X\in M_{n,n}(\C)\right\}.
\end{array}
$$

Le tore maximal $\tT\subset \upU_{2n}$ des matrices diagonales est adapté à $\tilde{\sigma}$: $\ttgot$ est stable pour $\tilde{\sigma}$, et 
$\tagot:=\tfrac{1}{i}\ttgot^{-\tilde{\sigma}}:=\left\{\diag(x,-x),\ x\in\R^n\right\}$ est un sous-espace abélien maximal de $\tpgot$. On choisit comme chambre de Weyl 
de $\ttgot$ le cône 
$$
\ttgot_+:=\Ad(\Lambda_n)\left(\left\{i\diag(z),\ z\in\R^{2n}_+\right\}\right),
$$
de telle manière que $\tagot_+:=\tfrac{1}{i}\left(\ttgot^{-\tilde{\sigma}}\cap\ttgot_+\right)=\{\diag(x,-x),\ x\in\R^n_{++}\}$ paramètre les $\tK$-orbites dans $\tpgot$.

\medskip

Soient $U_\C:=\GL_{n}(\C)\times \GL_{n}(\C)$ et le morphisme $\iota : U_\C\to\tU_\C$ défini par 
$\iota(g,h):=\begin{pmatrix} g&  0 \\ 0& h \end{pmatrix}$. La projection $\pi: \glgot_{2n}(\C)\to \glgot_{n}(\C)\times \glgot_{n}(\C)$ 
correspondante envoie $\begin{pmatrix} X&  Y \\ Z& T \end{pmatrix}$ sur $(X,T)$.

On munit $U_\C$ de l'involution 
$\sigma(g,h)=((h^*)^{-1},(g^*)^{-1})$, de telle manière que $\iota\circ \sigma = \tilde{\sigma}\circ \iota$. Le groupe 
$G:=\{(g, (g^*)^{-1}),\ g\in \GL_{n}(\C)\}$ est le sous-groupe fixé par $\sigma$. Son sous-groupe compact maximal est 
$K:= \{(k, k),\ k\in \upU_n)\}$, sa décomposition de Cartan est $\ggot=\kgot\oplus\pgot$ avec $\pgot:=\{(Y,-Y),\ Y\in \herm(n))\}$. 
Le tore maximal $T\subset\upU_n\times \upU_n$ des matrices diagonales est adapté à $\sigma$: $\tgot$ est stable pour $\sigma$, et 
$\agot:=\tfrac{1}{i}\tgot^{-\sigma}:=\left\{(\diag(y),\diag(-y)),\ y\in\R^n\right\}$ est un sous-espace abélien maximal de $\pgot$. 
On choisit comme chambre de Weyl de $\tgot$ le cône  $\tgot_+:=\left\{i\diag(x),\ z\in\R^{n}_+\right\}\times \left\{-i\diag(x),\ z\in\R^{n}_+\right\}$
de telle manière que $\tagot_+:=\tfrac{1}{i}\left(\ttgot^{-\tilde{\sigma}}\cap\ttgot_+\right)=\{(\diag(y),-\diag(y)),\ x\in\R^n_{+}\}$ paramètre les $\tK$-orbites dans $\tpgot$.

Le résultat suivant démontre la proposition \ref{prop:horn-B-n}.

\begin{lem}Pour tout $(x,y)\in\R^n_{++}\times \R^n_{+}$, les relations suivantes sont équivalentes
\begin{enumerate}
\item $(x,y)\in \Bcal(n)$,
\item $K\cdot (\diag(y),\diag(-y))\subset \pi\left(\tK\cdot \diag(x,-x)\right)$,
\item $\upU_n\times\upU_n\cdot \,(\diag(y),\diag(-y))\subset \pi\left(\upU_{2n}\cdot \diag(x,-x)\right)$,
\item $(\,\widehat{x}\, ,y, y^\vee)\in\LR(n,n)$.
\end{enumerate}
\end{lem}

\begin{proof} L'équivalence $1. \Longleftrightarrow 2.$ vient du fait que l'orbite $\tK\cdot \diag(x,-x)$ est égal à l'ensemble des matrices 
$$
\begin{pmatrix} \frac{X+X^*}{2}&  \frac{X-X^*}{2} \\ \frac{-X+X^*}{2}& \frac{-X-X^*}{2} \end{pmatrix}
$$
où $X\in M_{n,n}(\C)$ vérifie $\s(X)=x$. L'équivalence $2. \Longleftrightarrow 3.$ découle de la proposition \ref{prop:OSS-orbites}. 
Finalement, l'équivalence $3. \Longleftrightarrow 4.$ découle du fait que l'orbite $\upU_{2n}\cdot \diag(x,-x)$ correspond à l'ensemble 
des matrices $Z\in\herm(2n)$ telles que $\e(Z)=\widehat{x}$.
\end{proof}

\subsection{Preuve du théorème \ref{theo:pep-B-n}}

Le groupe de Weyl de $(\tU,\tT)$ est $\tW=\Sgot_{2n}$. Notons $\upsilon_n\in\Sgot_{2n}$, l'élément d'ordre 2 associée à la matrice $\Upsilon_n$. 
Alors, le groupe de Weyl restreint s'identifie au sous-groupe centralisateur $\tW_{\tagot}:=\{\tw \in \Sgot_{2n}, \tw\upsilon_n=\upsilon_n \tw\}$. On voit que 
$\tW_{\tagot}$ est isomorphe au produit semi-direct $\Sgot_n\ltimes \{\pm 1\}^n$ entre le sous-groupe $\{(w,w),w\in\Sgot_n\}$ et le sous-groupe engendré 
par les transpositions $(i,i+n), i\in [n]$. Le groupe de Weyl de $(U,T)$ est $W=\Sgot_{n}\times \Sgot_n$ et le groupe de Weyl restreint 
 est égal à $W_{\agot}:=\{(w,w),w\in\Sgot_n\}$

Le quotient $\tggot_\C/\ggot_\C$ est isomorphe, en tant que $\ggot$-module à l'ensemble des matrices 
$\begin{pmatrix} 0&  X \\Y&0 \end{pmatrix}, X,Y\in M_{n,n}(\C)$. Cela permet de voir que l'ensemble des racines relatives à l'action de $\agot$ sur 
$\tggot/\ggot$ est égal à $\Sigma(\tggot/\ggot)=\{\pm(e^*_i+e^*_j),\ i,j\in [n]\}$. Comme $\Sigma(\tggot/\ggot)$ engendre $\agot^*$, un élément rationnel
$\zeta\in\agot$ est admissible si $\vect(\Sigma(\tggot/\ggot)\cap \zeta^\perp)=\zeta^\perp$.

\begin{lem}
Les éléments admissibles $\zeta\in\agot$ sont égaux, modulo l'action de $\Q^{>0}\times W_{\agot}$, à l'un des éléments suivants:
\begin{itemize}
\item $\zeta_1=\diag(v_1,-v_1)$ où $v_1=(1,0,\ldots,0)$.
\item $-\zeta_1$.
\item $\zeta_{r,s}=\diag(v_{r,s},-v_{r,s})$ où $v_{r,s}=(\underbrace{-1,\ldots,-1}_{r\ termes},\underbrace{0,\ldots,0}_{n-r-s\ termes},\underbrace{1,\ldots,1}_{s\ termes})$, avec 
$r,s\geq 1$ et $r+s\leq n$.
\end{itemize}
\end{lem}
\begin{proof}Pour toute partie $A\subset [n]$, posons $1_A:=\sum_{i\in A} e_i$. Un vecteur non-nul $\zeta\in\agot$ s'écrit 
$\zeta=\sum_{\alpha>0} \alpha (1_{A_\alpha}-1_{B_\alpha})$ où les sous-ensembles $A_\alpha, B_\alpha$ sont disjoints et 
$A_\alpha\cup B_\alpha\neq\emptyset$. Soit $C\subset [n]$ le complémentaire de $\bigcup_{\alpha>0}A_\alpha\cup B_\alpha$. 
On voit que $\Sigma(\tggot/\ggot)\cap \zeta^\perp$ est égal à la réunion disjointe 
$$
\underbrace{\big\{\pm e_i,\ i\in C\big\}}_{\Sigma_0}\bigcup\bigcup_{\alpha>0} \underbrace{\big\{\pm(e_i+e_j), (i,j)\in A_\alpha\times B_\alpha\big\}}_{\Sigma_\alpha}.
$$
On remarque que $\dim \vect(\Sigma_0)= \sharp C$ tandis que 
$$
\dim \vect(\Sigma_\alpha) =
\begin{cases}
0\quad \mathrm{si}\quad A_\alpha=\emptyset\quad \mathrm{ou} \quad B_\alpha=\emptyset\\
\sharp A_\alpha+\sharp B_\alpha - 1 \quad \mathrm{sinon}.
\end{cases}
$$
On voit ainsi que $\Sigma(\tggot/\ggot)\cap \zeta^\perp$ engendre un sous-espace de codimension $1$ seulement dans deux situations: soit $\sharp C= n-1$ ou bien 
$\zeta=\alpha (1_{A_\alpha}-1_{B_\alpha})$ avec $A_\alpha\neq\emptyset$ et $B_\alpha\neq\emptyset$. Notre lemme est démontré.
\end{proof}

\subsubsection{Inégalités associée à $\pm\zeta_1$}

Dans la suite, pour tout $I\subset [n]$, on note $\C^I\subset \C^n$ le sous-espace vectoriel engendré par $e_i,i\in I$.
L'application $(g,h)\mapsto (g(\C^{[n]-[1]}),h(\C^{[1]}))$ permet  d'identifier le quotient $\GL_n(\C)\times\GL_n(\C)/P(\zeta_1)$ avec 
$\G(n-1,n)\times \G(1,n)$. De manière similaire, $k\in \GL_{2n}(\C)\mapsto (k(\C^{\{n+1\}})\subset k(\C^{[2n]-[1]}))$ 
identifie le quotient $\GL_{2n}(\C)/\tP(\zeta_1)$ avec
$\F(1,2n-1,2n)$. Ainsi, le morphisme $\GL_n(\C)\times\GL_n(\C)/P(\zeta_1)\to\GL_{2n}(\C)/\tP(\zeta_1)$ correspond à
l'application $\iota_{\zeta_1}:\G(n-1,n)\times \G(1,n)\to \F(1,2n-1,2n)$ qui envoie $(D,E)$ sur $(D\subset \C^n\oplus E)$.

\`A $(w,w)\in W_{\agot}$, on associe la classe de cohomologie $[\Xgot_{w,\zeta_1}]\in H^*(\GL_n(\C)\times\GL_n(\C)/P(\zeta_1))$. Le groupe parabolique 
$P(\zeta_1)$ est égal à $J_n P_{n-1}J_n^{-1}\times P_1$, où $P_k\subset \GL_{n}(\C)$ désigne le sous-groupe parabolique qui fixe le drapeau $\C^k\subset \C^n$.
D'autre part, le sous-groupe de Borel de $\GL_n(\C)\times\GL_n(\C)$ compatible avec notre choix de chambre de Weyl est $B_n\times J_n B_nJ_n^{-1}$. Ainsi 
$$
[\Xgot_{w,\zeta_1}]=\left[\overline{B_n wJ_n P_{n-1}/P_{n-1}}\right]\times \left[\overline{B_n J_nw P_{1}/P_{1}}\right]
$$
En d'autre termes, la classe de cohomologie $[\Xgot_{w,\zeta_1}]$ est égale à
$[\Xgot_{i^c}]\times [\Xgot_{i^o}]\in H^*(\G(n-1,n))\times H^*(\G(1,n))$, où $i=w(1)$.

\`A $\tw\in \tW_{\tagot}$, on associe la classe de cohomologie $[\tXgot_{\tw,\zeta_1}]\in H^*(\GL_{2n}(\C)/\tP(\zeta_1))$. 
Soit $u_n\in\Sgot_{2n}$ l'élément associé à la matrice de permutation $\Lambda_n\Theta_n$. On voit que le groupe parabolique 
$\tP(\zeta_1)$ est égal à $u_n P_{1,1}u_n^{-1}$, où $P_{1,1}\subset \GL_{2n}(\C)$ désigne le sous-groupe parabolique qui fixe le drapeau 
$\C^1\subset \C^{2n-1}\subset \C^{2n}$. Le sous-groupe de Borel de $\GL_{2n}(\C)$ compatible avec notre choix de chambre de Weyl est 
$\Lambda_n  B_{2n}\Lambda_n$. Ainsi 
$$
[\tXgot_{\tw,\zeta_1}]=\left[\overline{B_{2n}\Lambda_n \tw\, u_n P_{1,1}/P_{1,1}}\right]
$$

Posons $\ell=\tw(1)\in [2n]$ et $w':=\Lambda_n \tw\, u_n\in\Sgot_{2n}$. Alors 
\begin{itemize}
\item si $\ell\leq n$, on a  $w'(1)=2n+1-\ell$ et $w'(2n)=\ell$,
\item si $\ell\geq n+1$, on a  $w'(1)=\ell-n$ et $w'(2n)=3n+1-\ell$,
\end{itemize}
On a ainsi montré que 
\begin{itemize}
\item si $\ell\leq n$, on a  $[\tXgot_{\tw,\zeta_1}]=[\tXgot_{\{2n+1-\ell\}\subset \{\ell\}^c}]\in H^*(\F(1,2n-1,2n))$.
\item si $\ell\geq n+1$, on a  $[\tXgot_{\tw,\zeta_1}]=[\tXgot_{\{\ell-n\}\subset \{3n+1-\ell\}^c}]\in H^*(\F(1,2n-1,2n))$.
\end{itemize}

Pour des raisons de dimension, la relation $[\Xgot_{i^o}]\times[\Xgot_{i^c}]\cdot \iota_{\zeta_1}^*\left([\tXgot_{\{\ell-n\}\subset \{3n+1-\ell\}^c}]\right)=[pt]$
ne peut pas se produire. \`A l'opposé, grâce au corollaire \ref{coro:pour-B-n}, on voit que la relation
$$
[\Xgot_{i^o}]\times[\Xgot_{i^c}]\cdot \iota_{\zeta_1}^*\left([\tXgot_{\{2n+1-\ell\}\subset \{\ell\}^c}]\right)=[pt],
$$
plus la condition de Lévi-mobilité, sont satisfaits si et seulement si $\ell+i=n+1$. 
Un calcul élémentaire montre que l'inégalité correspondante est $x_\ell\geq y_\ell$.

Un travail similaire avec l'élément admissible $-\zeta_1$ fournit les inégalités $y_\ell\geq - x_{n+1-\ell},\forall \ell\in [n]$.

\subsubsection{Inégalités associée à $\zeta_{r,s}$}

Considérons l'élément admissible $\hgamma:=\zeta_{r,s}$ lorsque $r+s<n$. Le cas $r+s=n$ se traite de manière similaire. 
Ici $\Fcal_{\hgamma}\simeq \F(r,n-s,n)\times \F(s,n-r,n)$ et 
$\tFcal_{\hgamma}\simeq \F(r+s,2n-r-s,2n)$, et le morphisme 
$$
\iota_{\hgamma}:\F(r,n-s,n)\times \F(s,n-r,n)\longrightarrow \F(r+s,2n-r-s,2n)
$$ 
envoie $(E\subset F, E'\subset F')$ sur $E\oplus E'\subset F\oplus F'$.

Lorsque l'on se donne une partition $2n=a_1+\cdots+ a_6$ avec $a_i\geq 1$, on peut représenter $[2n]$ au moyen du tableau
\begin{tabular}{|c|c|c|c|c|c|}
  \hline
  $a_1$   & $a_2$    & $a_3$ & $a_4$   & $a_5$ & $a_6$ \\
  \hline
\end{tabular}
où la case représentée par $a_i$ désigne l'ensemble des entiers $k\geq 1$ vérifiant $a_1+\cdots+ a_i+1\leq k\leq a_1+\cdots+ a_{i+1}$.

Prenons $a_1=a_6=r$, $a_2=a_5=s$, et $a_3=a_4=n-r-s$. Considérons la permutation $w'\in \Sgot_{2n}$ qui envoie 
\begin{tabular}{|c|c|c|c|c|c|}
  \hline
  $a_1$   & $a_2$    & $a_3$ & $a_4$   & $a_5$ & $a_6$ \\
  \hline
\end{tabular} sur 
\begin{tabular}{|c|c|c|c|c|c|}
  \hline
  $a_1$   & $a_3$    & $a_5$ & $a_6$   & $a_4$ & $a_2$ \\
  \hline
\end{tabular}.
Alors $\hgamma=\Ad(w')\gamma^{r+s,r+s}$, et le groupe parabolique $\tP(\,\hgamma\,)$ est égal à 
$\Ad(w')(P_{r+s,r+s})$, où $P_{r+s,r+s}\subset \GL_{2n}(\C)$ désigne le sous-groupe parabolique qui fixe le drapeau
$\C^{r+s}\subset\C^{2n-r-s}\subset\C^{2n}$.

\`A $\tw\in \tW_{\tagot}$, on associe la classe de cohomologie $[\tXgot_{\tw,\zeta_{r,s}}]\in H^*(\GL_{2n}(\C)/\tP(\zeta_{r,s}))$. 
Le sous-groupe de Borel de $\GL_{2n}(\C)$ compatible avec notre choix de chambre de Weyl est 
$\Lambda_n  B_{2n}\Lambda_n$. Ainsi $[\tXgot_{\tw,\zeta_1}]=\left[\overline{B_{2n}\Lambda_n \tw\, w' P_{r+s,r+s}/P_{r+s,r+s}}\right]$
est égal à la classe 
$$
\left[\Xgot_{L\subset L'}\right]\ \in\ H^*(\F(r+s,2n-r-s,2n))
$$
avec $L=u([r+s])$, $L'=u([2n-r-s])$ et $u=\Lambda_n \tw\, w'$.

Le résultat suivant est laissé à la discrétion du lecteur.
\begin{lem}
Nous avons $L'=L^{o,c}$, et donc $L\subset L^{o,c}$. 
\end{lem}

\`A $(w,w)\in W_{\agot}$, on associe la classe de cohomologie $[\Xgot_{w,\zeta_{r,s}}]\in H^*(\GL_n(\C)\times\GL_n(\C)/P(\zeta_{r,s}))$. Le groupe parabolique 
$P(\zeta_{r,s})$ est égal à $P_{r,s}\times J_nP_{s,r}J_n^{-1}$. D'autre part, le sous-groupe de Borel de $\GL_n(\C)\times\GL_n(\C)$ compatible avec 
notre choix de chambre de Weyl est $B_n\times J_n B_nJ_n^{-1}$. Ainsi 
$$
[\Xgot_{w,\zeta_{r,s}}]=\left[\overline{B_n w P_{r,s}/P_{r,s}}\right]\times \left[\overline{B_n J_n w J_nP_{s,r}/P_{s,r}}\right]
$$
En d'autre termes, la classe de cohomologie $[\Xgot_{w,\zeta_{r,s}}]$ est égale à
$[\Xgot_{I\subset J}]\times [\Xgot_{J^{o,c}\subset I^{o,c}}]\in H^*(\F(r,n-s,n))\times H^*(\F(s,n-r,n))$, où $I=w([r])$ et $J=w([n-s])$.

Nous nous intéressons aux couples $(w,\tw)$ pour lesquels 
$$
[\Xgot_{I\subset J}]\times [\Xgot_{J^{o,c}\subset I^{o,c}}]\cdot \iota_{\hgamma}^*\left(\left[\Xgot_{L\subset L^{o,c}}\right]\right)=[pt]
$$
et la donnée $(I\subset J,J^{o,c}\subset I^{o,c},L\subset L^{o,c})$ est Lévi-mobile. Grâce à la proposition \ref{prop:levi-mobile-application-2}, on sait que 
ces deux conditions sont satisfaites si et seulement si on a:
\begin{enumerate}\setlength{\itemsep}{8pt}
\item $\cc^{L^c}_{J,I^{o,c}}= 1$.
\item $\cc^{L^o\natural L^{c}}_{I\natural J,J^{o,c}\natural I^{o,c}}= 1$,
\item $|\mu(I)|+|\mu(J^{o,c})|+|\mu(L)|= \dim \G(r+s,2n-r-s)$.
\end{enumerate}

Ici, comme $|\mu(J^{o,c})|=|\mu(J)|$ et $|\mu(I^{o,c})|=|\mu(I)|$, on voit que la condition 3. est impliquée par la relation $\cc^{L^c}_{J,I^{o,c}}= 1$.
Finalement, un calcul direct montre que l'inégalité $((x,-x),\tw\zeta_{r,s})\geq ((y,-y),w_0w\zeta_{r,s})$ correspond à 
$$
|x|_{L^o\cap [n]}-|x|_{L\cap [n]}\geq |y|_{J^{o,c}}-|x|_{I^o}.
$$
Cela termine la preuve du théorème \ref{theo:pep-B-n}.

\chapter{Valeurs singulières versus valeurs singuli\`eres}

Soit $p\geq q\geq 1$ et $n=p+q$. Une matrice $X\in M_{n,n}(\C)$ s'écrit par blocs 
$X=\begin{pmatrix}
X_{11}& X_{12}\\
X_{21}& X_{22}
\end{pmatrix}$
où $X_{11}\in M_{p,p}(\C)$, $X_{22}\in M_{q,q}(\C)$, $X_{12}\in M_{p,q}(\C)$ et $X_{21}\in M_{q,p}(\C)$.

Le but de cette section est la description des cônes suivants :

$$
\Scal(p,q)=\Big\{(\s(X), \s(X_{12}),\s(X_{21})), \ X\in M_{n,n}(\C)\Big\},
$$
$$
\Tcal(p,q)=\Big\{(\s(X), \s(X_{11}),\s(X_{22})), \ X\in M_{n,n}(\C)\Big\}.
$$

\medskip

Dans cette section, nous utilisons les notations suivantes: 

\begin{itemize}
\item Si $x=(x_1,\ldots,x_m)$, on pose $x^*=(x_m,\ldots,x_1)$.
\item \`A tout $x\in \R^q$, on associe $\widehat{x}\,^{p,q}:=(x_1,\cdots,x_q,0,\cdots,0,-x_q,\cdots,-x_1)\in \R^n$.

\item Pour tout $z\in \R^n$, on pose $\widehat{z}\,^{n,n}:=(z_1,\cdots,z_n,-z_n,\cdots,-z_1)\in \R^{2n}$ 

\item Si $Y$ est une matrice complexe de taille $a\times b$, nous notons\footnote{On note $0_{ab}$ la matrice nulle de taile $a\times b$.}
$\widehat{Y}^{a,b}:=\begin{pmatrix}0_{aa}& Y\\Y^*& 0_{bb}\end{pmatrix}$
la matrice hermitienne de taille $a+b\times a+b$.  

\item \`A tout $x\in \R^a$, on associe la matrice anti-diagonale 
$\adiag(x)=\begin{pmatrix}
0&\cdots & x_1\\
\vdots &\reflectbox{$\ddots$}& \vdots\\
x_q&\cdots & 0\\
\end{pmatrix}
$.
\item $B_m\subset \GL_m(\C)$ est le sous-groupe de Borel des matrices triangulaires supérieures.
\end{itemize}

\section{Les cônes $\Scal(p,q)$}

Commençons par donner la description de $\Scal(p,q)$ obtenue avec le théorème d'O'Shea-Sjamaar (voir \cite[Section 4]{pep-toshi}).

\begin{prop}\label{prop:equivalence-S-p-q}
Pour $(z,x,y)\in \R^n_{++}\times\R^q_{++}\times\R^q_{++}$, on a l'équivalence suivante :
\begin{equation}\label{eq:equivalence-S-p-q}
(z,x,y)\in \Scal(p,q)\quad \Longleftrightarrow\quad (\,\widehat{z}\,^{n,n},\widehat{x}\,^{p,q}, \widehat{y}\, ^{p,q})\in \LR(n,n).
\end{equation}
\end{prop}

La description du cône $\LR(n,n)$ obtenue au théorème \ref{theo:LR-m-n}, permet de voir que $(z,x,y)\in \Scal(p,q)$ si et seulement si les conditions suivantes sont satisfaites
\begin{enumerate}\setlength{\itemsep}{8pt}
\item[i)] $z_k\geq \max(x_k,y_k)$,\quad $\forall k\in [q]$,
\item[ii)] $|z |_{L\cap[n]} \, -\,  |z |_{L^o\cap[n]}\ \geq\   | x |_{I\cap[q]} \,-\,  | x |_{I^o\cap[q]} \,+\,   | y |_{J\cap[q]} \,-\,  | y |_{J^o\cap[q]}$, pour tout triplet de sous ensembles stricts $L\subset [2n]$, $I,J\subset [n]$, satisfaisant $\sharp L =\sharp I +\sharp J$, et $\cc^{L}_{I,J}=1$.
\end{enumerate}

\medskip

Expliquons maintenant la description plus précise de $\Scal(p,q)$ obtenue avec le théorème \ref{theo:main-reel}. 
Pour cela, on a besoin des notations utilisées à la section \ref{sec:preuve-singular-horn}.
Un sous-ensemble $X_\bullet\subset [q]$ est dit polarisé, s'il admet une partition $X_\bullet=X_+\coprod X_-$. \`A un ensemble polarisé $X_\bullet\subset [q]$ de cardinal $r$, on associe 
\begin{itemize}\setlength{\itemsep}{8pt}
\item $X^p_\bullet=X_+\coprod \{p+q+1-x,x\in X_-\}\ \in\  \Bcal(r,p,q)$,
\item $\widetilde{X}^p_\bullet= X^p_\bullet \natural (X^p_\bullet)^{o}\  \in\  \Pcal(r,p+q-r)$.
\end{itemize}

\medskip

\`A chaque triplet $(I_\bullet\subset [q],J_\bullet\subset [q],L_\bullet\subset [n])$ de sous-ensembles polarisés, on associe l'inégalité
$$
(\dagger)_{I_\bullet,J_\bullet,L_\bullet} :\quad  \qquad    
|z |_{L_+} \, -\,  |z |_{L_{-}}\ \geq\   | x |_{I_+} \,-\,  | x |_{I_{-}} \,+\,   | y |_{J_+} \,-\,  | y |_{J_{-}},\qquad (z,x,y)\in \R^n\times \R^q\times\R^q.
$$

Voici le résultat principal de cette section.

\begin{theorem}\label{theo:S-p-q}
Un élément $(z,x,y)\in\R^n_{++}\times\R^q_{++}\times\R^q_{++}$ appartient à $\Scal(p,q)$ si et seulement si 
$z_k\geq \max(x_k,y_k)$,\ $\forall k\in [q]$, et si les inégalités $(\dagger)_{I_\bullet,J_\bullet,L_\bullet}$ sont vérifiées
pour tout triplet de sous-ensembles polarisés $(I_\bullet,J_\bullet,L_\bullet)$  satisfaisant les conditions suivantes :
\begin{enumerate}
\item $I_\bullet\subset [q]$, $J_\bullet \subset [q]$ et $L_\bullet\subset [n]$ sont des sous-ensembles stricts tels que 
$\sharp L_\bullet = \sharp I_\bullet + \sharp J_\bullet$.
\item $\cc^{L^n_\bullet}_{I^p_\bullet,J^p_\bullet}=1$.
\item $\cc^{\widetilde{L}^n_\bullet}_{\widetilde{I}^p_\bullet,\widetilde{J}^p_\bullet}=1$.
\end{enumerate}
\end{theorem}

\begin{rem}
Le théorème précédent est encore vrai avec les conditions plus faibles : $\cc^{L^n_\bullet}_{I^p_\bullet,J^p_\bullet}\neq 0$ et  
$\cc^{\widetilde{L}^n_\bullet}_{\widetilde{I}^p_\bullet,\widetilde{J}^p_\bullet}\neq 0$. 
La proposition \ref{prop:equivalence-S-p-q} permet de voir que le résultat du théorème \ref{theo:S-p-q} est encore valable 
si on ne considère que les conditions {\em 1.} et {\em 2.}.
\end{rem}

Le reste de cette section est consacrée à la preuve du théorème \ref{theo:S-p-q}.

\subsection{Les groupes $G:= U(p,q)\times U(q,p)\ \hookrightarrow\ \tG:= U(n,n)$}

Nous travaillons avec le morphisme de groupes $\iota: \GL_n(\C)\times \GL_n(\C)\to  \GL_{2n}(\C)$ défini comme suit:
\begin{equation}\label{eq:embedding-p-q}
\iota(g,h)= \begin{pmatrix}
g_{11}& 0 & g_{12}\\
0& h & 0\\
g_{21}& 0 & g_{22}
\end{pmatrix},\qquad \mathrm{lorsque}\qquad g= \begin{pmatrix}
g_{11}&  g_{12}\\
g_{21}&  g_{22}
\end{pmatrix}.
\end{equation}
Ici, $g_{11}\in M_{p,p}(\C)$, $g_{12}\in M_{p,q}(\C)$, $g_{2,1}\in M_{q,p}(\C)$ et $g_{22}\in M_{q,q}(\C)$.

Le groupe $\GL_n(\C)\times \GL_n(\C)$ est muni de l'involution anti-homolorphe 
$$
\sigma(g,h):=\left(I_{p,q}(g^*)^{-1}I_{p,q}, I_{q,p}(h^*)^{-1}I_{q,p}\right).
$$ 
Le sous-groupe fixé par $\sigma$ est $G:= U(p,q)\times U(q,p)$. Le groupe $\GL_{2n}(\C)$ est muni de l'involution anti-homolorphe 
$\tilde{\sigma}(u):=I_{n,n}(u^*)^{-1}I_{n,n}$, ainsi $\tG:= U(n,n)$ est le sous-groupe fixé par $\tilde{\sigma}$. 
Comme $\iota\circ\sigma=\tilde{\sigma}\circ \iota$, on voit que $\iota$ induit un morphisme de groupe 
$\iota: U(p,q)\times U(q,p)\to U(n,n)$.

\subsubsection*{Choix pour $U(n,n)$}

Considérons la complexification $GL_{2n}(\C)$ de $G=U(n,n)$. Son sous-groupe compact maximal $U_{2n}$ est équipé de l'involution 
$\tilde{\sigma}(g)=I_{n,n}(g^*)^{-1}I_{n,n}$. Le sous-espace $\tpgot=\{\widehat{X}^{n,n},X\in M_n(\C)\}\subset\glgot_{2n}(\C)$ admet une action canonique de 
$\tK=U_n\times U_n= (U_{2n})^\sigma$.

À tout $(x,z)\in \R^n\times \R^n$, nous associons la matrice $X(x,z):=\begin{pmatrix}\diag(x)& A(z)\\ A(z^*)& \diag(x^*)\end{pmatrix}$. Nous travaillons avec 
le tore maximal $T_0\subset U_{2n}$ avec algèbre de Lie $\tgot_0:=\left\{i X(x,z), \ x,z\in\R^n\right\}$. Remarquons que $T_0$ 
est invariant sous l'action de $\sigma$ et que $\frac{1}{i}\tgot_0^{-\sigma}$ et égal à la sous-algèbre abélienne maximale 
$\tagot=\{X(0,z), z\in \R^n\}\subset\tpgot$.

L'involution $\Ad(\theta_0)$, avec $\theta_0=\tfrac{1}{\sqrt{2}}\begin{pmatrix}I_n& J_n\\ J_n& -I_n\end{pmatrix}$, définit un isomorphisme entre 
$T_0$ et le tore maximal $T\subset U_{2n}$ formé par les matrices diagonales. En particulier, $\Ad(\theta_0)(X(x,z))=\diag(x+z,x^*-z^*)$. Le groupe de Weyl restreint $W_{\tagot}$ s'identifie, à travers $Ad(\theta)$, au sous-groupe de 
$W_0\subset\Sgot_{2n}$ formé des permutations $w$ satisfaisant $w(k^o)=w(k)^o, \forall k\in [2n]$.

Nous travaillons avec la chambre de Weyl $\tgot_{0,+}=\Ad(\theta_0)(\tgot^{2n}_+)$ où $\tgot^{2n}_+=\{i\diag(u),u\in \R^{2n}_+\}$. Ainsi 
$\frac{1}{i}\left(\tgot_0^{-\sigma}\cap\tgot_{0,+}\right)$ est égal à la chambre de Weyl restreinte $\tagot_+:=\left\{ X(0,z), \ z\in\R^n_{++}\right\}$. 
Le sous-groupe de Borel correspondant
$\tB\subset GL_{2n}(\C)$ est égal à $\Ad(\theta_0)(B_{2n})$.

\subsubsection*{Choix pour $U(p,q)$}

Nous considérons maintenant $G_1=U(p,q)$ et sa complexification $GL_{n}(\C)$ : son sous-groupe compact maximal $U_{n}$ est équipé de l'involution 
$\sigma_1(h)=I_{p,q}(h^*)^{-1}I_{p,q}$. Le sous-espace $\pgot_1=\{\widehat{X}^{p,q},X\in M_{p,q}(\C)\}\subset\glgot_{n}(\C)$ admet une action canonique de 
$K_1=U_p\times U_q= (U_{n})^{\sigma_1}$.

À tout $(x,y,z)\in \R^q\times\R^{p-q}\times \R^q$, nous associons la matrice
$$
X_1(x,y,z)=\begin{pmatrix} 
\diag(x)&  0& A(z)\\
0& \diag(y)& 0\\
 A(z^*)&  0& \diag(x^*)
\end{pmatrix}.
$$

Nous travaillons avec le tore maximal $T_1\subset U_{n}$ avec algèbre de Lie $\tgot_1:=\{i\,X_1(x,y,z), \ (x,y,z)\in \R^q\times\R^{p-q}\times \R^q\}$. 
Remarquons que $T_1$ est invariant sous $\sigma_1$ : nous avons $\sigma_1(iX_1(x,y,z))=iX_1(x,y,-z)$, ainsi 
$\agot_1:=\frac{1}{i}\tgot_1^{-\sigma_1}=\{X_1(0,0,z), \ z\in \R^q\}$ est une sous-algèbre abélienne maximale de $\pgot_1$.

L'involution $\Ad(\theta_1)$, avec $\theta_1\in O(n)$ définie par (\ref{eq:theta}), 
détermine un isomorphisme entre $T_1$ et le tore maximal $T\subset U_{n}$ formé par les matrices diagonales. L'image de $X_1(x,y,z)$
par $\Ad(\theta_1)$ est égale à $\diag(x+z,y,x^*-z^*)$. Le groupe de Weyl restreint $W_{\agot_1}$ s'identifie, à travers $\Ad(\theta_1)$, 
au sous-groupe de $W_1\subset\Sgot_{n}$ formé des permutations $w$ satisfaisant $w(k^o)=w(k)^o, \forall k\in [q]$ et $w(k)=k, \forall k\in [q+1,\ldots,p]$.

Nous travaillons avec la chambre de Weyl $\tgot_{1,+}=\Ad(\theta_1)(\tgot_+)$ où $\tgot^n_+=\{i\diag(u),u\in \R^{n}_+\}$. Ainsi 
$\frac{1}{i}\left(\tgot_1^{-\sigma_1}\cap\tgot_{1,+}\right)$ est égal à la chambre de Weyl restreinte 
$\agot_{1,+}:=\{X_1(0,0,z), z\in \R^q_{++}\}$. 
Le sous-groupe de Borel correspondant
$B_1\subset GL_{n}(\C)$ est égal à $\Ad(\theta_1)(B_{n})$.

\subsubsection*{Choix pour $U(q,p)$}

Nous considérons maintenant $G_2=U(q,p)$ et sa complexification $GL_{n}(\C)$ : son sous-groupe compact maximal $U_{n}$ est équipé de l'involution 
$\sigma_2(h)=I_{q,p}(h^*)^{-1}I_{q,p}$. Le sous-espace $\pgot_2=\{\widehat{X}^{q,p},X\in M_{q,p}(\C)\}\subset\glgot_{n}(\C)$ admet une action canonique de 
$K_2=U_q\times U_p= (U_{n})^{\sigma_2}$.

À tout $(x,y,z)\in \R^q\times\R^{p-q}\times \R^q$, nous associons la matrice
$$
X_2(x,y,z)=\begin{pmatrix} 
\diag(x)& A(z) & 0\\
A(z^*)&  \diag(x^*) & 0 \\
0& 0& \diag(y)
\end{pmatrix}.
$$

Nous travaillons avec le tore maximal $T_2\subset U_{n}$ avec algèbre de Lie $\tgot_2:=\{i\,X_2(x,y,z), \ (x,y,z)\in \R^q\times\R^{p-q}\times \R^{q}\}$. 
Remarquons que $T_2$ est invariant sous $\sigma_2$ : nous avons $\sigma_2(iX_2(x,y,z))=iX_2(x,y,-z)$, donc 
$\agot_2:=\frac{1}{i}\tgot_2^{-\sigma_2}=\{X_2(0,0,z), \ z\in \R^q\}$ 
est une sous-algèbre abélienne maximale de $\pgot_2$.

L'involution $\Ad(\theta_2)$, avec 
$$
\theta_2=\frac{1}{\sqrt{2}}
\begin{pmatrix}
I_q& J_q&0\\
J_q& -I_q & 0\\
0& 0 &\sqrt{2}I_{p-q} 
\end{pmatrix},
$$
définit un isomorphisme entre $T_2$ et le tore maximal $T\subset \upU_{n}$ formé par les matrices diagonales. L'image de $X_2(x,y,z)$
par $\Ad(\theta_2)$ est égale à $\diag(x+z,x^*- z^*, y)$. Soit $\upsilon\in \Sgot_n$ la permutation définie comme suit : $\upsilon(k)=k$ si 
$k\leq q$, $\upsilon(k)=k+q$ si $q+1\leq k\leq p$, et $\upsilon(k)=k+q-p$ si $p+1\leq k\leq n$. 
Soit $Q_\upsilon\in N(T)$ la matrice de permutation associée à $\upsilon$. Ainsi, L'image de $X_2(x,y,z)$
par $\Ad(Q_\upsilon\theta_2)$ est égale à $\diag(x+z,y,x^*- z^*)$.

Le groupe de Weyl restreint $W_{\agot_2}$ s'identifie, à travers $Ad(Q_\upsilon\theta_2)$, 
au sous-groupe $W_2\subset \Sgot_{n}$ formé des permutations $w$ satisfaisant $w(k^o)=w(k)^o, \forall k\in [q]$ et $w(k)=k, \forall k\in [q+1,\ldots,p]$. 
Nous travaillons avec la chambre de Weyl $\tgot_{2,+}=\Ad(\theta_2Q_\upsilon^{-1})(\tgot^n_+)$. Ainsi 
$\frac{1}{i}\left(\tgot_2^{-\sigma_2}\cap\tgot_{2,+}\right)$ est égal à la chambre de Weyl restreinte 
$\agot_{2,+}:=\{X_2(0,0,z), z\in \R^q_{++}\}$. Le sous-groupe de Borel correspondant est égal à $B_2:=\Ad(\theta_2Q_\upsilon^{-1})(B_{n})$.

\subsection{\'Eléments admissibles}

Le groupe $\tK:=\tG\cap \upU_{2n}\simeq \upU_n\times \upU_n$ est un sous-groupe compact maximal de $\tG$, et le sous-espace 
$\tpgot$ admet une identification $\tK$-équivariante avec $M_{n,n}(\C)$ à travers l'application $X\mapsto \widehat{X}^{n,n}$.

Le sous-groupe $K= K_1\times K_2$, avec $K_1:=U(p,q)\cap \upU_{n}\simeq \upU_{p}\times \upU_{q}$ et 
$K_2:=U(q,p)\cap \upU_{n}\simeq \upU_{q}\times \upU_{p}$, est un sous-groupe compact maximal de $G$, et le sous-espace 
$\pgot= \pgot_1\times\pgot_2$ admet une identification $K$-équivariante avec $M_{p,q}(\C)\times M_{q,p}(\C)$:
$$
(Y,Z)\in M_{p,q}(\C)\times M_{q,p}(\C)\longmapsto (\widehat{Y}^{p,q},\widehat{Z}^{q,p})\in \pgot_1\times\pgot_2.
$$

Notons encore $\iota: \glgot_n(\C)\times \glgot_n(\C)\croc \glgot_{2n}(\C)$ le morphisme d'algèbre de Lie défini par (\ref{eq:embedding-p-q}). 
On vérifie que pour tout $(Y,Z)\in M_{p,q}(\C)\times M_{q,p}(\C)$, on a la relation
$$
\iota\left(\widehat{Y}^{p,q},\widehat{Z}^{q,p}\right)=\widehat{X}^{n,n}\quad\mathrm{avec}\quad 
X:=\begin{pmatrix}
0_{pp}&  Y\\
Z&  0_{qq}
\end{pmatrix}.
$$

Considérons les isomorphismes $\varphi:\R^n\to \agot$, $\varphi_1:\R^q\to \agot_1$, et $\varphi_2:\R^q\to \agot_2$ définis respectivement par 
$\varphi(z)=X(0,0,z)$, $\varphi_1(x)=X_1(0,0,x)$ et $\varphi_2(y)=X_2(0,0,y)$. Le morphisme $\iota:\pgot\to\tpgot$ envoie $(\varphi_1(x),\varphi_2(y))$ vers $\varphi(z)$ 
où $z=(x,0,\ldots,0,y)$.

Soit $\Sigma(\tggot/\ggot)\subset\agot^*$ l'ensemble des racines non nulles par rapport à l'action adjointe de $\agot$ sur $\tggot/\ggot$. L'isomorphisme $\varphi_1$ permet d'identifier l'espace dual 
$\agot_1^*$ avec $(\R^q)^*$ : soit $e_i^*,i\in [q]$ la base canonique correspondante de $\agot_1^*$. De même, l'isomorphisme $\varphi_2$ permet de considérer une base, $f_j^*,j\in [q]$, de $\agot_1^*$.

\begin{lem}
Nous avons
$$
\Sigma(\tggot/\ggot)=\left\{\pm e_i^*\pm f_j^*,\ i,j\in[q]\right\}\bigcup\left\{\pm e_i^*,\ i\in[q]\right\}\bigcup\left\{\pm f_j^*,\ j\in[q]\right\}.
$$
\end{lem}

\begin{proof} Considérons l'involution $\tau=\Ad(\diag(I_p,-I_q,-I_p,I_q))$ sur $\tG$, de sorte que $G$ soit le sous-groupe fixé par $\tau$.
Le sous-espace $\tagot$ est invariant par $\tau$ et $\agot\simeq \tagot^\tau$. L'ensemble des racines restreintes $\Sigma(\tggot)\subset\tagot^*$ comprend toutes les formes
$\pm h_i^* \pm h_j^*, i,j\in [n]$ : ici $h_j^*,j\in[n]$ est la base canonique de $\tagot^*$ associée à l'isomorphisme $\varphi:\R^n\to\tagot$. 
En utilisant la décomposition $\tggot=\tagot\oplus\sum_{\tilde{\alpha}\in\Sigma(\tggot)}\tggot_{\tilde{\alpha}}$, on voit que
 
$$
\tggot/\ggot\simeq \tggot^{-\tau}=
\tagot^{-\tau}\oplus\sum_{\tau(\tilde{\alpha})\neq\tilde{\alpha}}\left(\tggot_{\tilde{\alpha}}\oplus \tggot_{\tau(\tilde{\alpha})}\right)^{-\tau}\oplus 
\sum_{\tau(\tilde{\alpha})=\tilde{\alpha}}\left(\tggot_{\tilde{\alpha}}\right)^{-\tau}
$$
Ainsi, $\Sigma(\tggot/\ggot)$ est formé des restrictions $\tilde{\alpha}\vert_{\agot}$, où $\tilde{\alpha}\in\Sigma(\tggot)$ 
satisfait l'une des conditions suivantes: $\tau(\tilde{\alpha})\neq\tilde{\alpha}$ ou $\tau(\tilde{\alpha})=\tilde{\alpha}$ et 
$\left(\tggot_{\tilde{\alpha}}\right)^{-\tau}\neq 0$.
Nous voyons que $\{\tilde{\alpha}\vert_{\agot},\ \tau(\tilde{\alpha})\neq\tilde{\alpha}\}$ est égal à $\{\pm e_i^*,\ i\in[q]\}\cup\{\pm f_j^*,\ j\in[q]\}\cup\{0\}$. 
En utilisant la description des espaces propres $\tggot_{\tilde{\alpha}}$ donnée dans \cite{Knapp-book}, \S VI.4,  nous obtenons que
 $\{\tilde{\alpha}\vert_{\agot},\ \tau(\tilde{\alpha})=\tilde{\alpha}\ \mathrm{et}\ \left(\tggot_{\tilde{\alpha}}\right)^{-\tau}\neq 0\}$ est égal à 
$\{\pm e_i^*\pm f_j^*,\ i,j\in[q]\}$.
\end{proof}

Nous voulons maintenant déterminer l'ensemble des éléments admissibles de $\agot$  par rapport à l'ensemble des racines $\Sigma(\tggot/\ggot)$, c'est-à-dire les vecteurs non nuls $\zeta\in \agot$ satisfaisant $\vect(\Sigma(\tggot/\ggot)\cap\zeta^\perp)=\zeta^{\perp}$. L'ensemble des éléments admissibles est invariant sous l'action du groupe de Weyl restreint $W_\agot=W_{\agot_1}\times W_{\agot_2} \simeq (\Sgot_q\rtimes \{\pm 1\}^q)^2$.

Pour des entiers $r,s\leq q$, nous définissons les vecteurs $\zeta_r^1=\varphi_1(-\sum_{i\leq r} e_i)\in \agot_1$, $\zeta_s^2=\varphi_2(-\sum_{j\leq s} f_j)\in \agot_2$.

\begin{lem}\label{lem:admissible-S-p-q} $\zeta\in\agot$ est un élément rationnel admissible si et seulement si il existe $(q,w)\in \Q^{>0}\times W_\agot$ et 
$(r,s)\in \{(0,1),(1,0)\}\cup [q]\times[q]$, tels que $\zeta= q\, w(\zeta_r^1,\zeta_s^2)$.
\end{lem}
\begin{proof}Soit $\zeta=(\zeta_1,\zeta_2)$ un élément rationnel admissible. \'Ecrivons
$\zeta_1=\sum_{a\in B_1} a \mathbf{1}_{X(a)}$, $\zeta_2=\sum_{a\in B_2} b \mathbf{1}_{Y(b)}$, où $X(a),Y(b)\subset [q]$ sont des sous-ensembles non-vides, $B_1,B_2\subset \R-\{0\}$ sont finis, 
$1_{X(a)}=\sum_{i\in X(a)}e_i^*$ et $1_{Y(b)}=\sum_{j\in Y(b)}f_j^*$. Quitte à multiplier $(\zeta_r^1,\zeta_s^2)$ par un élément de $W_\agot$, on peut supposer que $B_1,B_2\subset \R^{>0}$. Posons 
$X=\cup_{a\in B_1}X(a)$ et $Y=\cup_{b\in B_2}Y(a)$. On voit alors que $\vect(\Sigma(\tggot/\ggot)\cap\zeta^\perp)$ admet comme base la famille 
$$
\{e_i, i\notin X\}\cup\{f_j, j\notin Y\}\cup\bigcup_{a\in B_1\cap B_2}\{e_i-f_j, (i,j)\in X(a)\times Y(a)\}.
$$
Ainsi, la dimension de $\vect(\Sigma(\tggot/\ggot)\cap\zeta^\perp)$ est égale à $N:=\sharp X^c + \sharp Y^c + \sum_{a\in B_1\cap B_2}(\sharp X(a) + \sharp Y(a) -1)$ tandis que la dimension de 
$\agot$ est égale à $N':=\sharp X^c + \sharp Y^c + \sum_{a\in B_1}\sharp X(a) +\sum_{b\in B_2}\sharp Y(b)$.  Un élément rationnel $\zeta\in\agot$ est admissible si et seulement si $N'=N-1$, c'est à dire si 
$$
\sum_{a\in B_1\cap B_2}(\sharp X(a) + \sharp Y(a) -1)=\sum_{a\in B_1}\sharp X(a) +\sum_{b\in B_2}\sharp Y(b) -1.
$$
Cette dernière relation, qui est équivalente à 
$$
\sum_{a\in B_1-B_2}\sharp X(a)+\sum_{b\in B_2-B_1}\sharp Y(b)=1 -\sharp B_1\cap B_2,
$$
est satisfaite seulement dans les cas suivants: $B_1=\emptyset$ et $\exists (b,j)\in \Q-\{0\}\times [q], \zeta_2=b f_j^*$, $B_2=\emptyset$ et $\exists (a,i)\in \Q-\{0\}\times [q],\zeta_1=a e_i^*$, et 
$B_1=B_2$ est un singleton. Cela termine la preuve du lemme.
\end{proof}

\subsection{Calcul de Schubert}

Fixons $(r,s)\in [q]\times[q]$. On associe aux vecteurs $\zeta_r^i\in\agot_i$, les sous-groupe paraboliques $P(\zeta_r^i)\subset \GL_{n}(\C)$.
Notons $\zeta_{r,s}=\iota(\zeta_r^1,\zeta_s^2)\in\tagot$, et $\tP(\zeta_{r,s})\subset \GL_{2n}(\C)$ le sous-groupe parabolique correspondant.

Le morphisme de groupe (\ref{eq:embedding-p-q}) induit le plongement $\iota_{r,s}: \Fcal_r\times\Fcal_s\to \tFcal_{r,s}$, 
où $\Fcal_r:= \GL_{n}(\C)/P(\zeta_r^1)$, $\Fcal_s:= \GL_{n}(\C)/P(\zeta_s^2)$ et $\tFcal_{r,s}:=\GL_{2n}(\C)/\tP(\zeta_{r,s})$. \`A chaque triplet 
$(w_1,w_2,\tw)\in W_{\agot_1}\times W_{\agot_2}\times \tW_{\tagot}$, on associe les variétés de Schubert
$$
\Xgot_{w_1}:=\overline{B_1w_1P(\zeta_r^1)/P(\zeta_r^1)}\subset \Fcal_r,\quad \Xgot_{w_2}:=\overline{B_2w_2P(\zeta_s^2)/P(\zeta_s^2)}\subset \Fcal_s, \quad\mathrm{et}\quad
\tXgot_{\tw}:=\overline{\tB\tw\tP(\zeta_{r,s})/\tP(\zeta_{r,s})}\subset \Fcal_{r,s}.
$$
Nous allons maintenant analyser la condition cohomologique 
\begin{equation}\label{eq:cohomo:S-p-q}
[\Xgot_{w_1}]\times [\Xgot_{w_2}]\cdot \iota_{r,s}^*\left([\tXgot_{\tw}]\right)=[pt].
\end{equation}

\`A tout $r,s\in [q]$, on associe le sous-ensemble $A_{r,s}\subset [n]$ de cardinal $r+s$ : $k\in A_{r+s}$ si 
$1\leq k\leq r$ ou bien $p+1\leq k\leq p+s$. En consid\'erant $A_{r,s}$ comme un sous-ensemble $[2n]$, on d\'efinit
$(A_{r,s})^o=\{2n+1-k, k\in A_{r,s}\}\subset [2n]$ et $(A_{r,s})^{o,c}=[2n]-(A_{r,s})^o$.

Utilisons les isomorphismes suivants:
\begin{enumerate}
\item $\psi_1:  \Fcal_r\to \F(r,n-r; n)$, $gP(\zeta_r^1)\mapsto g\theta_1(\C^r\subset \C^{n-r})$,
\item $\psi_2:  \Fcal_s\to \F(s,n-s; n)$, $hP(\zeta_s^2)\mapsto hQ_\upsilon\theta_2(\C^s\subset \C^{n-s})$,
\item $\psi:  \tFcal_{r,s}\to \F(r+s,2n-r-s; 2n)$, $k\tP(\zeta_{r,s})\mapsto k\theta(E_{r,s}\subset F_{r,s})$,
\end{enumerate}
o\`u $E_{r,s}=\vect\{e_i, i\in A_{r,s}\}\subset \C^n$ et $F_{r,s}=\vect\{e_i, i\in (A_{r,s})^{o,c}\}\subset \C^{2n}$.

Le morphisme $\iota: \F(r,n-r; n)\times\F(s,n-s; n)\to \F(r+s,2n-r-s; 2n)$, d\'efinit par le diagramme commutatif \ref{diagramme-S-p-q}, 
est l'application qui envoie qui envoie le couple $(E\subset F,E'\subset F')$ sur $(E\oplus E'\subset F\oplus F')$.

\begin{equation}\label{diagramme-S-p-q}
\begin{tikzcd}
  \Fcal_r \arrow[d, "\psi_1"] &\times & \Fcal_s \arrow[d, "\psi_2"]  \arrow[r, "\iota_{r,s}"]& \tFcal_{r,s}\arrow[d, "\psi"] \\
  \F(r,n-r; n) &\times & \F(s,n-s; n)\arrow[r, "\iota"]&\F(r+s,2n-r-s; 2n).
\end{tikzcd}
\end{equation}

Consid\'erons les \'el\'ements des groupes de Weyl restreints : 
$u_1=\Ad(\theta_1)(w_1)\in W_1$, $u_2=\Ad(\theta_2Q_\upsilon^{-1})(w_2)\in W_2$ et 
$u=\Ad(\theta)(\tw)\in W_0$. On leur associe les sous ensembles polaris\'es: 
\begin{itemize}
\item $I_\bullet\subset [q]$ de cardinal $r$, tel que $I_\bullet^p=u_1([r])$.
\item $J_\bullet\subset [q]$ de cardinal $s$, tel que $J_\bullet^p=u_2([s])$.
\item $L_\bullet\subset [n]$ de cardinal $r+s$,  o\`u $L_\bullet^n= u((A_{r,s})^{o})$.
\end{itemize}

Un calcul direct montre que la relation \ref{eq:cohomo:S-p-q} est équivalente \`a
\begin{equation}\label{eq:cohomo:S-p-q-bis}
[\Xgot_{I_\bullet^p\subset (I_\bullet^p)^{o,c}}]\times [\Xgot_{J_\bullet^p\subset (J_\bullet^p)^{o,c}}]\cdot \iota^*
\left([\tXgot_{(L_\bullet^n)^{o}\subset (L_\bullet^n)^{c}}]\right)=[pt].
\end{equation}
 
Finalement, gr\^ace \`a la proposition \ref{prop:levi-mobile-application-2}, on sait que la relation \ref{eq:cohomo:S-p-q-bis} 
plus la condition de L\'evi-mobilit\'e sont satisfaits 
si et seulement si les conditions suivantes sont r\'ealis\'ees
\begin{enumerate}
\item $\cc^{(L_\bullet^n)^{o,c}}_{(I_\bullet^p)^{o,c},(J_\bullet^p)^{o,c}}=1$.
\item $\cc^{\widetilde{L}^n_\bullet}_{\widetilde{I}^p_\bullet,\widetilde{J}^p_\bullet}=1$.
\item $\mu(I_\bullet^p)+\mu(J_\bullet^p)+\mu(L_\bullet^n)=(r+s)(2n-r-s)$.
\end{enumerate}

Ici, la consition 1. est \'equivalente \`a $\cc^{L_\bullet^n}_{I_\bullet^p,J_\bullet^p}=1$, et cette derni\`ere relation implique la relation 3.

Finalement, un calcul direct montre que l'inégalité $((z,-z),\tw\zeta_{r,s})\geq ((x,-x),w_0w_1\zeta_{r})+ ((y,-y),w_0w_2\zeta_{s})$ correspond à 
\`a l'in\'egalit\'e $(\dagger)_{I_\bullet,J_\bullet,L_\bullet}$.

Lorsque l'on travaille avec les variables $(r,s)\in \{(0,1),(1,0)\}$, le m\^eme type de m\'ethode permet d'obtenir les conditions 
$z_k\geq \max(x_k,y_k)$,\ $\forall k\in [q]$.

Cela termine la preuve du théorème \ref{theo:S-p-q}

\section{Les cônes $\Tcal(p,q)$}

Commençons par donner la description de $\Tcal(p,q):=\Big\{(\s(X), \s(X_{11}),\s(X_{22})), \ X\in M_{n,n}(\C)\Big\}$ obtenue avec le théorème d'O'Shea-Sjamaar (voir \cite[Section 4]{pep-toshi}).

\begin{prop}\label{prop:equivalence-T-p-q}
Pour $(z,x,y)\in \R^n_{++}\times\R^p_{++}\times\R^q_{++}$, on a l'équivalence suivante :
\begin{equation}\label{eq:equivalence-T-p-q}
(z,x,y)\in \Tcal(p,q)\quad \Longleftrightarrow\quad (\,\widehat{z}\,^{n,n},\widehat{x}\,^{p,p}, \widehat{y}\, ^{q,q})\in \LR(2p,2q).
\end{equation}
\end{prop}

La description du cône $\LR(2p,2q)$ obtenue au théorème \ref{theo:LR-m-n}, permet de voir que $(z,x,y)\in \Tcal(p,q)$ si et seulement 
si les conditions suivantes sont satisfaites :
\begin{equation}\label{eq:T-p-q-basique}
\boxed{z_k\geq x_k, \ \forall k\leq p},\quad \boxed{z_j\geq y_j, \ \forall j\leq q}, \quad \boxed{z_{2q+\ell}\leq x_\ell, \ \forall \ell\leq p-q},
\end{equation}
et $| z |_{A\cap[n]} \, -\,  | z |_{A^o\cap[n]}\ \geq\   | x |_{B\cap[p]} \,-\,  | x |_{B^o\cap[p]} \,+\,   | y |_{C\cap[q]} \,-\,  | y |_{C^o\cap[q]}$
pour tout triplet $(A\subset [2n], B\subset [2p],C\subset [2q])$ vérifiant $\sharp A = \sharp B + \sharp C$ et $\cc^{A}_{B,C}\neq 0$.

\medskip

Expliquons maintenant la description plus précise de $\Tcal(p,q)$ obtenue avec le théorème \ref{theo:main-reel}. 
Pour cela, on utilise les notations de la section \ref{sec:preuve-singular-horn}.

\begin{theorem}\label{theo:T-p-q} Un élément $(z,x,y)\in \R^n_{++}\times\R^p_{++}\times\R^q_{++}$ appartient à $\Tcal(p,q)$ 
si et seulement si les inégalités (\ref{eq:T-p-q-basique}) sont vérifiées et 
si nous avons $| z |_{L_+} - | z  |_{L_-}\geq | x |_{I_+} -  | x |_{I_-} + | y |_{J_+} -  | y |_{J_-}$ pour tout triplet de sous-ensembles polarisés 
$(I_\bullet,J_\bullet, L_\bullet)$  satisfaisant les conditions suivantes :
\begin{enumerate}
\item $I_\bullet\subset [p]$, $J_\bullet \subset [q]$ et $L_\bullet\subset [n]$ sont des sous-ensembles stricts tels que 
$\sharp L_\bullet = \sharp I_\bullet + \sharp J_\bullet$.
\item $\cc^{L^n_\bullet}_{I^p_\bullet,J^q_\bullet}=1$.
\item $\cc^{\widetilde{L}^n_\bullet}_{\widetilde{I}^p_\bullet,\widetilde{J}^q_\bullet}=1$.
\end{enumerate}
\end{theorem}

Le reste de cette section est consacrée à la preuve de ce théorème.

\subsection{Les groupes $G:= U(p,p)\times U(q,q)\ \hookrightarrow\ \tG:= U(n,n)$}

Nous travaillons avec le morphisme de groupes $\iota: \GL_{2p}(\C)\times \GL_{2q}(\C)\to  \GL_{2n}(\C)$ défini comme suit:
\begin{equation}\label{eq:embedding-T-p-q}
\iota(g,h)= \begin{pmatrix}
g_{11}& 0 & g_{12}\\
0& h & 0\\
g_{21}& 0 & g_{22}
\end{pmatrix},\qquad \mathrm{lorsque}\qquad g= \begin{pmatrix}
g_{11}&  g_{12}\\
g_{21}&  g_{22}
\end{pmatrix}.
\end{equation}
Ici, le matrices $g_{ij}$ appartiennent à $M_{p,p}(\C)$.

Le groupe $\GL_{2q}(\C)\times \GL_{2q}(\C)$ est muni de l'involution anti-homolorphe 
$$
\sigma(g,h):=\left(I_{p,p}(g^*)^{-1}I_{p,p}, I_{q,q}(h^*)^{-1}I_{q,q}\right).
$$ 
Le sous-groupe fixé par $\sigma$ est $G:= U(p,q)\times U(q,p)$. Le groupe $\GL_{2n}(\C)$ est muni de l'involution anti-homolorphe 
$\tilde{\sigma}(u):=I_{n,n}(u^*)^{-1}I_{n,n}$, ainsi $\tG:= U(n,n)$ est le sous-groupe fixé par $\tilde{\sigma}$. 
Comme $\iota\circ\sigma=\tilde{\sigma}\circ \iota$, on voit que $\iota$ induit un morphisme de groupe 
$\iota: U(p,p)\times U(q,q)\to U(n,n)$.

\subsubsection*{Choix pour $U(n,n)$}

Considérons la complexification $GL_{2n}(\C)$ de $G=U(n,n)$. Son sous-groupe compact maximal $U_{2n}$ est équipé de l'involution 
$\tilde{\sigma}(g)=I_{n,n}(g^*)^{-1}I_{n,n}$. Le sous-espace $\tpgot=\{\widehat{X}^{n,n},X\in M_n(\C)\}\subset\glgot_{2n}(\C)$ admet une action canonique de 
$\tK=U_n\times U_n= (U_{2n})^{\tilde{\sigma}}$.

À tout $(x,z)\in \R^n\times \R^n$, nous associons la matrice $X(x,z):=\begin{pmatrix}\diag(x)& A(z)\\ A(z^*)& \diag(x^*)\end{pmatrix}$. Nous travaillons avec 
le tore maximal $T_0\subset U_{2n}$ avec algèbre de Lie $\tgot_0:=\left\{i X(x,z), \ x,z\in\R^n\right\}$. Remarquons que $T_0$ 
est invariant sous l'action de $\sigma$ et que $\frac{1}{i}\tgot_0^{-\sigma}$ et égal à la sous-algèbre abélienne maximale 
$\tagot=\{X(0,z), z\in \R^n\}\subset\tpgot$.

L'involution $\Ad(\theta_0)$, avec $\theta_0=\tfrac{1}{\sqrt{2}}\begin{pmatrix}I_n& J_n\\ J_n& -I_n\end{pmatrix}$, définit un isomorphisme entre 
$T_0$ et le tore maximal $T\subset U_{2n}$ formé par les matrices diagonales. En particulier, $\Ad(\theta_0)(X(x,z))=\diag(x+z,x^*-z^*)$. Le groupe de Weyl restreint $W_{\tagot}$ s'identifie, à travers $Ad(\theta)$, au sous-groupe de 
$W_0\subset\Sgot_{2n}$ formé des permutations $w$ satisfaisant $w(k^o)=w(k)^o, \forall k\in [2n]$.

Nous travaillons avec la chambre de Weyl $\tgot_{0,+}=\Ad(\theta_0)(\tgot^{2n}_+)$ où $\tgot^{2n}_+=\{i\diag(u),u\in \R^{2n}_+\}$. Ainsi 
$\frac{1}{i}\left(\tgot_0^{-\sigma}\cap\tgot_{0,+}\right)$ est égal à la chambre de Weyl restreinte $\tagot_+:=\left\{ X(0,z), \ z\in\R^n_{++}\right\}$. 
Le sous-groupe de Borel correspondant
$\tB\subset GL_{2n}(\C)$ est égal à $\Ad(\theta_0)(B_{2n})$.

\subsubsection*{Choix pour $U(p,p)$}

Nous considérons maintenant $G_1=U(p,p)$ et sa complexification $GL_{2p}(\C)$ : son sous-groupe compact maximal $U_{2p}$ est équipé de l'involution 
$\sigma_1(h)=I_{p,p}(h^*)^{-1}I_{p,p}$. Le sous-espace $\pgot_1=\{\widehat{X}^{p,p},X\in M_{p,p}(\C)\}\subset\glgot_{n}(\C)$ admet une action canonique de 
$K_1=U_p\times U_p= (U_{2p})^{\sigma_1}$.

À tout $(z,x)\in \R^p\times \R^p$, nous associons la matrice $X_1(x,z):=\begin{pmatrix}\diag(z)& A(x)\\ A(x^*)& \diag(z^*)\end{pmatrix}$. Nous travaillons avec 
le tore maximal $T_1\subset U_{2p}$ avec algèbre de Lie $\tgot_1:=\left\{i X_1(z,x), \ x,z\in\R^p\right\}$. Remarquons que $T_1$ 
est invariant sous l'action de $\sigma_1$ et que $\frac{1}{i}\tgot_1^{-\sigma_1}$ et égal à la sous-algèbre abélienne maximale 
$\tagot_1:=\{X_1(0,x), x\in \R^p\}\subset\tpgot$.

L'involution $\Ad(\theta_1)$, avec $\theta_1=\tfrac{1}{\sqrt{2}}\begin{pmatrix}I_p& J_p\\ J_p& -I_p\end{pmatrix}$, définit un isomorphisme entre 
$T_1$ et le tore maximal $T\subset U_{2p}$ formé par les matrices diagonales. Le groupe de Weyl restreint $W_{\tagot_1}$ s'identifie, à travers 
$Ad(\theta_1)$, au sous-groupe de $W_1\subset\Sgot_{2p}$ formé des permutations $w$ satisfaisant $w(k^o)=w(k)^o, \forall k\in [2p]$.

Nous travaillons avec la chambre de Weyl $\tgot_{1,+}=\Ad(\theta_1)(\tgot^{2p}_+)$ où $\tgot^{2p}_+=\{i\diag(u),u\in \R^{2p}_+\}$. Ainsi 
$\frac{1}{i}\left(\tgot_1^{-\sigma_1}\cap\tgot_{1,+}\right)$ est égal à la chambre de Weyl restreinte $\tagot_{1,+}:=\left\{ X_1(0,x), \ x\in\R^p_{++}\right\}$. 
Le sous-groupe de Borel correspondant $\tB_1\subset GL_{2p}(\C)$ est égal à $\Ad(\theta_1)(B_{2p})$.

\subsubsection*{Choix pour $U(q,q)$}

Nous considérons maintenant $G_2=U(q,q)$ et sa complexification $GL_{2q}(\C)$ : son sous-groupe compact maximal $U_{2p}$ est équipé de l'involution 
$\sigma_2(h)=I_{q,q}(h^*)^{-1}I_{q,q}$. Le sous-espace $\pgot_2=\{\widehat{X}^{q,q},X\in M_{q,q}(\C)\}\subset\glgot_{n}(\C)$ admet une action canonique de 
$K_1=U_p\times U_p= (U_{2p})^{\sigma_1}$.

À tout $(z,y)\in \R^p\times \R^p$, nous associons la matrice $X_1(z,y):=\begin{pmatrix}\diag(z)& A(y)\\ A(y^*)& \diag(z^*)\end{pmatrix}$. Nous travaillons avec 
le tore maximal $T_2\subset U_{2q}$ avec algèbre de Lie $\tgot_1:=\left\{i X_2(z,y), \ z,y\in\R^q\right\}$. Remarquons que $T_2$ 
est invariant sous l'action de $\sigma_2$ et que $\frac{1}{i}\tgot_2^{-\sigma_2}$ et égal à la sous-algèbre abélienne maximale 
$\tagot_2:=\{X_2(0,y), y\in \R^q\}\subset\tpgot$.

L'involution $\Ad(\theta_2)$, avec $\theta_2=\tfrac{1}{\sqrt{2}}\begin{pmatrix}I_q& J_q\\ J_q& -I_q\end{pmatrix}$, définit un isomorphisme entre 
$T_2$ et le tore maximal $T\subset U_{2q}$ formé par les matrices diagonales. Le groupe de Weyl restreint $W_{\tagot_2}$ s'identifie, à travers 
$Ad(\theta_2)$, au sous-groupe de $W_2\subset\Sgot_{2q}$ formé des permutations $w$ satisfaisant $w(k^o)=w(k)^o, \forall k\in [2q]$.

Nous travaillons avec la chambre de Weyl $\tgot_{2,+}=\Ad(\theta_2)(\tgot^{2q}_+)$ où $\tgot^{2q}_+=\{i\diag(u),u\in \R^{2q}_+\}$. Ainsi 
$\frac{1}{i}\left(\tgot_2^{-\sigma_2}\cap\tgot_{2,+}\right)$ est égal à la chambre de Weyl restreinte $\tagot_{2,+}:=\left\{ X_1(0,y), \ y\in\R^q_{++}\right\}$. 
Le sous-groupe de Borel correspondant $\tB_2\subset GL_{2q}(\C)$ est égal à $\Ad(\theta_2)(B_{2q})$.

\subsection{\'Eléments admissibles}

Le groupe $\tK:=\tG\cap \upU_{2n}\simeq \upU_n\times \upU_n$ est un sous-groupe compact maximal de $\tG$, et le sous-espace 
$\tpgot$ admet une identification $\tK$-équivariante avec $M_{n,n}(\C)$ à travers l'application $X\mapsto \widehat{X}^{n,n}$.

Le sous-groupe $K= K_1\times K_2$, avec $K_1:=U(p,p)\cap \upU_{2p}\simeq \upU_{p}\times \upU_{p}$ et 
$K_2:=U(q,q)\cap \upU_{2q}\simeq \upU_{q}\times \upU_{q}$, est un sous-groupe compact maximal de $G$, et le sous-espace 
$\pgot= \pgot_1\times\pgot_2$ admet une identification $K$-équivariante avec $M_{p,p}(\C)\times M_{q,q}(\C)$:
$$
(Y,Z)\in M_{p,p}(\C)\times M_{q,q}(\C)\longmapsto (\widehat{Y}^{p,p},\widehat{Z}^{q,q})\in \pgot_1\times\pgot_2.
$$

Notons encore $\iota: \glgot_{2p}(\C)\times \glgot_{2q}(\C)\croc \glgot_{2n}(\C)$ le morphisme d'algèbre de Lie défini par (\ref{eq:embedding-T-p-q}). 
On vérifie que pour tout $(Y,Z)\in M_{p,q}(\C)\times M_{q,p}(\C)$, on a la relation
$$
\iota\left(\widehat{Y}^{p,p},\widehat{Z}^{q,q}\right)=\widehat{X}^{n,n}\quad\mathrm{avec}\quad 
X:=\begin{pmatrix}
0&  Y\\
Z&  0
\end{pmatrix}.
$$

Considérons les isomorphismes $\varphi:\R^n\to \agot$, $\varphi_1:\R^p\to \agot_1$, et $\varphi_2:\R^q\to \agot_2$ définis respectivement par 
$\varphi(z)=X(0,z)$, $\varphi_1(x)=X_1(0,x)$ et $\varphi_2(y)=X_2(0,y)$. Le morphisme $\iota:\pgot\to\tpgot$ envoie 
$(\varphi_1(x),\varphi_2(y))$ vers $\varphi(z)$ où $z=(x,y)$, et il r\'ealise un isomorphisme $\agot\simeq \tagot$.

Soit $\Sigma(\tggot/\ggot)\subset\agot^*$ l'ensemble des racines non nulles par rapport à l'action adjointe de $\agot$ sur $\tggot/\ggot$. 
L'isomorphisme $\varphi_1$ permet d'identifier l'espace dual 
$\agot_1^*$ avec $(\R^p)^*$ : soit $e_i^*,i\in [p]$ la base canonique correspondante de $\agot_1^*$. 
De même, l'isomorphisme $\varphi_2$ permet de considérer une base, $f_j^*,j\in [q]$, de $\agot_2^*$.

\begin{lem}
Nous avons
$$
\Sigma(\tggot/\ggot)=\left\{\pm e_i^*\pm f_j^*,\ (i,j)\in[p]\times [q]\right\}.
$$
\end{lem}

\begin{proof} Considérons l'involution $\tau=\Ad(\diag(I_p,-I_q,-I_q,I_p))$ sur $\tG$, de sorte que $G$ soit le sous-groupe fixé par $\tau$. Ici,
$\tau$ agit trivialement sur $\tagot\simeq\agot$. En utilisant la décomposition $\tggot=\tagot\oplus\sum_{\tilde{\alpha}\in\Sigma(\tggot)}\tggot_{\tilde{\alpha}}$, 
on voit que 
$$
\tggot/\ggot\simeq \tggot^{-\tau}=\sum_{\tilde{\alpha}\in\Sigma(\tggot)}\left(\tggot_{\tilde{\alpha}}\right)^{-\tau}
$$
Ainsi, $\Sigma(\tggot/\ggot)$ est des racines $\tilde{\alpha}\in\Sigma(\tggot)$ satisfaisant $\left(\tggot_{\tilde{\alpha}}\right)^{-\tau}\neq 0$.
En utilisant la description des espaces propres $\tggot_{\tilde{\alpha}}$ donnée dans \cite{Knapp-book}, \S VI.4,  nous obtenons que
$\Sigma(\tggot/\ggot)=\{\pm e_i^*\pm f_j^*,\ (i,j)\in[p]\times [q]\}$.
\end{proof}

Nous voulons maintenant déterminer l'ensemble des éléments admissibles de $\agot$  par rapport à l'ensemble des racines $\Sigma(\tggot/\ggot)$, c'est-à-dire les vecteurs non nuls $\zeta\in \agot$ satisfaisant $\vect(\Sigma(\tggot/\ggot)\cap\zeta^\perp)=\zeta^{\perp}$. L'ensemble des éléments admissibles est invariant sous l'action du groupe de Weyl restreint $W_\agot=W_{\agot_1}\times W_{\agot_2} \simeq (\Sgot_p\rtimes \{\pm 1\}^p)\times (\Sgot_q\rtimes \{\pm 1\}^q)$.

Pour des entiers $r\leq q$ et $s\leq q$, nous définissons les vecteurs 
$\zeta_r^1=\varphi_1(-\sum_{i\leq r} e_i)\in \agot_1$, $\zeta_s^2=\varphi_2(-\sum_{j\leq s} f_j)\in \agot_2$.

\begin{lem}\label{lem:admissible-T-p-q} $\zeta\in\agot$ est un élément rationnel admissible si et seulement si il existe $(q,w)\in \Q^{>0}\times W_\agot$ et 
$(r,s)\in \{(0,1),(1,0)\}\cup [p]\times[q]$, tels que $\zeta= q\, w(\zeta_r^1,\zeta_s^2)$.
\end{lem}

\begin{proof}C'est la même preuve que celle du Lemme \ref{lem:admissible-S-p-q}.

\end{proof}

\subsection{Calcul de Schubert}

Fixons $(r,s)\in [p]\times[q]$. On associe aux vecteurs $\zeta_r^i\in\agot_i$, les sous-groupe paraboliques $P(\zeta_r^1)\subset \GL_{2p}(\C)$ et 
$P(\zeta_r^2)\subset \GL_{2q}(\C)$. Notons $\zeta_{r,s}=\iota(\zeta_r^1,\zeta_s^2)\in\tagot$, et $\tP(\zeta_{r,s})\subset \GL_{2n}(\C)$ 
le sous-groupe parabolique correspondant.

Le morphisme de groupe (\ref{eq:embedding-p-q}) induit le plongement $\iota_{r,s}: \Fcal_r\times\Fcal_s\to \tFcal_{r,s}$, 
où $\Fcal_r:= \GL_{2p}(\C)/P(\zeta_r^1)$, $\Fcal_s:= \GL_{2q}(\C)/P(\zeta_s^2)$ et $\tFcal_{r,s}:=\GL_{2n}(\C)/\tP(\zeta_{r,s})$. \`A chaque triplet 
$(w_1,w_2,\tw)\in W_{\agot_1}\times W_{\agot_2}\times \tW_{\tagot}$, on associe les variétés de Schubert
$$
\Xgot_{w_1}:=\overline{B_1w_1P(\zeta_r^1)/P(\zeta_r^1)}\subset \Fcal_r,\quad \Xgot_{w_2}:=\overline{B_2w_2P(\zeta_s^2)/P(\zeta_s^2)}\subset \Fcal_s, \quad\mathrm{et}\quad
\tXgot_{\tw}:=\overline{\tB\tw\tP(\zeta_{r,s})/\tP(\zeta_{r,s})}\subset \Fcal_{r,s}.
$$

Consid\'erons les \'el\'ements des groupes de Weyl restreints : 
$u_1=\Ad(\theta_1)(w_1)\in W_1$, $u_2=\Ad(\theta_2)(w_2)\in W_2$ et 
$u=\Ad(\theta)(\tw)\in W_0$. On leur associe les sous ensembles polaris\'es: 
\begin{itemize}
\item $I_\bullet\subset [p]$ de cardinal $r$, tel que $I_\bullet^p=u_1([r])$.
\item $J_\bullet\subset [q]$ de cardinal $s$, tel que $J_\bullet^q=u_2([s])$.
\item $L_\bullet\subset [n]$ de cardinal $r+s$,  o\`u $L_\bullet^n= u((A_{r,s})^{o})$.
\end{itemize}

Comme \`a la section précédente, un calcul direct montre que l'identit\'e $[\Xgot_{w_1}]\times [\Xgot_{w_2}]\cdot \iota_{r,s}^*\left([\tXgot_{\tw}]\right)=[pt]$ 
est équivalente \`a
\begin{equation}\label{eq:cohomo:T-p-q}
[\Xgot_{I_\bullet^q\subset (I_\bullet^q)^{o,c}}]\times [\Xgot_{J_\bullet^p\subset (J_\bullet^p)^{o,c}}]\cdot \iota^*
\left([\tXgot_{(L_\bullet^n)^{o}\subset (L_\bullet^n)^{c}}]\right)=[pt].
\end{equation}
 Ici, $\iota$ d\'esigne le morphisme $\F(r,2p-r; 2p)\times\F(s,2q-s; 2q)\to \F(r+s,2n-r-s; 2n)$.
 
Finalement, gr\^ace \`a la proposition \ref{prop:levi-mobile-application-2}, on sait que la relation \ref{eq:cohomo:S-p-q-bis} 
plus la condition de L\'evi-mobilit\'e sont satisfaits 
si et seulement si $\cc^{L_\bullet^n}_{I_\bullet^p,J_\bullet^q}=1=\cc^{\widetilde{L}^n_\bullet}_{\widetilde{I}^p_\bullet,\widetilde{J}^q_\bullet}=1$.

De plus, un calcul direct montre que l'inégalité $((z,-z),\tw\zeta_{r,s})\geq ((x,-x),w_0w_1\zeta_{r})+ ((y,-y),w_0w_2\zeta_{s})$ correspond à  
l'in\'egalit\'e $(\dagger)_{I_\bullet,J_\bullet,L_\bullet}$.

Lorsque l'on travaille avec les variables $(r,s)\in \{(0,1),(1,0)\}$, le m\^eme type de m\'ethode permet d'obtenir les inégalités (\ref{eq:T-p-q-basique}).

Cela termine la preuve du théorème \ref{theo:T-p-q}.

\chapter{Annexe}

\section*{Annexe A}

Dans cette monographie, $U_\C$ désigne un groupe de Lie réductif complexe connexe avec un sous-groupe compact maximal $U$. 
Soit $\sigma$ une involution complexe conjuguée sur $U_\C$ laissant le sous-groupe $U$ invariant. Soit $T$ un tore maximal de $U$ 
stable sous l'action de $\sigma$ : nous supposons que le sous-espace vectoriel $\tgot^{-\sigma}=\{X\in\tgot,\sigma(X)=-X\}$ a une dimension maximale. 

Prenons une représentation fidèle $\rho :U\to U(V)$ où $V$ est un espace vectoriel hermitien. 
Elle s'étend à un morphisme $\rho :U_\C\to \mathrm{GL}(V)$, de sorte que $\rho(U_\C)$ est un sous-groupe 
fermé stable sous l'involution de Cartan $\Theta:\mathrm{GL}(V)\to\mathrm{GL}(V)$. Nous considérons la application bilinéaire réelle 

\begin{equation}\label{eq:bilinear}
b:\ugot_\C\times \ugot_\C\longrightarrow \R
\end{equation}

définie par $b(X,Y)= \mathrm{Re}\Big(\tr(d\rho(X)d\rho(Y)) +\tr(d\rho(\sigma(X))d\rho(\sigma(Y)))\Big)$. Nous voyons alors que 

\begin{itemize}
\item $b(\sigma(X),\sigma(Y))=b(X,Y)$, $\forall X,Y\in\ugot_\C$, 
\item $b(uX,uY)=b(X,Y)$, $\forall u\in U_\C$, $\forall X, Y\in\ugot_\C$, 
\item $b(iX,iY)=-b(X,Y)$, si $X,Y\in\ugot_\C$,\item $b(X,Y)\in \Z$, si $X,Y\in\wedge:=\frac{1}{2\pi}\ker(\exp:\tgot\to T)$.
\end{itemize} 

Dans cet article, nous travaillons avec le produit scalaire $U$-invariant suivant sur $\ugot_\C$ :
$$
(X,Y)= -b(X,\Theta(Y)).
$$

Si $V\subset \ugot_\C$ est un sous-espace vectoriel réel, nous avons un isomorphisme $\xi\in V^* =\hom(V,\R)\to \xi^\flat \in V$ défini par la relation 
$$
\langle\xi, X\rangle=(\xi^\flat,X),\qquad \forall (\xi,X)\in V^*\times V.
$$

Soit $G$ le groupe réductif réel égal à la composante connexe de $(U_\C)^\sigma$ : son sous-groupe compact maximal est $K=G\cap U$. 
Au niveau des algèbres de Lie, nous avons $\ggot=\kgot\oplus \pgot$ où $\kgot=\ugot^{\sigma}$ et $\pgot=i\ugot^{-\sigma}$. 
L'application bilinéaire (\ref{eq:bilinear}) définit une application bilinéaire $G$-invariante $b:\ggot\times\ggot\to\R$ qui est définie positive sur $\pgot$ 
et définie négative sur $\kgot$.  

La commutativité du diagramme suivant est fréquemment utilisée dans le corps du présent article :
$$
\xymatrixcolsep{5pc}\xymatrix{(\ugot^*)^{-\sigma}\ar[d]^{\flat} \ar[r]^{j^*} & \pgot^*\ar[d]^{\flat} \\\ugot^{-\sigma}          & \pgot\ar[l]^{j} .}
$$

Nous considérons le système de racines restreint $\Sigma(\ggot)$ associé à un sous-espace abélien maximal $\agot\subset\pgot$ et 
un système de racines positives $\Sigma^+$ associé à un choix de chambre de Weyl $\agot^*_+$.  Soit ${\tsgot}$ une face de $\agot^*_+$ : l'ensemble des racines positives orthogonales à ${\tsgot}$ est noté $\Sigma^+_{\tsgot}$. Nous considérons $$\rho_{\tsgot}:=\sum_{\alpha\in\Sigma^+_{\tsgot}}\alpha\quad \in\agot^*$$et l'élément dual correspondant $(\rho_{\tsgot})^\flat=-\zeta_{\tsgot}\in \agot$ (voir \S \ref{sec:preuve-th-real-RP} où l'élément $\zeta_{\tsgot}$ est utilisé) .

\begin{lem}\label{lem:appendix}
Nous avons $\langle\alpha,\zeta_{\tsgot}\rangle < 0$ pour tout $\alpha\in\Sigma^+_{\tsgot}$.
\end{lem}

\begin{proof}
Soit $\alpha_o$ une racine simple de $\Sigma^+_{\tsgot}$. Soit $s_{\alpha_o}:\agot^*\to\agot^*$ la symétrie orthogonale associée. Pour tout $\beta\in \Sigma^+_{\tsgot}$, des calculs standard donnent que $s_{\alpha_o}(\beta)\in \Sigma^+_{\tsgot}$ si $\beta$ n'est pas proportionnel à $\alpha$, et 
$s_{\alpha_o}(\beta)=-\beta$ si $\beta$ est proportionnel à $\alpha$. Cela montre que $s_{\alpha_o}(\rho_{\tsgot})=\rho_{\tsgot}-2N\alpha_o$ 
pour un entier $N\geq 1$, donc $(\alpha_o,\rho_{\tsgot})>0$. Cela implique que $(\alpha,\rho_{\tsgot})>0$ pour tout $\alpha\in \Sigma^+_{\tsgot}$. 
\end{proof}

\section*{Annexe B}

Soit $M$ une variété complexe de dimension $n$ équipée d'une action holomorphe d'un groupe réductif $U_\C$ et d'une involution anti-holomorphe $\tau$.  
Un sous-ensemble $\Xcal\subset M$ est analytique si $\Xcal$ est fermé et si, pour tout $x\in \Xcal$, il existe un voisinage ouvert $V$ de $x$ dans $M$ et une collection finie de fonctions holomorphes $h_1,\cdots,h_s \in\Ocal(V)$ telles que $\Xcal\cap V=\{m\in V, h_1(m)=\cdots=h_s(m)=0\}$. Un sous-ensemble ouvert de Zariski $\Vcal^\C$ de $M$ est un sous-ensemble ouvert de la forme $\Vcal^\C =M\setminus \Xcal$ où $\Xcal$ est un sous-ensemble analytique.

\begin{lem}\label{lem:zariski} Soit $\Vcal^\C$ un sous-ensemble ouvert de Zariski non vide de $M$. 
\begin{enumerate}
\item Pour tout sous-groupe $H\subset U_\C$, $H\cdot \Vcal^\C$ est un sous-ensemble ouvert de Zariski invariant par $H$. 
\item $\Vcal^\C\cup \tau(\Vcal^\C)$ et $\Vcal^\C\cap \tau(\Vcal^\C)$ sont des sous-ensembles ouverts de Zariski invariants par $\tau$. 
\item $M^\tau\cap \Vcal^\C$ est un sous-ensemble ouvert dense de $M^\tau$. 
\end{enumerate}
\end{lem}

\begin{proof}
Soit $\Xcal$ le sous-ensemble analytique tel que $\Vcal^\C =M\setminus \Xcal$.Alors $H\cdot \Vcal^\C= M\setminus \Xcal'$, où 
$\Xcal'=\cap_{h\in H}h\cdot\Xcal$. Comme $U_\C$ agit de manière holomorphe sur $M$, chaque $h\cdot\Xcal$ est analytique. 
Comme une intersection arbitraire de sous-ensembles analytiques est analytique, nous pouvons conclure que $H\cdot \Vcal^\C$ 
est un  sous-ensemble ouvert de Zariski. Pour le deuxième point, il suffit d'expliquer pourquoi $\tau(\Xcal)$ est analytique. 
Soit $x\in\Xcal$ et $x'=\tau(x)\in \tau(\Xcal)$. Considérons un voisinage ouvert $V$ de $x$ dans $M$ et une collection finie de fonctions holomorphes 
$h_1,\cdots,h_s \in\Ocal(V)$ telles que $\Xcal\cap V=\{m\in V, h_1(m)=\cdots=h_s(m)=0\}$. Soit $V'=\tau(V)$, et définissons $h'_i\in \Ocal(V')$ par les relations $h'_i(m')=\overline{h_i(\tau(m'))}, \forall m'\in V'$. On voit alors que $\tau(\Xcal)\cap V'=\{m'\in V', h'_1(m')=\cdots=h'_s(m')=0\}$. Prouvons le dernier point. Considérons un élément $x\in M^\tau\cap \Xcal$, un voisinage ouvert $V$ de $x$ dans $M$ et une collection finie de fonctions holomorphes $h_1,\cdots,h_s \in\Ocal(V)$ telles que $\Xcal\cap V=\{m\in V, h_1(m)=\cdots=h_s(m)=0\}$. Nous devons prouver que $M^\tau\cap V\cap \Vcal^\C$ n'est pas vide. Supposons que le contraire soit vrai :  $M^\tau\cap V\subset \Xcal\cap V$,  en d'autres termes, toutes les fonctions holomorphes $h_i \in\Ocal(V)$ s'annulent sur $M^\tau\cap V$. Dans le lemme suivant, nous vérifions que cela implique que les fonctions holomorphes $h_i \in\Ocal(V)$ sont toutes identiques à $0$. Ce dernier point est en contradiction avec le fait que $\Vcal^\C$ est dense dans $M$, donc $\Vcal^\C\cap V\neq\emptyset$ : il existe $m\in V$ et $h_i \in\Ocal(V)$ tels que $h_i(m)\neq 0$. Notre dernière étape est réglée par le lemme suivant

\begin{lem} 
Soit $h\in\Ocal(V)$ tel que $h\vert_{M^\tau\cap V}=0$. Alors $h=0$.
\end{lem}

\begin{proof} 
Puisque $h\vert_{M^\tau\cap V}=0$, la différentielle $\T_m h:\T_m M\to\C$ s'annule sur $\T_m M^\tau$ pour tout $m\in M^\tau\cap V$. Mais $\T_m M =\T_m M^\tau\oplus \J(\T_m M^\tau), \forall m\in M^\tau\cap V$ et $h$ est holomorphe, donc on obtient $\T_m h=0, \forall m\in M^\tau\cap V$. Supposons que $V$ soit suffisamment petit pour que nous ayons des coordonnées holomorphes $z_1,\cdots, z_n\in\Ocal (V)$. Les arguments précédents montrent que si $h\vert_{M^\tau\cap V}=0$, alors $\frac{\partial h}{\partial z_k}\vert_{M^\tau\cap V}=0$ pour tout $k$. Nous en déduisons alors par récurrence que $\frac{\partial^\alpha h}{\partial z^\alpha}\vert_{M^\tau\cap V}=0$ pour tout $\alpha\in\N^n$. Soit $x\in M^\tau\cap V$. Comme $h\in\Ocal(V)$ est une fonction analytique, les identités $\frac{\partial^\alpha h}{\partial z^\alpha}(x)=0,\forall\alpha\in\N^n$ impliquent que $h=0$. 
\end{proof}
\end{proof}


\chapter{Liste de symboles}
\begin{tabular}{ll}
    $\sharp A$ & le cardinal de $A$\\
    $[\ell]$ & l'ensemble des entiers naturels $\{1,\dots,\ell\}$\\
    $I^o$ & l'ensemble $\{\ell+1-i, i\in I\}$ associ\'e à $I\subset [\ell]$\\
     $I^c$ & l'ensemble $[\ell]-I$ associ\'e à $I\subset [\ell]$\\
     $I^{o,c}$ & l'ensemble $[\ell]-I^o$ associ\'e à $I\subset [\ell]$\\
     $\mu(I)$ & la partition $(i_r-r\geq \cdots\geq i_1-1)$ associ\'ee à $I=\{i_1<\cdots<i_r\}\subset\N-\{0\}$\\
    $x^*$ & le vecteur $(x_n,\ldots, x_1)$ associ\'e à $x=(x_1,\ldots, x_n)$\\
    $x^\vee$ & le vecteur $(-x_n,\ldots, -x_1)$ associ\'e à $x=(x_1,\ldots, x_n)$\\
    $\widehat{x}\,^{p,q}$ & le vecteur $(x_1,\cdots,x_q,0,\cdots,0,-x_q,\cdots,-x_1)\in\R^{p+q}$ associé à $x=(x_1,\ldots, x_q)$\\
    $\R^\ell_{+}$ & l'ensemble des suites $x=(x_1\geq \cdots \geq x_\ell)$ de nombres réels \\
    $\R^\ell_{++}$ & l'ensemble des suites $x=(x_1\geq \cdots \geq x_\ell\geq 0)$ de nombres réels \\
    $|x|$ & le réel $\sum_{i=1}^n x_i$ associé à $x\in \R^n$\\
    $|x|_I$ & le réel $\sum_{i\in I} x_i$ associé à $x\in \R^n$ et à $I\subset [n]$\\
    $\Pcal(r,n)$ & l'ensemble des sous-ensembles de cardinal $r$ de $[n]$\\
    $\Bcal(r,p,q)$ & l'ensemble des sous-ensembles $I\subset [n]$ satisfaisant: $\sharp I=r$, 
    $I\cap I^o=\emptyset$ et $I\cap \{q+1,\ldots, p\}=\emptyset$\\
    $M_{a,b}(\C)$ & l'espace vectoriel des matrices complexes de taille $a\times b$\\
    $\herm(n)$ & l'espace vectoriel des matrices hermitiennes $n\times n$\\
    $\sym(n)$ & l'espace vectoriel  des matrices symétriques $n\times n$\\
    $\e(X)\in \R^n_{+}$ & le vecteur des valeurs propres d'une matrice hermitienne $X\in\herm(n)$\\
    $\s(Y)\in \R^q_{++}$ & le vecteur des valeurs singulières d'une matrice $Y\in M_{p,q}(\C)$, avec $q\leq p$\\
    $P(\gamma)$ & le sous-groupe parabolique associé à $\gamma$\\
    $ \tr\left(\gamma \circlearrowright E\right)$ & la trace de l'action linéaire de $\gamma$ sur un espace vectoriel $E$\\
    $B_n$ & le sous-groupe de Borel de $GL_n(\C)$ formé des matrices triangulaires supérieures\\
    $\Fcal_\gamma$ & une variété de drapeaux\\
    $\G(r;n)$ & la grassmanienne des sous-espaces vectoriels $F\subset \C^n$ de dimension $r$\\
    $\F(a,b; n)$ & la variété des couples de sous-espaces vectoriels $E\subset F\subset \C^{n}$ où $\dim E= a$ et $\dim F= b$\\
    $\Xgot_{I\subset J}$ & la variété de Schubert de $\F(a,b; n)$ associée à $I\subset J\subset [n]$, $\sharp I=a$, $\sharp J=b$\\
    $\cc_{I,J}^L$ & le coefficient de Littlewood-Richardson associé à $I,J,L\subset\N-\{0\}$\\
\end{tabular}


\end{document}